\documentclass[11pt]{amsart}
\usepackage{amssymb}
\usepackage{amscd}
\usepackage{amsmath}
\usepackage{amsmath}
\usepackage[all]{xy}
\usepackage{mathrsfs}
\usepackage{graphicx}
 \usepackage{color}

\theoremstyle{plain}
\newtheorem{theorem}{Theorem}[section]
\newtheorem{lemma}[theorem]{Lemma}
\newtheorem{corollary}[theorem]{Corollary}
\newtheorem{proposition}[theorem]{Proposition}
\newtheorem{definition}[theorem]{Definition}

\newtheorem{conjecture}[theorem]{Conjecture}
\newtheorem{prediction}[theorem]{Prediction}
\newtheorem{example}[theorem]{Example}
\newtheorem{examples}[theorem]{Examples}
\newtheorem{hyp}[theorem]{Hypothesis}

\newtheorem{remark}[theorem]{Remark}
\newtheorem{remarks}[theorem]{Remark}

\font\russ=wncyr10  1
\def\sha{\hbox{\russ\char88}}

\newcommand{\lra}{\longrightarrow}
\newcommand{\Zp}{{\mathbb{Z}_p}}
\newcommand{\Qp}{{\mathbb{Q}_p}}

\newcommand{\Ze}{{\mathbb{Z}}}

\newcommand{\Qu}{\mathbb{Q}}

\newcommand{\calP}{\mathcal{P}}

\newcommand{\A}{\mathfrak A}

\DeclareMathOperator{\Sel}{Sel}

\DeclareMathOperator{\Tr}{Tr}

\DeclareMathOperator{\Nrd}{Nrd}

\newcommand{\calO}{\mathcal{O}}

\newcommand{\eins}{\boldsymbol{1}}
\DeclareMathOperator{\Aut}{Aut}

\DeclareMathOperator{\Gal}{Gal}

\DeclareMathOperator{\Hom}{Hom}
\DeclareMathOperator{\End}{End}

\DeclareMathOperator{\im}{im}

\newcommand{\CC}{\mathbb{C}}
\newcommand{\bc}{\mathbb{C}}

\newcommand{\QQ}{\mathbb{Q}}

\newcommand{\RR}{\mathbb{R}}
\newcommand{\br}{\mathbb{R}}
\newcommand{\ZZ}{\mathbb{Z}}
\newcommand{\Z}{\mathbb{Z}}

\newcommand{\calL}{\mathcal{L}}

\newcommand{\frp}{\mathfrak{p}}

\newcommand{\cok}{\text{cok}}

\newcommand{\id}{\mathrm{id}}

\newcommand{\bz}{\mathbb{Z}}

\def\bigcapp{\raise1ex\hbox{\rotatebox{180}{$\biguplus$}}}
 \def\bigcappd{\raise1ex\hbox{\rotatebox{180}{$\displaystyle\biguplus$}}}

\begin{document}

\title[]{On refined conjectures \\
 of Birch and Swinnerton-Dyer type \\
 for Hasse-Weil-Artin $L$-series}


 \author{David Burns and Daniel Macias Castillo}

\begin{abstract} We consider refined conjectures of Birch and Swinnerton-Dyer type for the Hasse-Weil-Artin $L$-series of abelian varieties over general number fields. We shall, in particular, formulate several new such conjectures and establish their precise relation to previous conjectures, including to the relevant special case of the equivariant Tamagawa number conjecture. We also derive a wide range of concrete interpretations and explicit consequences of these conjectures that, in general, involve a thoroughgoing mixture of difficult archimedean considerations related to refinements of the conjecture of Deligne and Gross and delicate $p$-adic congruence relations that involve the bi-extension height pairing of Mazur and Tate and are related to key aspects of non-commutative Iwasawa theory. In important special cases we provide strong evidence, both theoretical and numerical, in support of the conjectures. We also point out an inconsistency in a conjecture of Bradshaw and Stein regarding Zhang's Theorem on Heegner points and suggest a possible correction. \end{abstract}

\address{King's College London,
Department of Mathematics,
London WC2R 2LS,
U.K.}
\email{david.burns@kcl.ac.uk}

\address{Departamento de Matem\'aticas, 
Universidad Aut\'onoma de Madrid, 28049 Madrid (Spain);
and Instituto de Ciencias Matem\'aticas, 28049 Madrid (Spain).}
\email{daniel.macias@uam.es}


\thanks{Mathematics Subject Classification: 11G40 (primary), 11G05, 11G10, 11G25, 11R34 (secondary).}


\maketitle

\tableofcontents

\section{Introduction}\label{intro}


\subsection{The aim of this article}

\subsubsection{}\label{history}

Let $A$ be an abelian variety defined over a number field $k$.

Then, by a celebrated theorem of Mordell and Weil, the abelian group that is formed by the set $A(k)$ of points of $A$ with coefficients in $k$ is finitely generated.

It is also conjectured that the Hasse-Weil $L$-series $L(A,z)$ for $A$ over $k$ has a meromorphic continuation to the entire complex plane and satisfies a functional equation with central point $z=1$ and, in addition, that the Tate-Shafarevich group $\sha(A_k)$ of $A$ over $k$ is finite.

Assuming these conjectures to be true, the Birch and Swinnerton-Dyer Conjecture concerns the leading coefficient $L^\ast(A,1)$ in the Taylor expansion of $L(A,z)$ at $z=1$. This conjecture was originally formulated for elliptic curves in a series of articles of Birch and Swinnerton-Dyer in the 1960's (see, for example, \cite{birch}) and then reinterpreted and extended to the setting of abelian varieties by Tate in \cite{tate} to give the following prediction.
\medskip

\noindent{}{\bf Conjecture} (Birch and Swinnerton-Dyer)
\begin{itemize}
\item[(i)] The order of vanishing of $L(A,z)$ at $z=1$ is equal to the rank of $A(k)$.
\item[(ii)] One has
\begin{equation}\label{bsd equality} L^\ast(A,1) = \frac{\Omega_A\cdot R_A\cdot {\rm Tam}(A)}{\sqrt{D_k}^{{\rm dim}(A)}\cdot |A(k)_{\rm tor}|\cdot |A^t(k)_{\rm tor}|}\cdot |\sha(A_k)|.\end{equation}
\end{itemize}
\medskip

Here $\Omega_A$ is the canonical period of $A$, $R_A$ the discriminant of the canonical N\'eron-Tate height pairing on $A$, ${\rm Tam}(A)$ the product over the (finitely many) places of $k$ at which $A$ has bad reduction of local `Tamagawa numbers', $D_k$ the absolute value of the discriminant of $k$,  $A(k)_{\rm tor}$ the torsion subgroup of $A(k)$ and $A^t$ the dual abelian variety of $A$.

It should be noted at the outset that, even if one assumes $L(A,z)$ can be meromorphically continued to $z=1$ and $\sha(A_k)$ is finite, so that both sides of (\ref{bsd equality}) make sense, the predicted equality is itself quite remarkable.

For instance, the $L$-series is defined via an Euler product over places of $k$ so that its leading coefficient at $z=1$ is intrinsically local and analytic in nature whilst the most important terms on the right hand side of (\ref{bsd equality}) are both global and algebraic in nature.

In addition, whilst isogenous abelian varieties give rise to the same $L$-series, the individual terms that occur in the `Euler characteristic' on the right hand side of  (\ref{bsd equality}) are not themselves isogeny invariant and it  requires a difficult theorem of Tate to show that the validity of (\ref{bsd equality}) is invariant under isogeny.

For these, and many other, reasons, the above conjecture, which we abbreviate to ${\rm BSD}(A_k)$, is regarded as one of the most important problems in arithmetic geometry today.

Nevertheless, there are various natural contexts in which it seems likely that ${\rm BSD}(A_k)$ does not encompass the full extent of the interplay between the analytic and algebraic invariants of $A$. Moreover, a good understanding of the finer connections that can arise could lead to much greater insight into concrete questions such as, for example, the growth of ranks of Mordell-Weil groups in extensions of number fields.

For instance, if $A$ has a large endomorphism ring, then it seems reasonable to expect there to be a version of ${\rm BSD}(A_k)$ that reflects the existence of such endomorphisms.

The earliest example of such a refinement is Gross's formulation in \cite{G-BSD} of an equivariant
Birch and Swinnerton-Dyer conjecture for elliptic curves $A$ with complex multiplication by the maximal order $\mathcal{O}$ of an imaginary quadratic field.

This conjecture incorporates natural refinements of ${\rm BSD}(A_k)$(i) and ${\rm BSD}(A_k)$(ii) and is supported by numerical evidence obtained by Gross and Buhler in \cite{GrossBuhler} and by theoretical evidence obtained by Rubin in \cite{rubin2}.

In a different direction one can study the leading coefficients of the Hasse-Weil-Artin $L$-series $L(A,\psi,z)$ that are obtained from $A$ and finite dimensional complex characters $\psi$ of the absolute Galois group $G_k$ of $k$.

In this setting, general considerations led Deligne and Gross to the expectation that for any finite dimensional character $\chi$ of $G_k$  over a number field $E$ the order of vanishing ${\rm ord}_{z=1}L(A,\sigma\circ\chi,z)$ at $z=1$ of $L(A,\sigma\circ\chi,z)$ should be independent of the choice of an embedding $\sigma: E \to \CC$. This prediction in turn led them naturally to the conjecture that for each complex character $\psi$ one should have
\begin{equation}\label{dg equality} {\rm ord}_{z=1}L(A,\psi,z) = {\rm dim}_\CC\bigl( \Hom_{\CC[\Gal(F/k)]}(V_\psi,\CC\otimes_\ZZ A^t(F))\bigr)\end{equation}
(cf. \cite[p.127]{rohrlich}). Here $F$ is any finite Galois extension of $k$ such that $\psi$ factors through the projection $G_k \to \Gal(F/k)$ and $V_\psi$ is any  $\CC[\Gal(F/k)]$-module of character $\psi$ (see also Conjecture \ref{conj:ebsd}(ii) below).

This prediction generalizes ${\rm BSD}(A_k)$(i) and has important, and explicit, consequences for ${\rm ord}_{z=1}L(A,\psi,z)$ (see, for example, the recent article \cite{bisatt} of Bisatt and Dokchitser).

In addition, for rational elliptic curves $A$ and characters $\psi$ for which $L(A,\psi,1)$ does not vanish, there is by now strong evidence for the conjecture of Deligne and Gross.

Such evidence has been obtained by Bertolini and Darmon \cite{BD} in the setting of ring-class characters of imaginary quadratic fields, by Kato \cite{kato} in the setting of linear characters of $\QQ$ (in this regard see also Rubin \cite[\S8]{rubin}), by Bertolini, Darmon and Rotger \cite{bdr} for odd, irreducible two-dimensional
Artin representations of $\QQ$ and by Darmon and Rotger \cite{dr} for certain self-dual Artin representations of $\QQ$ of dimension at most four.
 (We also recall in this context that, in the setting of \cite{bdr}, recent work of Kings, Loeffler and Zerbes \cite{klz} proves the finiteness of components of the $p$-primary part of the Tate-Shafarevich group of $A$ over $F$ for a large set of primes $p$.)

Write $\mathcal{O}_\psi$ for the ring of integers of the number field generated by the values of $\psi$. Then, as a refinement of the conjectural equality (\ref{dg equality}), and a natural analogue of ${\rm BSD}(A_k)$(ii) relative to $\psi$, it would be of interest to understand a precise conjectural formula in terms of suitable `$\psi$-components' of the standard algebraic invariants of $A$ for the fractional $\mathcal{O}_\psi$-ideal that is generated by the product of the leading coefficient of $L(A,\psi,z)$ at $z=1$ and an appropriate combination of `$\psi$-isotypic' periods and regulators.

Such a formula might also reasonably be expected to lead to concrete predictions concerning the behaviour of natural arithmetic invariants attached to the abelian variety.

For example, Dokchitser, Evans and Wiersema \cite{vdrehw} have recently shown that inexplicit versions of such a formula  lead, under suitable hypotheses, to predictions concerning the non-triviality of Tate-Shafarevich groups and the existence of points of infinite order on $A$ over extension fields of $k$.

However, so far, the formulation of an explicit such conjecture has been straightforward only if one avoids the $p$-primary support of such fractional ideals for primes $p$ that divide the degree of the extension of $k$ that corresponds to the kernel of $\psi$.

In addition, such a conjectural formula would not itself take account of any connections that might exist between the leading coefficients of $L(A,\psi,z)$ for characters $\psi$ that are not in the same orbit under the action of $G_\QQ$.

In this direction, Mazur and Tate \cite{mt} have in the special case that $k =\QQ$, $A$ is an elliptic curve and $\psi$ is linear predicted an explicit family of such congruence relations that refine ${\rm BSD}(A_k)$(ii). These congruences rely heavily on an explicit formula in terms of modular symbols for the values $L(A,\psi,1)$ for certain classes of tamely ramified Dirichlet characters $\psi$ that Mazur had obtained in \cite{mazur79}. They are expressed in terms of the discriminants of integral group-ring valued pairings constructed by using the geometrical theory of bi-extensions and are closely linked to earlier work of Mazur, Tate and Teitelbaum in \cite{mtt} regarding the formulation of $p$-adic analogues of ${\rm BSD}(A_k)$(ii).

The conjecture of Mazur and Tate has in turn motivated much subsequent work and the formulation of several new conjectures involving the values $L(A,\psi,1)$. Such conjectures include the congruence relations that are formulated by Bertolini and Darmon in \cite{bert} and \cite{bert2} and involve a natural notion of `derived height pairings' and links to the Galois structure of Selmer modules that are predicted by Kurihara in \cite{kuri}.

It has, however, proved to be much more difficult to formulate explicit refinements of ${\rm BSD}(A_k)$(ii) that involve congruence relations between the values of derivatives of Hasse-Weil-Artin $L$-series.

In this direction, Darmon \cite{darmon} uses the theory of Heegner points to formulate an analogue of the Mazur-Tate congruence conjecture for the first derivatives of Hasse-Weil-Artin $L$-series that arise from rational elliptic curves and ring class characters of imaginary quadratic fields. However, aside from this example, the only other such explicit study  we are aware of is due to Kisilevsky and Fearnley who in \cite{kisilevsky} and \cite{kisilevsky2} formulated, and studied numerically, conjectures for the `algebraic parts' of the leading coefficients of Hasse-Weil-Artin $L$-series that arise from rational elliptic curves and certain families of Dirichlet characters.

\subsubsection{}In a more general setting, the formulation by Bloch and Kato \cite{bk} of the `Tamagawa number conjecture' for the motive $h^1(A_k)(1)$ offers a different approach to the formulation of ${\rm BSD}(A_k)$.

In particular, the subsequent re-working of this conjecture by Fontaine and Perrin-Riou in \cite{fpr}, and its `equivariant' extension to motives with coefficients, as described by Flach and the first author in \cite{bufl99}, in principle provides a systematic means of studying refined versions of ${\rm BSD}(A_k)$.

In this setting it is known, for example, that the conjectures of Gross in \cite{G-BSD} are equivalent to the equivariant Tamagawa number conjecture for the motive $h^1(A_k)(1)$ with respect to the coefficient ring $\mathcal{O}$ (cf. \cite[\S4.3, Rem. 10]{bufl99}).

To study Hasse-Weil-Artin $L$-series it is convenient to fix a finite Galois extension $F$ of $k$ of group $G$.
Then the equivariant Tamagawa number conjecture for $h^1(A_{F})(1)$ with respect to the integral group ring $\ZZ[G]$ is formulated as an equality of the form
\begin{equation}\label{etnc eq} \delta(L^\ast(A_{F/k},1)) = \chi(h^1(A_{F})(1),\ZZ[G]).\end{equation}

\noindent{}Here $\delta$ is a canonical homomorphism from the unit group $\zeta(\br [G])^\times$ of the centre of $\RR[G]$ to the relative algebraic $K$-group of the ring extension $\bz [G] \subseteq \br
[G]$ and $L^\ast(A_{F/k},1)$ is an element of $\zeta(\RR[G])^\times$ that is defined using the leading coefficients $L^\ast(A,\psi,1)$ for each irreducible complex character $\psi$ of $G$. Also, $\chi(h^1(A_{F})(1),\ZZ[G])$ is an adelic Euler characteristic that is constructed by
combining virtual objects (in the sense of Deligne) for each prime $p$ of the compactly supported \'etale cohomology of the $p$-adic Tate modules of $A$ together with the N\'eron-Tate height pairing and period isomorphisms for $A$ over $F$, as well as an analysis of the finite support cohomology groups introduced by Bloch and Kato.

The equality (\ref{etnc eq}) is constitutes a strong and simultaneous refinement of the conjectures ${\rm BSD}(A_L)$ as $L$ ranges over the intermediate fields of $F/k$. However, the rather technical, and inexplicit, nature of this equality means that it has proved to be very difficult to interpret in a concrete way, let alone to verify either theoretically or numerically.

For example, if $A$ has strictly positive rank over $F$ it has still only been verified numerically in a small number of cases by Navilarekallu in \cite{tejaswi}, by Bley in \cite{Bley1} and \cite{Bley2}, by Bley and the second author in \cite{bleymc} and by Wuthrich and the present authors in \cite{bmw}.

In addition, the only theoretical evidence for the conjecture in this setting is its verification for dihedral families of the form $F/\QQ$ where $F$ is an unramified abelian extension of an imaginary quadratic field (this is the main result of \cite{bmw} and relies on the theorem of Gross and Zagier). In particular, the restriction on ramification that the latter result imposes on $F/k$ means that many of the more subtle aspects of the conjecture are avoided.

To proceed we note that the conjectural equality (\ref{etnc eq}) naturally decomposes into `components', one for each rational prime $p$, in a way that will be made precise in Appendix \ref{consistency section}, and that each such $p$-component (which for convenience we refer to as `(\ref{etnc eq})$_p$' in the remainder of this introduction) is itself of some interest.

For example, if $A$ has good ordinary reduction at $p$, then the compatibility result proved by Venjakob and the first named author in~\cite[Th. 8.4]{BV2} shows, modulo the assumed non-degeneracy of classical $p$-adic height pairings, that the equality (\ref{etnc eq})$_p$ is a consequence of the main conjecture of
non-commutative Iwasawa theory for $A$, as formulated by Coates et al in \cite{cfksv} with respect to any compact $p$-adic Lie extension of $k$ that contains the cyclotomic $\ZZ_p$-extension of $F$.

This means that the study of (\ref{etnc eq}), and of its more explicit consequences, is relevant to attempts to properly understand the content of the main conjecture of non-commutative Iwasawa theory. It also shows that the $p$-adic congruence relations that are proved numerically by Dokchitser and Dokchitser in \cite{dokchitsers} are related to the equality (\ref{etnc eq}).

To study congruences in a more general setting we fix an embedding of $\RR$ into the completion $\CC_p$ of the algebraic closure of $\QQ_p$. Then the long exact sequence of relative $K$-theory implies that the equality (\ref{etnc eq})$_p$ determines the image of $L^\ast(A_{F/k},1)$ in $\zeta(\CC_p[G])^\times$ modulo the image under the natural reduced norm map of the Whitehead group $K_1(\ZZ_p[G])$ of $\ZZ_p[G]$.

In view of the explicit description of the latter image that is obtained by Kakde in \cite{kakde} or, equivalently, by the methods of Ritter and Weiss in \cite{rw}, this means that (\ref{etnc eq}) is essentially equivalent to an adelic family of (albeit inexplicit) congruence relations between the leading coefficients $L^\ast(A,\psi, 1)$, suitably normalised by a product of explicit equivariant regulators and periods, as $\psi$ varies over the set of irreducible complex characters of $G$.

This is also the reason why the study of congruence relations between suitably normalised derivatives of Hasse-Weil-Artin $L$-series should be related to the construction of `$p$-adic $L$-functions' in the setting of non-commutative Iwasawa theory.

\subsubsection{}The main aim of the present article is then to develop general techniques that will allow one to understand the above congruence relations in a more explicit way, and in a much wider setting, than has previously been possible.

In this way we are led to the formulation (in Conjecture \ref{conj:ebsd}) of a seemingly definitive refinement of the Birch and Swinnerton-Dyer formula (\ref{bsd equality}) in the setting of Hasse-Weil-Artin $L$-series. We then derive a range of concrete consequences of this conjecture that are amenable to explicit investigation, either theoretically or numerically, in cases that (for the first time) involve a thoroughgoing mixture of difficult archimedean considerations that are related to refinements of the conjectural equality (\ref{dg equality}) of Deligne and Gross, and of delicate $p$-adic congruence relations that are related to aspects of non-commutative Iwasawa theory.

In particular, we shall show that this family of predictions both refines and extends the explicit refinements of (\ref{dg equality}) that were recalled in \S\ref{history}. It also gives insight into the more subtle aspects of the conjectural equality (\ref{etnc eq}), and hence (via the results of \cite{BV2}) of the main conjecture of non-commutative Iwasawa theory, that go well beyond the sort of concrete congruence conjectures that have been considered previously in connection to the central conjectures of either \cite{bufl99} or \cite{cfksv}.

We also believe that some general results obtained here can contribute towards establishing a proper framework for the subsequent investigation of these important questions.

%
%

\subsection{The main contents}

\subsubsection{}







As a key part of our approach, we shall first associate two natural notions of Selmer complex to the $p$-adic Tate module of an abelian variety.

 The `classical Selmer complex' that we define in \S\ref{selmer section} is closely related to the `finite support cohomology' that was introduced by Bloch and Kato in \cite{bk} and, as a result, its cohomology can be explicitly described in terms of Mordell-Weil groups and Selmer groups.

 Nevertheless, this complex is not well-suited to certain $K$-theoretical calculations since it is not always `perfect' over the relevant $p$-adic group ring.

 For this reason we shall in \S\ref{selmer section} also associate a notion of `Nekov\'a\v r-Selmer complex' to certain choices of $p$-adic submodules of the groups of semi-local points. This construction is motivated by the general approach of Nekov\'a\v r in \cite{nek} and gives a complex that is always perfect and has cohomology that can be described in terms of the Selmer modules studied by Mazur and Rubin in \cite{MRkoly}. Such Nekov\'a\v r-Selmer complexes will then play an important role in several subsequent $K$-theoretical computations.

In \S\ref{selmer section} we also explain how a suitably compatible family over all primes $p$ of $p$-adic modules of semi-local points, or a `perfect Selmer structure' for $A$ and $F/k$ as we shall refer to it, gives rise to a canonical perfect complex of $G$-modules.

We shall then show that such structures naturally arise from a choice of global differentials and compute the cohomology groups of the associated Selmer complexes.

In \S\ref{ref bsd section} we formulate the Birch and Swinnerton-Dyer Conjecture for the variety $A$ and Galois extension $F$ of $k$, or `${\rm BSD}(A_{F/k})$' as we shall abbreviate it.

Under the assumed validity of an appropriate case of the Generalized Riemann Hypothesis, we shall first associate a $K_1$-valued leading coefficient element to the data $A$ and $F/k$.

After fixing a suitable choice of global differentials we can also associate a $K_1$-valued `period' element to $A$ and $F/k$.

The central conjecture of this article then asserts that the image under the natural connecting homomorphism of the quotient of these $K_1$-valued invariants is equal to the Euler characteristic in a relative $K$-group of a pair comprising a Nekov\'a\v r-Selmer complex constructed from the given set of differentials and the classical N\'eron-Tate height pairing for $A$ over $F$.

This conjectural equality also involves a small number of `Fontaine-Messing' correction terms that we use to
compensate for the choice of a finite set of places of $k$ that is necessary to state ${\rm BSD}(A_{F/k})$.


It can be directly shown that this conjecture recovers ${\rm BSD}(A_{k})$ in the case that $F=k$, is consistent in several key respects and has good functorial properties under change of Galois extension $F/k$. The conjecture can also be interpreted as a natural analogue of the `refined Birch and Swinnerton-Dyer Conjecture' for abelian varieties over global function fields that was recently formulated, and in some important cases proved, by Kakde, Kim and the first author in \cite{bkk}.

In \S\ref{k theory period sect} and \S\ref{local points section} we prove several technical results that will subsequently help us to derive explicit consequences from the assumed validity of ${\rm BSD}(A_{F/k})$.

In \S\ref{k theory period sect} these results include establishing the precise link between $K_1$-valued periods, classical periods, Galois resolvents and suitably modified Galois-Gauss sums.


In \S\ref{local points section} we shall prepare for subsequent $K$-theoretical computations with classical Selmer complexes by studying the cohomological-triviality of local points on ordinary abelian varieties. In particular, we use these results to introduce a $K$-theoretical invariant, the `\'etale discrepancy class', of the twist matrix of such a variety that, in essence, measures the difference between the image of the formal logarithm  and an Euler characteristic arising from the \'etale cohomology of $\mathbb{G}_m$. We give a partial computation of these invariants and also explain how, in the case of elliptic curves, the (assumed) compatibility under unramified twist of suitable cases of the local epsilon constant conjecture (as formulated in its most general form by Fukaya and Kato in \cite{fukaya-kato}) leads to an explicit description of the \'etale discrepancy class.

In \S\ref{tmc} we impose several mild hypotheses on both the reduction types of $A$ and the ramification invariants of $F/k$ that together ensure that the classical Selmer complex defined in \S\ref{selmer section} is perfect over the relevant $p$-adic group ring.

Working under these hypotheses, we combine the results of \S\ref{k theory period sect} and \S\ref{local points section} together with a strengthening of the main computations of Wuthrich and the present authors in \cite{bmw} to derive a more explicit interpretation of ${\rm BSD}(A_{F/k})$. These results are in many respects the technical heart of this article and rely heavily on the subtle, and still for the most part conjectural, arithmetic properties of wildly ramified Galois-Gauss sums. 

The $K$-theoretical computations in \S\ref{tmc} also constitute a natural equivariant refinement and generalisation of several earlier computations in this area including those that are made by Venjakob in~\cite[\S3.1]{venjakob}, by the first author in \cite{ltav}, by Bley in~\cite{Bley1}, by Kings in~\cite[Lecture 3]{kings} and by the second author in \cite{dmc}.

In \S\ref{ecgs} we discuss concrete consequences of ${\rm BSD}(A_{F/k})$ concerning both the explicit Galois structure of Selmer complexes and modules and the formulation of precise refinements of the Deligne-Gross Conjecture. In particular, in this section we address a problem explicitly raised by Dokchitser, Evans and Wiersema in \cite{vdrehw} (see, in particular, Remark \ref{evans}).

In \S\ref{congruence sec} and \S\ref{mrsconjecturesection} we then specialise to consider abelian extensions $F/k$ and combine our approach with general techniques recently developed by Sano, Tsoi and the first author in \cite{bst} in order to derive from ${\rm BSD}(A_{F/k})$ several explicit congruence relations between the suitably normalized derivatives of Hasse-Weil-Artin $L$-series.

In \S\ref{comparison section} we then prove that the pairing constructed by Mazur and Tate in \cite{mt} using the theory of bi-extensions coincides with the inverse of a canonical `Nekov\'a\v r height pairing' that we define by using Bockstein homomorphisms arising naturally from Galois descent considerations. This comparison result relies, in part, on earlier results of Bertolini and Darmon \cite{bert2} and of Tan \cite{kst} and is, we believe, of some independent interest.

The relations discussed in \S\ref{congruence sec} and \S\ref{mrsconjecturesection} often take a very explicit form (see, for example, the discussion in \S\ref{explicit examples intro} below) and, when combined with the results of \S\ref{comparison section}, can be seen to extend and refine earlier conjectures of Mazur and Tate \cite{mt} and Darmon \cite{darmon0} amongst others.

This approach also shows that for certain cyclic and dihedral extensions $F/k$ the key formula that is predicted by ${\rm BSD}(A_{F/k})$  is equivalent to the validity of a family of explicit congruence relations that simultaneously involve both the N\'eron-Tate and Mazur-Tate height pairings.

In this way we shall for the first time render refined versions of the Birch and Swinnerton-Dyer Conjecture accessible to numerical verification in cases in which they involve an intricate mixture of both archimedean phenomenon and delicate $p$-adic congruences.

In particular, in \S\ref{mrsconjecturesection} we give details of several such numerical verifications of the `$p$-component' of ${\rm BSD}(A_{F/k})$ for primes $p$ that divide the degree of $F/k$ that Werner Bley has been able to perform by using this approach (see, in particular, Remark \ref{bleyexamples rem} and Examples \ref{bleyexamples}).

In \S\ref{mod sect} and \S\ref{HHP} we then specialise to consider applications of our general approach in two classical settings.

Firstly, in \S\ref{mod sect} we consider rational elliptic curves over fields that are both abelian and tamely ramified over $\QQ$. In this case
we can use the theory of modular symbols to give an explicit reinterpretation of ${\rm BSD}(A_{F/\QQ})$ and thereby describe precise conditions under which the conjecture is valid. As a concrete application of this result we then use it to deduce from Kato's theorem \cite{kato} that for every natural number $n$ there are infinitely many primes $p$ and, for each such $p$, infinitely many abelian extensions $F/\QQ$ for which the $p$-component of ${\rm BSD}(A_{F/\QQ})$ is valid whilst the degree and discriminant of $F/\QQ$ are each divisible by at least $n$ distinct primes and the Sylow $p$-subgroup of $\Gal(F/\QQ)$ has exponent at least $p^n$ and rank at least $n$. This result strengthens the main result of Bley in \cite{Bley3}.

Then in \S\ref{HHP} we consider abelian extensions of imaginary quadratic fields and elliptic curves that satisfy the Heegner hypothesis. The main result of this section is a significant extension of the main result of Wuthrich and the present authors in \cite{bmw} and relies on Zhang's generalization of the theorem of Gross and Zagier relating first derivatives of Hasse-Weil-Artin $L$-series to the heights of Heegner points. In this section we shall also point out an inconsistency in the formulation of a conjecture of Bradshaw and Stein in \cite{BS} regarding Zhang's formula and offer a possible correction.

The article also contains three appendices. In Appendix \ref{consistency section} we use techniques developed by Wuthrich and the present authors in \cite{bmw} to explain the precise link between our central conjecture ${\rm BSD}(A_{F/k})$ and the conjectural equality (\ref{etnc eq}).

This technical result may perhaps be of interest in its own right but also allows us to deduce from the general theory of equivariant Tamagawa numbers that our formulation of ${\rm BSD}(A_{F/k})$ is consistent in several key respects.

Then, in Appendix \ref{ptduality} we make explicit certain standard constructions relating to Poitou-Tate duality.
The results of Appendix \ref{ptduality} are for the most part routine but nevertheless play an important role in the arguments that we use to compare height pairings in \S\ref{comparison section}.

Finally, in Appendix \ref{bocksection} we describe the general construction of algebraic height pairings coming from Bockstein homomorphisms that leads to the definition of the Nekov\'a\v r height pairing in \S\ref{comparison section}.

\subsubsection{}\label{explicit examples intro}To end the introduction we shall give some concrete examples of the sort of congruence predictions that result from our approach (all taken from the more general material given in \S\ref{congruence sec} and \S\ref{mrsconjecturesection}). 

To do this we fix a finite abelian extension of number fields $F/k$ of group $G$ and an elliptic curve $A$ over $k$. We also fix an odd prime $p$ that does not divide the order of the torsion subgroup of $A(F)$ and an isomorphism of fields $\CC\cong \CC_p$ (that we do not explicitly indicate in the sequel). We write $\widehat{G}$ for the set of irreducible complex characters of $G$.

The first prediction concerns the values at $z=1$ of Hasse-Weil-Artin $L$-series. To state it we set $F_p := \QQ_p\otimes_\QQ F$ and write ${\rm log}_{A,p}$ for the formal group logarithm of $A$ over $\QQ_p\otimes_\QQ k$ and $\Sigma(k)$ for the set of embeddings $k \to \CC$. For each subset $x_\bullet = \{x_{\sigma}: \sigma \in \Sigma(k)\}$ of $A(F_p)$ and each character $\psi$ in $\widehat{G}$ we then define a `$p$-adic logarithmic resolvent' by setting

\begin{equation}\label{log resol abelian} \mathcal{LR}_\psi(x_\bullet) := {\rm det}\left(\bigl(\sum_{g \in G} \hat \sigma(g^{-1}({\rm log}_{A,p}(x_{\sigma'})))\cdot \psi(g)\bigr)_{\sigma,\sigma' \in \Sigma(k)}\right),\end{equation}
where we fix an ordering of $\Sigma(k)$ and an extension $\hat \sigma$ to $F$ of each $\sigma$ in $\Sigma(k)$.

Now if $S$ is any finite set of places of $k$ that contains all archimedean places, all that ramify in $F$, all at which $A$ has bad reduction and all above $p$, and $L_{S}(A,\check\psi,z)$ is the $S$-truncated $L$-series attached to $A$ and the contragredient $\check\psi$ of $\psi$, then our methods predict that for any $x_\bullet$ the sum

\begin{equation}\label{first predict} \sum_{\psi \in \widehat{G}}\frac{L_{S}(A,\check\psi,1)\cdot \mathcal{LR}_\psi(x_\bullet)}{\Omega^\psi_A\cdot w_\psi}\cdot e_\psi \end{equation}
belongs to $\ZZ_p[G]$ and annihilates the $p$-primary part $\sha(A_{F})[p^\infty]$ of the Tate-Shafarevich group of $A$ over $F$. Here $e_\psi$ denotes the idempotent $|G|^{-1}\sum_{g \in G}\check\psi(g)g$ of $\CC[G]$ and the periods $\Omega^\psi_A$ and Artin root numbers $w_\psi$ are as explicitly defined in \S\ref{k theory period sect2}.

In particular, if one finds  for some choice of $x_\bullet$ that the sum in (\ref{first predict}) is a unit of $\ZZ_p[G]$, then this implies that $\sha(A_F)[p^\infty]$ should be trivial.

More generally, the fact that each sum in (\ref{first predict}) should belong to $\ZZ_p[G]$ implies that for every $g$ in $G$, and every set of local points $x_\bullet$, there should be a congruence
\[ \sum_{\psi \in \widehat{G}}\psi(g)\frac{L_{S}(A,\check\psi,1)\cdot \mathcal{LR}_\psi(x_\bullet)}{\Omega_A^\psi\cdot w_\psi}
 \equiv 0 \,\,\,({\rm mod}\,\, |G|\cdot \ZZ_p).\]

In concrete examples these congruences are strong restrictions on the values $L_{S}(A,\check\psi,1)$ that can be investigated  numerically but cannot be deduced by solely considering Birch and Swinnerton-Dyer type formulas for individual Hasse-Weil-Artin $L$-series.


In general, our analysis leads to a range of congruence predictions that are both finer than the above and also involve the values at $z=1$ of higher derivatives of Hasse-Weil-Artin $L$-series, suitably normalised by a product of explicit regulators and periods.

To give an example of this sort of prediction, we shall focus on the simple case that $F/k$ is cyclic of degree $p$ (although entirely similar predictions can be made in the setting of cyclic extensions of arbitrary $p$-power order and also for certain dihedral families of extensions).

In this case, under certain natural, and very mild, hypotheses on $A$ and $F/k$ relative to $p$ there exist non-negative integers $m_0$ and $m_1$ with the property that the pro-$p$ completion $A(F)_p$ of $A(F)$ is isomorphic as a $\ZZ_p[G]$-module to a direct sum of $m_0$ copies of $\ZZ_p$ and $m_1$ copies of $\ZZ_p[G]$.

In particular, if we further assume $m_0=2$ and $m_1= 1$, then we may fix points $P_{0}^1$ and $P_{0}^2$ in $A(k)$ and $P_1$ in $A(F)$ such that $A(F)_p$ is the direct sum of the $\ZZ_p[G]$-modules generated by $P_{0}^1,$ $P_{0}^2$ and $P_1.$ (In Example \ref{wuthrich example} the reader will find explicit examples of such pairs $A$ and $F/k$ for the prime $p=3$.)


Then, writing $L^{(1)}_{S}(A,\check\psi,1)$ for the value at $z=1$ of the first derivative of $L_{S}(A,\check\psi,z)$, our methods predict that, under mild additional hypotheses, there should exist an element $x$ of $\ZZ_p[G]$ that annihilates $\sha(A_{F})[p^\infty]$ and is such that
\begin{equation}\label{second predict} \sum_{\psi \in \widehat{G}}\frac{L^{(1)}_{S}(A,\check\psi,1)\cdot\tau^*(\QQ,\psi)}{\Omega^\psi_A\cdot i^{r_2}}\cdot e_\psi  = x\cdot \sum_{g\in G}\langle g(P_1),P_1\rangle_{A_F}\cdot g^{-1},\end{equation}
where $\langle -,-\rangle_{A_F}$ denotes the N\'eron-Tate height pairing of $A$ relative to $F$, $\tau^*(\QQ,\psi)$ is the (modified, global) Galois-Gauss sum of the character of $G_\QQ$ that is obtained by inducing $\psi$ and $r_2$ is the number of complex places of $k$.

To be more explicit, we write $R_A$ for the determinant of the N\'eron-Tate regulator matrix of $A$ over $k$ with respect to the ordered $\QQ$-basis $\{P_{0}^1,P_{0}^2,\sum_{g\in G}g(P_1)\}$ of $\QQ\cdot A(k)$ and for each non-trivial $\psi$ in $\widehat{G}$ we define a non-zero complex number by setting
\[ h^\psi(P_1):=\sum_{g\in G}\langle g(P_1),P_1\rangle_{A_F}\cdot\psi(g)^{-1}.\]

We finally write $S_{\rm r}$ for the set of places of $k$ that ramify in $F$, $d_k$ for the discriminant of $k$ and $I_p(G)$ for the augmentation ideal of $\ZZ_p[G]$. Then, under mild hypotheses, our methods predict that there should be containments
\[
\sum_{\psi\neq {\bf 1}_G}\frac{L_{S_{\rm r}}^{(1)}(A,\check\psi,1)\cdot\tau^*(\QQ,\psi)}{\Omega^\psi_{A}\cdot i^{r_2} \cdot h^\psi(P_1)}\cdot e_\psi  \in I_p(G)^2\,\,\text{ and }\,\, \frac{L_{S_{\rm r}}^{(3)}(A,1)\sqrt{|d_k|}}{\Omega_{A}\cdot R_{A}}\in \ZZ_p,\]
and a congruence modulo $I_p(G)^3$ of the form
\begin{equation}\label{examplecongruent}
\sum_{\psi\neq {\bf 1}_G}\frac{L_{S_{\rm r}}^{(1)}(A,\check\psi,1)\cdot\tau^*(\QQ,\psi)}{ \Omega^\psi_{A}\cdot i^{r_2} \cdot h^\psi(P_1)}\cdot\! e_\psi
 \equiv \!\frac{L_{S_{\rm r}}^{(3)}(A,1)\sqrt{|d_k|}}{\Omega_{A}\cdot R_{A}}\cdot\det\left(\begin{array}{cc}
\langle P_{0}^1,P_{0}^1\rangle^{\rm MT} & \langle P_{0}^1,P_{0}^2\rangle^{\rm MT}
\\
\langle P_{0}^2,P_{0}^1\rangle^{\rm MT} & \langle P_{0}^2,P_{0}^2\rangle^{\rm MT}
\end{array}\right)\!.
\end{equation}
Here $L_{S_{\rm r}}^{(3)}(A,1)$ denotes the value at $z=1$ of the third derivative of $L_{S_{\rm r}}(A,z)$ and
\[ \langle\,,\rangle^{\rm MT}:A(k)\times A(k)\to I_p(G)/I_p(G)^2\]
is the canonical pairing that Mazur and Tate define in \cite{mt0} by using the geometrical theory of bi-extensions.

Further, if $\sha(A_F)[p^\infty]$ is trivial, then the $p$-component of ${\rm BSD}(A_{F/k})$ is valid if and only if (\ref{examplecongruent}) holds and, in addition, the $p$-component of the Birch and Swinnerton-Dyer Conjecture is true for $A$ over both $k$ and $F$.

%
%
%
%

We remark that even in the simplest possible case that $k = \QQ$ and $p=3$, these predictions strongly refine those made by Kisilevsky and Fearnley in \cite{kisilevsky} and cannot be deduced by simply considering leading term formulas for individual Hasse-Weil-Artin $L$-series.


\subsection{General notation} For the reader's convenience we give details here of some of the general notation and terminology that will be used throughout the article.

\subsubsection{}We write $|X|$ for the cardinality of a finite set $X$.

For an abelian group $M$ we write $M_{\rm tor}$ for its torsion subgroup and $M_{\rm tf}$ for the quotient of $M$ by $M_{\rm tor}$. 

For a prime $p$ and natural number $n$ we write $M[p^n]$ for the subgroup $\{m \in M: p^nm =0\}$ of the Sylow $p$-subgroup $M[p^{\infty}]$ of $M_{\rm tor}$.

We set $M_p := \ZZ_p\otimes_\ZZ M$, write $M^\wedge_p$ for the  pro-$p$-completion $\varprojlim_n M/p^n M$ of $M$ and denote the Pontryagin dual $\Hom(M,\QQ/\ZZ)$ of $M$ by $M^\vee$. 

If $M$ is finite of exponent dividing $p^m$, then $M^\vee$ identifies with $\Hom_{\ZZ_p}(M,\QQ_p/\ZZ_p)$ and we shall (without explicit comment) use the canonical identification $\QQ_p/\ZZ_p=\varinjlim_n \ZZ/p^n\ZZ$ to identify elements of $M^\vee$ with their canonical image in the linear dual $\Hom_{\ZZ/p^m\ZZ}(M,\ZZ/p^m\ZZ)$.

If $M$ is finitely generated, then for a field extension $E$ of $\QQ$ we shall often abbreviate $E\otimes_\ZZ M$ to $E\cdot M$.


If $M$ is a $\Gamma$-module for some group $\Gamma$, then we always endow $M^\vee$ with the natural contragredient action of $\Gamma$.


We recall that if $\Gamma$ is finite, then a $\Gamma$-module $M$ is said to be `cohomologically-trivial' if for all subgroups $\Delta$ of $\Gamma$ and all integers $i$ the Tate cohomology group $\hat H^i(\Delta,M)$ vanishes.

\subsubsection{}For any ring $R$ we write $R^\times$ for its multiplicative group and $\zeta(R)$ for its centre.

Unless otherwise specified we regard all $R$-modules as left $R$-modules. We write ${\rm Mod}(R)$ for the abelian category of $R$-modules and ${\rm Mod}^{\rm fin}(R)$ for the abelian subcategory of ${\rm Mod}(R)$ comprising all $R$-modules that are finite.

We write $D(R)$ for the derived category of complexes of $R$-modules. If $R$ is noetherian, then we write $D^{\rm perf}(R)$ for the full triangulated subcategory of $D(R)$ comprising complexes that are `perfect' (that is, isomorphic in $D(R)$ to a bounded complex of finitely generated projective $R$-modules).

For a natural number $n$ we write $\tau_{\le n}$ for the exact truncation functor on $D(R)$ with the property that for each object $C$ in $D(R)$ and each integer $i$ one has
\[ H^i(\tau_{\le n}(C)) = \begin{cases} H^i(C), &\text{if $i \le n$}\\
                                          0, &\text{otherwise.}\end{cases}\]

\subsubsection{}For a Galois extension of number fields $L/K$ we set $G_{L/K} := \Gal(L/K)$. We also fix an algebraic closure $K^c$ of $K$ and set $G_K := G_{K^c/K}$.

For each non-archimedean place $v$ of a number field we write $\kappa_v$ for its residue field  and denote its absolute norm $|\kappa_v|$ by ${\rm N}v$. We write $\ell(v)$ for the residue characteristic of $v$.

We write the dual of an abelian variety $A$ as $A^t$. 
If $A$ is defined over a number field $k$, then for each extension $L$ of $k$ we write the Tate-Shafarevich group of $A$ over $L$ as $\sha(A_L)$. 

We shall also use the following notation regarding sets of places of $k$.

\begin{itemize}
\item[-] $S_k^\RR$ is the set of real archimedean places of $k$;
\item[-] $S_k^\CC$ is the set of complex archimedean places of $k$;
\item[-] $S_k^\infty (= S_k^\RR \cup S_k^\CC$) is the set of archimedean places of $k$;
\item[-] $S_k^f$ is the set of non-archimedean places of $k$;
\item[-] $S_k^v$ is the set of places of $k$ that extend a place $v$ of a given subfield of $k$. In particular,
\item[-] $S_k^p$ is the set of $p$-adic places of $k$ for a given prime number $p$;
\item[-] $S_k^L$ is the set of places of $k$ that ramify in a given extension $L$ of $k$;
\item[-] $S_k^A$ is the set of places of $k$ at which a given abelian variety $A$ has bad reduction.
\end{itemize}

 For a fixed set of places $S$ of $k$ we also write $S(L)$ for the set of places of $L$ which lie above a place in $S$.


Throughout this paper, we will consider the following situation: we are given a Galois extension of number fields $F/k$ and an abelian variety $A$ of dimension $d:={\rm dim}(A)$ that is defined over $k$. We set $G: = G_{F/k}$.

For a place $v$ of $k$ we set $G_{v} := G_{k_v^c/k_v}$. We also write $k_v^{\rm un}$ for the maximal unramified extension of $k$ in $k_v^c$ and set $I_{v} := G_{k_v^c/k_v^{\rm un}}$.

We fix a place $w$ of $F$ above $v$ and a corresponding embedding $F\to k_v^c$. We write $G_w$ and $I_w$ for the images of $G_{v}$ and $I_{v}$ under the induced homomorphism $G_{v} \to G$. We also fix a lift $\Phi_v$ to $G$ of the Frobenius automorphism in $G_w/I_w$.


We also write $\Sigma(k)$ for the set of embeddings $k \to \CC$ and $\Sigma_\sigma(F)$ for each $\sigma$ in $\Sigma(k)$ for the set of embeddings $F \to \CC$ that extend $\sigma$.

For each $v$ in $S_k^\infty$ we fix a corresponding embedding $\sigma_v$ in $\Sigma(k)$ and an embedding $\sigma'_v$ in $\Sigma_{\sigma_v}(F)$.

We then write $Y_{v,F}$ for the module $\prod_{\Sigma_{\sigma_v}(F)}\ZZ$ endowed with its natural action of $G\times G_v$ (via which $G$ and $G_v$ respectively act via pre-composition and post-composition with the embeddings in $\Sigma_{\sigma_v}(F)$).

For each $v$ in $S_k^\infty$ we set
\[ H_v(A_{F/k}) := H^0(k_v,Y_{v,F}\otimes_{\ZZ}H_1((A^t)^{\sigma_v}(\CC),\ZZ)),\]
regarded as a $G$-module via the given action on $Y_{v,F}$. We note that this $G$-module is free of rank $2d$ if $v$ is in $S_k^\CC$ and spans a free $\ZZ[1/2][G]$-module of rank $d$ if $v$ is in $S_k^\RR$.

We then define a $G$-module by setting
\[ H_\infty(A_{F/k}) := \bigoplus_{v \in S_k^\infty}H_v(A_{F/k}).\]

Finally we write $\Sigma_k(F)$ for the set of $k$-embeddings $F \to k^c$ and $Y_{F/k}$ for the module $\prod_{\Sigma_k(F)}\ZZ$, endowed with its natural action of $G\times G_k$.

\subsection{Acknowledgements} We are very grateful to St\'ephane Vigui\'e for his help with aspects of the argument presented in Appendix \ref{ptduality} and to Werner Bley for providing us with the material in \S\ref{ell curve sect} and for pointing out a sign-error in an earlier version of the argument in \S\ref{comparison section}.

We are also grateful to both Werner Bley and Christian Wuthrich for their interest, helpful correspondence and tremendous generosity regarding numerical computations.

In addition, we would like to thank Rob Evans, Masato Kurihara, Jan Nekov\'a\v r, Takamichi Sano, Kwok-Wing Tsoi and Stefano Vigni for helpful discussions and correspondence.

We are also very grateful to the anonymous referees for their careful reading of the article and for making several helpful observations and suggestions.

Finally, it is a great pleasure to thank Dick Gross for his strong encouragement regarding this project and for several insightful remarks.

The second author acknowledges financial support from the Spanish Ministry of Science and Innovation, through the `Severo Ochoa Programme for Centres of Excellence in R\&D' [SEV-2015-0554] and [CEX-2019-000904-S] as well as through projects [MTM2016-79400-P] and [PID2019-108936GB-C21].

\section{Selmer complexes}\label{selmer section}

\subsection{The aim of this section} In the sequel, we shall say that a complex is `global' in nature if all of its cohomology groups are finitely generated abelian groups.  

Then, to formulate leading term conjectures in relative algebraic $K$-groups, it is necessary to use invariants of perfect global complexes rather than of modules. In the setting of an abelian variety $A$, there are two distinct approaches that one can adopt to construct appropriate families of complexes. 
\medskip

\noindent{}$\bullet$ For each prime $p$, use the Kummer map from local points to the Galois cohomology of the $p$-adic Tate module $T_p(A)$ of $A$ to construct a $p$-adic complex whose cohomology groups identify with the 
pro-$p$ completions of classical objects such as Mordell-Weil groups and the Pontryagin duals of Selmer groups. Such complexes are directly linked to the construction of `cohomology with finite support' of Bloch and Kato and can be combined over all $p$ to specify a global complex (that is unique up to isomorphism in the appropriate derived category) whose chomology groups can be directly described in terms of Mordell-Weil groups and Selmer groups. 
\medskip

\noindent{}$\bullet$ Choose a collection of well-behaved submodules (a so-called `Perfect Selmer Structure') in both the Betti cohomology and local points of $A$, and by relating them (via Kummer maps) to the Galois cohomology of $T_p(A)$ (for every $p$), construct a global complex that is necessarily perfect, but whose cohomology groups do not precisely agree with Mordell-Weil groups or (Pontryagin duals of) Selmer groups. This approach is motivated by Nekov\'a\v r's general theory of Selmer complexes.  
\medskip

We shall refer to the $p$-adic and global complexes that arise from these two approaches as `classical Selmer complexes' (given the nature of their cohomology groups) and `Nekov\'a\v r-Selmer complexes' (given the nature of their construction) respectively. 

Roughly speaking, if the classical Selmer complex is perfect over the relevant algebra, then it will be best suited to derive explicit predictions from our general approach. However, this situation arises only if one imposes certain hypotheses on $A$ and, in order to avoid such restrictions (for example, to develop the general framework or derive the widest class of numerically verifiable predictions), it is necessary for us to use Nekov\'a\v r-Selmer complexes. In addition, whilst a choice of global differentials specifies a canonical perfect Selmer structure for which the resulting global Nekov\'a\v r-Selmer complex is closely related to classical Selmer complexes, we find that a different choice of perfect Selmer structure can also be useful in the derivation of explicit predictions.  

With this in mind, in this section we shall discuss the general construction of Selmer complexes and establish some of their key properties.

\subsection{Classical Selmer complexes}\label{p-adiccomplexes} In this section we fix a prime number $p$.

\subsubsection{}We first record a straightforward (and well-known) result regarding pro-$p$ completions that will be useful in the sequel.

We let $B$ denote either $A$ or its dual variety $A^t$ and write $T_p(B)$ for the $p$-adic Tate module of $B$.

\begin{lemma}\label{v not p} For each non-archimedean place $w'$ of $F$ the following claims are valid.
\end{lemma}
\begin{itemize}
\item[(i)] If $w'$ is not $p$-adic then the natural Kummer map $B(F_{w'})^\wedge_p \to H^1(F_{w'},T_p(B))$ is bijective.
\item[(ii)] There exists a canonical short exact sequence
\[ 0 \to H^1(\kappa_{w'}, T_p(B)^{I_{w'}}) \to B(F_{w'})^\wedge_p \to H^0(F_{w'}, H^1(I_{w'}, T_{p}(B))_{\rm tor})\to 0.\]
\end{itemize}

\begin{proof} We note first that if $w'$ does not divide $p$, then the module $H^1(F_{w'},T_p(B))$ is finite. This follows, for example, from Tate's local Euler characteristic formula, the vanishing of $H^0(F_{w'},T_p(B))$ and the fact that local duality identifies $H^2(F_{w'},T_p(B))$ with the finite module $H^0(F_{w'},T_p(B^t)\otimes_{\ZZ_p}\QQ_p/\ZZ_p)^\vee$.

Given this observation, claim (i) is obtained directly upon passing to the inverse limit over $m$ in the natural Kummer theory exact sequence
\begin{equation}\label{kummerseq} 0 \to B(F_{w'})/p^m \to H^1(F_{w'},T_p(B)/p^m) \to H^1(F_{w'},B)[p^m]\to 0.\end{equation}
%

Claim (ii) is established by Flach and the first author in \cite[(1.38)]{bufl95} (after recalling the fact that the group denoted $H^1_f(F_{w'},T_p(B))_{\rm BK}$ in loc. cit. is equal to the image of $B(F_{w'})^\wedge_p$ in $H^1(F_{w'},T_p(B))$ under the injective Kummer map). 
\end{proof}

\begin{remark}\label{Tamagawa remark}{\em The cardinality of each module $H^0(F_{w'}, H^1(I_{w'}, T_{p}(B))_{\rm tor})$ that occurs in Lemma \ref{v not p}(ii) is the maximal power of $p$ that divides the Tamagawa number of $B$ at $w'$. }\end{remark}

\subsubsection{}If $B$ denotes either $A$ or $A^t$, then for any subfield $E$ of $k$ and any place $v$ in $S_E^f$ we obtain a $G$-module by setting
%
\[ B(F_v) := \prod_{w' \in S_F^v}B(F_{w'}).\]

For later purposes we note that if $E = k$, then this module is isomorphic to the module of $G_w$-coinvariants
$Y_{F/k}\otimes_{\ZZ[G_w]} B(F_w)$ of the tensor product $Y_{F/k}\otimes_{\ZZ} B(F_w)$, upon which $G$ acts only the first factor but $G_k$ acts diagonally on both.

In a similar way,  the $p$-adic Tate module of the base change of $B$ through $F/k$ is equal to
\[ T_{p,F}(B) := Y_{F/k,p}\otimes_{\ZZ_p}T_p(B)\]
(where, again, $G$ acts only on the first factor of the tensor product whilst $G_k$ acts on both). We set $V_{p,F}(B) := \QQ_p\cdot T_{p,F}(B)$.

%
%

We can now introduce a notion of Selmer complex that will play an important role in the sequel.

\begin{definition}\label{bkdefinition}{\em For any finite subset $\Sigma$ of $S_k^f$ that contains each of $S_k^p, S_k^F\cap S_k^f$ and $S_k^A$, the `classical $p$-adic Selmer complex' ${\rm SC}_{\Sigma,p}(A_{F/k})$ for the data $A, F/k$ and $\Sigma$ is the mapping fibre of the morphism
\begin{equation}\label{bkfibre}
\tau_{\le 3}(R\Gamma(\mathcal{O}_{k,S_k^\infty\cup \Sigma},T_{p,F}(A^t))) \oplus \left(\bigoplus_{v\in\Sigma} A^t(F_v)^\wedge_p\right)[-1]  \xrightarrow{(\lambda,\kappa)} \bigoplus_{v \in \Sigma} R\Gamma (k_v, T_{p,F}(A^t))
\end{equation}
in $D(\ZZ_p[G])$. Here $\lambda$ is the natural diagonal localisation morphism and $\kappa$ is induced by the Kummer theory maps
$A^t(F_v)^\wedge_p\to H^1(k_v,T_{p,F}(A^t))$ (and the fact that $H^0(k_v, T_{p,F}(A^t))$ vanishes for all $v$ in $\Sigma$).

}
\end{definition}

\begin{remark}{\em If $p$ is odd, then $R\Gamma(\mathcal{O}_{k,S_k^\infty\cup \Sigma},T_{p,F}(A^t)))$ is acyclic in degrees greater than two and so the natural morphism $\tau_{\le 3}(R\Gamma(\mathcal{O}_{k,S_k^\infty\cup \Sigma},T_{p,F}(A^t))) \to R\Gamma(\mathcal{O}_{k,S_k^\infty\cup \Sigma},T_{p,F}(A^t)))$ in $D(\ZZ_p[G])$ is an isomorphism. In this case the truncation functor $\tau_{\le 3}$ can therefore be omitted from the above definition.}\end{remark}

The following result shows that the Selmer complex ${\rm SC}_{\Sigma,p}(A_{F/k})$ is independent, in a natural sense, of the choice of the set of places $\Sigma$.

For this reason, in the sequel we shall usually write ${\rm SC}_{p}(A_{F/k})$ in place of ${\rm SC}_{\Sigma,p}(A_{F/k})$.

\begin{lemma}\label{independenceofsigma} Let $\Sigma$ and $\Sigma'$ be any finite subsets of $S_k^f$ as in Definition \ref{bkdefinition} with $\Sigma\subseteq\Sigma'$. Then there is a canonical isomorphism ${\rm SC}_{\Sigma',p}(A_{F/k})\to {\rm SC}_{\Sigma,p}(A_{F/k})$ in $D(\ZZ_p[G])$.
\end{lemma}

\begin{proof} We recall that the compactly supported cohomology complex
\[ R\Gamma_{c,\Sigma}:=R\Gamma_c(\mathcal{O}_{k,S_k^\infty\cup\Sigma},T_{p,F}(A^t))\]
is defined to be the mapping fibre of the diagonal localisation morphism
\begin{equation}\label{compactloc} R\Gamma(\mathcal{O}_{k,S_k^\infty\cup\Sigma},T_{p,F}(A^t)) \to \bigoplus_{v \in S_k^\infty\cup\Sigma}R\Gamma(k_v,T_{p,F}(A^t))\end{equation}
in $D(\ZZ_p[G])$.

We further recall that $R\Gamma_{c,\Sigma}$ is acyclic outside degrees $1, 2$ and $3$ (see, for example, \cite[Prop. 1.6.5]{fukaya-kato}) and that
 $R\Gamma(k_v,T_{p,F}(A^t))$ for each $v$ in $\Sigma$ is acyclic outside degrees $1$ and $2$ and hence that the natural morphisms
\[ \tau_{\le 3}(R\Gamma_{c,\Sigma}) \to R\Gamma_{c,\Sigma}\,\,\text{ and } \,\,\tau_{\le 3}(R\Gamma(k_v,T_{p,F}(A^t))) \to R\Gamma(k_v,T_{p,F}(A^t))\]
in $D(\ZZ_p[G])$ are  isomorphisms.

Upon comparing these facts with the definition of ${\rm SC}_{\Sigma,p}(A_{F/k})$ one deduces the existence of a canonical exact triangle in $D(\ZZ_p[G])$ of the form
\begin{equation}\label{comparingtriangles}R\Gamma_{c,\Sigma}\to {\rm SC}_{\Sigma,p}(A_{F/k})\to \left(\bigoplus_{ v\in \Sigma}A^t(F_v)_p^\wedge\right)[-1]\oplus \tau_{\le 3}(R\Gamma_\infty) \to R\Gamma_{c,\Sigma}[1],
\end{equation}
where we abbreviate $\bigoplus_{ v\in S_k^\infty} R\Gamma(k_v,T_{p,F}(A^t))$ to $R\Gamma_\infty$.


In addition, by the construction of \cite[(30)]{bufl99}, there is a canonical exact triangle in $D(\ZZ_p[G])$ of the form
\begin{equation}\label{independencetriangle}\bigoplus\limits_{ v\in \Sigma'\setminus\Sigma}R\Gamma(\kappa_v,T_{p,F}(A^t)^{I_v})[-1]\to R\Gamma_{c,\Sigma'}\to R\Gamma_{c,\Sigma}\to\bigoplus\limits_{ v\in \Sigma'\setminus\Sigma}R\Gamma(\kappa_v,T_{p,F}(A^t)^{I_v}).\end{equation}

Finally we note that, since the choice of $\Sigma$ implies that $A$ has good reduction at each place $v$ in $\Sigma'\setminus\Sigma$, the module $H^0(F_{w'}, H^1(I_{w'}, T_{p}(A^t))_{\rm tor})$ vanishes for every $w'$ in $S_F^v$.

Thus, since for each $v$ in $\Sigma'\setminus\Sigma$ the complex $R\Gamma(\kappa_v,T_{p,F}(A^t)^{I_v})$ is acyclic outside degree one, the exact sequences in Lemma \ref{v not p}(ii) induce a canonical isomorphism
\begin{equation}\label{firstrow}A^t(F_v)_p^\wedge[-1]\to R\Gamma(\kappa_v,T_{p,F}(A^t)^{I_v})\end{equation}
in $D(\ZZ_p[G])$.

These three facts combine to give a canonical commutative diagram in $D(\ZZ_p[G])$ of the form

\begin{equation*}\label{complexesdiag}\xymatrix{
\bigoplus\limits_{ v\in \Sigma'\setminus\Sigma}A^t(F_v)_p^\wedge[-2] \ar@{^{(}->}[d] \ar[r]^{\hskip-0.4truein\sim}  &
\bigoplus\limits_{ v\in \Sigma'\setminus\Sigma}R\Gamma(\kappa_v,T_{p,F}(A^t)^{I_v})[-1] \ar[d] &
\\
\bigoplus\limits_{ v\in \Sigma'}A^t(F_v)_p^\wedge[-2]\oplus \tau_{\le 3}(R\Gamma_\infty)[-1] \ar@{->>}[d] \ar[r] &
R\Gamma_{c,\Sigma'} \ar[d] \ar[r]  &
{\rm SC}_{\Sigma',p}(A_{F/k})
\\
\bigoplus\limits_{ v\in \Sigma}A^t(F_v)_p^\wedge[-2]\oplus \tau_{\le 3}(R\Gamma_\infty)[-1]  \ar[r] &
R\Gamma_{c,\Sigma}  \ar[r]  &
{\rm SC}_{\Sigma,p}(A_{F/k}).
}\end{equation*}
Here the first row is induced by the isomorphisms (\ref{firstrow}), the second and third rows by the triangles (\ref{comparingtriangles}) for $\Sigma'$ and $\Sigma$ respectively, the first column is the obvious short exact sequence and the second column is given by the triangle (\ref{independencetriangle}).

In particular, since all rows and columns in this diagram are exact triangles, its commutativity implies the existence of a canonical isomorphism in $D(\ZZ_p[G])$ from
${\rm SC}_{\Sigma',p}(A_{F/k})$ to ${\rm SC}_{\Sigma,p}(A_{F/k})$, as required. 
\end{proof}

Taken together, Lemma \ref{v not p}(ii) and Remark \ref{Tamagawa remark} imply that if $p$ is odd and no Tamagawa numbers of $A$ over $F$ are divisible by $p$, then the complex ${\rm SC}_{p}(A_{F/k})$ canonically identifies with the `finite support cohomology' complex $R\Gamma_f(k,T_{p,F}(A))$ that was defined (for odd $p$) and played a key role in the article \cite{bmw} of Wuthrich and the present authors. Otherwise, these complexes differ slightly.

For such $p$ we have preferred to use ${\rm SC}_{p}(A_{F/k})$ rather than $R\Gamma_f(k,T_{p,F}(A))$ in this article since it is more amenable to certain explicit constructions that we have to make in later sections.

For the moment, we record only the following facts about ${\rm SC}_{p}(A_{F/k})$ that will be established in Propositions \ref{explicitbkprop} and \ref{explicitbkprop2} below. We write $\Sel_p(A_{F})$ for the classical $p$-primary Selmer group of $A$ over $F$. Then ${\rm SC}_{p}(A_{F/k})$ is acyclic outside degrees one, two and three and, assuming the Tate-Shafarevich group $\sha(A_F)$ of $A$ over $F$ to be finite, there are canonical identifications for each odd $p$ of the form
\begin{equation}\label{bksc cohom} H^i({\rm SC}_{p}(A_{F/k})) = \begin{cases} A^t(F)_p, &\text{if $i=1$,}\\
\Sel_p(A_F)^\vee, &\txt{if $i=2$,}\\
A(F)[p^{\infty}]^\vee, &\text{if $i=3$,}\end{cases}\end{equation}
whilst for $p=2$ there is a canonical identification $H^1({\rm SC}_{2}(A_{F/k})) =  A^t(F)_2$ and a canonical homomorphism $\Sel_2(A_F)^\vee \to H^2({\rm SC}_{2}(A_{F/k}))$ with finite kernel and cokernel, and the module $H^3({\rm SC}_{2}(A_{F/k}))$ is finite.

\begin{remark}\label{indeptremark}{\em A closer analysis of the argument in Lemma \ref{independenceofsigma} shows that, with respect to the identifications (\ref{bksc cohom}) that are established (under the hypothesis that $\sha(A_F)$ is finite)  in Proposition \ref{explicitbkprop} below, the isomorphism ${\rm SC}_{\Sigma',p}(A_{F/k})\to {\rm SC}_{\Sigma,p}(A_{F/k})$ constructed in Lemma \ref{independenceofsigma} induces the identity map on all degrees of cohomology.}\end{remark}



\subsection{Nekov\'a\v r-Selmer complexes}\label{NSc} In this section we again fix a prime number $p$.

Whilst the modules that occur in (\ref{bksc cohom}) are the primary objects of interest in the theory of abelian varieties, the complex ${\rm SC}_{p}(A_{F/k})$ is not always well-suited to our purposes since, except in certain special cases (that will be discussed in detail in \S\ref{tmc}), it does not belong to $D^{\rm perf}(\ZZ_p[G])$.

For this reason, we find it convenient to introduce the following alternative notion of Selmer complexes.

This construction is motivated by the general approach developed by Nekov\'a\v r in \cite{nek}.

\begin{definition}\label{selmerdefinition}{\em  We fix a finite set of places $S$ of $k$ with
$$ S_k^\infty\cup S_k^F \cup S_k^A\subseteq S$$ as well as $\ZZ_p[G]$-submodules $X$ of $A^t(F_p)^\wedge_p$ and $X'$ of $H_{\infty}(A_{F/k})_p$. Then the `Nekov\'a\v r-Selmer complex' ${\rm SC}_{S}(A_{F/k};X,X')$ of the data $(A,F,S,X,X')$ is the mapping fibre of the  morphism
\begin{equation}\label{fibre morphism}
R\Gamma(\mathcal{O}_{k,S\cup S_k^p},T_{p,F}(A^t)) \oplus X[-1] \oplus X'[0] \xrightarrow{(\lambda, \kappa_1,\kappa_2)} \bigoplus_{v \in S \cup S_k^p} R\Gamma (k_v, T_{p,F}(A^t))
\end{equation}
 in $D(\ZZ_p[G])$. Here $\lambda$ is again the natural diagonal localisation morphism, $\kappa_1$ is the morphism
 \[ X[-1]\rightarrow \bigoplus_{v \in S_k^p}R\Gamma (k_v, T_{p,F}(A^t))\]
 induced by the sum over $v$ of the local Kummer maps (and the fact each group $H^0(k_v,T_{p,F}(A^t))$ vanishes) and $\kappa_2$ is the morphism
\[ X'[0] \to \bigoplus_{v \in S_k^\infty}R\Gamma (k_v, T_{p,F}(A^t))\]
that is induced by the canonical comparison isomorphisms
\begin{equation}\label{cancompisom} Y_{v,F,p}\otimes_{\ZZ} H_1((A^t)^{\sigma_v}(\CC),\ZZ) \cong Y_{F/k,p}\otimes_{\ZZ_p}T_{p}(A^t)= T_{p,F}(A^t) \end{equation}
for each $v$ in $S_k^\infty$.
}\end{definition}


In the next result we establish the basic properties of these Nekov\'a\v r-Selmer complexes. In this result we shall write ${\rm Mod}^\ast(\ZZ_p[G])$ for the category ${\rm Mod}(\ZZ_p[G])$ in the case that $p$ is odd and for the quotient of ${\rm Mod}(\ZZ_2[G])$ by its subcategory ${\rm Mod}^{\rm fin}(\ZZ_2[G])$ in the case that $p = 2$.

\begin{proposition}\label{prop:perfect} Let $X$ be a finite index $\ZZ_p[G]$-submodule of $A^t(F_p)^\wedge_p$ that is cohomologically-trivial as a $G$-module.

Let $X'$ be a finite index projective $\ZZ_p[G]$-submodule of $H_\infty(A_{F/k})_p$, with $X' = H_\infty(A_{F/k})_p$ if $p$ is odd.

Then the following claims are valid.
\begin{itemize}
\item[(i)] ${\rm SC}_{S}(A_{F/k};X,X')$ is an object of $D^{\rm perf}(\ZZ_p[G])$ that is acyclic outside degrees one, two and three.
\item[(ii)] $H^3({\rm SC}_{S}(A_{F/k};X,X'))$ identifies with $A(F)[p^{\infty}]^\vee$.
\item[(iii)] If $\sha(A_F)$ is finite, then in ${\rm Mod}^\ast(\ZZ_p[G])$ there exists a canonical injective homomorphism
\[ H^1({\rm SC}_{S}(A_{F/k};X,X')) \to A^t(F)_p \]
that has finite cokernel and a canonical surjective homomorphism
\[  H^2({\rm SC}_{S}(A_{F/k};X,X')) \to {\rm Sel}_p(A_{F})^\vee\]
that has finite kernel.
\end{itemize}
\end{proposition}

\begin{proof} Set $C_{S} := {\rm SC}_{S}(A_{F/k};X,X')$.

Then, by comparing the definition of $C_{S}$ as the mapping fibre of (\ref{fibre morphism}) with the definition of the compactly supported cohomology complex $R\Gamma_c(A_{F/k}) := R\Gamma_c(\mathcal{O}_{k,S\cup S_k^p},T_{p,F}(A^t))$ as the mapping fibre of the morphism (\ref{compactloc})
one finds that there is an exact triangle in $D(\ZZ_p[G])$ of the form
\begin{equation}\label{can tri}  R\Gamma_c(A_{F/k})\to C_{S} \to X[-1] \oplus X'[0]\to R\Gamma_c(A_{F/k})[1].\end{equation}

To derive claim (i) from this triangle it is then enough to recall (from, for example, \cite[Prop. 1.6.5]{fukaya-kato}) that $R\Gamma_c(A_{F/k})$ belongs to $D^{\rm perf}(\ZZ_p[G])$ and is acyclic outside degrees one, two and three and note that both of the $\ZZ_p[G]$-modules $X$ and $X'$ are finitely generated and cohomologically-trivial.

The above triangle also gives a canonical identification
\begin{equation}\label{artinverdier} H^3(C_{S}) \cong H^3(R\Gamma_c(A_{F/k})) \cong H^0(k,T_{p,F}(A)\otimes_{\ZZ_p}\QQ_p/\ZZ_p)^\vee = A(F)[p^{\infty}]^\vee\end{equation}
where the second isomorphism is induced by the Artin-Verdier Duality Theorem.

In a similar way, if we set $\Sigma:=(S\cap S_k^f)\cup S_k^p$ and abbreviate the classical $p$-adic Selmer complex ${\rm SC}_{\Sigma,p}(A_{F/k})$ to $C'_\Sigma$, then a direct comparison of the definitions of $C_{S}$ and $C_\Sigma'$ shows that $C_{S}$ is isomorphic in $D(\ZZ_p[G])$ to the mapping fibre of the morphism
\begin{equation}\label{selmer-finite tri} C'_\Sigma \oplus X[-1] \oplus X'[0]
\xrightarrow{(\lambda', \kappa'_1,\kappa_2)}
 \bigoplus_{v \in \Sigma} A^t(F_v)^\wedge_p[-1] \oplus \bigoplus_{v \in S_k^\infty}\tau_{\le 3}(R\Gamma(k_v,T_{p,F}(A^t)))\end{equation}
where $\lambda'$ is the canonical morphism
\[ C'_\Sigma \to \bigoplus_{v \in \Sigma} A^t(F_v)^\wedge_p[-1]\]
determined by the definition of of $C'_\Sigma$ as the mapping fibre of (\ref{bkfibre}),
and the morphism $$\kappa'_1: X[-1]\rightarrow \bigoplus_{v \in S_k^p} A^t(F_v)^\wedge_p[-1]$$ is induced by the given inclusion
 $X \subseteq A^t(F_p)^\wedge_p$.

This description of $C_{S}$ gives rise to a canonical long exact sequence of $\ZZ_p[G]$-modules

\begin{multline}\label{useful1} 0 \to {\rm cok}(H^0(\kappa_2)) \to H^1(C_{S}) \to H^1(C_\Sigma') \\
 \to (A^t(F_p)^\wedge_p/X) \oplus\bigoplus_{v \in (S\cap S_k^f)\setminus S_k^p} A^t(F_v)^\wedge_p \oplus  \bigoplus_{v \in S_k^\infty}H^1(k_v,T_{p,F}(A^t))\\ \to H^2(C_{S})\to H^2(C_\Sigma') \to \bigoplus_{v \in S_k^\infty}H^2(k_v,T_{p,F}(A^t)). \end{multline}

In addition, for each $v \in S_k^\infty$ the groups $H^1(k_v,T_{p,F}(A^t))$ and $H^2(k_v,T_{p,F}(A^t))$ vanish if $p$ is odd and are finite if $p=2$, whilst our choice of $X'$ ensures that ${\rm cok}(H^0(\kappa_2))$ is also a finite group of $2$-power order.

Claim (iii) therefore follows upon combining the above sequence with the identifications of $H^1(C_\Sigma')$ and $H^2(C_\Sigma')$ given in (\ref{bksc cohom}) for odd $p$, and in the subsequent remarks for $p=2$, that are valid whenever $\sha(A_F)$ is finite.
\end{proof}

\begin{remark}\label{mrselmer}{\em If $p$ is odd, then the proof of Proposition \ref{prop:perfect} shows that the cohomology group
$H^1({\rm SC}_{S}(A_{F/k};X,H_\infty(A_{F/k})_p))$ coincides with the Selmer group $H^1_{\mathcal{F}_X}(k,T_{p,F}(A^t))$ in the sense of Mazur and Rubin \cite{MRkoly}, where $\mathcal{F}_X$ is the Selmer structure with $\mathcal{F}_{X,v}$ equal to the image of $X$ in $H^1(k_v,T_{p,F}(A^t))$ for $v\in S_k^p$ and equal to $0$ for $v \in S\setminus S_k^p$.} \end{remark}


\subsection{Perfect Selmer structures and global complexes}\label{perfect selmer integral} 

In this section we introduce a notion of perfect Selmer structures as a means of constructing global perfect complexes from the $p$-adic Nekov\'a\v r-Selmer complexes that were discussed in the previous section. 

In subsequent sections, our main interest will be in the perfect Selmer structures that arise from a choice of global differentials via the construction discussed in \S \ref{perf sel sect}. However, we note that, when deriving numerically verifiable predictions from the central conjecture of this article, it also proves useful for us to consider the perfect Selmer structures that arise from a more general choice of semi-local points.

We recall that $\ell(v)$ denotes the residue characteristic of a non-archimedean place $v$ of $k$.

\begin{definition}\label{pgss def}{\em A `perfect Selmer structure' for the pair $A$ and $F/k$ is a collection
\[ \mathcal{X} := \{\mathcal{X}(v): v \}\]
over all places $v$ of $k$ of $G$-modules that satisfy the following conditions.

\begin{itemize}
\item[(i)] For each $v$ in $S_k^\infty$ the module $\mathcal{X}(v)$ is projective and a submodule of $H_v(A_{F/k})$ of finite $2$-power index.
\item[(ii)] For each $v$ in $S_k^f$ the module $\mathcal{X}(v)$ is cohomologically-trivial and a finite index $\ZZ_{\ell(v)}[G]$-submodule of $A^t(F_v)^\wedge_{\ell(v)}$.
\item[(iii)] For almost all (non-archimedean) places $v$ one has $\mathcal{X}(v) = A^t(F_v)^\wedge_{\ell(v)}.$
\end{itemize}
We thereby obtain a  projective $G$-submodule
\[ \mathcal{X}(\infty) := \bigoplus_{v \in S_k^\infty}\mathcal{X}(v)\]
of $H_\infty(A_{F/k})$ of finite $2$-power index and, for each rational prime $\ell$, a finite index cohomologically-trivial $\ZZ_\ell[G]$-submodule %
\[  \mathcal{X}(\ell) := \bigoplus_{v\in S_k^\ell}\mathcal{X}(v)\]
of $A^t(F_\ell)^\wedge_{\ell}$.}
\end{definition}

\begin{remark}{\em The conditions (ii) and (iii) in Definition \ref{pgss def} are consistent since if $\ell$ does not divide $|G|$, then any $\ZZ_{\ell}[G]$-module is automatically cohomologically-trivial for $G$.} \end{remark}

In the following result we write $X_\ZZ(A_F)$ for the `integral Selmer group' of $A$ over $F$ defined by Mazur and Tate in \cite{mt}.

We recall that, if the Tate-Shafarevich group $\sha(A_F)$ is finite, then $X_\ZZ(A_F)$ is a finitely generated $G$-module and there exists an isomorphism of $\hat \ZZ[G]$-modules
\[ \hat\ZZ\otimes_\ZZ X_\ZZ(A_F) \cong {\rm Sel}(A_F)^\vee\]
that is unique up to automorphisms that induce the identity map on both the submodule $X_\ZZ(A_F)_{\rm tor} = \sha(A_F)^\vee$ and quotient module $X_\ZZ(A_F)_{\rm tf} = \Hom_\ZZ(A(F), \ZZ)$. (Here $\hat\ZZ$ denotes the profinite completion of $\ZZ$).

We identify ${\rm Mod}^{\rm fin}(\ZZ_2[G])$ as an abelian subcategory of ${\rm Mod}(\ZZ[G])$ in the obvious way and write ${\rm Mod}^\ast(\ZZ[G])$ for the associated quotient category.

\begin{proposition}\label{prop:perfect2} Assume that $\sha(A_F)$ is finite. Then for any perfect Selmer structure $\mathcal{X}$ for $A$ and $F/k$ and any set $S$ as in Definition \ref{selmerdefinition}, 
 there exists a complex $C_S(\mathcal{X}) = {\rm SC}_{S}(A_{F/k};\mathcal{X})$ in $D^{\rm perf}(\ZZ[G])$ that is unique up to isomorphisms in $D^{\rm perf}(\Z[G])$ that induce the identity map in all degrees of cohomology and has all of the following properties.
\begin{itemize}
\item[(i)] For each prime $\ell$ there is a canonical isomorphism in $D^{\rm perf}(\ZZ_\ell[G])$
\[ \ZZ_\ell\otimes_\ZZ C_S(\mathcal{X}) \cong {\rm SC}_{S}(A_{F/k};\mathcal{X}(\ell),\mathcal{X}(\infty)_\ell).\]
\item[(ii)] $C_S(\mathcal{X})$ is acyclic outside degrees one, two and three.
\item[(iii)] There is a canonical identification $H^3(C_S(\mathcal{X})) = (A(F)_{\rm tor})^\vee$.

\item[(iv)] In ${\rm Mod}^\ast(\ZZ[G])$ there exists a canonical injective homomorphism
\[ H^1(C_S(\mathcal{X})) \to A^t(F) \]
that has finite cokernel and a canonical surjective homomorphism
\[ H^2(C_S(\mathcal{X})) \to X_\ZZ(A_F)\]
that has finite kernel.
\item[(v)] If $\mathcal{X}(v)\subseteq A^t(F_v)$ for all $v$ in $S\cap S_k^f$ and $\mathcal{X}(v) = A^t(F_v)^\wedge_{\ell(v)}$ for all $v\notin S$, then there exists an exact sequence in ${\rm Mod}^\ast(\ZZ[G])$ of the form
\[ 0 \to H^1(C_S(\mathcal{X})) \to A^t(F) \xrightarrow{\Delta_{S,\mathcal{X}}} \bigoplus_{v \in S\cap S_k^f}\frac{A^t(F_v)}{\mathcal{X}(v)} \to
H^2(C_S(\mathcal{X})) \to X_\ZZ(A_F)\to 0\]
in which $\Delta_{S,\mathcal{X}}$ is the natural diagonal map.
\end{itemize}
\end{proposition}

\begin{proof} We write $\hat \ZZ$ for the profinite completion of $\ZZ$ and for each prime $\ell$ set $C_S(\ell) := {\rm SC}_{S}(A_{F/k};\mathcal{X}(\ell),\mathcal{X}(\infty)_\ell)$.

To construct a suitable complex $C_S(\mathcal{X})$ we shall use the general result of \cite[Lem. 3.8]{bkk} with the complex $\widehat C$ in loc. cit. taken to be the object $\prod_\ell C_S(\ell)$ of $D(\hat \ZZ[G])$.

In fact, since $\mathcal{X}$ satisfies the conditions (i) and (ii) in Definition \ref{pgss def}, Proposition \ref{prop:perfect}(i) implies that each complex $C_S(\ell)$ belongs to $D^{\rm perf}(\ZZ_\ell[G])$ and is acyclic outside degrees one, two and three and so to apply \cite[Lem. 3.8]{bkk} it is enough to specify for each $j \in \{1,2,3\}$ a finitely generated $G$-module $M^j$ together with an isomorphism of $\hat \Z[G]$-modules of the form $\iota_j: \hat \Z\otimes_\Z M^j \cong \prod_\ell H^j(C_S(\ell))$.

By Proposition \ref{prop:perfect}(ii) it is clear that one can take $M^3 = A(F)_{\rm tor}^\vee$ and $\iota_3$ the canonical identification induced by the decomposition $A(F)_{\rm tor}^\vee = \prod_\ell A(F)[\ell^\infty]^\vee$.

To construct suitable modules $M^1$ and $M^2$,
%
%
 %
 %
we note first that the proof of Proposition \ref{prop:perfect}(iii) combines with the fact that $\mathcal{X}$ satisfies condition (iii) in Definition \ref{pgss def} to give rise to a homomorphism of $\hat\ZZ[G]$-modules
\[ \prod_\ell H^1(C_S(\ell)) \xrightarrow{\theta_1} \hat\ZZ\otimes_\ZZ A^t(F) \]
with the property that $\ker(\theta_1)$ is finite of $2$-power order and ${\rm cok}(\theta_1)$ is finite, and to a diagram of homomorphisms of $\hat\ZZ[G]$-modules
\begin{equation}\label{derived diag}
\prod_\ell H^2(C_S(\ell)) \xrightarrow{\theta_2} \prod_\ell H^2({\rm SC}_{\Sigma_\ell,\ell}(A_{F/k})) \xleftarrow{\theta_3} \hat\ZZ\otimes_\ZZ X_{\ZZ}(A_F)\end{equation}
in which $\ker(\theta_2)$ is finite whilst ${\rm cok}(\theta_2), \ker(\theta_3)$ and ${\rm cok}(\theta_3)$ are all finite of $2$-power order. Here for each prime number $\ell$ we have also set $\Sigma_\ell:=(S\cap S_k^f)\cup S_k^\ell$.

It is then straightforward to construct a commutative (pull-back) diagram of $G$-modules
\begin{equation}\label{useful2 diagrams} \begin{CD}
 M^1 @> >> A^t(F)\\
 @V \iota_{11} VV @VV\iota_{12} V\\
  \prod_\ell H^1(C_S(\ell)) @> \theta_1>> \hat\ZZ\otimes_\ZZ A^t(F)\end{CD}\end{equation}
in which $M^1$ is finitely generated, the upper horizontal arrow has finite kernel of $2$-power order and finite cokernel, the morphism $\iota_{12}$ is the natural inclusion and the morphism $\iota_{11}$ induces an isomorphism of $\hat\ZZ[G]$-modules $\iota_1$ of the required sort.

In a similar way, there is a pull-back diagram of $G$-modules
\begin{equation*} \begin{CD}
 M_2 @> \theta_2' >> \theta_3(X_{\ZZ}(A_F))\\
 @V \iota_{21} VV @VV\iota_{22} V\\
 \prod_\ell H^2(C_S(\ell)) @> \theta_2 >> \prod_\ell H^2({\rm SC}_{\Sigma_\ell,\ell}(A_{F/k}))\end{CD}\end{equation*}
in which $M_2$ is finitely generated, $\iota_{22}$ is the natural inclusion, $\ker(\theta_2')$ is finite, ${\rm cok}(\theta_2')$ is finite of $2$-power order and the morphism $\iota_{21}$ induces a short exact sequence
\[0\to \hat \ZZ\otimes_\ZZ M_2 \to \prod_\ell H^2(C_S(\ell)) \to M_2' \to 0 \]
in which $M_2'$ is finite of $2$-power order. Then, since $\hat \ZZ$ is a flat $\ZZ$-module, one has
\[ {\rm Ext}^1_{G}(M_2',M_2) = \hat\ZZ\otimes_\ZZ{\rm Ext}^1_{G}(M_2',M_2) = {\rm Ext}^1_{\hat \ZZ[G]}(M_2',\hat\ZZ\otimes_\ZZ M_2)\]
and so there exists an exact commutative diagram of $G$-modules
\[ \begin{CD}
0 @> >> \hat \ZZ\otimes_\ZZ M_2 @> \iota_{21} >> \prod_\ell H^2(C_S(\ell)) @> >>  M_2' @> >> 0\\
& & @A AA @A\iota_2 AA @\vert\\
0 @> >> M_2 @> >> M^2 @> >>  M_2' @> >> 0\end{CD}\]
in which the left hand vertical arrow is the natural inclusion and $M^2$ is finitely generated.

It is then clear that $\iota_2$ induces an isomorphism $\hat\ZZ\otimes_\ZZ M^2 \cong \prod_\ell H^2(C_S(\ell))$ and that the diagram
\[ M^2 \xleftarrow{\iota_{21}} M_2 \xrightarrow{\theta_2'} \theta_3(X_{\ZZ}(A_F)) \xleftarrow{\theta_3} X_{\ZZ}(A_F)\]
constitutes a morphism in ${\rm Mod}^\ast(\ZZ[G])$. This morphism is surjective, has finite kernel and lies in a commutative diagram in ${\rm Mod}^\ast(\ZZ[G])$
\begin{equation}\label{derived diag2} \begin{CD} M^2 @> >> X_\ZZ(A_F)\\
              @V \iota_2 VV @VV V\\
              \prod_{\ell}H^2(C_S(\ell)) @> >> \hat\ZZ\otimes_\ZZ X_\ZZ(A_F)\end{CD}\end{equation}
in which the right hand vertical arrow is the inclusion map and the lower horizontal arrow corresponds to the diagram (\ref{derived diag}).

These observations show that we can apply \cite[Lem. 3.8]{bkk} in the desired way in order to obtain a complex $C_S(\mathcal{X})$ in  $D^{\rm perf}(\ZZ[G])$ that has $H^j({\rm SC}_{S}(A_{F/k};\mathcal{X})) = M^j$ for each $j$ in $\{1,2,3\}$ and satisfies all of the stated properties in claims (i)-(iv).

Turning to claim (v) we note that the given conditions on the modules $\mathcal{X}(v)$ imply that for each $v$ in $S\cap S_k^f$ there is a direct sum decomposition of finite modules
\[ \frac{A^t(F_v)}{\mathcal{X}(v)} = \frac{A^t(F_v)^\wedge_{\ell(v)}}{\mathcal{X}(v)} \oplus \bigoplus_{\ell \not= \ell(v)}A^t(F_v)^\wedge_\ell\]
and hence also a direct sum decomposition over all primes $\ell$ of the form
\[ \bigoplus_{v \in S\cap S_k^f} \frac{A^t(F_v)}{\mathcal{X}(v)} = \bigoplus_\ell \left( \bigoplus_{v \in S_k^\ell}\frac{A^t(F_v)^\wedge_{\ell(v)}}{\mathcal{X}(v)} \oplus \bigoplus_{v \in (S\cap S_k^f)\setminus S_k^\ell}A^t(F_v)^\wedge_\ell\right).\]

This shows that the kernel and cokernel of the map $\Delta_{S,\mathcal{X}}$ in claim (v) respectively coincide with the intersection over all primes $\ell$ of the kernel and the direct sum over all primes $\ell$ of the cokernel of the diagonal map
\[ A^t(F) \to \bigoplus_{v \in S_k^\ell}\frac{A^t(F_v)^\wedge_\ell}{\mathcal{X}(v)} \oplus \bigoplus_{v \in (S\cap S_k^f)\setminus S_k^\ell}A^t(F_v)^\wedge_\ell\]
that occurs in the sequence (\ref{useful1}) (with $p$ replaced by $\ell$ and $X$ by $\mathcal{X}(\ell)$).

Given this fact, the exact sequence follows from the commutativity of the diagrams (\ref{useful2 diagrams}) and (\ref{derived diag2}) and the exactness of the sequence (\ref{useful1}).
\end{proof}

\subsection{Global differentials and perfect Selmer structures}\label{perf sel sect} 

With a view to the subsequent formulation (in \S\ref{statement of conj section}) of our central conjecture we explain how a choice of global differentials gives rise to a natural perfect Selmer structure for $A$ and $F/k$.

In the sequel we shall for a natural number $m$ write $[m]$ for the (ordered) set of integers $i$ that satisfy $1 \le i\le m$.

\subsubsection{}\label{gamma section}For each $v$ in $S_k^\RR$ we fix ordered $\ZZ$-bases
\[ \{\gamma_{v,a}^+: a\in [d]\}\,\,\,\text{ and }\,\,\,\{\gamma_{v,a}^-: a\in [d]\}\]
of $H_1((A^t)^{\sigma_v}(\CC),\ZZ)^{c=1}$ and
%
%
of $H_1((A^t)^{\sigma_v}(\CC),\ZZ)^{c=-1}$ respectively, where $c$ denotes complex conjugation.

For each $v$ in $S_k^\CC$ we fix an ordered $\ZZ$-basis
\[ \{\gamma_{v,a}: a\in [2d]\}\]
of $H_1((A^t)^{\sigma_v}(\CC),\ZZ)$.

%

For each $v$ in $S_k^\infty$ we then fix $\tau_v\in G$ with $\tau_v(\sigma_v')=c\circ\sigma_v'$ and write $H_v(\gamma_\bullet)$ for the free $G$-module with basis
\begin{equation}\label{gamma basis}\begin{cases} \{ (1+\tau_v)\sigma_v'\otimes \gamma^+_{v,a} + (1-\tau_v)\sigma_v'\otimes \gamma^-_{v,a}: a \in [d]\}, &\text{ if $v$ is real}\\
                 \{\sigma_v'\otimes\gamma_{v,a}:a \in [2d]\}, &\text{ if $v$ is complex.}\end{cases}\end{equation}

The direct sum
\begin{equation}\label{25} H_\infty(\gamma_\bullet) := \bigoplus_{v \in S_k^\infty}H_v(\gamma_\bullet)\end{equation}
is then a free $G$-submodule of $H_\infty(A_{F/k})$ of finite $2$-power index.

To specify an ordered $\ZZ[G]$-basis of $H_\infty(\gamma_\bullet)$ we fix an ordering of $S_k^\infty$ and then order the union of the sets
 (\ref{gamma basis}) lexicographically.

\subsubsection{}\label{perf sel construct}We next fix a N\'eron model $\mathcal{A}^t$ for $A^t$ over $\mathcal{O}_k$ and, for each non-archimedean place $v$ of $k$, a N\'eron model $\mathcal{A}_v^t$ for $A^t_{/k_v}$ over $\mathcal{O}_{k_v}$.

For any subfield $E$ of $k$ and any non-archimedean place $v$ of $E$ we set $\mathcal{O}_{F,v}:=\prod_{w'\in S_F^v}\mathcal{O}_{F_{w'}}$.

For each non-archimedean place $v$ of $k$ we then set
\begin{equation}\label{mathcalD} \mathcal{D}_F(\mathcal{A}^t_v) :=  \mathcal{O}_{F,v}\otimes_{\mathcal{O}_{k_v}}\Hom_{\mathcal{O}_{k_v}}(H^0(\mathcal{A}_v^t,\Omega^1_{\mathcal{A}_v^t}), \mathcal{O}_{k_v}).\end{equation}
%


We finally fix an ordered $\QQ[G]$-basis $\omega_\bullet$ of the space of invariant differentials
\[ H^0(A^t_F,\Omega^1_{A^t_F}) \cong F\otimes_k H^0(A^t,\Omega^1_{A^t})\]
%
%
and write $\mathcal{F}(\omega_\bullet)$ for the $G$-module generated by the elements of $\omega_\bullet$. In the sequel we often identify $\omega_\bullet$ with its dual ordered $\QQ[G]$-basis in $\Hom_{F}(H^0(A^t_F,\Omega^1_{A^t_F}),F)$ and $\mathcal{F}(\omega_\bullet)$ with the $G$-module generated by this dual basis.

In the sequel, for any subfield $E$ of $k$ and any place $v$ in $S_E^f$ we set $F_v:=\prod_{w'\in S_F^v}F_{w'}$.

For each non-archimedean place $v$ of $k$ we write $\mathcal{F}(\omega_\bullet)_v$ for the $\ZZ_{\ell(v)}$-closure of the image of $\mathcal{F}(\omega_\bullet)$ in $F_{v}\otimes_k\Hom_{k}(H^0(A^t,\Omega^1_{A^t}), k)$ and
\begin{equation*}\label{classical exp} {\rm exp}_{A^t,F_v}: F_v\otimes_k \Hom_k(H^0(A^t,\Omega^1_{A^t}),k) \cong \Hom_{F_v}(H^0(A^t_{F_v},\Omega^1_{A^t_{F_v}}),F_v) \cong \QQ_{\ell(v)}\cdot A^t(F_v)^\wedge_{\ell(v)}\end{equation*}
for the exponential map of $A^t_{F_v}$ relative to some fixed $\mathcal{O}_{k_v}$-basis of $H^0(\mathcal{A}_v^t,\Omega^1_{\mathcal{A}_v^t})$.

Then, if necessary after multiplying each element of $\omega_\bullet$ by a suitable natural number, we may, and will, assume that the following conditions are satisfied:

\begin{itemize}
\item[(i$_{\omega_\bullet}$)] for each $v$ in $S_k^f$ one has $\mathcal{F}(\omega_\bullet)_{v}\subseteq \mathcal{D}_F(\mathcal{A}_v^t)$;
\item[(ii$_{\omega_\bullet}$)] for each $v$ in $S\cap S_k^f$, the map ${\rm exp}_{A^t,F_v}$ induces an isomorphism of $\mathcal{F}(\omega_\bullet)_{v}$ with a submodule of $A^t(F_v)$.
\end{itemize}

%
%
%
%
%
%

%
%

We then define $\mathcal{X}=\mathcal{X}_S(\omega_\bullet) = \mathcal{X}_S(\{\mathcal{A}^t_v\}_v,\omega_\bullet,\gamma_\bullet)$ to be the perfect Selmer structure for $A$, $F/k$ and $S$ that has the following properties:
\begin{itemize}
\item[(i$_\mathcal{X}$)] If $v\in S_k^\infty$, then $\mathcal{X}(v) = H_v(\gamma_\bullet)$.
\item[(ii$_\mathcal{X}$)] If $v \in S\cap S_k^f$, then $\mathcal{X}(v) = {\rm exp}_{A^t,F_v}(\mathcal{F}(\omega_\bullet)_v)$.
\item[(iii$_\mathcal{X}$)] If $v \notin S$, then $\mathcal{X}(v) = A^t(F_v)^\wedge_{\ell(v)}$.
\end{itemize}

\begin{remark}{\em This specification does define a perfect Selmer structure for $A$ and $F/k$ since if $v$ does not belong to $S$, then the $G$-module $A^t(F_v)^\wedge_{\ell(v)}$ is cohomologically-trivial (by Lemma \ref{useful prel}(ii) below).} \end{remark}

\begin{remark}\label{can structure groups}{\em The perfect Selmer structure $\mathcal{X}_S(\omega_\bullet) = \mathcal{X}_S(\{\mathcal{A}^t_v\}_v,\omega_\bullet,\gamma_\bullet)$ defined above satisfies the conditions of Proposition \ref{prop:perfect2}(v). As a consequence, if one ignores finite modules of $2$-power order, then the cohomology modules of the Selmer complex
$C_S(\mathcal{X}(\omega_\bullet)) = {\rm SC}_{S}(A_{F/k};\mathcal{X}_S(\omega_\bullet))$ can be described as follows: \

\noindent{} - $H^1(C_S(\mathcal{X}(\omega_\bullet)))$ is the submodule of $A^t(F)$ comprising all elements $x$ with the property that, for each $v$ in $S$, the image of $x$ in $A^t(F_v)$ belongs to the subgroup ${\rm exp}_{A^t,F_v}(\mathcal{F}(\omega_\bullet)_v)$.

\noindent{} - $H^2(C_S(\mathcal{X}(\omega_\bullet)))$ is an extension of the integral Selmer group $X_\ZZ(A_F)$ by the (finite) cokernel of the diagonal map $A^t(F) \to \bigoplus_{v \in S}\bigl(A^t(F_v)/{\rm exp}_{A^t,F_v}(\mathcal{F}(\omega_\bullet)_v)\bigr)$.

\noindent{} - $H^3(C_S(\mathcal{X}(\omega_\bullet)))$ is equal to $(A(F)_{\rm tor})^\vee$.}\end{remark}

%

\section{The refined Birch and Swinnerton-Dyer Conjecture}\label{ref bsd section}

In this section we formulate (as Conjecture \ref{conj:ebsd}) a precise refinement of the Birch and Swinnerton-Dyer Conjecture.

\subsection{Relative $K$-theory} For the reader's convenience, we first quickly review some relevant facts of algebraic $K$-theory.

\subsubsection{}\label{Relative $K$-theory}

For a Dedekind domain $R$ with field of fractions $F$, an $R$-order $\mathfrak{A}$ in a finite dimensional separable $F$-algebra $A$ and a field extension $E$ of $F$ we set $A_E := E\otimes_F A$.

The relative algebraic $K_0$-group $K_0(\mathfrak{A},A_E)$ of the ring inclusion $\mathfrak{A}\subset A_E$ is described explicitly in terms of generators and relations by Swan in \cite[p. 215]{swan}.

For any extension field $E'$ of $E$ there exists a canonical commutative diagram
\begin{equation} \label{E:kcomm}
\begin{CD} K_1(\mathfrak{A}) @> >> K_1(A_{E'}) @> \partial_{\mathfrak{A},A_{E'}} >> K_0(\mathfrak{A},A_{E'}) @> \partial'_{\mathfrak{A},A_{E'}} >> K_0(\mathfrak{A})\\
@\vert @A\iota AA @A\iota' AA @\vert\\
K_1(\mathfrak{A}) @> >> K_1(A_E) @> \partial_{\mathfrak{A},A_E}  >> K_0(\mathfrak{A},A_E) @> \partial'_{\mathfrak{A},A_E}  >> K_0(\mathfrak{A})
\end{CD}
\end{equation}
in which the upper and lower rows are the respective long exact sequences in relative $K$-theory of the inclusions $\mathfrak{A}\subset A_E$ and $\mathfrak{A}\subset A_{E'}$ and both of the vertical arrows are injective and induced by the inclusion $A_E \subseteq A_{E'}$. (For more details see \cite[Th. 15.5]{swan}.)


In particular, if $R = \ZZ$ and for each prime $\ell$ we set $\mathfrak{A}_\ell := \ZZ_\ell\otimes_\ZZ \mathfrak{A}$ and $A_\ell:=
\QQ_\ell\otimes _\QQ A$, then we can regard each group $K_0(\mathfrak{A}_\ell,A_\ell)$ as a subgroup of $K_0(\mathfrak{A},A)$ by means of the canonical composite homomorphism
\begin{equation}\label{decomp}
\bigoplus_\ell K_0(\mathfrak{A}_\ell,A_\ell) \cong K_0(\mathfrak{A},A)\subset K_0(\mathfrak{A},A_\RR),
\end{equation}
where $\ell$ runs over all primes, the isomorphism is as described in the discussion following \cite[(49.12)]{curtisr} and the inclusion is induced by the relevant case of $\iota'$.

For an element $x$ of $K_0(\mathfrak{A},A)$ we write $(x_\ell)_\ell$ for its image in $\bigoplus_\ell K_0(\mathfrak{A}_\ell,A_\ell)$ under the isomorphism in (\ref{decomp}).

Then, if $G$ is a finite group and $E$ is a field of characteristic zero, taking reduced norms over the semisimple algebra $E[G]$ induces (as per the discussion in \cite[\S 45A]{curtisr}) an injective homomorphism
\[ {\rm Nrd}_{E[G]}: K_1(E[G]) \to \zeta(E[G])^\times. \]
This homomorphism is bijective if $E$ is either algebraically closed or complete.

\subsubsection{}\label{nad sec} We shall also use a description of $K_0(\mathfrak{A},A_E)$ in terms of the formalism of `non-abelian determinants' that is given by Fukaya and Kato in \cite[\S1]{fukaya-kato}.

We recall, in particular, that any pair comprising an object $C$ of $D^{\rm perf}(\mathfrak{A})$ and a morphism of non-abelian determinants $\theta: {\rm Det}_{A_E}(E\otimes_R C) \to {\rm Det}_{A_E}(0)$ gives rise to a canonical element of $K_0(\mathfrak{A},A_E)$ that we shall denote by $\chi_\mathfrak{A}(C,\theta)$.

If $E\otimes_RC$ is acyclic, then one obtains in this way a canonical element $\chi_\mathfrak{A}(C,0)$ of $K_0(\mathfrak{A},A_E)$.

More generally, if $E\otimes_RC$ is acyclic outside of degrees $a$ and $a+1$ for any integer $a$, then a choice of isomorphism of $A_E$-modules $h: E\otimes_RH^a(C) \cong E\otimes_RH^{a+1}(C)$ gives rise to a morphism $h^{\rm det}: {\rm Det}_{A_E}(E\otimes_R C) \to {\rm Det}_{A_E}(0)$ of non-abelian determinants and we set
\[ \chi_\mathfrak{A}(C,h) := \chi_\mathfrak{A}(C,h^{\rm det}).\]

We recall the following general result concerning these elements (which follows directly from \cite[Lem. 1.3.4]{fukaya-kato}) since it will be used often in the sequel.

\begin{lemma}\label{fk lemma} Let $C_1 \to C_2 \to C_3 \to C_1[1]$ be an exact triangle in $D^{\rm perf}(\mathfrak{A})$ that satisfies the following two conditions:
\begin{itemize}
\item[(i)] there exists an integer $a$ such that each $C_i$ is acyclic outside degrees $a$ and $a+1$;
\item[(ii)] there exists an exact commutative diagram of  $A_E$-modules
\[\begin{CD}
0 @> >> E\otimes_R H^a(C_1) @> >> E\otimes_R H^a(C_2) @> >> E\otimes_R H^a(C_3) @> >> 0\\
@. @V h_1VV @V h_2VV @V h_3VV \\
0 @> >> E\otimes_R H^{a+1}(C_1) @> >> E\otimes_R H^{a+1}(C_2) @> >> E\otimes_R H^{a+1}(C_3) @> >> 0\end{CD}\]
in which each row is induced by the long exact cohomology sequence of the given exact triangle and each map $h_i$ is bijective.
\end{itemize}

Then in $K_0(\mathfrak{A},A_E)$ one has $\chi_\mathfrak{A}(C_2,h_2) = \chi_\mathfrak{A}(C_1,h_1) + \chi_\mathfrak{A}(C_3,h_3)$.
\end{lemma}


%
%
%
%
%
%

\begin{remark}\label{comparingdets}{\em If $\mathfrak{A}$ is commutative, then $K_0(\mathfrak{A},A_E)$ identifies with the multiplicative group of invertible $\mathfrak{A}$-submodules of $A_E$. If, in this case, $C$ is acyclic outside degrees one and two, then for any isomorphism of $A_E$-modules $h: E\otimes_RH^1(C)\to E\otimes_RH^2(C)$ one finds that the element $\chi_{\mathfrak{A}}(C,h)$ defined above corresponds under this identification to the inverse of the ideal $\vartheta_{h}({\rm Det}_{\mathfrak{A}}(C))$ that is defined in \cite[Def. 3.1]{bst}.}\end{remark}

For convenience, we shall often abbreviate the notations $\chi_{\ZZ[G]}(C,h)$ and $\chi_{\ZZ_p[G]}(C,h)$ to $\chi_G(C,h)$ and $\chi_{G,p}(C,h)$ respectively.

When the field $E$ is clear from context, we also write $\partial_{G}$, $\partial'_{G}$, $\partial_{G,p}$ and $\partial'_{G,p}$ in place of $\partial_{\ZZ[G],E[G]}$, $\partial'_{\ZZ[G],E[G]}$, $\partial_{\ZZ_p[G],E[G]}$ and $\partial'_{\ZZ_p[G],E[G]}$ respectively.



\subsection{Statement of the conjecture}\label{statement of conj section}

In the sequel we fix a finite set of places $S$ of $k$ with
\begin{equation}\label{revisionS}S_k^\infty\cup S_k^F \cup S_k^A\subseteq S.\end{equation} We also fix cycles $\gamma_\bullet$ and differentials $\omega_\bullet$ as in \S\ref{perf sel sect}.

\subsubsection{}\label{revisionO}
%

We write $\Omega_{\omega_\bullet}(A_{F/k})$ for the element of $K_1(\RR[G])$ that is represented by the matrix of the canonical `period' isomorphism of $\RR[G]$-modules
\begin{multline*} \RR\otimes_\ZZ H_\infty(\gamma_\bullet) = \RR\otimes_\ZZ \bigoplus_{v \in S_k^\infty}H^0(k_v,Y_{v,F}\otimes_{\ZZ}H_1((A^t)^{\sigma_v}(\CC),\ZZ))\\
 \cong \RR\otimes_{\QQ} \Hom_F(H^0(A_F^t,\Omega^1_{A_F^t}),F),\end{multline*}
with respect to the ordered $\ZZ[G]$-basis of the module $H_\infty(\gamma_\bullet)$ in (\ref{25}) that is specified by (\ref{gamma basis}), and with respect to the ordered $\QQ[G]$-basis $\omega_\bullet$ of $ \Hom_F(H^0(A_F^t,\Omega^1_{A_F^t}),F)$.

This element $\Omega_{\omega_\bullet}(A_{F/k})$ constitutes a natural `$K$-theoretical period' and can be explicitly computed in terms of the classical periods that are associated to $A$ (see Lemma \ref{k-theory period} below).


To take account of the local behaviour of the differentials $\omega_\bullet$ we define a $G$-module 
\begin{equation}\label{revisionQ} \mathcal{Q}(\omega_\bullet)_S := \bigoplus_{v \notin S} \mathcal{D}_F(\mathcal{A}_v^t)/\mathcal{F}(\omega_\bullet)_v,\end{equation}
where $v$ runs over all places of $k$ that do not belong to $S$ (and each module is as in \S \ref{perf sel construct}).

It is easily seen that almost all terms in this direct sum vanish and hence that $\mathcal{Q}(\omega_\bullet)_S$ is finite. This $G$-module is also cohomologically-trivial since $\mathcal{D}_F(\mathcal{A}_v^t)$ and $\mathcal{F}(\omega_\bullet)_v$ are both free $\ZZ_{\ell(v)}[G]$-modules for each $v$ outside $S$.

Assuming the Tate-Shafarevich group $\sha(A_{F})$ to be finite, we can therefore use Proposition \ref{prop:perfect2} to define an object of $D^{\rm perf}(\ZZ[G])$ by setting
\begin{equation}\label{revisionSC} {\rm SC}_{S,\omega_\bullet}(A_{F/k}) := {\rm SC}_S(A_{F/k},\mathcal{X}_S(\omega_\bullet)) \oplus \mathcal{Q}(\omega_\bullet)_S[0],\end{equation}
where we abbreviate the perfect Selmer structure $\mathcal{X}_S(\{\mathcal{A}^t_v\}_v,\omega_\bullet,\gamma_\bullet)$ defined by the conditions (i$_\mathcal{X}$), (ii$_\mathcal{X}$) and (iii$_\mathcal{X}$) in \S\ref{perf sel sect} to $\mathcal{X}_S(\omega_\bullet)$.

We next write
 \[
h_{A,F}: A(F)\times A^t(F) \to \RR
\]
for the classical N\'eron-Tate height-pairing for $A$ over $F$.

This pairing is non-degenerate and hence, again assuming $\sha(A_{F})$ to be finite, it combines with the properties of the Selmer complex
${\rm SC}_{S}(A_{F/k},\mathcal{X}(\omega_\bullet))$ established in Proposition \ref{prop:perfect2} (i) and Proposition \ref{prop:perfect}(ii) to induce a canonical isomorphism of $\RR[G]$-modules
%
\begin{multline} \label{height triv}
 h_{A,F}': \RR\otimes_\ZZ H^1({\rm SC}_{S,\omega_\bullet}(A_{F/k})) = \RR\otimes_\ZZ A^t(F)\\ \cong \RR\otimes_\ZZ\Hom_\ZZ(A(F),\ZZ) =   \RR\otimes_\ZZ H^2({\rm SC}_{S,\omega_\bullet}(A_{F/k})).\end{multline}
This isomorphism then gives rise via the formalism recalled in \S\ref{nad sec} to a canonical element
\[ \chi_{G}({\rm SC}_{S,\omega_\bullet}(A_{F/k}),h_{A,F}) := \chi_{G}({\rm SC}_{S,\omega_\bullet}(A_{F/k}),h'_{A,F})\]
of the relative algebraic $K$-group $K_0(\ZZ[G],\RR[G])$.

Our conjecture will predict an explicit formula for this element in terms of Hasse-Weil-Artin $L$-series.

%
%


\subsubsection{}For every prime $\ell$ the reduced norm maps ${\rm Nrd}_{\QQ_\ell[G]}$ and ${\rm Nrd}_{\CC_\ell[G]}$ discussed in \S\ref{Relative $K$-theory} are bijective and so there exists a composite homomorphism
\begin{equation}\label{G,O hom} \delta_{G,\ell}: \zeta(\CC_\ell[G])^\times \to K_1(\CC_\ell[G]) \xrightarrow{\partial_{\ZZ_\ell[G],\CC_\ell[G]}}
K_0(\ZZ_\ell[G],\CC_\ell[G]) \end{equation}
in which the first map is the inverse of ${\rm Nrd}_{\CC_\ell[G]}$. This homomorphism maps $\zeta(\QQ_\ell[G])^\times$ to the subgroup $K_0(\ZZ_\ell[G],\QQ_\ell[G])$ of $K_0(\ZZ[G],\QQ[G])$.

If now $v$ is any place of $k$ that does not belong to $S$, then $v$ is unramified in $F/k$ and so the finite $G$-modules $$\kappa_{F_v}:=\prod_{w'\in S_F^v}\kappa_{F_{w'}}\,\,\,\, \text{ and } \,\,\,\,\tilde A^t_v(\kappa_{F_v}):=\prod_{w'\in S_F^v}\tilde A^t(\kappa_{F_{w'}})$$ are both cohomologically-trivial by  Lemma \ref{useful prel}(i) below. Here for any place $w'$ in $S_F^v$, $\tilde A^t$ denotes the reduction of $A^t_{/F_{w'}}$ to $\kappa_{F_{w'}}$.

For any such $v$ we may therefore define an element of the subgroup $K_0(\ZZ_{\ell(v)}[G],\QQ_{\ell(v)}[G])$ of $K_0(\ZZ[G],\QQ[G])$ by setting
\begin{equation}\label{localFM} \mu_{v}(A_{F/k}) := \chi_{G,\ell(v)}\bigl(\kappa_{F_v}^d[0]\oplus\tilde A^t_v(\kappa_{F_v})_{\ell(v)}[-1],0\bigr)-\delta_{G,\ell(v)}(L_v(A,F/k))\end{equation}
where $L_v(A,F/k)$ is the element of $\zeta(\QQ[G])^\times$ that is equal to the value at $z=1$ of the $\zeta(\CC[G])$-valued $L$-factor at $v$ of the motive $h^1(A_{F})(1)$, regarded as defined over $k$ and with coefficients $\QQ[G]$, as discussed in \cite[\S4.1]{bufl99}.

The sum
\begin{equation}\label{revisionMU} \mu_{S}(A_{F/k}) := \sum_{v\notin S}\mu_{v}(A_{F/k})\end{equation}
will play an important role in our conjecture.

We shall refer to this sum as the `Fontaine-Messing correction term' for the data $A, F/k$ and $S$ since, independently of any conjecture, the theory developed by Fontaine and Messing in \cite{fm} implies that $\mu_{v}(A_{F/k})$ vanishes for all but finitely many $v$ and hence that $\mu_S(A_{F/k})$ is a well-defined element of $K_0(\ZZ[G],\QQ[G])$. (For details see Lemma \ref{fm} below).

\subsubsection{}We write $\widehat{G}$ for the set of irreducible complex characters of $G$. In the sequel, for each $\psi$ in $\widehat{G}$ we fix a $\CC[G]$-module $V_\psi$ of character $\psi$.

We recall that a character $\psi$ in $\widehat{G}$ is said to be `symplectic' if the subfield of
$\bc$ that is generated by the values of $\psi$ is totally real
and $\End_{\br [G]}(V_\psi)$ is isomorphic to the division ring
of real Quaternions. We write $\widehat{G}^{\rm s}$ for the subset of
$\widehat{G}$ comprising such characters.

For each $\psi$ in $\widehat{G}$ we write $\check\psi$ for its contragredient character and
\begin{equation}\label{revisionIDEM} e_\psi:=\frac{\psi(1)}{|G|}\sum_{g\in G}\psi(g^{-1})g\end{equation}
for the associated central primitive idempotents of $\CC[G]$.

These idempotents induce an identification of $\zeta(\CC[G])$ with $\prod_{\widehat{G}}\CC$ and we write $x = (x_\psi)_\psi$ for the corresponding decomposition of each element $x$ of $\zeta(\CC[G])$.

For each $\psi$ in $\widehat{G}$ we write $L_{S}(A,\psi,z)$ for the Hasse-Weil-Artin $L$-series of $A$ and $\psi$, truncated by removing the Euler factors corresponding to places in $S$.

We can now state the central conjecture of this article.

\begin{conjecture}\label{conj:ebsd} 
The following claims are valid.
\begin{itemize}
\item[(i)] The group $\sha(A_F)$ is finite.
\item[(ii)] For all $\psi$ in $\widehat{G}$ the function $L(A,\psi,z)$ has an analytic continuation to $z=1$ where it has a zero of order
 $\psi(1)^{-1}\cdot {\rm dim}_{\CC}(e_\psi(\CC\otimes_\ZZ A^t(F)))$.
\item[(iii)] For all $\psi$ in $\widehat{G}^{\rm s}$ the leading coefficient $L^*_S(A,\psi,1)$ at $z=1$ of the function $L_S(A,\psi,z)$ is a strictly positive real number. In particular, there exists a unique element $L^*_{S}(A_{F/k},1)$ of $K_1(\RR[G])$ with
\[ {\rm Nrd}_{\RR[G]}(L_S^*(A_{F/k},1))_\psi = L_S^*(A,\check\psi,1)\]
for all $\psi$ in $\widehat{G}$.
\item[(iv)] In $K_0(\ZZ[G],\RR[G])$ one has
\[ \partial_G\left(\frac{L_S^*(A_{F/k},1)}{\Omega_{\omega_\bullet}(A_{F/k})}\right) = \chi_G({\rm SC}_{S,\omega_\bullet}(A_{F/k}),h_{A,F}) + \mu_{S}(A_{F/k}).\]
\end{itemize}
\end{conjecture}

In the sequel we shall refer to this conjecture as the `Birch and Swinnerton-Dyer Conjecture for the pair $(A,F/k)$' and abbreviate it to ${\rm BSD}(A_{F/k})$.



\begin{remark}{\em The assertion of ${\rm BSD}(A_{F/k})$(i) is the celebrated Shafarevich-Tate conjecture. The quantity $\psi(1)^{-1}\cdot {\rm dim}_{\CC}(e_\psi(\CC\otimes_\ZZ A^t(F)))$ is equal to the multiplicity with which the character $\psi$ occurs in the rational representation $\QQ\otimes_\ZZ A^t(F)$ of $G$ (and hence to the right hand side of the equality (\ref{dg equality})) and so the assertion of ${\rm BSD}(A_{F/k})$(ii) coincides with a conjecture of Deligne and Gross (cf. \cite[p. 127]{rohrlich}).}\end{remark}

\begin{remark}{\em Write $\tau$ for complex conjugation. Then, by the Hasse-Schilling-Maass Norm Theorem (cf. \cite[(7.48)]{curtisr}), the image of ${\rm Nrd}_{\RR[G]}$ is the subset of $\prod_{\widehat{G}}\CC^\times$ comprising $x$ with the property that $x_{\psi^\tau} = \tau(x_\psi)$ for all $\psi$ in $\widehat{G}$ and also that $x_\psi$ is a strictly positive real number for all $\psi$ in $\widehat{G}^{\rm s}$. This means that the second assertion of ${\rm BSD}(A_{F/k})$(iii) follows immediately from the first assertion, the injectivity of ${\rm Nrd}_{\RR[G]}$ and the fact that $L_S^*(A,\psi^\tau,1) = \tau(L_S^*(A,\psi,1))$ for all $\psi$ in $\widehat{G}$.

The first assertion of ${\rm BSD}(A_{F/k})$(iii) is itself motivated by the fact that if $\psi$ belongs to $\widehat{G}^{\rm s}$, and $[\psi]$ denotes the associated Artin motive over $k$, then one can show that $L^*_S(A,\psi,1)$ is a strictly positive real number whenever the motive $h^1(A)\otimes [\psi]$ validates the `Generalized Riemann Hypothesis' discussed by Deninger in \cite[(7.5)]{den}. However, since this fact does not itself provide any more evidence for ${\rm BSD}(A_{F/k})$(iii) we omit the details.} \end{remark}

\begin{remark}\label{weaker BSD}{\em It is possible to formulate a version of ${\rm BSD}(A_{F/k})$ that omits claim (iii) and hence avoids any possible reliance on the validity of the Generalized Riemann Hypothesis. To do this we recall that the argument of \cite[\S4.2, Lem. 9]{bufl99} constructs a canonical `extended boundary homomorphism' of relative $K$-theory $\delta_G: \zeta(\RR[G])^\times \to K_0(\ZZ[G],\RR[G])$ that lies in a commutative diagram

\[ \xymatrix{
K_1(\RR[G]) \ar@{^{(}->}[d]^{{\rm Nrd}_{\RR[G]}} \ar[rr]^{\hskip -0.2truein\partial_G} & & K_0(\ZZ[G],\RR[G])\\
\zeta(\RR[G])^\times . \ar[urr]^{\delta_G}}\]
%

Hence, to obtain a version of the conjecture that omits claim (iii) one need only replace the term on the left hand side of the equality in claim (iv) by the difference
\[ \delta_G\bigl(\calL_S^*(A_{F/k},1)\bigr) - \partial_G\bigl(\Omega_{\omega_\bullet}(A_{F/k})\bigr)\]
where $\calL_S^*(A_{F/k},1)$ denotes the element of $\zeta(\RR[G])^\times$ with $\calL_S^*(A_{F/k},1)_\psi = L_S^*(A,\check\psi,1)$ for all $\psi$ in $\widehat{G}$.}\end{remark}

\begin{remark}\label{rbsd etnc rem}{\em The approach developed by Wuthrich and the present authors in \cite[\S4]{bmw} can be extended to show that
 the weaker version of ${\rm BSD}(A_{F/k})$ discussed in the last remark is equivalent to the validity of the equivariant Tamagawa number conjecture for the pair $(h^1(A_F)(1),\ZZ[G])$, as formulated in \cite[Conj. 4]{bufl99} (for details see Appendix A). Taken in conjunction with the results of Venjakob and the first author in \cite{BV2}, this observation implies that the study of ${\rm BSD}(A_{F/k})$ and its consequences is relevant if one wishes to properly understand the content of the main conjecture of non-commutative Iwasawa theory, as formulated by Coates et al in \cite{cfksv}.}\end{remark}

\begin{remark}\label{cons1}{\em If, for each prime $\ell$, we fix an isomorphism of fields $\CC\cong \CC_\ell$, then the exactness of the lower row in (\ref{E:kcomm}) with $\mathfrak{A} = \ZZ_\ell[G]$ and $A_E = \CC_\ell[G]$ implies that the equality in ${\rm BSD}(A_{F/k})$(iv) determines the image of $(L^*_S(A,\psi, 1))_{\psi\in \widehat{G}}$ in $\zeta(\CC_\ell[G])^\times$ modulo the image under the reduced norm map ${\rm Nrd}_{\QQ_\ell[G]}$ of $K_1(\ZZ_\ell[G])$. In view of the explicit description of the latter image that is obtained by Kakde in \cite{kakde} (or, equivalently, by the methods of Ritter and Weiss in \cite{rw}), this means ${\rm BSD}(A_{F/k})$(iv) implicitly incorporates families of congruence relations between the leading coefficients $L^*_S(A,\psi, 1)$ for varying $\psi$ in $\widehat{G}$.}\end{remark}

\begin{remark}\label{consistency remark}{\em The formulation of ${\rm BSD}(A_{F/k})$ is consistent in the following respects.
\begin{itemize}
\item[(i)] Its validity is independent of the choices of set $S$ and ordered $\QQ[G]$-basis $\omega_\bullet$.

\item[(ii)] Its validity for the pair $(A,F/k)$ implies its validity for $(A_E,F/E)$ for any intermediate field $E$ of $F/k$ and for $(A,E/k)$ for any such $E$ that is Galois over $k$.

\item[(iii)] Its validity for the pair $(A,k/k)$ is equivalent, up to sign, to the Birch and Swinnerton-Dyer Conjecture for $A$ over $k$.
\end{itemize}
Each of these statements can be proven directly but also follows from the observation in Remark \ref{rbsd etnc rem} (see Remark \ref{consistency} for more details).}\end{remark}

\begin{remark}{\em A natural analogue of ${\rm BSD}(A_{F/k})$ has been formulated, and in some important cases proved, in the setting of abelian varieties over global function fields by Kakde, Kim and the first author in \cite{bkk}.} \end{remark}


%

Motivated at least in part by Remark \ref{rbsd etnc rem}, our main aim in the rest of this article will be to describe, and in important special cases provide evidence for, a range of explicit consequences that would follow from the validity of ${\rm BSD}(A_{F/k})$.

\subsection{$p$-components}\label{pro-p sect} To end this section we show that the equality in ${\rm BSD}(A_{F/k})$(iv) can be checked by considering separately `$p$-primary components' for each prime $p$. 

For each prime $p$ and each isomorphism of fields $j: \CC\cong \CC_p$, the inclusion $\RR \subset \CC$ combines with the functoriality of $K$-theory to induce a homomorphism
\[ K_1(\RR[G]) \to K_1(\CC_p[G])\]
and also pairs with the inclusion $\ZZ \to \ZZ_p$ to induce a homomorphism
\[ K_0(\ZZ[G],\RR[G]) \to K_0(\ZZ_p[G],\CC_p[G]).\]
In the sequel we shall, for convenience, use $j_\ast$ to denote both of these homomorphisms as well as the inclusion $\zeta(\RR[G])^\times \to \zeta(\CC_p[G])^\times$ and isomorphism $\zeta(\CC[G])^\times\cong \zeta(\CC_p[G])^\times$ that are induced by the action of $j$ on coefficients.
%


\begin{lemma}\label{pro-p lemma} Fix $S$ as in (\ref{revisionS}) and fix $\omega_\bullet$ as in \S\ref{perf sel sect}. 
Then, to verify ${\rm BSD}(A_{F/k})$(iv) it suffices to prove, for every prime $p$ and every isomorphism of fields $j:\CC\cong \CC_p$, that
\begin{multline}\label{displayed pj} \partial_{G,p}\left(\frac{j_*(L_S^*(A_{F/k},1))}{j_*(\Omega_{\omega_\bullet}(A_{F/k}))}\right) = \chi_{G,p}({\rm SC}_S(A_{F/k},\mathcal{X}(p),\mathcal{X}(\infty)_p),h^j_{A,F})\\ +\chi_{G,p}( \mathcal{Q}(\omega_\bullet)_{S,p} [0],0) + \mu_{S}(A_{F/k})_p,\end{multline}
with $\mathcal{X} = \mathcal{X}_S(\omega_\bullet)$, $h^{j}_{A,F} = \CC_p\otimes_{\RR,j}h^{{\rm det}}_{A,F}$, and $\mathcal{Q}(\omega_\bullet)_{S}$ defined as in (\ref{revisionQ}).
\end{lemma}

\begin{proof} We consider the diagonal homomorphism of abelian groups
\begin{equation}\label{local iso} K_0(\ZZ[G],\RR[G]) \xrightarrow{(\prod j_*)_p}  \prod_p\left(\prod_{j: \CC\cong \CC_p} K_0(\ZZ_p[G],\CC_p[G])\right),\end{equation}
where the products run over all primes $p$ and all choices of isomorphism $j$.

The key fact that we shall use is that this map is injective. This fact is certainly well-known but, given its importance, we shall, for completeness, prove it.

We consider the exact sequences that are given by the lower row of (\ref{E:kcomm}) with $\mathfrak{A}= R[G]$ and $A_E = E[G]$ for the
pairs $(R,E)=(\ZZ,\QQ)$, $(\ZZ,\RR)$, $(\ZZ_p,\QQ_p)$ and
$(\ZZ_p,\CC_p)$ and the maps between these sequences which are
induced by the obvious inclusions and by an embedding
$j:\RR\to\CC_p$.

By an easy diagram chase one obtains a
commutative diagram of short exact sequences
\begin{equation*}
\xymatrix{
0 \ar[r] & K_0(\ZZ[G],\QQ[G]) \ar[r] \ar[d] & K_0(\ZZ[G],\RR[G]) \ar[r] \ar[d] &
K_1(\RR[G])/K_1(\QQ[G]) \ar[r] \ar[d] & 0 \\
0 \ar[r] & K_0(\ZZ_p[G],\QQ_p[G]) \ar[r] & K_0(\ZZ_p[G],\CC_p[G]) \ar[r] &
K_1(\CC_p[G])/K_1(\QQ_p[G]) \ar[r] & 0.
}
\end{equation*}
Therefore it suffices to show that the maps
\begin{equation}
\label{equation_K_injectivity_left}
K_0(\ZZ[G],\QQ[G])\to\prod_{p,j} K_0(\ZZ_p[G],\QQ_p[G])
\end{equation}
and
\begin{equation}
\label{equation_K_injectivity_right}
K_1(\RR[G])/K_1(\QQ[G])\to\prod_{p,j}
K_1(\CC_p[G])/K_1(\QQ_p[G])
\end{equation}
are injective. The injectivity of (\ref{equation_K_injectivity_left})
follows immediately from the relevant case of the isomorphism in (\ref{decomp}).

Let $x\in K_1(\RR[G])$ be such that for all $p$ and all $j$ one has
\[ j_*(x)\in K_1(\QQ_p[G])\subseteq K_1(\CC_p[G]).\]

We now use the (injective) maps ${\rm Nrd}_{\RR[G]}$ and ${\rm Nrd}_{\QQ[G]}$ to identify $K_1(\RR[G])$ and $K_1(\QQ[G])$ with $\im({\rm Nrd}_{\RR[G]})$ and $\im({\rm Nrd}_{\QQ[G]})$ respectively.

Then, $x=\sum_{g\in G} c_gg$ is an element of $\im({\rm Nrd}_{\RR[G]})$ such that
\begin{equation}
\label{equation_lemma_K_injectivity}
j_*(x)=\sum_{g\in G} j(c_g)g\in\zeta(\QQ_p[G])^\times.
\end{equation}

We claim that $\sum_{g\in G} c_gg\in\QQ[G]$. Let $g\in G$ and consider the
coefficient $c_g$.

If, firstly, $c_g$ was transcendental over $\QQ$, then
there would be an embedding $j:\RR\to\CC_p$ such that
$j(c_g)\not\in\QQ_p$, thereby contradicting
(\ref{equation_lemma_K_injectivity}).

Therefore $c_g$ is algebraic
over $\QQ$. Now $j(c_g)\in\QQ_p$ for all $p$ and embeddings
$j$ implies that all primes are completely split in the
number field $\QQ(c_g)$ and therefore $\QQ(c_g)=\QQ$.

Hence $x$ belongs to $\im({\rm Nrd}_{\RR[G]})\cap\QQ[G]$ which, by the Hasse-Schilling-Maass Norm Theorem, is equal to $\im({\rm Nrd}_{\QQ[G]})$.

This shows the injectivity of (\ref{equation_K_injectivity_right}) and hence also of the map (\ref{local iso}).
The injectivity of (\ref{local iso}) in turn implies that the equality of ${\rm BSD}(A_{F/k})$(iv) is valid if and only if its image under each maps $j_*$ is valid.

Set $\mathcal{X} := \mathcal{X}_S(\omega_\bullet)$. Then
\begin{multline*} j_*(\chi_{G}({\rm SC}_{S,\omega_\bullet}(A_{F/k}),h_{A,F})) = \chi_{G,p}(\ZZ_p\otimes_\ZZ {\rm SC}_{S,\omega_\bullet}(A_{F/k}), \CC_p\otimes_{\RR,j}h_{A,F}))\\
=\chi_{G,p}({\rm SC}_S(A_{F/k},\mathcal{X}(p),\mathcal{X}(\infty)_p),h^j_{A,F}) +\chi_{G,p}( \mathcal{Q}(\omega_\bullet)_{S,p}[0],0),\end{multline*}
where the first equality is by definition of the map $j_*$ and the second by Proposition \ref{prop:perfect2}(i).

Given this, the claim follows from the obvious equality $j_*(\mu_{S}(A_{F/k})) = \mu_{S}(A_{F/k})_p$ and the commutativity of the diagram
\begin{equation*}\label{commute K thry} \begin{CD} K_1(\RR[G]) @> \partial_G >> K_0(\ZZ[G],\RR[G])\\
@VV j_{*} V @VV j_{*}V\\
K_1(\CC_p[G]) @> \partial_{G,p} >> K_0(\ZZ_p[G],\CC_p[G]).\end{CD}\end{equation*}
\end{proof}


\begin{remark}{\em In the sequel we shall say, for any given prime $p$, that the `$p$-primary component' {\rm BSD}$_p(A_{F/k})$(iv) of the equality in {\rm BSD}$(A_{F/k})$(iv) is valid if for every choice of isomorphism of fields $j:\CC\cong \CC_p$ the equality (\ref{displayed pj}) is valid.} \end{remark}

\section{Periods and Galois-Gauss sums}\label{k theory period sect}

To prepare for arguments in subsequent sections, we shall now explain the precise link between the $K$-theoretical period $\Omega_{\omega_\bullet}(A_{F/k})$ that occurs in ${\rm BSD}(A_{F/k})$ and the classical periods that are associated to $A$ over $k$. 

\subsection{Periods and Galois resolvents}\label{k theory period sect2} At the outset we fix an ordered $k$-basis $\{\omega'_j: j \in [d]\}$ of $H^0(A^t,\Omega^1_{A^t})$.

For each $v$ in $S_\RR^k$ we then set
\[ \Omega_{A,v}^+ := {\rm det}\left(\left(\int_{\gamma_{v,a}^{+}}\sigma_{v,*}(\omega'_b)\right)_{a,b}\right)\,\,\,\text{ and }\,\,\, \Omega_{A,v}^- := {\rm det}\left(\left(\int_{\gamma_{v,a}^{-}}\sigma_{v,*}(\omega'_b)\right)_{a,b}\right),\]
where the elements $\gamma_{v,a}^+$ and $\gamma_{v,a}^-$ of $H_1((A^t)^{\sigma_v}(\CC),\ZZ)$ are as specified in \S\ref{gamma section} and in both matrices $(a,b)$ runs over $[d]\times [d]$.

For each $v$ in $S_k^\CC$ we also set
\[ \Omega_{A,v} := {\rm det}\left(\left(\int_{\gamma_{v,a}}\sigma_{v,*}(\omega'_b),c\!\left(\int_{\gamma_{v,a}}\sigma_{v,*}(\omega'_b)\right)\right)_{a,b}\right)\]
where the elements $\gamma_{v,a}$ of $H_1((A^t)^{\sigma_v}(\CC),\ZZ)$ are again as specified in \S\ref{gamma section} and $(a,b)$ runs over $[2d]\times [d]$.

We note that, by explicitly computing integrals, the absolute values of these determinants can be explicitly related to the periods that occur in the classical formulation of the Birch and Swinnerton-Dyer conjecture (see, for example, Gross \cite[p. 224]{G-BSD}).

For each archimedean place $v$ of $k$ and character $\psi$ we then set
\[\Omega^\psi_{A,v} := \begin{cases} \Omega_{A,v}^{\psi(1)}, &\text{ if $v \in S_k^\CC$,}\\
                                            (\Omega^+_{A,v})^{1-\psi_v^-(1)}(\Omega^-_{A,v})^{\psi_v^-(1)}, &\text{ if $v \in S_k^\RR$}\end{cases} \]
with
\[\psi_v^-(1) := \psi(1) - {\rm dim}_\CC(H^0(G_w,V_\psi)),\]
where again $V_\psi$ is a fixed choice of $\CC[G]$-module of character $\psi$.

For each $\psi$ we set
\[ \Omega_A^\psi := \prod_{v \in S_k^\infty}\Omega^\psi_{A,v}\]
and we then finally define an element of $\zeta(\CC[G])^\times$ by setting
\begin{equation}\label{period def} \Omega_A^{F/k} := \sum_{\psi \in \widehat{G}}\Omega^\psi_A\cdot e_\psi\end{equation}
with the idempotents $e_\psi$ as in (\ref{revisionIDEM}).

For each $v$ in $S_k^\RR$, resp. in $S_k^\CC$, we also set
\[ w_{v,\psi} := \begin{cases} i^{\psi^-_v(1)}, &\text{if $v\in S_k^\RR$,}\\
                              i, &\text{if $v\in S_k^\CC$.}\end{cases}\]

For each character $\psi$ we then set
\[ w_{\psi} :=  \prod_{v \in S_k^\infty}w_{v,\psi}\]
and then define an element of $\zeta(\CC[G])^\times$ by setting
%
\begin{equation}\label{root number def} w_{F/k} := \sum_{\psi\in \widehat{G}} w_\psi\cdot e_\psi .\end{equation}

%

\begin{lemma}\label{k-theory period} Set $n := [k:\QQ]$. Fix an ordered $\QQ[G]$-basis $\{z_i: i \in [n]\}$ of $F$ and write $\omega_\bullet$ for the (lexicographically ordered) $\QQ[G]$-basis  $\{ z_i\otimes \omega'_j: i \in [n], j \in [d]\}$ of $H^0(A_F^t,\Omega^1_{A_F^t})$. Then in $\zeta(\RR[G])^\times$ one has
\[ {\rm Nrd}_{\RR[G]}(\Omega_{\omega_\bullet}(A_{F/k})) = \Omega_A^{F/k}\cdot w_{F/k}^d\cdot {\rm Nrd}_{\QQ[G]}
\left( \left( (\sum_{g \in G} \hat \sigma(g^{-1}(z_i)))\cdot g\right)_{\sigma\in \Sigma(k),i\in [n]}\right)^{-d} \]
where we have fixed an extension $\hat\sigma$ to $\Sigma(F)$ of each embedding $\sigma$ in $\Sigma(k)$.\end{lemma}

\begin{proof} This follows from the argument of \cite[Lem. 4.5]{bmw}.\end{proof}

\subsection{Galois resolvents and Galois-Gauss sums}

Under suitable conditions, one can also choose the $\QQ[G]$-basis $\{z_i: i \in [n]\}$ of $F$ so that the reduced norm of the Galois resolvent matrix that occurs in Lemma \ref{k-theory period} can be explicitly described in terms of Galois-Gauss sums.

Before explaining this we first recall the relevant notions of Galois-Gauss sums.

\subsubsection{}\label{mod GGS section} 

The `global Galois-Gauss sum of $F/k$' is the element
\[ \tau(F/k) :=\sum_{\psi \in \widehat{G}}\tau(\QQ,\psi)\cdot e_\psi\]
of $\zeta(\QQ^c[G])^\times$, with the idempotents $e_\psi$ as in (\ref{revisionIDEM}).

Here we regard each character $\psi$ of $G$ as a character of $G_k$ via the projection $G_k \to G$ and then write
$\tau(\QQ,\psi)$ for the global Galois-Gauss sum (as defined by Martinet in \cite{martinet})
of the induction $\chi:={\rm Ind}_k^\QQ(\psi)$ of $\psi$ to $G_\QQ$. Explicitly, one has
$$\tau(\QQ,\psi)=W(\check\chi)\cdot\sqrt{{\rm N}f(\chi)}\cdot w_\chi,$$
where $W(\check\chi)$ is the Artin root number of $\check\chi$ while $f(\chi)$ is the Artin conductor of $\chi$.


To define suitable modifications of these sums 
we then define the `unramified characteristic' of $v$ at each character $\psi$ in $\widehat{G}$ by setting
\[ u_{v,\psi} := {\rm det}(-\Phi_v^{-1}\mid V_\psi^{I_w})\in \QQ^{c,\times}.\]
%

For each character $\psi$ in $\widehat{G}$ we set
\[ u_\psi :=  \prod_{v\in S_k^F}u_{v,\psi}.\]

We then define elements of $\zeta(\QQ[G])^\times$ by setting
\begin{equation}\label{u def} u_v(F/k) := \sum_{\psi\in \widehat{G}}u_{v,\psi}\cdot e_\psi\end{equation}
and
\[ u_{F/k} := \prod_{v\in S_k^F}u_v(F/k) = \sum_{\psi\in \widehat{G}}u_\psi\cdot e_\psi.\]

We finally define the `modified global Galois-Gauss sum of $\psi$' to be the element
\[ \tau^\ast(\QQ,\psi) := u_\psi\cdot \tau(\QQ,\psi)\]
of $\QQ^c$, and the `modified global Galois-Gauss sum of $F/k$' to be the element
\[ \tau^\ast(F/k) := u_{F/k}\cdot \tau(F/k) = \sum_{\psi\in \widehat{G}}\tau^\ast(\QQ,\psi)\cdot e_\psi \]
of $\zeta(\QQ^c[G])^\times$.

%
%
%
%
%
%
%

\begin{remark}{\em The modified Galois-Gauss sums $\tau^\ast(\QQ,\psi)$ defined above play a key role in the proof of the main results of classical Galois module theory, as discussed by Fr\"ohlich in  \cite{frohlich}. In Lemma \ref{imprimitive GGS} below, one can also find a more concrete reason as to why such terms should arise naturally in the setting of leading term conjectures.} 
\end{remark}

\subsubsection{} The next result shows that under mild hypotheses the Galois-resolvent matrix that occurs in Lemma \ref{k-theory period} can be explicitly interpreted in terms of the elements $\tau^\ast(F/k)$ introduced above.

\begin{proposition}\label{lms}The following claims are valid.

\begin{itemize}
\item[(i)] For any ordered $\QQ[G]$-basis $\omega_\bullet$ of $H^0(A_F^t,\Omega^1_{A_F^t})$ there exists an element $u(\omega_\bullet)$ of $\zeta(\QQ[G])^\times$ such that
\[ {\rm Nrd}_{\RR[G]}(\Omega_{\omega_\bullet}(A_{F/k})) = u(\omega_\bullet)\cdot \Omega_A^{F/k}\cdot w_{F/k}^d\cdot \tau^\ast(F/k)^{-d}.\]

\item[(ii)] Fix a prime $p$ and set $\mathcal{O}_{F,p} := \ZZ_p\otimes_\ZZ\mathcal{O}_F.$ Then if no $p$-adic place of $k$ is wildly ramified in $F$, there is an ordered $\ZZ_p[G]$-basis $\{z^p_{i}\}_{i \in [n]}$ of $\mathcal{O}_{F,p}$ for which one has
\[ {\rm Nrd}_{\QQ_p[G]}\left( \left( (\sum_{g \in G} \hat \sigma(g^{-1}(z^p_i)))\cdot g\right)_{\sigma\in \Sigma(k),i\in [n]}\right) = \tau^\ast(F/k).\]
\end{itemize}
\end{proposition}

\begin{proof} It is enough to prove claim (i) for any choice of $\QQ[G]$-basis $\omega_\bullet$. Then, choosing $\omega_\bullet$ as in Lemma
\ref{k-theory period}, the latter result implies that it is enough to prove that the product
\[ {\rm Nrd}_{\QQ_p[G]}\left( \left( (\sum_{g \in G} \hat \sigma(g^{-1}(z_i)))\cdot g\right)_{\sigma\in \Sigma(k),i\in [n]}\right)\cdot \tau^\ast(F/k)^{-1}\]
belongs to $\zeta(\QQ[G])^\times$ and this follows from the argument used by Bley and the first author to prove \cite[Prop. 3.4]{bleyburns}.

Turning to claim (ii) we note that if no $p$-adic place of $k$ is wildly ramified in $F$, then the $\ZZ_p[G]$-module $\mathcal{O}_{F,p}$ is free of rank $n$ by Noether's Theorem (cf. \cite[\S I.3, Cor. 2]{frohlich}) and so we may fix an ordered $\ZZ_p[G]$-basis $z^p_\bullet := \{z^p_{i}: i \in [n]\}$.

The matrix
\[ M(z^p_\bullet) := (  (\sum_{g \in G} \hat \sigma(g^{-1}(z^p_b)))\cdot g)_{\sigma\in \Sigma(k),b\in [n]})\]
in ${\rm GL}_{n}(\CC_p[G])$ then represents, with respect to the bases $z^p_\bullet$ of $F_p$ and $\{\hat \sigma: \sigma \in \Sigma(k)\}$ of $Y_{F/k,p}$, the isomorphism of $\CC_p[G]$-modules
\[ \mu_{F,p}: \CC_p\otimes_{\QQ_p} F_p \cong \CC_p\otimes_{\ZZ_p}Y_{F/k,p}\]
that sends each $z\otimes f$ to $(z\hat\sigma(f))_{\sigma\in\Sigma(k)}$.

Hence one has
\begin{align*} \delta_{G,p}\bigl({\rm Nrd}_{\CC_p[G]}\bigl(M(z^p_\bullet)\bigr)\bigr) = \, &\partial_{G,p}\bigl([M(z^p_\bullet)]\bigr)\\
                                                                     = \, &[\mathcal{O}_{F,p}, Y_{F/k,p}; \mu_{F,p}]\\
                                                                     = \, &\delta_{G,p}(\tau^\ast(F/k)),\end{align*}
where $[M(z^p_\bullet)]$ denotes the class of $M(z^p_\bullet)$ in $K_1(\CC_p[G])$ and the last equality follows from the proof of \cite[Th. 7.5]{bleyburns}.

Now the exact sequence of relative $K$-theory implies that kernel of $\delta_{G,p}$ is equal to the image of $K_1(\ZZ_p[G])$ under the map ${\rm Nrd}_{\QQ_p[G]}$.

In addition, the ring  $\ZZ_p[G]$ is semi-local and so the natural map ${\rm GL}_{n}(\ZZ_p[G]) \to K_1(\ZZ_p[G])$ is surjective.

It follows that there exists a matrix $U$ in ${\rm GL}_n(\ZZ_p[G])$ with
\[ {\rm Nrd}_{\CC_p[G]}(M(z^p_\bullet))\cdot {\rm Nrd}_{\CC_p[G]}(U) = \tau^\ast(F/k)\]
and so it suffices to replace the basis $z^p_\bullet$ by its image under the automorphism of $\mathcal{O}_{F,p}$ that corresponds to the matrix $U$. \end{proof}

Taken together, Lemma \ref{k-theory period} and Proposition \ref{lms}(ii) give an explicit interpretation of the $K$-theoretical periods that occur in the formulation of {\rm BSD}$(A_{F/k})$.

However, the existence of $p$-adic places that ramify wildly in $F$ makes the situation more complicated and this leads to technical difficulties in later sections.

\section{Local points on ordinary varieties}\label{local points section}

In \S\ref{tmc} we will impose several mild hypotheses on the reduction types of $A$ and the ramification invariants of $F/k$ which together ensure that the classical Selmer complex is perfect over $\ZZ_p[G]$. Under these hypotheses, we will then give a more explicit interpretation of the equality in ${\rm BSD}(A_{F/k})$(iv).

As a necessary preparation for these results, in this section we establish several results concerning the properties of local points on varieties with good ordinary reduction.

We recall that if $M$ is a finite extension of $\QQ_p$ for some prime $p$ and $B$ is an abelian variety of dimension $d$ that is defined over $M$ and has good ordinary reduction, then the formal group $\hat B$ of $B$ is toroidal and so there exists an isomorphism of formal groups 
\[ f: \hat B \cong \mathbb{G}_m^d\]
that is defined over the valuation ring of the completion of the maximal unramified extension $M^{\rm un}$ of $M$.

The Frobenius automorphism $\Phi_M$ in $G_{M^{\rm un}/M}$ therefore acts on the coefficients of $f$, and the `twist matrix' $u$ of $B$ is defined via the equality $$f^{\Phi_M} = u\circ f.$$ We recall, in particular, that $u$ belongs to ${\rm GL}_d(\ZZ_p)$ and depends only on the reduction $\tilde B$ of $B$ (see the argument of Mazur in \cite[p. 216]{m}).


\subsection{Cohomological-triviality}In this section we assume to be given a finite Galois extension $N/M$ of $p$-adic fields and set $\Gamma := G_{N/M}$. We fix a Sylow $p$-subgroup $\Delta$ of $\Gamma$. We write $\Gamma_0$ for the inertia subgroup of $\Gamma$ and set $N_0 := N^{\Gamma_0}$.

We also assume to be given an abelian variety $B$, of dimension $d$, over $M$ that has good reduction and write $\tilde B$ for the corresponding reduced variety.

\begin{lemma}\label{useful prel} The following claims are valid. 
\begin{itemize}
\item[(i)] If $N/M$ is unramified, then the $\Gamma$-modules $B(N)$, $\tilde B(\kappa_{N})$ and $\kappa_N$ are cohomologically-trivial.
\item[(ii)] If $N/M$ is at most tamely ramified, then the $\ZZ_p[\Gamma]$-modules $B(N)^\wedge_p$ and $\tilde B(\kappa_{N})[p^\infty]$ are cohomologically-trivial.
\item[(iii)] If the variety $B$ is ordinary and $\tilde B(\kappa_{N^\Delta})[p^\infty]$ vanishes, then the $\ZZ_p[\Gamma]$-module $B(N)^\wedge_p$ is cohomologically-trivial.
\item[(iv)] Assume that $B$ is ordinary and write $u$ for its twist matrix in ${\rm GL}_{d}(\ZZ_p)$. 
If $\tilde B(\kappa_{N^\Delta})[p^\infty]$ vanishes, then $B(N)^\wedge_p$ is torsion-free, hence projective over $\ZZ_p[\Gamma]$, if and only if for any non-trivial $d$-fold vector $\underline{\zeta}$ of $p$-th roots of unity in $N^{\rm un}$ one has
\[ \Phi_N(\underline{\zeta}) \not= \underline{\zeta}^u,\]
where $\Phi_N$ is the Frobenius automorphism in $G_{N^{\rm un}/N}$. In particular, this is the case if any $p$-power root of unity in $N^{\rm un}$ belongs to $N$.
\end{itemize}
\end{lemma}

\begin{proof} A standard Hochschild-Serre spectral sequence argument combines with the criterion of \cite[Thm. 9]{cf} to show that  claim (i) is valid provided that each of the modules $B(N)$, $\tilde B(\kappa_{N})$ and $\kappa_N$ is cohomologically-trivial with respect to every  subgroup $C$ of $\Gamma$ of prime order (see the proof of \cite[Lem. 4.1]{bmw} for a similar argument).

We therefore fix a subgroup $C$ of $\Gamma$ that has prime order. Now cohomology over $C$ is periodic of order 2 and each of the modules $B(N)$, $\tilde B(\kappa_{N})$ and $\kappa_N$ span free $\QQ[\Gamma]$-modules.
It thus follows from \cite[Cor. to Prop. 11]{cf} that the Herbrand Quotient with respect to $C$ of each of these modules is equal to 1.
To prove claim (i) it is enough to show that the natural norm maps $B(N)\to B(N^C)$, $\tilde B(\kappa_{N}) \to \tilde B(\kappa_{N^C})$ and $\kappa_{N}\to \kappa_{N^C}$ are surjective.

Since the extension $N/N^C$ is unramified, this surjectivity is well-known for the module $\kappa_N$ and for the modules $B(N)$ and $\tilde B(\kappa_{N})$ it follows directly from the result of Mazur in \cite[Cor. 4.4]{m}.

To prove claim (ii) we assume that $N/M$ is tamely ramified. In this case the order of $\Gamma_0$ is prime to $p$ and so the same standard Hochschild-Serre spectral sequence argument as in claim (i) implies claim (ii) is true if the modules $B(N_0)^\wedge_p = (B(N)^\wedge_p)^{\Gamma_0}$ and $\tilde B(\kappa_{N})[p^\infty] = \tilde B(\kappa_{N})[p^\infty]^{\Gamma_0}$ are cohomologically-trivial with respect to every subgroup $C$ of $\Gamma/\Gamma_0$ of order $p$. Since $N_0/N_0^C$ is unramified, this follows from the argument in claim (i).

In a similar way, to prove claim (iii) one is reduced to showing that if $\tilde B(\kappa_{N^\Delta})[p^\infty]$ vanishes, then for each subgroup $C$ of $\Gamma_0$ of order $p$, the norm map ${\rm N}_C: B(N)^\wedge_p \to B(N^C)^\wedge_p$ is surjective.

Now the main result of Lubin and Rosen in \cite{LR} implies that the cokernel of ${\rm N}_C$ is isomorphic to the cokernel of the natural action of ${\rm I}_d-u$ on the direct sum of $d$-copies of $C$ and from the proof of \cite[Th. 2]{LR} one knows that ${\rm det}({\rm I}_d-u)$ is a $p$-adic divisor of $|\tilde B(\kappa_N)|$.
 But if $\tilde B(\kappa_{N^\Delta})[p^\infty]$ vanishes, then $\tilde B(\kappa_{N})[p^\infty]$ also vanishes (as $\Delta$ is a $p$-group) and so
 ${\rm det}({\rm I}_d-u)$ is a $p$-adic unit. It follows that ${\rm cok}(N_C)$ vanishes, as required to prove claim (iii).

To prove claim (iv) we assume $\tilde B(\kappa_{N^\Delta})[p^\infty]$ vanishes. Then claim (ii) implies $B(N)^\wedge_p$ is a projective  $\ZZ_p[\Gamma]$-module if and only if $B(N)^\wedge_p[p^\infty]$ vanishes. In addition, from the lemma in \cite[\S1]{LR} (with $L = K = N$), we know that the group $B(N)^\wedge_p[p^\infty]$ is isomorphic to the subgroup of $(N^{{\rm un},\times})^d$ comprising $p$-torsion elements $\underline{\eta}$ which satisfy
$\Phi_N(\underline{\eta}) = \underline{\eta}^u$.

This directly implies the first assertion of claim (iv) and the second assertion then follows  because ${\rm det}({\rm I}_d-u)$ is a $p$-adic unit and so $u\not\equiv 1$ (mod $p$). \end{proof}

\begin{remark}{\em A more general analysis of the cohomological properties of formal groups was recently given by Ellerbrock and Nickel in \cite{ellerbrocknickel}.}\end{remark}

\subsection{The \'etale discrepancy class}\label{twist inv prelim}

In this section we fix an abelian variety $B$ over $M$ of dimension $d$.
We assume $B$ has good ordinary reduction and is such that $\tilde B(\kappa_{N^\Delta})[p^\infty]$ vanishes, where, as in the previous section, $\Delta$ is a fixed Sylow $p$-subgroup of $\Gamma:=G_{N/M}$.

Under these hypotheses, we use Lemma \ref{useful prel} to define a $K$-theoretical invariant of the twist matrix of $B$ that, in essence, measures the discrepancy between the image of the formal logarithm of $B$ and a classical invariant that arises from the \'etale cohomology of $\mathbb{G}_m$. In particular, since, in many cases, this `\'etale discrepancy class' either vanishes or can be computed explicitly (see Proposition \ref{basic props} and the observation following (\ref{curve local eps conj}) below), and the invariant related to $\mathbb{G}_m$ is already much studied in the literature, the material in this section will help us, under suitable hypotheses, to compute explicitly the difference between Nekov\'a\v r-Selmer complexes and classical Selmer complexes, and thereby play a key role in the derivation of explicit predictions from ${\rm BSD}(A_{F/k})$.

At the outset we recall that the \'etale cohomology complex $R\Gamma(N,\ZZ_p(1))$ belongs to $D^{\rm perf}(\ZZ_p[\Gamma])$. Hence, following Lemma \ref{useful prel}(iii), we obtain a complex in $D^{\rm perf}(\ZZ_p[\Gamma])$ by setting
\[ C_{B,N}^{\bullet} := R\Gamma(N,\ZZ_p(1))^d[1] \oplus B(N)^\wedge_p[-1].\]

This complex is acyclic outside degrees zero and one. In addition, Kummer theory gives an identification $H^1(N,\ZZ_p(1)) = (N^\times)^\wedge_p$ and the invariant map ${\rm inv}_N$ of $N$ an isomorphism $ H^2(N,\ZZ_p(1)) \cong \ZZ_p$.

We next fix a choice of isomorphism of $\QQ_p[\Gamma]$-modules $\lambda_{B,N}$  which lies in a commutative diagram 
\begin{equation}\label{lambda diag}\begin{CD}
0 @> >> \QQ_p\cdot (U^{(1)}_{N})^d @> \subset >> \QQ_p\cdot H^0(C_{B,N}^{\bullet}) @> ({\rm val}_N)^d >> \QQ_p^d @> >> 0\\
@. @V {\rm exp}_{B,N}VV @V \lambda_{B,N} VV @V \times f_{N/M}VV\\
0 @> >> \QQ_p \cdot B(N)^\wedge_p @> \subset >> \QQ_p\cdot H^1(C_{B,N}^{\bullet}) @> {\rm can} >> \QQ_p^d @> >> 0.\end{CD}\end{equation}
Here $U_N^{(1)}$ is the group of 1-units of $N$, ${\rm val}_N: \QQ_p\cdot (N^\times)^\wedge_p\to \QQ_p$ is the canonical valuation map on $N$, $f_{N/M}$ is the residue degree of $N/M$, `{\rm can}' is induced by ${\rm inv}_N$ and ${\rm exp}_{B,N}$ is the composite isomorphism
\[ \QQ_p\cdot (U^{(1)}_{N})^d \cong N^d \cong \QQ_p\cdot B(N)^\wedge_p\]
where the first isomorphism is induced by the $p$-adic logarithm on $N$ and the second by the exponential map of the formal group of $B$ over $N$.

We now introduce a useful general convention: for each element $x$ of $\zeta(\CC_p[\Gamma])$ we write $^\dagger x$ for the unique element of $\zeta(\CC_p[\Gamma])^\times$ with the property that for each $\mu$ in $\widehat{\Gamma}$ one has
\begin{equation}\label{dagger eq} e_\mu (^\dagger x) = \begin{cases} e_\mu x, &\text{ if $e_\mu x\not= 0$,}\\
  e_\mu, &\text{ otherwise.}\end{cases}\end{equation}
(This construction is written as $x \mapsto ^*\!\! x$ in \cite{bleyburns,breuning2}).

We then define an element
\[ c_{N/M} := \frac{^\dagger((|\kappa_M|-\Phi_{N/M})e_{\Gamma_0})}{^\dagger((1-\Phi_{N/M})e_{\Gamma_0})}\]
of $\zeta(\QQ[\Gamma])^\times$. Here and in the sequel, $\Phi_{N/M}$ is a fixed lift to $\Gamma$ of the Frobenius automorphism in $\Gamma/\Gamma_0$ and, for any subgroup $J$ of $\Gamma$, $e_{J}$ denotes the idempotent $(1/|J|)\sum_{\gamma\in J}\gamma$.

We may finally define our desired class in $K_0(\ZZ_p[\Gamma],\QQ_p[\Gamma])$. 
%
%
%
\begin{definition}\label{etaledisc}{\em Assume that $B/M$ has good ordinary reduction and $\tilde B(\kappa_{N^\Delta})[p^{\infty}]$ vanishes. The `\'etale discrepancy class' of $(B,N/M)$ is the element
\[ R_{N/M}(\tilde B) := \chi_{\Gamma,p}(C_{B,N}^{\bullet},\lambda_{B,N}) +d\cdot\delta_{\Gamma,p}(c_{N/M}) \]
of $K_0(\ZZ_p[\Gamma],\QQ_p[\Gamma])$.
}\end{definition}

\begin{remark}{\em If $v$ is a non-archimedean place of $k$ such that the restriction $A_v$ of $A$ to $k_v$ satisfies the hypotheses in Definition \ref{etaledisc}, then we shall refer to the \'etale discrepancy class of the data $(A_v,F_w/k_v)$ as the `\'etale discrepancy class of $A_F$ at $v$'.}\end{remark}
%

%

The basic properties of the \'etale discrepancy class are described by the following result.

\begin{proposition}\label{basic props} Assume $B$ is ordinary and $\tilde B(\kappa_{N^\Delta})[p^{\infty}]$ vanishes.
\begin{itemize}
\item[(i)] $R_{N/M}(\tilde B)$ depends only upon $N/M$ and the reduced variety $\tilde B$.
\item[(ii)] $R_{N/M}(\tilde B)$ has finite order.
\item[(iii)] If $N/M$ is tamely ramified, then 
 $R_{N/M}(\tilde B)$ vanishes.
\end{itemize}
\end{proposition}

\begin{proof}  We set $C^\bullet := C_{B,N}^\bullet$, $\lambda := \lambda_{B,N}$ and $f := f_{N/M}$, and also write $\wp = \wp_N$ for the maximal ideal in the valuation ring of $N$.

Then, whilst $\lambda$ can be chosen in many different ways to ensure that (\ref{lambda diag}) commutes, it is straightforward to check that $\chi_{\Gamma,p}(C^{\bullet},\lambda)$ is independent of this choice. The fact that this element depends only on the (twist matrix of the) reduced variety $\tilde B$ follows from Lemma \ref{twist dependence} below. This proves claim (i).

It is convenient to prove claim (iii) first and so we assume $N/M$ is tamely ramified. Then, writing $\hat B$ for the formal group of $B$, in this case Lemma \ref{ullom}(i) below implies that for each natural number $n$ the $\ZZ_p[\Gamma]$-modules $U^{(n)}:= \mathbb{G}_m(\wp^n)$ and $V^{(n)}:= \hat B(\wp^n)$ are cohomologically-trivial and there exist exact triangles in $D^{\rm perf}(\ZZ_p[\Gamma])$ of the form
\begin{equation*}\label{use tri}
\begin{cases} (U^{(n)})^d[0]\oplus V^{(n)}[-1]\xrightarrow{\alpha} C^{\bullet}\to C_{\alpha}^\bullet \to (U^{(n)})^d[1]\oplus V^{(n)}[0]\\
U^{(1)}[0] \xrightarrow{\beta} R\Gamma(N,\ZZ_p(1))[1] \to C^\bullet_{N,1} \to U_N^1[1]\\
(U^{(1)}/U^{(n)})^d[0] \xrightarrow{\gamma} C_{\alpha}^\bullet \to (C^\bullet_{N,1})^d\oplus (V^{(1)}/V^{(n)})[-1] \to (U^{(1)}/U^{(n)})^d[1].\end{cases}\end{equation*}
Here $\alpha$ is the unique morphism such that $H^0(\alpha)$ and $H^1(\alpha)$ are respectively induced by the inclusions $U^{(n)}\subset (N^\times)^\wedge_p$ and $V^{(n)} \subseteq B(N)^\wedge_p$ and so that the cohomology sequence of the first triangle induces identifications of $H^0(C_{\alpha}^\bullet)$ and $H^1(C_{\alpha}^\bullet)$ with $((N^\times)^\wedge_p/U^{(n)})^d$ and $\ZZ_p^d \oplus V^{(1)}/V^{(n)}$; $\beta$ is the unique morphism so that $H^0(\beta)$ is induced by the inclusion $U^{(1)} \subset (N^\times)^\wedge_p$ and so the cohomology sequence of the second  triangle induces identifications of $H^0(C_{N,1}^\bullet)$ and $H^1(C_{N,1}^\bullet)$ with $(N^\times)^\wedge_p/U^{(1)}$ and $\ZZ_p$ respectively; $\gamma$ is the unique morphism so that $H^{0}(\gamma)$ is the inclusion $(U^{(1)}/U^{(n)})^d \subset H^0(C_{\alpha}^\bullet)$.

In particular, if $n$ is sufficiently large, then we may apply Lemma \ref{fk lemma} to the first and third of the above triangles to deduce that

\begin{align}\label{interm} &\chi_{\Gamma,p}(C^{\bullet},\lambda)\\ = \,&\chi_{\Gamma,p}((U^{(n)})^d[0]\oplus V^{(n)}[-1],{\rm exp}_{B,N}) +  \chi_{\Gamma,p}(C_{\alpha}^{\bullet},\lambda_{\alpha})\notag\\
= \,&\chi_{\Gamma,p}(C_{\alpha}^{\bullet},\lambda_{\alpha})\notag\\
= \,& \chi_{\Gamma,p}((U^{(1)}/U^{(n)})^d[0],0) + d\cdot \chi_{\Gamma,p}(C_{N,1}^{\bullet},f\cdot{\rm val}_N) + \chi_{\Gamma,p}((V^{(1)}/V^{(n)})[-1],0)\notag\\
= \,& \chi_{\Gamma,p}((U^{(1)}/U^{(n)})^d[0],0) + d\cdot \chi_{\Gamma,p}(C_{N,1}^{\bullet},f\cdot{\rm val}_N) - \chi_{\Gamma,p}((V^{(1)}/V^{(n)})[0],0)\notag\\
= \, &d\cdot \chi_{\Gamma,p}(C_{N,1}^{\bullet},f\cdot{\rm val}_N),\notag
 \end{align}
where we write $\lambda_{\alpha}$ for the isomorphism of $\QQ_p[\Gamma]$-modules
\[ \QQ_p\cdot H^0(C_{\alpha}^\bullet) = \QQ_p\cdot ((N^\times)^\wedge_p/U^{(n)})^d \cong \QQ_p^d = \QQ_p\cdot H^1(C_{\alpha}^\bullet)\]
that is induced by the map $f\cdot {\rm val}_N$ and the second and last equalities in (\ref{interm}) follow from Lemma \ref{ullom}.

But $$\chi_{\Gamma,p}(C_{N,1}^{\bullet},f\cdot {\rm val}_N) = \chi_{\Gamma,p}(C_{N,1}^{\bullet},{\rm val}_N) + \delta_{\Gamma,p}(^\dagger(f\cdot e_\Gamma))$$ whilst from \cite[Th. 4.3]{bleyburns} one has

\[ \chi_{\Gamma,p}(C_{N,1}^{\bullet},{\rm val}_N) = -\delta_{\Gamma,p}(c_{N/M}\cdot ^\dagger\!(f\cdot e_\Gamma))= - \delta_{\Gamma,p}(c_{N/M}) - \delta_{\Gamma,p}(^\dagger\!(f\cdot e_\Gamma)).\]
Claim (iii) is thus obtained by substituting these facts into the equality (\ref{interm}).


To deduce claim (ii) from claim (iii) we recall (from \cite[Thm. 4.1]{ewt}) that an element $\xi$ of $K_0(\ZZ_p[\Gamma],\QQ_p[\Gamma])$ has finite order if and only if for cyclic subgroup $\Upsilon$ of $\Gamma$ and every quotient $\Omega = \Upsilon/\Upsilon'$ of
order prime to $p$ one has $(q^\Upsilon_{\Omega}\circ\rho^{\Gamma}_\Upsilon)(\xi) = 0$. Here
$$\rho^{\Gamma}_\Upsilon:K_0(\ZZ_p[\Gamma],\QQ_p[\Gamma])\to K_0(\ZZ_p[\Upsilon],\QQ_p[\Upsilon])$$ is the natural restriction map, and 

$$q^\Upsilon_{\Omega}:K_0(\ZZ_p[\Upsilon],\QQ_p[\Upsilon])\to K_0(\ZZ_p[\Omega],\QQ_p[\Omega])$$
maps the class of a triple $(P,\phi,Q)$ to the class of $(P^{\Upsilon'},\phi^{\Upsilon'},Q^{\Upsilon'})$. 

Since the extension $N^{\Upsilon'}/N^\Upsilon$ is tamely ramified, it is thus enough to show that
\[ (q^\Upsilon_{\Omega}\circ\rho^{\Gamma}_\Upsilon)(R_{N/M}(\tilde B))=R_{N^{\Upsilon'}/N^\Upsilon}(\tilde B).\]

This is proved by a routine computation in relative $K$-theory that uses the same ideas as in \cite[Rem. 2.9]{breuning2}.
 In fact, the only point worth mentioning explicitly in this regard is that if $\Gamma'$ is normal in $\Gamma$, and we set $N' := N^{\Gamma'}$, then the natural projection isomorphism $\iota:\ZZ_p[\Gamma/\Gamma']\otimes^{\mathbb{L}}_{\ZZ_p[\Gamma]}R\Gamma(N,\ZZ_p(1)) \cong R\Gamma(N',\ZZ_p(1))$ in $D^{\rm perf}(\ZZ_p[\Gamma/\Gamma'])$ gives a commutative diagram of (trivial) $\QQ_p[\Gamma/\Gamma']$-modules

\[ \begin{CD} \QQ_p\cdot H^2(N,\ZZ_p(1))^{\Gamma'} @> {\rm inv}_N >> \QQ_p\\
@V H^2(\iota)VV @VV \times f_{N/N'}V\\
\QQ_p\cdot H^2(N',\ZZ_p(1)) @> {\rm inv}_{N'}>> \QQ_p.\end{CD}\]
%
\end{proof}

\begin{lemma}\label{ullom} If $N/M$ is tamely ramified, the following claims are valid for all natural numbers $a$.
\begin{itemize}
\item[(i)] The $\ZZ_p[\Gamma]$-modules $U^{(a)}$ and $V^{(a)}$ are cohomologically-trivial.
\item[(ii)] One has $d\cdot \chi_{\Gamma,p}((U^{(1)}/U^{(a)})[0],0) = \chi_{\Gamma,p}((V^{(1)}/V^{(a)})[0],0)$.
\item[(iii)] For all sufficiently large $a$ one has $\chi_{\Gamma,p}((U^{(a)})^d[0]\oplus V^{(a)}[-1],{\rm exp}_{B,N})=0$.
\end{itemize}
\end{lemma}

\begin{proof} The key fact in this case is that for every integer $i$ the $\ZZ_p[\Gamma]$-module $\wp^i$ is cohomologically-trivial (by Ullom \cite{Ullom}).

In particular, if we write $i_0$ for the least integer with $i_0 \ge e/(p-1)$, where $e$ is the ramification degree of $N/\QQ_p$, then for any integer $a \ge i_0$ the formal logarithm ${\rm log}_{B}$ and $p$-adic exponential map restrict to give isomorphisms of $\ZZ_p[\Gamma]$-modules
\begin{equation}\label{iso1} V^{(a)}\cong (\wp^{a})^d,\,\,\,\,\,\,\,\,\wp^a \cong U^{(a)}\end{equation}
and so the $\ZZ_p[\Gamma]$-modules $V^{(a)}$ and $U^{(a)}$ are cohomologically-trivial.

In addition, for all $a$ the natural isomorphisms
\begin{equation}\label{iso2} U^{(a)}/U^{(a+1)} \cong\wp^a/\wp^{a+1},\,\,\,\,\,\,\,\, \bigl(\wp^a/\wp^{a+1}\bigr)^d \cong V^{(a)}/V^{(a+1)}\end{equation}
imply that these quotient modules are also cohomologically-trivial. By using the tautological exact sequences for each $a < i_0$
\begin{equation}\label{filter1} \begin{cases} &0 \to U^{(a+1)}/U^{(i_0)} \to U^{(a)}/U^{(i_0)} \to U^{(a)}/U^{(a+1)} \to 0,\\
 &0 \to V^{(a+1)}/V^{(i_0)} \to V^{(a)}/V^{(i_0)} \to V^{(a)}/V^{(i+1)} \to 0\end{cases}\end{equation}
%
one can therefore deduce (by a downward induction on $a$, starting at $i_0$) that all modules $U^{(a)}$ and $V^{(a)}$ are cohomologically-trivial. This proves claim (i).

In addition, by repeatedly using the exact sequences (\ref{filter1}) and isomorphisms (\ref{iso2}) one computes that $d\cdot \chi_{\Gamma,p}((U^{(1)}/U^{(a)})[0],0)$ is equal to

\begin{align*} d\cdot\sum_{b=1}^{b=a-1}\chi_{\Gamma,p}((U^{(b)}/U^{(b+1)})[0],0)                                                            = \, &\sum_{b=1}^{b=a-1}\chi_{\Gamma,p}(\bigl((U^{(b)}/U^{(b+1)})\bigr)^d[0],0)\\
                                                            = \, &\sum_{b=1}^{b=a-1}\chi_{\Gamma,p}((V^{(b)}/V^{(b+1)})[0],0)\\
                                                            = \, &\chi_{\Gamma,p}((V^{(1)}/V^{(b)})[0],0),\end{align*}
as required go prove claim (ii).

Finally, claim (iii) is a direct consequence of the isomorphisms (\ref{iso1}).
\end{proof}

\begin{lemma}\label{twist dependence} Let $B$ and $B'$ be abelian varieties over $M$, of the same dimension $d$, that have good ordinary reduction and are such that $\tilde B(\kappa_{N^\Delta})[p^\infty]$ and $\tilde B'(\kappa_{N^\Delta})[p^{\infty}]$ both vanish. Then the following claims are valid.

\begin{itemize}
\item[(i)] The $\ZZ_p[\Gamma]$-modules $B(N)^\wedge_p$ and $B'(N)^\wedge_p$ are cohomologically-trivial and the formal group logarithms induce an isomorphism of $\QQ_p[\Gamma]$-modules
\[ \QQ_p\cdot B(N)^\wedge_p \xrightarrow{{\rm log}_{B,N}} N^d \xrightarrow{{\rm exp}_{B',N}} \QQ_p\cdot B'(N)^\wedge_p.\]

\item[(ii)] If the reduced varieties $\tilde B$ and $\tilde B'$ are isomorphic, then in $K_0(\ZZ_p[\Gamma],\QQ_p[\Gamma])$ one has
\[ \chi_{\Gamma,p}(B(N)^\wedge_p[0] \oplus B'(N)^\wedge_p[-1], {\rm exp}_{B',N}\circ {\rm log}_{B,N}) = 0.\]
%
\end{itemize}
\end{lemma}

\begin{proof} Claim (i) follows directly from Lemma \ref{useful prel}(iii).

To prove claim (ii) we write $M^{\rm un}$ for the  maximal unramified extension of $M$, $\hat M^{\rm un}$ for its completion and $\mathcal{O}$ for the valuation ring of $\hat M^{\rm un}$. We write $\Phi_M$ for the Frobenius automorphism in $G_{M^{\rm un}/M}$.

Then the formal groups $\hat B$ of $B$ and $\hat B'$ of $B'$ are toroidal and, if $\tilde B'$ is isomorphic to $\tilde B$, there exist isomorphisms of formal groups $f_1: \hat B \cong \mathbb{G}_m^d$ and $f_2: \hat B' \cong \mathbb{G}_m^d$ over $\mathcal{O}$ that satisfy $f_1^{\Phi_M} = u\circ f_1$ and $f_2^{\Phi_M} = u\circ f_2$, where $u$ is the twist matrix of $B$ in ${\rm GL}_d(\ZZ_p)$.



We now consider the isomorphism $\phi:= f_2^{-1}\circ f_1$ from $\hat B$ to $\hat B'$ over $\mathcal{O}$. We fix an element $x$ of $\mathcal{O}_N\mathcal{O}^{\rm un}$ and an element $g$  of $G_{N^{\rm un}/M}$ whose image in $G_{M^{\rm un}/M}$ is an integral power $\Phi_M^a$ of $\Phi_M$.

Then one has  
\begin{multline*} g(\phi(x)) = \phi^g(g(x)) = ((f_2^{\Phi_M^a})^{-1}\circ f_1^{\Phi_M^a})(g(x))\\ = ((u^a\circ f_2)^{-1}\circ (u^a\circ f_1))(g(x)) =
(f_2^{-1}\circ f_1)(g(x)) = \phi(g(x)).\end{multline*}

This means that $\phi$ is an isomorphism of $\ZZ_p[[G_{N^{\rm un}/M}]]$-modules and so restricts to give an isomorphism of $\Gamma$-modules
\[ \hat B(N) = \hat B(N^{\rm un})^{G_{N^{\rm un}/N}} \cong \hat B'(N^{\rm un})^{G_{N^{\rm un}/N}} = \hat B'(N).\]

Upon passing to pro-$p$-completions, and noting that the groups $\tilde B(\kappa_{N})[p^\infty]$ and $\tilde B'(\kappa_{N})[p^\infty]$ vanish, we deduce that $\phi$ induces an isomorphism of $\ZZ_p[\Gamma]$-modules
\[ \phi_p: B(N)^\wedge_p = \hat B(N)^\wedge_p\cong \hat B'(N)^\wedge_p = B'(N)^\wedge_p.\]

There is also a commutative diagram of formal group isomorphisms

\[
\begin{CD} \hat B @> f_1  >> \mathbb{G}_m^d @> (f_2)^{-1} >> \hat B' \\
@V {\rm log}_B VV @V {\rm log}_{\mathbb{G}_m} VV @VV {\rm log}_{B'}V\\
\mathbb{G}_a^d @> \times f_1'(0)  >> \mathbb{G}_a^d @> \times f_2'(0)^{-1} >> \mathbb{G}_a^d\end{CD}\]

Taken together with the isomorphism $\phi_p$, this diagram implies that the element
\[ \chi_{\Gamma,p}(B(N)^\wedge_p[0] \oplus B'(N)^\wedge_p[-1], {\rm exp}_{B',N}\circ {\rm log}_{B,N}) \]
is equal to the image under $\partial_{\Gamma,p}$ of the automorphism of the $\QQ_p[\Gamma]$-module $N^d$ that corresponds to the matrix
 $f_2'(0)^{-1}f_1'(0)$.

It is thus enough to note that, since the latter matrix belongs to ${\rm GL}_d(\mathcal{O}_M)$ it is represented by a matrix in ${\rm GL}_{d[M:\QQ_p]}(\ZZ_p[\Gamma])$ and so belongs to the kernel of $\partial_{\Gamma,p}$, as required. \end{proof}


\subsection{Elliptic curves}\label{ell curve sect} If $B$ is an (ordinary) elliptic curve over $\QQ_p$, then it is possible in certain cases to formulate a precise conjectural formula for the \'etale discrepancy class $R_{N/M}(\tilde B)$ of Definition \ref{etaledisc}.

This aspect of the theory will be considered in detail elsewhere. However, to give a brief idea of the general approach we recall that, 
if $\Phi$ is the Frobenius automorphism in $G_{\Qu_p^{\rm un} / \Qp}$, the twist matrix of $B$ is the unique element $u$ of $\ZZ_p^\times$ for which the composite $f^\Phi\circ f^{-1}$ is equal to the endomorphism $[u]_{\mathbb{G}_m}$ of $\mathbb{G}_m$, for any given isomorphism of formal groups $f: \hat{B} \to \mathbb{G}_m$. 


\begin{lemma} $\hat{B}$ is a Lubin-Tate formal group with respect to the parameter $u^{-1}p$.
\end{lemma}
\begin{proof} By using the equalities
\[
f^\Phi \circ [u^{-1}p]_{\hat{B}} \circ f^{-1} = [u^{-1}p]_{\mathbb{G}_m} \circ f^\Phi \circ f^{-1} = [p]_{\mathbb{G}_m}
\]
one computes that

\begin{eqnarray*}
  [u^{-1}p]_{\hat{B}} &=& \left( f^\Phi \right)^{-1} \circ [p]_{\mathbb{G}_m} \circ f \\
 &\equiv& \left( f^\Phi \right)^{-1} \circ X^p \circ f \pmod{p} \\
&=& \left( f^{-1} \right)^{\Phi} \left( f(X)^p \right)  \\
&=& \left( f^{-1}(f(X)) \right)^p \\
&=& X^p.
\end{eqnarray*}

Thus, since $ [u^{-1}p]_{\hat{B}}  \equiv u^{-1}p X \pmod{\deg 2}$, it follows that $[u^{-1}p]_{\hat{B}} $ is a Lubin-Tate power series with respect to $u^{-1}p$, as claimed.
\end{proof}


We write $\chi^{\rm ur}$ for the restriction to $G_M$ of the character
\[
\chi_\Qp^{\rm ur} \colon G_\Qp \lra \Ze_p^\times, \quad \Phi \mapsto u^{-1}.
\]

We assume that the restriction of $\chi^{\rm ur}$ to $G_N$ is non-trivial and write $T$ for the (unramified) twist $\Zp(\chi^{\rm ur})(1)$ of the representation $\Zp(1)$. Then, by \cite[Prop.~2.5]{IV} or \cite[Lem. 3.2.1]{BC2}, the complex $R\Gamma(N, T)$ is acyclic outside
degrees one and two and there are canonical identifications
\[
 H^i(N, T) = \begin{cases} \hat{B}(\frp_N), &\text{ if $i=1$,}\\
  \bigl( \Zp / p^{\omega_N} \Zp \bigr) (\chi^{ur}),&\text{ if $i=2$,}\end{cases}\]
where $\omega_N$ denotes the $p$-adic valuation of the element $1 - \chi^{ur}(\Phi^{f_{N/\Qp}})$.

These explicit descriptions allow one to interpret the \'etale discrepancy class $R_{N/M}(\tilde B)$ in terms of differences between elements that occur in the formulations of the local epsilon constant conjecture for the representations $\ZZ_p(1)$ and $T$, as studied by Benois and Berger \cite{benoisberger}, Bley and Cobbe \cite{BC2} and Izychev and Venjakob \cite{IV}.

In this way one finds that the (assumed) compatibility of these conjectures for the representations $\ZZ_p(1)$  and $T$ implies the following equality

\begin{equation}\label{curve local eps conj} R_{N/M}(\tilde B) = \delta_{\Gamma,p}\bigl(\bigl(\sum_{\chi\in\widehat{G}}u^{f_{M/\Qp}(s_M\chi(1) + m_\chi)}e_\chi\bigr)
\frac{^\dagger((1 - (u\cdot\Phi^{-1})^{f_{M/\Qp}})e_{\Gamma_0})}
{^\dagger((|\kappa_M| - (u\cdot\Phi^{-1})^{-f_{M/\Qp}})e_{\Gamma_0})}\bigr)
\end{equation}
where the conductor of each character $\chi$ is $\pi_M^{m_\chi}\calO_M$ and the different of $M/\QQ_p$ is $\pi_M^{s_M}\calO_M$, the idempotents $e_\chi$ are as in (\ref{revisionIDEM}) and $e_{\Gamma_0}=(1/|\Gamma_0|)\sum_{\gamma\in\Gamma_0}\gamma$.

In particular, the results of \cite{BC2} imply that the equality (\ref{curve local eps conj}) is unconditionally valid for certain natural families of wildly ramified extensions $N/M$.

\section{Classical Selmer complexes and refined BSD}\label{tmc}

In this section we study ${\rm BSD}(A_{F/k})$ under the assumption that $A$ and $F/k$ satisfy the following list of hypotheses.

In this list we fix an {\em odd} prime number $p$ and an intermediate field $K$ of $F/k$ such that $\Gal(F/K)$ is a Sylow $p$-subgroup of $G$.
\begin{itemize}
\item[(H$_1$)] The Tamagawa number of $A_{K}$ at each place in $S_K^A$ is not divisible by $p$;
\item[(H$_2$)] $S_K^A \cap S_K^p = \emptyset$ (that is, no place of bad reduction for $A_{K}$ is $p$-adic);
\item[(H$_3$)] For all $v$ in $S_K^p$ above a place in $S_k^F$ the reduction is ordinary and $A(\kappa_v)[p^\infty]$ vanishes;
\item[(H$_4$)] For all $v$ in $S_K^f\setminus S_K^p$ above a place in $S_k^F$ the group $A(\kappa_v)[p^\infty]$ vanishes;
\item[(H$_5$)] $S_k^A\cap S_k^F = \emptyset$ (that is, no place of bad reduction for $A$ is ramified in $F$);
\item[(H$_6$)] $\sha(A_F)$ is finite.
\end{itemize}

\begin{remark}\label{satisfying H} {\em For a fixed abelian variety $A$ over $k$ and extension $F/k$ the hypotheses (H$_1$) and (H$_2$) are clearly satisfied by all but finitely many odd primes $p$
, (H$_4$) and (H$_5$) constitute a mild restriction on the ramification of $F/k$ and (H$_6$) coincides with the claim of ${\rm BSD}(A_{F/k})$(i). However, the hypothesis~ (H$_3$) excludes the case that is called `anomalous' by Mazur in~\cite{m} and, for a given $A$, there may be infinitely many primes $p$ for which there are $p$-adic places $v$ at which $A$ has good ordinary reduction but $A(\kappa_v)[p]$ does not vanish.
Nevertheless, it is straightforward to describe examples of abelian varieties $A$ for which there are only finitely many such anomalous places -- see, for example, the result of Mazur and Rubin in~\cite[Lem. A.5]{mr}.}
\end{remark}

\begin{remark} {\em The validity of each of the hypotheses listed above is equivalent
to the validity of the corresponding hypothesis with $A$ replaced by $A^t$ and we will often use this fact without explicit comment.}
\end{remark}

In this section we first verify and render fully explicit the computation (\ref{bksc cohom}) of the cohomology of the Selmer complex introduced in Definition \ref{bkdefinition}, thereby extending the computations given by Wuthrich and the present authors in \cite[Lem. 4.1]{bmw}.

Such an explicit computation will be useful in the proof of the main result of \S\ref{comparison section} below. We shall also use Lemma \ref{useful prel} to ensure that, under the hypotheses listed above, this complex belongs to the category $D^{\rm perf}(\ZZ_p[G])$.

In the main result of this section we shall then re-interpret ${\rm BSD}(A_{F/k})$ in terms of invariants that can be associated to the classical Selmer complex under the above listed hypotheses.

\subsection{The classical Selmer complex}\label{explicitbk} We fix an odd prime number $p$ and a finite set of non-archimedean places $\Sigma$ of $k$ with
\[ S_k^p\cup (S_k^F\cap S_k^f) \cup S_k^A \subseteq \Sigma.\]

For any such set $\Sigma$ the classical Selmer complex ${\rm SC}_{\Sigma,p}(A_{F/k})$ is defined as the mapping fibre of the morphism (\ref{bkfibre}) in $D(\ZZ_p[G])$.

We further recall from Lemma \ref{independenceofsigma} that this complex is, in a natural sense, independent of the choice of $\Sigma$ and so will be abbreviated to ${\rm SC}_{p}(A_{F/k})$.

In the next result we describe consequences of Lemma \ref{useful prel} for this complex and also give a description of its cohomology that will be useful in the computations that are carried out in \S \ref{comparison section} below. In claim (iv) of this result we will write $\Sel_p(A_{F})$ for the classical $p$-primary Selmer group of $A$ over $F$.

\begin{proposition}\label{explicitbkprop} 

Set $C:= {\rm SC}_{p}(A_{F/k})$.
Then the following claims are valid.
\begin{itemize}
\item[(i)] The complex $C$ is acyclic outside degrees one, two and three and there is a canonical identification $H^3(C)=A(F)[p^{\infty}]^\vee$ and a canonical inclusion of $H^1(C)$ into $H^1\bigl(\mathcal{O}_{F,S_k^\infty(F)\cup \Sigma(F)},T_{p}(A^t)\bigr)$.
\item[(ii)] Assume that $A$, $F/k$ and $p$ satisfy the hypotheses (H$_1$)-(H$_5$). Then for every non-archimedean place $v$ of $k$ the $G$-modules $A^t(F_v)^\wedge_p$ and $\ZZ_p\otimes_\ZZ A^t(F_v)$ are cohomologically-trivial. In addition, the module $A^t(F_v)^\wedge_p$ vanishes for every place $v$ in $S_k^F\setminus S_k^p$.

In particular, the complex $C$ belongs to $D^{\rm perf}(\ZZ_p[G])$.
\item[(iii)] Assume that $A$, $F/k$ and $p$ satisfy the hypotheses (H$_1$)-(H$_5$). Then for each normal subgroup $J$ of $G$ there is a natural isomorphism in $D^{\rm perf}(\ZZ_p[G/J])$ of the form
\[ \ZZ_p[G/J]\otimes^{\mathbb{L}}_{\ZZ_p[G]}{\rm SC}_{p}(A_{F/k}) \cong {\rm SC}_{p}(A_{F^J/k}).\]
\item[(iv)] If $\sha(A_F)$ is finite, then $H^1(C)$ identifies with the image of the injective Kummer map $A^t(F)_p\to H^1\bigl(\mathcal{O}_{F,S_k^\infty(F)\cup \Sigma(F)},T_{p}(A^t)\bigr)$ and there is a canonical isomorphism of $H^2(C)$ with $\Sel_p(A_F)^\vee$ (that is described in detail in the course of the proof below).
    \end{itemize}\end{proposition}

\begin{proof} 
Throughout this argument we abbreviate the rings $\mathcal{O}_{k,S_k^\infty\cup\Sigma}$ and $\mathcal{O}_{F,S_k^\infty(F)\cup \Sigma(F)}$ to $U_{k}$ and $U_{F}$ respectively.

Since the complexes $R\Gamma (k_v, T_{p,F}(A^t))$ for $v$ in $\Sigma$ are acyclic in degrees greater than two, the first assertion of claim (i) follows directly from the definition of $C$ as the mapping fibre of the morphism (\ref{bkfibre}).

In addition, the description of the complex ${\rm SC}_{S_k^\infty\cup\Sigma}(A_{F/k},X)$ (for any module $X$ as in Proposition \ref{prop:perfect}) as the mapping fibre of the morphism (\ref{selmer-finite tri}) in $D(\ZZ_p[G])$ also implies that $H^3(C)$ is canonically isomorphic to $H^3({\rm SC}_{S_k^\infty\cup\Sigma}(A_{F/k},X))$. Proposition \ref{prop:perfect}(ii) thus implies that $H^3(C)$ identifies with $A(F)[p^\infty]^\vee$.

Finally, the explicit definition of $C$ as a mapping fibre (combined with Lemma \ref{v not p}(i)) also gives an associated canonical long exact sequence
\begin{multline}\label{longexact}0 \to H^1(C) \to H^1\bigl(U_F,T_{p}(A^t)\bigr) \to
\bigoplus\limits_{w'\in S_F^p}T_p\bigl(H^1(F_{w'},A^t)\bigr) \stackrel{\delta}{\to} H^2(C)\\
 \to H^2\bigl(U_{F},T_{p}(A^t)\bigr) \to \bigoplus\limits_{ w'\in \Sigma(F)}H^2\bigl(F_{w'},T_{p}(A^t)\bigr) \to
 H^3(C) \to 0
\end{multline}
in which the third and sixth arrows are the canonical maps induced by localisation. In this sequence each term $T_p\bigl(H^1(F_{w'},A^t)\bigr)$ denotes the $p$-adic Tate module of $H^1(F_{w'},A^t)$, which we have identified with the quotient of $H^1(F_{w'},T_{p}(A^t))$ by the image of $A^t(F_{w'})_p^\wedge$ under the canonical Kummer map.
%


In particular, the sequence (\ref{longexact}) gives a canonical inclusion $H^1(C) \subseteq H^1\bigl(U_{F},T_{p}(A^t)\bigr)$ and this completes the proof of claim (i).

Turning to claim (ii) we note first that if $v$ does not belong to $S_k^A\cup S_k^F$ then the cohomological-triviality of $A^t(F_v)^\wedge_p$ follows directly from Lemma \ref{useful prel}(ii).

In addition, if $v$ is $p$-adic, then $A^t(F_v)^\wedge_p$ is cohomologically-trivial as a consequence of Lemma \ref{useful prel}(ii) and (iii) and the given hypotheses (H$_2$) and (H$_3$).

It suffices therefore to consider the $G$-modules $A^t(F_v)^\wedge_p$ for places in $(S_k^A\cup S_k^F)\setminus S_k^p$. For each such $v$ we consider the direct sum 
\[ C_v(A_F) := \bigoplus_{w'\in S_k^v}H^0(F_{w'}, H^1(I_{w'}, T_{p}(A^t))_{\rm tor})\]
of the modules that occur in the exact sequence of Lemma \ref{v not p}(ii).

Now if $v$ belongs to $S_k^A$, then (H$_5$) implies $v$ is unramified in $F/k$ and so the $\ZZ_p[G]$-module
$$T_{p,F}(A^t)^{I_v}\cong\ZZ_p[G]\otimes_{\ZZ_p}T_p(A^t)^{I_v}$$
is free. In this case therefore, the natural exact sequence
$$0\to T_{p,F}(A^t)^{I_v}\stackrel{1-\Phi_v^{-1}}{\longrightarrow}T_{p,F}(A^t)^{I_v}\to H^1(\kappa_v,T_{p,F}(A^t)^{I_v})\to 0$$
implies that the $G$-module $H^1(\kappa_v,T_{p,F}(A^t)^{I_v})$ is cohomologically-trivial.

Since the conditions (H$_1$) and (H$_5$) combine in this case to imply that $C_v(A_F)$ vanishes (as in the proof of \cite[Lem. 4.1(ii)]{bmw}) the cohomological-triviality of $A^t(F_v)^\wedge_p$ therefore follows from the exact sequence in Lemma \ref{v not p}(ii).

Finally, we claim that $A^t(F_v)^\wedge_p$ vanishes for each $v$ that belongs to $S_k^F\setminus S_k^p$. To see this we note that, in this case, (H$_5$) implies $v$ does not belong to $S_k^A$ so that $C_v(A_F)$ vanishes whilst the conditions (H$_4$) and (H$_5$) also combine (again as in the proof of \cite[Lem. 4.1(i)]{bmw}) to imply $H^1(\kappa_v,T_{p,F}(A^t)^{I_v})$ vanishes. From the exact sequence of Lemma \ref{v not p}(ii) we can therefore deduce that $A^t(F_v)^\wedge_p$ vanishes, as claimed.

At this stage we have proved that for every non-archimedean place $v$ of $k$, the $G$-module $A^t(F_v)^\wedge_p$ is cohomologically-trivial. Since each $\ZZ_p[G]$-module $A^t(F_v)^\wedge_p$ is finitely generated this implies that each complex $A^t(F_v)^\wedge_p[-1]$ is an object of $D^{\rm perf}(\ZZ_p[G])$.

Given this fact, the final assertion of claim (ii) is a consequence of the definition of $C$ as the mapping fibre of (\ref{bkfibre}) and the fact that, since $p$ is odd, the complexes $R\Gamma(U_k,T_{p,F}(A^t))$ and $R\Gamma (k_v, T_{p,F}(A^t))$ for each $v$ in $\Sigma$ each belong to $D^{\rm perf}(\ZZ_p[G])$ (as a consequence, for example, of \cite[Prop. 1.6.5(2)]{fukaya-kato}).

To complete the proof of claim (ii) we fix a non-archimedean place $v$ of $k$ and consider instead the $G$-module $\ZZ_p\otimes_\ZZ A(F_v)$. We recall that  there exists a short exact sequence of $G$-modules of the form
\begin{equation*}\label{finalassertion}0\to\mathcal{O}_{F,v}^d\to A(F_v)\to C\to 0\end{equation*}
in which the group $C$ is finite. From this exact sequence one may in turn derive short exact sequences
\begin{equation}\label{completions}0\to((\mathcal{O}_{F,v})^\wedge_p)^d\to A(F_v)^\wedge_p\to C^\wedge_p\to 0\end{equation} and
\begin{equation}\label{tensorproducts}0\to(\ZZ_p\otimes_\ZZ\mathcal{O}_{F,v})^d\to\ZZ_p\otimes_\ZZ A(F_v)\to \ZZ_p\otimes_\ZZ C\to 0.\end{equation}

We assume first that $v$ is $p$-adic. In this case the canonical maps $\ZZ_p\otimes_\ZZ\mathcal{O}_{F,v}\to(\mathcal{O}_{F,v})^\wedge_p$ and $\ZZ_p\otimes_\ZZ C\to C^\wedge_p$ are bijective and hence the exactness of the above sequences implies that the canonical map $\ZZ_p\otimes_\ZZ A(F_v)\to A(F_v)^\wedge_p$ is also an isomorphism. The $G$-module $\ZZ_p\otimes_\ZZ A(F_v)$ is thus cohomologically-trivial, as required.

We finally assume $v$ is not $p$-adic. In this case, the exact sequence (\ref{completions}) gives an isomorphism $A(F_v)^\wedge_p\cong C^\wedge_p=\ZZ_p\otimes_\ZZ C$ and so (\ref{tensorproducts}) gives a short exact sequence
$$0\to(\ZZ_p\otimes_\ZZ\mathcal{O}_{F,v})^d\to\ZZ_p\otimes_\ZZ A(F_v)\to A(F_v)^\wedge_p\to 0.$$
Since we have already established the cohomological-triviality of $A(F_v)^\wedge_p$, we know that the $G$-module $\ZZ_p\otimes_\ZZ A(F_v)$ is cohomologically-trivial if and only if the $G$-module $\ZZ_p\otimes_\ZZ\mathcal{O}_{F,v}$ is cohomologically-trivial. But the latter module is naturally a $\QQ$-vector-space, and therefore is indeed cohomologically-trivial. This completes the proof of claim (ii).

Turning to claim (iii) we note that the cohomological-triviality of the $\ZZ_p[G]$-module $A^t(F_v)^\wedge_p$ for each $v$ in $\Sigma$ (as is proved by claim (ii) under the given hypotheses) implies that there are natural isomorphisms in $D(\ZZ_p[G/J])$ of the form
\begin{align*} \ZZ_p[G/J]\otimes^{\mathbb{L}}_{\ZZ_p[G]}A^t(F_v)^\wedge_p[-1] \cong\, &(\ZZ_p[G/J]\otimes_{\ZZ_p[G]}A^t(F_v)^\wedge_p)[-1]\\
\cong\, &H_0(J,A^t(F_v)^\wedge_p)[-1]\\
\cong\, &H^0(J,A^t(F_v)^\wedge_p)[-1]\\
= \, & A^t(F^J_v)^\wedge_p[-1],\end{align*}
where the third isomorphism is induced by the map sending each element $x$ of $A^t(F_v)^\wedge_p$ to its image under the action of $\sum_{g \in J}g$.

The existence of the isomorphism in claim (iii) is then deduced by combining these isomorphisms together with the explicit definitions of the complexes ${\rm SC}_{p}(A_{F/k})$ and ${\rm SC}_{p}(A_{F^J/k})$ as mapping fibres and the fact (recalled, for example, from \cite[Prop. 1.6.5(3)]{fukaya-kato}) that there are standard Galois descent isomorphisms in $D(\ZZ_p[G/J])$ of the form
\begin{equation}\label{global descent} \ZZ_p[G/J]\otimes^{\mathbb{L}}_{\ZZ_p[G]}R\Gamma(U_k,T_{p,F}(A^t)) \cong R\Gamma(U_k,T_{p,F^J}(A^t))\end{equation}
and
\begin{equation}\label{local descent} \ZZ_p[G/J]\otimes^{\mathbb{L}}_{\ZZ_p[G]}R\Gamma (k_v, T_{p,F}(A^t))\cong R\Gamma (k_v, T_{p,F^J}(A^t))\end{equation}
for each $v$ in $\Sigma$.

To prove claim (iv) we assume $\sha(A_F)$ is finite and first prove that
the image of $H^1(C)$ in $H^1\bigl(U_{F},T_{p}(A^t)\bigr)$ coincides with the image of the injective Kummer map $$\kappa:A^t(F)_p\to H^1\bigl(U_{F},T_{p}(A^t)\bigr).$$
In order to do so, 
we identify ${\rm cok}(\kappa)$ with the $p$-adic Tate module $T_p\bigl(H^1\bigl(U_F,A^t\bigr)\bigr)$ of $H^1\bigl(U_F,A^t\bigr)$.

It is clear that any element of ${\rm im}(\kappa)$ is mapped to the image of $A^t(F_{w'})_p^\wedge$ in $H^1\bigl(F_{w'},T_{p}(A^t)\bigr)$ by localising at any place $w'$ in $S_F^p$.

The exactness of (\ref{longexact}) therefore implies that ${\rm im}(\kappa)$ is contained in $H^1(C)$ and furthermore that we have a commutative diagram
with exact rows
\begin{equation}\label{sha diag} \xymatrix{
0 \ar[r] & A^t(F)_p \ar[r] \ar[d]^{\kappa} & H^1\bigl(U_F,T_{p}(A^t)\bigr) \ar[r] \ar@{=}[d] & T_p\bigl(H^1\bigl(U_F,A^t\bigr)\bigr)  \ar[r] \ar[d] & 0\\
0 \ar[r] & H^1(C) \ar[r] & H^1\bigl(U_F,T_{p}(A^t)\bigr) \ar[r] & \bigoplus\limits_{w'\in S_F^p}T_p\bigl(H^1(F_{w'},A^t)\bigr),
}\end{equation}
where the right-most vertical arrow is induced by the localisation maps.
But the assumed finiteness of $\sha(A_F)$ (combined with \cite[Ch. I, Cor 6.6]{milne}) then implies that this arrow is injective, and therefore the Snake Lemma implies that $\im(\kappa)=H^1(C)$, as required.

To conclude the proof of claim (iv) we use the canonical exact triangle
\begin{multline}\label{compacttriangle}R\Gamma_c\bigl(U_F,T_p(A^t)\bigr)\to R\Gamma\bigl(U_k,T_{p,F}(A^t)\bigr)\to \bigoplus\limits_{ v\in S_k^\infty\cup\Sigma} R\Gamma(k_v,T_{p,F}(A^t))\\ \to R\Gamma_c\bigl(U_F,T_p(A^t)\bigr)[1]\end{multline}
in $D(\ZZ_p[G])$. We also write $\Delta$ for the canonical composite homomorphism
$$\bigoplus\limits_{ v\in \Sigma}A^t(F_v)_p^\wedge\to\bigoplus\limits_{ w'\in \Sigma(F)}H^1(F_{w'},T_{p}(A^t))\to H^2_c(U_F,T_p(A^t)),$$
with the first arrow given by the local Kummer maps and the second arrow given by the long exact cohomology sequence associated to the triangle (\ref{compacttriangle}). We then claim that there is a canonical commutative diagram

\begin{equation}\label{Selmerdiagram}\xymatrix{
H^2_c\bigl(U_F,T_p(A^t)\bigr) \ar[d] \ar@{=}[r] &
 H^2_c\bigl(U_F,T_{p}(A^t)\bigr) \ar[d] \ar[r]^-{w\circ s^{-1}} &
H^1\bigl(U_F,A[p^\infty])^\vee
 \ar[d] \\
H^2(C) \ar^-{\sim}[r] &
\cok(\Delta) \ar[r]^-{\sim} &
\Sel_p(A_F)^\vee
.
}\end{equation}
Here the second and third vertical arrows are the canonical projection maps and $w$ and $s$ are the isomorphisms defined in (\ref{themapw}) and (\ref{themaps}) in Appendix \ref{ptduality} below.

The composition of the horizontal arrows in the bottom row of the diagram (\ref{Selmerdiagram}) will then define the desired canonical isomorphism of $H^2(C)$ with $\Sel_p(A_F)^\vee$.

To verify the existence of the diagram (\ref{Selmerdiagram}) we use the canonical exact sequence
\begin{multline}\label{comparingsequence}0\to\bigoplus\limits_{ v\in S_k^\infty} H^0(k_v,T_{p,F}(A^t)) \to H^1_c\bigl(U_F,T_p(A^t)\bigr)\to H^1(C)\\ \to\bigoplus\limits_{ v\in \Sigma}A^t(F_v)_p^\wedge\stackrel{\Delta}{\to}H^2_c\bigl(U_F,T_p(A^t)\bigr)\to H^2(C)\to 0\end{multline} associated to the exact triangle (\ref{comparingtriangles}).
(Here we have used the fact that, as $p$ is odd, the group $H^i(k_v,T_{p,F}(A^t))$ vanishes for every $v$ in $S_k^\infty$ and every $i > 0$.)

This exact sequence induces the desired canonical isomorphism of $H^2(C)$ with $\cok(\Delta)$ and, by construction, the last map occurring in the sequence gives a vertical map making the first square of the diagram (\ref{Selmerdiagram}) commute.

It is finally straightforward, using the commutativity of the diagram in Corollary \ref{Tatepoitouexplicit} below, to deduce that the isomorphism $$w\circ s^{-1}:H^2_c\bigl(U_{F},T_{p}(A^t)\bigr) \to
H^1\bigl(U_{F},A[p^\infty])^\vee$$ induces an isomorphism $$\cok(\Delta)\stackrel{\sim}{\to}\Sel_p(A_F)^\vee.$$ This induced isomorphism completes the construction of the diagram (\ref{Selmerdiagram}) and thus also the proof of claim (iv).\end{proof}

In the next result we shall (exceptionally for \S\ref{tmc}) consider the prime $2$ and describe an analogue of Proposition \ref{explicitbkprop} in this case.

\begin{proposition}\label{explicitbkprop2} The following claims are valid for the complex $C:= {\rm SC}_{2}(A_{F/k})$.
\begin{itemize}
\item[(i)] $C$ is acyclic outside degrees one, two and three.
\item[(ii)] If $\sha(A_F)$ is finite, then $H^1(C)$ identifies with the image of the injective Kummer map $A^t(F)_2\to H^1\bigl(\mathcal{O}_{F,S_k^\infty(F)\cup \Sigma(F)},T_{2}(A^t)\bigr)$ and there exists a
    canonical homomorphism $\Sel_2(A_F)^\vee \to H^2(C)$, the kernel and cokernel of which are both finite.
\item[(iii)]  The module $H^3(C)$ is finite.
\end{itemize}
\end{proposition}

\begin{proof} Claim (i) is established by the same argument that is used to prove the first assertion of Proposition \ref{explicitbkprop}(i).

In a similar way, the analysis concerning the diagram (\ref{sha diag}) is also valid in the case $p=2$ and proves the first assertion of claim (ii).

To prove the remaining claims we set $U_F := \mathcal{O}_{F,S_k^\infty(F)\cup \Sigma(F)}$ and note that the long exact cohomology sequence of the exact triangle (\ref{comparingtriangles}) gives rise in this case to an exact sequence
\begin{multline*} \bigoplus\limits_{ v\in \Sigma}A^t(F_v)_2^\wedge \oplus\bigoplus\limits_{ v\in S_k^\infty} H^1(k_v,T_{2,F}(A^t)) \to H^2_c\bigl(U_F,T_2(A^t)\bigr) \to H^2(C)\\
\to \bigoplus\limits_{ v\in S_k^\infty} H^2(k_v,T_{2,F}(A^t)) \to H^3_c(U_F,T_{2}(A^t)) \to H^3(C) \to \bigoplus_{ v\in S_k^\infty} H^3(k_v,T_{2,F}(A^t)).\end{multline*}
In addition, for each $v$ in $S_k^\infty$ and each $j \in \{1,2,3\}$ the group $H^j(k_v,T_{2,F}(A^t))$ is finite.

Given these facts, the second assertion of claim (ii) is a consequence of Artin-Verdier Duality (just as with the analogous assertion in Proposition \ref{explicitbkprop}(iv)) and claim (iii) follows directly from the isomorphism (\ref{artinverdier}).\end{proof}

\subsection{Statement of the main result}\label{somr tmc sec} We continue to assume that the  hypotheses (H$_1$)-(H$_6$) are satisfied.

In this case, for any isomorphism of fields $j:\CC\cong\CC_p$, the isomorphism
\[ h^{j}_{A,F}:=\CC_p\otimes_{\RR,j}h_{A,F}^{{\rm det}}\]
that is induced by the N\'eron-Tate height (\ref{height triv}) combines with the explicit descriptions given in Proposition \ref{explicitbkprop} to give a canonical element
\[ \chi_{G,p}({\rm SC}_p(A_{F/k}),h^{j}_{A,F})\]
of $K_0(\ZZ_p[G],\CC_p[G])$. By Lemma \ref{independenceofsigma} (and Remark \ref{indeptremark}) this element is in particular independent of the choice of set $\Sigma$ with respect to which ${\rm SC}_p(A_{F/k})={\rm SC}_{\Sigma,p}(A_{F/k})$ is defined. In the rest of \S\ref{tmc} we may and will thus set $$\Sigma:=S_k^p\cup (S_k^F\cap S_k^f) \cup S_k^A.$$

Our aim in the rest of \S\ref{tmc} is to interpret ${\rm BSD}_p(A_{F/k})$(iv) in terms of an explicit description of
 this element.


\subsubsection{} At the outset we note that Hypotheses (H$_2$) and (H$_3$) imply that for each $v$ in $S_k^p$ the restriction $A^t_v$ of $A^t$ to $k_v$ satisfies the conditions that are imposed in Definition \ref{etaledisc} and hence that the \'etale discrepancy class $R_{F_w/k_v}(\tilde A^t_v)$ of $A^t_F$ at $v$ is well-defined in $K_0(\ZZ_p[G_w],\QQ_p[G_w])$.

We again write $d$ for ${\rm dim}(A)$ and then define an element of $K_0(\ZZ_p[G],\QQ_p[G])$ by setting

\[ R_{F/k}(\tilde A^t_v) := {\rm ind}^G_{G_w}(d\cdot R_{F_w/k_v}+ R_{F_w/k_v}(\tilde A^t_v)).\]
Here ${\rm ind}^G_{G_w}$ is the induction homomorphism $K_0(\ZZ_p[G_w],\QQ_p[G_w])\to K_0(\ZZ_p[G],\QQ_p[G])$ and $R_{F_w/k_v}$ is the canonical element of $K_0(\ZZ_p[G_w],\QQ_p[G_w])$ that is defined by Breuning in \cite{breuning2} (and will be explicitly recalled in the course of the proof of Proposition \ref{heavy part} below).

In the sequel we will fix a finite set of places $S$ of $k$ with
$$ S_k^\infty\cup S_k^F \cup S_k^A\subseteq S,$$ as in the statement of Conjecture \ref{conj:ebsd}.

We abbreviate $S_k^F\cap S_k^f$ to $S_{\rm r}$ and set $S_{p,{\rm r}} := S_k^p\cap S_k^F$. 
We shall also write $S_{p,{\rm w}}$ and $S_{p,{\rm t}}$ for the (disjoint) subsets of $S_{p,{\rm r}}$ comprising places that are respectively wildly and tamely ramified in $F$ and $S_{p,{\rm u}}$ for the set $S_k^p\setminus S_{p,{\rm r}}$ of $p$-adic places in $k$ that do not ramify in $F$. 

For each place $v$ in $S_k^p$ and each character $\psi$ in $\widehat{G}$ we define a non-zero element
\begin{equation}\label{revisionVARRHO} \varrho_{v,\psi} := ({\rm N}v)^{\dim_\CC(V_\psi^{I_w})}\end{equation}
of $\QQ^c$.

We then define an invertible element of $\zeta(\CC[G])$ by setting
%
\begin{equation}\label{bkcharelement}
 \mathcal{L}^*_{A,F/k} := \sum_{\psi \in \widehat{G}} \frac{L^{*}_{S_{\rm r}}(A,\check{\psi},1)\cdot \tau^{\ast}(\QQ,\psi)^d\cdot \prod_{v\in S_{p,{\rm r}}}\varrho_{v,\psi}^d}{\Omega_A^\psi\cdot w_\psi^d}\cdot e_\psi\end{equation}
%
where, for each $\psi$ in $\widehat{G}$, the period $\Omega_A^\psi$ and root number $w_\psi$ are as defined in \S\ref{k theory period sect2}, the modified global Galois-Gauss sum $\tau^{\ast}(\QQ,\psi)$ is as defined in \S \ref{mod GGS section}, the Hasse-Weil-Artin $L$-series $L_{S_{\rm r}}(A,\check{\psi},z)$ is truncated by removing only the Euler factors corresponding to places in $S_{\rm r}$ and the idempotent $e_\psi$ is as defined in (\ref{revisionIDEM}).

%

\subsubsection{}

We can now state the main result of this section. In order to do so we use the homomorphism
$$\delta_{G,p}:\zeta(\CC_p[G])^\times\to K_0(\ZZ_p[G],\CC_p[G])$$ defined in (\ref{G,O hom}) and also the local Fontaine-Messing correction $\mu_v(A_{F/k})$ terms defined in (\ref{localFM}).

\begin{theorem}\label{bk explicit} Assume $A$, $F/k$ and $p$ satisfy all of the hypotheses (H$_1$)-(H$_6$).


Then the equality of ${\rm BSD}_p(A_{F/k})$(iv) is valid if and only if for every isomorphism of fields $j:\CC\cong \CC_p$ one has 
%
\[ \delta_{G,p}(j_\ast(\mathcal{L}^*_{A,F/k}))=\chi_{G,p}({\rm SC}_p(A_{F/k}),h^{j}_{A,F})  + \sum_{v \in S_{p,{\rm w}}}R_{F/k}(\tilde A^t_v)+
\sum_{v\in S_{p,{\rm u}}^*} \mu_v(A_{F/k}). \]
Here $S^*_{p,{\rm u}}$ is the subset of $S_{p,{\rm u}}$ comprising places that divide the different of $k/\QQ$.
\end{theorem}

\begin{remark}\label{emptysets}{\em If the sets $S_{p,{\rm w}}$ and $S^*_{p,{\rm u}}$ are empty, then the above result implies that the equality in ${\rm BSD}_p(A_{F/k})$(iv) is valid if and only if one has
\[ \delta_{G,p}(\mathcal{L}^*_{A,F/k})=\chi_{G,p}({\rm SC}_p(A_{F/k}),h^{j}_{A,F}).\]
This is, in particular, the case if $p$ is unramified in $F/\QQ$ and, in this way, Theorem \ref{bk explicit} recovers the results of Wuthrich and the present authors in \cite[Prop. 4.2 and Th. 4.3]{bmw}. More generally, the above equality is predicted whenever no $p$-adic place of $k$ is wildly ramified in $F$ and, in addition, $p$ is unramified in $k/\QQ$ (as is obviously the case if $k = \QQ$) and this case will play an important role in the special settings considered in \S\ref{mod sect} and \S\ref{HHP}. }\end{remark}

\begin{remark}\label{breuning remark}{\em In \cite[Conj. 3.2]{breuning2} Breuning has conjectured that the terms $R_{F_w/k_v}$ should always vanish. In \cite{breuning} and \cite{breuning2} he has proved this conjecture for all tamely ramified extensions, for all abelian
extensions of $\QQ_p$ with $p$ odd, for all $S_3$-extensions of $\QQ_p$ and for certain families of
dihedral and quaternion extensions. If $p$ is odd, then Bley and Debeerst \cite{bleydebeerst} have also given an algorithmic proof of the conjecture for all Galois extensions of $\QQ_p$ of degree at most $15$. More recently, Bley and Cobbe \cite{BC} have proved the conjecture for certain natural families of wildly ramified extensions. }\end{remark}


\begin{remark}{\em If $A$ is an elliptic curve, then Remark \ref{breuning remark} combines with the formula (\ref{curve local eps conj}) for the \'etale discrepancy classes (of $A^t_F$ at each $p$-adic place) to give a completely explicit description of the elements $R_{F/k}(\tilde A_v)$. However, whilst the results of \cite{BC2} imply that this description is unconditionally valid for certain families of wildly ramified extensions, it is, in general, conjectural.}\end{remark}

\begin{remark}\label{bsdinvariants}{\em If $F=k$ then it can be shown that the element (\ref{bkcharelement}) is equal to the product $(-1)^d\cdot(L^\ast(A,1)/\Omega_A)\cdot (\sqrt{|d_k|)}^{d}$ with $$\Omega_A=\prod_{v\in S_k^\CC}\Omega_{A,v}\cdot\prod_{v\in S_k^\RR}\Omega_{A,v}^+,$$ where the classical periods $\Omega_{A,v}$ and $\Omega_{A,v}^+$ are as defined in \S\ref{k theory period sect2}.}
\end{remark}


\subsection{The proof of Theorem \ref{bk explicit}} 

\subsubsection{}\label{clever peiods}

In view of Lemma \ref{pro-p lemma} it is enough for us to fix a field isomorphism $j:\CC\cong \CC_p$ and show that the displayed equality in Theorem \ref{bk explicit} is equivalent to (\ref{displayed pj}).

%

Taking advantage of Remark \ref{consistency remark}(i), we first specify the set $S$ to be equal to $S^\infty_k\cup S_k^F\cup S_k^A$. We shall next use the approach of \S\ref{k theory period sect} to make a convenient choice of differentials $\omega_\bullet$.

For each $v$ in $S_k^p$ we set
\[ \mathcal{D}_v := \Hom_{\mathcal{O}_{k_v}}(H^0(\mathcal{A}_v^t,\Omega^1_{\mathcal{A}_v^t}), \mathcal{O}_{k_v}),\]
where the N\'eron models $\mathcal{A}^t_v$ are as fixed at the beginning of \S\ref{perf sel sect}.

For each such $v$ we also fix a free (rank one) $\mathcal{O}_{k_v}[G]$-submodule $\mathcal{F}_v$ of $F_v = k_v\otimes_k F$ and we assume that for each $v \in S_{p,{\rm u}}$ one has
\[ \mathcal{F}_v = \mathcal{O}_{F,v} = \mathcal{O}_{k_v}\otimes_{\mathcal{O}_{k}}\mathcal{O}_{F}.\]

We then set
\[ \Delta(\mathcal{F}_v) := \mathcal{F}_v\otimes_{\mathcal{O}_{k_v}}\mathcal{D}_v.\]

For each place $w'$ in $S_F^p$ we write $\Sigma(F_{w'})$ for the set of $\QQ_p$-linear embeddings $F_{w'} \to \QQ_p^c$, we define a $\ZZ_p[G_{w'}]$-module $Y_{F_{w'}} := \prod_{\sigma \in \Sigma(F_{w'})}\ZZ_p$ (upon which $G_{w'}$ acts via precomposition with the embeddings) and write
\[ \pi_{F_{w'}}: \QQ_p^c\otimes_{\ZZ_p}F_{w'} \to \QQ_p^c\otimes_{\ZZ_p}Y_{F_{w'}}\]
for the isomorphism of $\QQ_p^c[G_{w'}]$-modules that sends each element $\ell\otimes f$ to $(\ell\otimes \sigma(f))_\sigma$.

For each $v$ in $S_k^p$ we then consider the isomorphism of $\QQ_p[G]$-modules
\[ \pi_{F_v}: \QQ_p^c\otimes_{\ZZ_p}F_v = \prod_{w'\in S_F^v}(\QQ_p^c\otimes_{\ZZ_p}F_{w'}) \xrightarrow{(\pi_{F_{w'}})_{w'}} \QQ_p^c\otimes_{\ZZ_p}\bigoplus_{w' \in S_F^v}Y_{F_{w'}} = \QQ_p^c\otimes_{\ZZ_p}Y_{F_v},  \]
where we set $Y_{F_v} := \bigoplus_{w'}Y_{F_{w'}}$.

After fixing an embedding of $\QQ^c$ into $\QQ_p^c$ we obtain an induced identification of $\bigoplus_{v \in S_k^p}Y_{F_v}$ with the module $Y_{F,p} := \bigoplus_{\Sigma(F)}\ZZ_p$, upon which $G$ acts via pre-composition on the embeddings.

We next fix
\begin{itemize}
\item[$\bullet$] an ordered $k$-basis $\{\omega'_j:j \in [d]\}$ of $H^0(A^t,\Omega^1_{A^t})$ that generates over $\mathcal{O}_{k,p}$ the module
 $\mathcal{D}_p := \prod_{v \in S_k^p}\mathcal{D}_v$, and
\item[$\bullet$] an ordered $\ZZ_p[G]$-basis $\{z_b:b \in [n]\}$ of $\mathcal{F}_p := \prod_{v\in S_k^p}\mathcal{F}_v$.
\end{itemize}

Then the (lexicographically ordered) set
\[ \omega_\bullet:= \{ z_b\otimes \omega'_j: b \in [n], j \in [d]\}\]
is a $\QQ_p[G]$-basis of $H^0(A_F^t,\Omega^1_{A_F^t}) = F\otimes_kH^0(A^t,\Omega^1_{A^t})$ and the 
 %
%
%
%
%
arguments of Lemma \ref{k-theory period} and Proposition \ref{lms} combine to show that
\begin{equation}\label{norm resolvents} \partial_{G,p}\left(j_*(\Omega_{\omega_\bullet}(A_{F/k}))\right)=\delta_{G,p}\left(j_*(\Omega_A^{F/k}\cdot w_{F/k}^d)\right) + \sum_{v\in S_k^p}d\cdot[\mathcal{F}_v,Y_{F_v};\pi_{F_v}]\end{equation}
in $K_0(\ZZ_p[G],\CC_p[G])$.

\subsubsection{}Now, if necessary, we can multiply $\mathcal{F}$ by a sufficiently large power of $p$ in order to ensure that for every $v$ in $S_{p,{\rm r}}$ the following two conditions are satisfied.

\begin{itemize}
\item[$\bullet$] the $p$-adic exponential map induces a (well-defined) injective homomorphism from $\mathcal{F}_v$ to $(F_v^\times)^\wedge_p$;
\item[$\bullet$] the formal group exponential ${\rm exp}_{A^t,F_v}$ that arises from the differentials $\{\omega'_j:j \in [d]\}$ induces an isomorphism of $\Delta(\mathcal{F}_v)$ with a submodule of $A^t(F_v)^\wedge_p$.
\end{itemize}

For each $v$ in $S_k^p$ we now set
\[ X(v) := \begin{cases}{\rm exp}_{A^t,F_v}(\Delta(\mathcal{F}_v)), &\text{ if $v \in S_{p,{\rm r}}$}\\
A^t(F_v)^\wedge_p, &\text{ if $v \in S_{p,{\rm u}}$.}
\end{cases}\]
Then it is clear that, for any choice of $\gamma_\bullet$ as in \S \ref{perf sel sect} and our specific choices of $S$ and $\omega_\bullet$, the module $X(v)$ coincides with $\mathcal{X}(v)$ for $\mathcal{X}=\mathcal{X}_S(\{\mathcal{A}_v^t\}_v,\omega_\bullet,\gamma_\bullet)$.

The description of the complex
\[ C_{X(p)} := {\rm SC}_{S}(A_{F/k};X(p),H_\infty(A_{F/k})_p)\]
as the mapping fibre of the morphism (\ref{selmer-finite tri}) gives rise to an exact triangle
\[ C_{X(p)} \to {\rm SC}_p(A_{F/k}) \oplus X(p)[-1] \xrightarrow{(\lambda', \kappa'_1)}
\bigoplus_{v \in S_k^p\cup S_k^A} A^t(F_v)_p^\wedge[-1]\to C_{X(p)}[1].\]
Here we have used the fact that, for $v\in S_{\rm r}\setminus S_k^p$, the module $A^t(F_v)_p^\wedge$ vanishes by Proposition \ref{explicitbkprop}(ii).

Further, since Proposition \ref{explicitbkprop}(ii) implies that the $\ZZ_p[G]$-modules
\[ X(p):= \prod_{v\in S_k^p}X(v) \,\,\text{ and }\,\, \bigoplus_{v \in S_k^p\cup S_k^A} A^t(F_v)_p^\wedge\]
are cohomologically-trivial and that the complex ${\rm SC}_p(A_{F/k})$ is perfect, Proposition \ref{prop:perfect}(i) implies that this is a triangle in $D^{\rm perf}(\ZZ_p[G])$. 

By applying Lemma \ref{fk lemma} to this exact triangle, we can therefore deduce that there is in $K_0(\ZZ_p[G],\CC_p[G])$ an equality
\begin{align}\label{first comp} &\chi_{G,p}(C_{X(p)},h^{j}_{A,F})\\
 =\, &\chi_{G,p}({\rm SC}_p(A_{F/k}),h^{j}_{A,F}) - \sum_{v \in S_k^A}\chi_{G,p}(A^t(F_v)_p^\wedge[-1],0)\notag\\
 &\hskip 2truein - \sum_{v\in S_k^p}\chi_{G,p}(X(v)[0]\oplus A^t(F_v)^\wedge_p[-1],{\rm id})\notag\\
= \, & \chi_{G,p}({\rm SC}_p(A_{F/k}),h^{j}_{A,F}) - \sum_{v \in S_k^A}\chi_{G,p}(A^t(F_v)_p^\wedge[-1],0)\notag\\
&\hskip 2truein - \sum_{v\in S_{p,{\rm r}}}\chi_{G,p}(\Delta(\mathcal{F}_v)[0]\oplus A^t(F_v)^\wedge_p[-1],{\rm exp}_{A^t,F_v})\notag\\
= \, & \chi_{G,p}({\rm SC}_p(A_{F/k}),h^{j}_{A,F}) - \delta_{G,p}\Bigl(\prod_{v\in S_k^A} L_v(A,F/k)\Bigr)\notag\\
&\hskip 2truein - \sum_{v\in S_{p,{\rm r}}}\chi_{G,p}(\Delta(\mathcal{F}_v)[0]\oplus A^t(F_v)^\wedge_p[-1],{\rm exp}_{A^t,F_v})\notag\end{align}
where the last equality holds because, for each $v\in S_k^A$, one has
\[ \chi_{G,p}(A^t(F_v)_p^\wedge[-1],0)=  \delta_{G,p}\Bigl( L_v(A,F/k)\Bigr).\]
This equality in turn follows upon combining the argument that gives \cite[(13)]{bmw} with the exactness of the sequence of Lemma \ref{v not p}(ii) for each place $w'$ of $F$ above a place in $S_k^A$ and the fact that the third term occurring in each of these sequences vanishes, as verified in the course of the proof of Proposition \ref{explicitbkprop}(ii).

\subsubsection{}The equalities (\ref{norm resolvents}) and (\ref{first comp}) lead us to consider for each place $v$ in $S_k^p$ the element
\[ c(F/k,\tilde A^t_v) := \begin{cases} d\cdot[\mathcal{F}_v,Y_{F_v};\pi_{F_v}]-\chi_{G,p}(\Delta(\mathcal{F}_v)[0]\oplus A^t(F_v)^\wedge_p[-1],{\rm exp}_{A^t,F_v}), &\text{if $v \in S_{p,{\rm r}}$,}\\
d\cdot[\mathcal{O}_{F,v},Y_{F_v};\pi_{F_v}], &\text{if $v \in S_{p,{\rm u}}$}.\end{cases}\]

It is straightforward to check that for each $v \in S_{p,{\rm r}}$ this element is independent of the choice of $\mathcal{F}_v$ and that Lemma \ref{twist dependence} implies its dependence on $A$ is restricted to (the twist matrix of) the reduction of $A^t$ at $v$.

The key step, however, in the proof of Theorem \ref{bk explicit} is the computation of this element in term of local Galois-Gauss sums that is described in the next result.

For each place $v$ in $S_k^f$ we define an
`equivariant local Galois-Gauss sum' by setting
\[ \tau_v(F/k) := \sum_{\psi \in \widehat{G}}\tau(\QQ_{\ell(v)},\psi_v)\cdot e_\psi\in \zeta(\QQ^c[G])^\times.\]
Here $\psi_v$ denotes the restriction of $\psi$ to $G_w$ and $\tau(\QQ_{\ell(v)},\psi_v)$ is the Galois-Gauss sum (as defined in \cite{martinet}) of the induction to $G_{\QQ_{\ell(v)}}$ of the character of $G_{k_v}$ that is obtained by composing $\psi_v$ with the natural projection $G_{k_v} \to G_w$.

We also define a modified local Galois-Gauss sum by setting
\[ \tau_v^p(F/k) := \begin{cases} \varrho_v(F/k)\cdot u_v(F/k)\cdot \tau_v(F/k), &\text{ if $v \in S_{p,{\rm r}}$,}\\
                                              u_v(F/k)\cdot \tau_v(F/k), &\text{ otherwise,}\end{cases}\]
where we set
\begin{equation}\label{varrho def} \varrho_{v}(F/k) := \sum_{\psi\in \widehat{G}}\varrho_{v,\psi}\cdot e_\psi\end{equation}
and the element $u_v(F/k)$ of $\zeta(\QQ[G])^\times$ is as defined in (\ref{u def}).

Finally, for each $p$-adic place $v$ of $k$ we set

\[ U_v(F/k) := {\rm ind}^{G}_{G_w}(U_{F_w/k_v}),\]
where $U_{F_w/k_v}$ is the `unramified term' in $K_0(\ZZ_p[G_w],\QQ_p^c[G_w])$ that is defined by Breuning in \cite[\S2.5]{breuning2}.

\begin{proposition}\label{heavy part} For each place $v$ in $S_k^p$ and any choice of $j$ one has
\[c(F/k,\tilde A^t_v) =  d\cdot \delta_{G,p}(j_*(\tau_v^p(F/k))) + d\cdot U_v(F/k) - R_{F/k}(\tilde A^t_v).\]
in $K_0(\ZZ_p[G],\QQ_p[G])$.
\end{proposition}

\begin{proof} The term $\delta_{G,p}(j_*(\tau^p_v(F/k)))$ is independent of the choice of $j$ by \cite[Lem. 2.2]{breuning2}. To prove the claimed equality we consider separately the cases $v \in S_{p,{\rm r}}$ and $v \in S_{p,{\rm u}}$.

We assume first that $v$ belongs to $S_{p,{\rm r}}$. Then for every place $w'$ in $S_F^v$ we consider the corresponding perfect complex of $\ZZ_p[G_{w'}]$-modules $R\Gamma(F_{w'},\ZZ_p(1))$ as described in \S \ref{twist inv prelim} and obtain a perfect complex of $\ZZ_p[G]$-modules
$$R\Gamma(F_v,\ZZ_p(1)):=\prod_{w'\in S_F^v}R\Gamma(F_{w'},\ZZ_p(1)).$$

Since Kummer theory canonically identifies the cohomology in degree one of this complex with $(F_v^\times)_p^\wedge$ we may define an additional perfect complex of $\ZZ_p[G]$-modules $C^\bullet_{\mathcal{F}_v}$ through the exact triangle
\[ \mathcal{F}_v[0] \xrightarrow{\alpha_v} R\Gamma(F_v,\ZZ_p(1))[1] \to C^\bullet_{\mathcal{F}_v} \to \mathcal{F}_v[1] \]
in $D^{\rm perf}(\ZZ_p[G])$, with $H^0(\alpha_v)$ induced by the $p$-adic exponential map ${\rm exp}_p$.

We write $f$ for the residue degree of our fixed place $w$ in $S_F^v$. Then the long exact cohomology sequence of the above triangle implies that the normalised valuation map ${\rm val}_{F/k,v} := (f\cdot({\rm val}_{F_{w'}}))_{w'\in S_F^v}$ 
induces an isomorphism of $\QQ_p[G]$-modules
\begin{multline*} \QQ_p \cdot H^0(C^\bullet_{\mathcal{F}_v})  \cong \QQ_p\cdot ((F_v^\times)^\wedge_p/{\rm exp}_p(\mathcal{F}_v)) \xrightarrow{ {\rm val}_{F/k,v}}\prod_{w'\in S_F^v}\QQ_p \\
\xleftarrow{ ({\rm inv}_{F_{w'}})_{w'}} \QQ_p\cdot \prod_{w'\in S_F^v}H^2(F_{w'},\ZZ_p(1)) \cong \QQ_p\cdot H^1(C^\bullet_{\mathcal{F}_v})\end{multline*}
which, by abuse of notation, we also denote by ${\rm val}_{F/k,v}$.

In addition, the chosen differentials $\{\omega'_a: a \in [d]\}$ induce an isomorphism of
 $\mathcal{O}_{k_v}$-modules $\mathcal{D}_{v}\cong \mathcal{O}^d_{k_v}$
and hence an isomorphism of $\mathcal{O}_{k_v}[G]$-modules $\omega_{v,*}: \Delta(\mathcal{F}_v) \cong \mathcal{F}_v^d$.

In particular, if we write $C_{A^t,F}^{v,\bullet}$ for the complex $\prod_{w'\in S_F^v} C_{A^t_v,F_{w'}}^\bullet$, where each complex $C_{A^t_v,F_{w'}}^\bullet$ is as defined at the beginning of \S\ref{twist inv prelim}, then there exists a canonical exact triangle
\[ \Delta(\mathcal{F}_v)[0]\oplus A^t(F_v)^\wedge_p[-1] \xrightarrow{\iota_v} C_{A^t,F}^{v,\bullet} \xrightarrow{\iota'_v} C^{\bullet,d}_{\mathcal{F}_v} \to \Delta(\mathcal{F}_v)[1]\oplus A^t(F_v)^\wedge_p[0]\]
in $D^{\rm perf}(\ZZ_p[G])$.
Here $C^{\bullet,d}_{\mathcal{F}_v}$ denotes the product of $d$ copies of $C^{\bullet}_{\mathcal{F}_v}$, $H^0(\iota_v)$ is the composite map $({\rm exp}_p)^d\circ \omega_{v,*}$ and $H^1(\iota_v)$ is the identity map between
$A^t(F_v)^\wedge_p = H^1(A^t(F_v)^\wedge_p[-1])$ and the direct summand $A^t(F_v)^\wedge_p$ of $H^1(C_{A^t,F}^{v,\bullet})$.

The long exact cohomology sequence of this triangle also gives an exact commutative diagram of $\QQ_p[G]$-modules
\[\minCDarrowwidth1em\begin{CD} 0 @> >>\! \QQ_p\cdot \Delta(\mathcal{F}_v)\! @> \QQ_p\otimes_{\ZZ_p}H^0(\iota_v) >>\! \QQ_p\cdot H^0(C_{A^t,F}^{v,\bullet})\! @> \QQ_p\otimes_{\ZZ_p}H^0(\iota'_v) >> \!\QQ_p\cdot H^0(C^{\bullet,d}_{\mathcal{F}_v})\! @> >>\! 0\\
@. @V{\rm exp}_{A^t,F_v} VV @V  \lambda^v_{A^t,F} VV @V ({\rm val}_{F/k,v})^d VV \\
0 @> >>\! \QQ_p\cdot A^t(F_v)_p^\wedge\! @> \QQ_p\otimes_{\ZZ_p}H^1(\iota_v) >>\! \QQ_p\cdot H^1(C_{A^t,F}^{v,\bullet})\! @> \QQ_p\otimes_{\ZZ_p}H^1(\iota'_v)>> \!\QQ_p\cdot H^1(C^{\bullet, d}_{\mathcal{F}_v}) @> >>\! 0,\end{CD}\]
in which $\lambda^v_{A^t,F} = (\lambda_{A^t_v,F_{w'}})_{w'\in S_F^v}$, where each map $\lambda_{A^t_v,F_{w'}}$ is fixed as in diagram (\ref{lambda diag}).

After recalling the definition of $R_{F_w/k_v}(\tilde A^t_v)$ and applying Lemma \ref{fk lemma} to this commutative diagram one can therefore derive an equality
%
%
%
\begin{align*}\label{third}
&\chi_{G,p}(\Delta(\mathcal{F}_v)[0]\oplus A^t(F_v)^\wedge_p[-1],{\rm exp}_{A^t,F_v})\\
= \, &\chi_{G,p}(C_{A^t,F}^{v,\bullet},\lambda_{A^t,F}^v) - d\cdot\chi_{G,p}(C^\bullet_{\mathcal{F}_v},{\rm val}_{F/k,v})\notag\\
= \, &{\rm ind}^G_{G_w}(\chi_{G_w,p}(C_{A^t_v,F_{w}}^\bullet,\lambda_{A^t_v,F_{w}})) - d\cdot\chi_{G,p}(C^\bullet_{\mathcal{F}_v},{\rm val}_{F/k,v})\notag\\
= \, &{\rm ind}^G_{G_w}(R_{F_w/k_v}(\tilde A^t_v))-{\rm ind}^G_{G_w}(d\cdot\delta_{G_w,p}(c_{F_w/k_v})) - d\cdot\chi_{G,p}(C^\bullet_{\mathcal{F}_v},{\rm val}_{F/k,v})\notag\\
=\, & {\rm ind}^G_{G_w}(R_{F_w/k_v}(\tilde A^t_v)) - d(\chi_{G,p}(C^\bullet_{\mathcal{F}_v},{\rm val}_{F/k,v})+\delta_{G,p}(c_{F_w/k_v})).\notag\end{align*}

It follows that
\[ c(F/k,\tilde A^t_v) = d\cdot[\mathcal{F}_v,Y_{F_v};\pi_{F_v}]-{\rm ind}^G_{G_w}(R_{F_w/k_v}(\tilde A^t_v)) + d(\chi_{G,p}(C^\bullet_{\mathcal{F}_v},{\rm val}_{F/k,v})+\delta_{G,p}(c_{F_w/k_v}))\]
and hence that the claimed result is true in this case if one has
\begin{multline}\label{wanted at last}  [\mathcal{F}_v,Y_{F_v};\pi_{F_v}] + \chi_{G,p}(C^\bullet_{\mathcal{F}_v},{\rm val}_{F/k,v})+\delta_{G,p}(c_{F_w/k_v})
\\ = \delta_{G,p}(j_*(\tau^p_v(F/k))) + U_v(F/k) - {\rm ind}^G_{G_w}(R_{F_w/k_v}).\end{multline}

To prove this we note that the very definition of $R_{F_w/K_v}$ in \cite[\S3.1]{breuning2} implies that
\begin{multline*} {\rm ind}^G_{G_w}(R_{F_w/k_v})\\ = \delta_{G,p}(j_*(\tau_v(F/k))) - [\mathcal{F}_v,Y_{F_v};\pi_{F_v}] - \chi_{G,p}(C^\bullet_{\mathcal{F}_v},{\rm val}'_{F/k,v}) + U_v(F/k) - \delta_{G,p}(m_{F_w/k_v}),\end{multline*}
where ${\rm val}'_{F/k,v}$ denotes the isomorphism of $\QQ_p[G]$-modules
\[ \QQ_p\cdot H^1(C^\bullet_{\mathcal{F}_v})\cong \QQ_p\cdot H^2(C^\bullet_{\mathcal{F}_v})\]
that is induced by the maps ($({\rm val}_{F_{w'}}))_{w'\in S_F^v}$ and we use the element
\[ m_{F_w/k_v}:= \frac{^\dagger(f\cdot e_{G_w})\cdot \, ^\dagger((1- \Phi_v\cdot {\rm N}v^{-1})e_{I_w})}{^\dagger((1-\Phi_v^{-1})e_{I_w})}\]
of $\zeta(\QQ[G])^\times$. (Here we use the notational convention introduced in (\ref{dagger eq}). In addition, to derive the above formula for ${\rm ind}^G_{G_w}(R_{F_w/k_v})$ we have relied on \cite[Prop. 2.6]{breuning2} and the fact that in loc. cit. Breuning uses the `opposite' normalization of Euler characteristics to that used here, so that the term $\chi_{G,p}(C^\bullet_{\mathcal{F}_v},{\rm val}'_{F/k,v})$ appears in the corresponding formulas in loc. cit. with a negative sign.)

To deduce the required equality (\ref{wanted at last}) from this formula it is then enough to note that
\[ \chi_{G}(C^\bullet_{\mathcal{F}_v},{\rm val}'_{F/k,p})
= \chi_{G}(C^\bullet_{\mathcal{F}_v},{\rm val}_{F/k,p}) - \delta_{G,p}(^\dagger(f\cdot e_{G_w})),\]
and that an explicit comparison of definitions shows that
\[  j_*(\tau^p_v(F/k))\cdot m_{F_w/k_v} = j_*(\tau_v(F/k))\cdot ^\dagger(f\cdot e_{G_w})\cdot  c_{F_w/k_v}.\]

Turning now to the case $v\in S_{p,{\rm u}}$ we only need to prove that for each such place $v$ one has
\[ d\cdot[\mathcal{O}_{F,v},Y_{F_v};\pi_{F_v}] = d\cdot \delta_{G,p}(j_*(u_v(F/k)\cdot\tau_v(F/k))) + d\cdot U_v(F/k) - R_{F/k}(\tilde A^t_v).\]

Now, since each such $v$ is unramified in $F/k$ the term $R_{F/k}(\tilde A^t_v)$ vanishes (as a consequence of Proposition \ref{basic props}(iii) and Remark \ref{breuning remark}) and so it is enough to note that
\[ [\mathcal{O}_{F,v},Y_{F_v};\pi_{F_v}] = \delta_{G,p}((u_v(F/k)\cdot \tau_v(F/k))) + U_v(F/k),\]
as is shown in the course of the proof of \cite[Th. 3.6]{breuning2}.
\end{proof}

%


\subsubsection{}We next record a result concerning the decomposition of global Galois-Gauss sums as a product of local Galois-Gauss sums.

\begin{lemma}\label{gauss} In $K_0(\ZZ_p[G],\QQ^c_p[G])$ one has
\[ \delta_{G,p}(j_*(\tau^\ast(F/k)\cdot \prod_{v \in S_{p,{\rm r}}}\varrho_v(F/k))) = \sum_{v \in S_{k}^p}(\delta_{G,p}(j_*(\tau_v^p(F/k))) + U_v(F/k)).\]
\end{lemma}

\begin{proof} We observe first that the difference $\xi$ of the left and right hand sides of this claimed equality belongs to $K_0(\ZZ_p[G],\QQ_p[G])$.

This follows from the fact that both the term
\[ \delta_{G,p}(j_*(\tau^\ast(F/k))) - \sum_{v \in S_k^p}[\mathcal{F}_v,Y_{F_v};\pi_{F_v}],\]
and for each $v \in S_k^p$ the term
\[ \delta_{G,p}(j_*(\tau_v^p(F/k))) + U_v(F/k) - [\mathcal{F}_v,Y_{F_v};\pi_{F_v}],\]
belong to $K_0(\ZZ_p[G],\QQ_p[G])$ (the former as a consequence of \cite[Prop. 3.4 and (34)]{bleyburns} and the latter as a consequence of  \cite[Prop. 3.4]{breuning2}).

Thus, by Taylor's Fixed Point Theorem for group determinants (as discussed in \cite[Chap. 8]{Taylor}) it is enough for us to show that $\xi$ belongs to the kernel of the natural homomorphism $\iota: K_0(\ZZ_p[G],\QQ_p^c[G]) \to K_0(\mathcal{O}^t_p[G],\QQ_p^c[G])$ where $\mathcal{O}_p^t$ is the valuation ring of the maximal tamely ramified extension of $\QQ_p$ in $\QQ_p^c$.

Now from \cite[Prop. 2.12(i)]{breuning2} one has $\iota(U_v(F/k)) = 0$ for all $v$ in $S_k^p$. In addition, for each non-archimedean place $v$ of $k$ that is not $p$-adic the vanishing of $\iota(\delta_{G,p}(j_*(u_v(F/k)\cdot \tau_v(F/k))))$ is equivalent to the result proved by  Holland and Wilson in
\cite[Th. 3.3(b)]{HW3} (which itself relies crucially on the work of Deligne and Henniart in \cite{deligne-henniart}).

The vanishing of $\iota(\xi)$ is thus a consequence of the fact that the classical decomposition of global Galois-Gauss sums as a product of local Galois-Gauss sums implies that
\[ \tau^\ast(F/k)\cdot \prod_{v \in S_{p,{\rm r}}}\varrho_v(F/k) = \prod_{v}\tau^p_v(F/k)\]
where $v$ runs over all places of $k$ that divide the discriminant of $F$, since for any place $v$ that does not ramify in $F$ one has $\tau(\QQ_{\ell(v)},\psi_v) = 1$ for all  $\psi$ in $\widehat{G}$.
\end{proof}

\subsubsection{}We can now complete the proof of Theorem \ref{bk explicit}.

To this end we note first that the definition (\ref{bkcharelement}) of $\mathcal{L}^*_{A,F/k}$ implies that %

\begin{align*} &\delta_{G,p}(j_\ast(\mathcal{L}^*_{A,F/k})) - \partial_{G,p}\left(\frac{j_*(L_S^*(A_{F/k},1))}{j_*(\Omega_{\omega_\bullet}(A_{F/k}))}\right)\\
=\, &\delta_{G,p}\bigl(\prod_{v\in S_k^A} L_v(A,F/k)\bigr) + d\cdot\bigl(\delta_{G,p}(j_*(\tau^\ast(F/k)\cdot \prod_{v \in S_{p,{\rm r}}}\varrho_v(F/k))) - \sum_{v\in S_k^p}[\mathcal{F}_v,Y_{F_v};\pi_{F_v}]\bigr)\\
= \, &\delta_{G,p}\bigl(\prod_{v\in S_k^A} L_v(A,F/k)\bigr) + d\cdot\sum_{v \in S_{k}^p}\bigl(\delta_{G,p}(j_*(\tau_v^p(F/k))) + U_v(F/k) - [\mathcal{F}_v,Y_{F_v};\pi_{F_v}]\bigr)
\\
= \, &\delta_{G,p}\bigl(\prod_{v\in S_k^A} L_v(A,F/k)\bigr) +\sum_{v \in S_{k}^p} \bigl(c(F/k,\tilde A^t_v) - d\cdot [\mathcal{F}_v,Y_{F_v};\pi_{F_v}]\bigr) +\sum_{v \in S_{k}^p}R_{F/k}(\tilde A^t_v)\\
= \, &\chi_{G,p}({\rm SC}_p(A_{F/k}),h^{j}_{A,F})+\sum_{v \in S_{k}^p}R_{F/k}(\tilde A^t_v) - \chi_{G,p}(C_{X(p)},h^{j}_{A,F})\\
= \, &\chi_{G,p}({\rm SC}_p(A_{F/k}),h^{j}_{A,F})+\sum_{v \in S_{p,{\rm w}}}R_{F/k}(\tilde A^t_v) - \chi_{G,p}(C_{X(p)},h^{j}_{A,F}).\end{align*}
Here the first equality follows from (\ref{norm resolvents}), the second from Lemma \ref{gauss}, the third from Proposition \ref{heavy part}, the fourth from the definition of the terms $c(F/k,\tilde A^t_v)$ and the equality (\ref{first comp}) and the last from the fact that $R_{F/k}(\tilde A^t_v)$ vanishes for each $v \in S_k^p\setminus S_{p,{\rm w}}$ as  consequence of Proposition \ref{basic props}(iii) and Remark \ref{breuning remark}.

It is thus sufficient to show that the above displayed equality implies that the equality (\ref{displayed pj}), with set of places $S$ and differentials $\omega_\bullet$ chosen as in \S\ref{clever peiods}, is equivalent to the equality stated in Theorem \ref{bk explicit}.

But this is true since our choice of periods $\omega_\bullet$ and lattices $\mathcal{F}_v$ as in \S\ref{clever peiods} implies that the module $Q(\omega_\bullet)_{S,p}$ vanishes and because

\[ \mu_S(A_{F/k})_p = \sum_{v \in S_k^p\setminus S}\mu_v(A_{F/k})=\sum_{v \in S_{p,{\rm u}}}\mu_v(A_{F/k})
  =\sum_{v \in S^\ast_{p,{\rm u}}} \mu_v(A_{F/k}),\]
where the last equality follows from Lemma \ref{fm} below.

This completes the proof of Theorem \ref{bk explicit}.

\begin{lemma}\label{fm} For any place $v\in S_k^A$ for which the residue characteristic $\ell(v)$ is unramified in $F$, the term
 $\mu_{v}(A_{F/k})$ vanishes. \end{lemma}

\begin{proof} We fix a place $v$ as in the statement of the lemma and set $p:= \ell(v)$. We write $\mathcal{O}_{F_v}$ for the integral closure of $\ZZ_p$ in $F_v$ and set $\wp_{F_v} := p\cdot\mathcal{O}_{F_v}$.

Then, since $p$ does not ramify in $F/\QQ$, the $\ZZ_p[G]$-modules $\mathcal{O}_{F_v}$ and $\wp_{F_v}$ are projective and $\wp_{F_v}$ is the direct sum of the maximal ideals of the valuation rings in each field component of $F_v=\prod_{w'\in S_F^v}F_{w'}$.

We use the canonical comparison isomorphism of $\QQ_p[G]$-modules
\[ \nu_v: \Hom_{F_v}(H^0(A^t_{F_v},\Omega^1_{A^t_{F_v}}),F_v) \cong {\rm DR}_v(V_{p,F}(A^t))/F^0\]
and the exponential map ${\rm exp}_{\rm BK}: {\rm DR}_v(V_{p,F}(A^t))/F^0\to H^1_f(k,V_{p,F}(A^t))$ of Bloch and Kato.

We recall, in particular, that in this case there is a natural identification of spaces $H^1_f(k,V_{p,F}(A^t)) = \QQ_p\cdot A^t(F_v)^\wedge_p$ under which the composite ${\rm exp}_{\rm BK}\circ\nu_v$ sends the free $\ZZ_p[G]$-lattice $\mathcal{D}_F(\mathcal{A}_v^t)$ defined in (\ref{mathcalD}) to ${\rm exp}_{A^t,F_v}((\mathcal{O}_{F_v})^d)$, where ${\rm exp}_{A^t,F_v}$ is the classical exponential map of $A^t$ over $F_{v}$ (cf. the result of Bloch and Kato in \cite[Exam. 3.11]{bk}).

In particular, since $A^t$ has good reduction over the absolutely unramified algebra $F_{v}$, the theory of Fontaine and Messing \cite{fm} implies (via the proof of \cite[Lem. 3.4]{bmw}) that there exists an exact sequence of $\ZZ_p[G]$-modules
\[ 0 \to {\rm exp}_{A^t,F_v}((\mathcal{O}_{F_v})^d) \xrightarrow{\subseteq} A^t(F_v)^\wedge_p \to N \to 0\]
where $N$ is a finite module that has finite projective dimension and is such that
\[ \chi_{G,p}(N[-1],0) = \delta_{G,p}(L_v(A,F/k))\]
in $K_0(\ZZ_p[G],\QQ_p[G])$.

On the other hand, since $F_{v}$ is absolutely unramified, the map ${\rm exp}_{A^t,F_v}$ restricts to give a short exact sequence of $\ZZ_p[G]$-modules
\[ 0 \to \wp_{F_v}^d \xrightarrow{{\rm exp}_{A^t,F_v}} A^t(F_v)^\wedge_p \to \tilde A^t_v(\kappa_{F_v})_p \to 0.\]

By comparing these exact sequences, and noting that $\mathcal{O}_{F_v}/\wp_{F_v}$ identifies with the ring $\kappa_{F_v}$, one obtains a short exact sequence of $\ZZ_p[G]$-modules
\[ 0 \to \kappa_{F_v}^d \to \tilde A^t_v(\kappa_{F_v})_p \to N \to 0\]
in which each term is both finite and of finite projective dimension.

Thus, upon taking Euler characteristics of this exact sequence, one finds that
\begin{align*} \delta_{G,p}(L_v(A,F/k)) =\,&\chi_{G,p}(N[-1],0)\\
 = \, &\chi_{G,p}(\tilde A^t_v(\kappa_{F_v})_p[-1],0) - \chi_{G,p}(\kappa_{F_v}^d[-1],0)\\
 = \, &\chi_{G,p}(\tilde A^t_v(\kappa_{F_v})_p[-1],0) + \chi_{G,p}(\kappa_{F_v}^d[0],0)\\
  =\, &\chi_{G,p}(\kappa_{F_v}^d[0]\oplus\tilde A^t_v(\kappa_{F_v})_p[-1],0),\end{align*}

\noindent{}and hence that the element $\mu_{v}(A_{F/k})$ vanishes, as required.
\end{proof}

\section{Euler characteristics and Galois structures}\label{ecgs}

In this section we consider consequences of ${\rm BSD}(A_{F/k})$ concerning both the Galois structure of Selmer complexes and modules and the formulation of refinements of the Deligne-Gross Conjecture.

\subsection{Galois structures of Selmer complexes}\label{Galoiscomplexes} In this section we fix a finite set of places $S$ of $k$ with
$$ S_k^\infty\cup S_k^F \cup S_k^A\subseteq S$$ as well as data $\{\mathcal{A}^t_v\}_v$ and $\omega_\bullet$ as in \S\ref{perf sel construct}. We then write
\[ \Upsilon = \Upsilon(\{\mathcal{A}^t_v\}_v,\omega_\bullet,S)\]
for the finite set of non-archimedean places $v$ of $k$ that are such that either $v$ belongs to $S$ or divides the different of $k/\QQ$ or the lattice $\mathcal{F}(\omega_\bullet)_{v}$ differs from $\mathcal{D}_F(\mathcal{A}_v^t)$ (with each of these modules as in \S \ref{perf sel construct}).

We then consider the perfect Selmer structure $\mathcal{X}_{\Upsilon}(\omega_\bullet)$ that is defined in \S\ref{perf sel construct}.

\begin{proposition}\label{gec} If ${\rm BSD}(A_{F/k})$ is valid, then for any data $S$, $\{\mathcal{A}^t_v\}_v$ and $\omega_\bullet$ as above the following claims are valid.

\begin{itemize}
\item[(i)] The Selmer complex ${\rm SC}_{\Upsilon}(A_{F/k};\mathcal{X}_{\Upsilon}(\omega_\bullet))$ is represented by a bounded complex of finitely generated free $G$-modules.
\item[(ii)] Set $\ZZ' := \ZZ[1/2]$. If neither of the groups $A(F)$ and $A^t(F)$ has an element of odd order, then the $\ZZ'[G]$-module $\ZZ'\otimes_\ZZ H^2({\rm SC}_{\Upsilon}(A_{F/k};\mathcal{X}_{\Upsilon}(\omega_\bullet)))$ has a presentation with the same number of generators and relations.
\end{itemize}
\end{proposition}

\begin{proof} We set $C_{\omega_\bullet} := {\rm SC}_{\Upsilon}(A_{F/k};\mathcal{X}_{\Upsilon}(\omega_\bullet))$ and write $\chi_G(C_{\omega_\bullet})$ for its Euler characteristic in $K_0(\ZZ[G])$.

Then, the definition of $\Upsilon$ implies immediately that the module $\mathcal{Q}(\omega_\bullet)_\Upsilon$ vanishes and also combines with Lemma \ref{fm} to imply that $\mu_{\Upsilon}(A_{F/k})$ vanishes.

The vanishing of $\mathcal{Q}(\omega_\bullet)_\Upsilon[0]$ implies ${\rm SC}_{\Upsilon,\omega_\bullet}(A_{F/k}) = C_{\omega_\bullet}$ and hence that
\[ \partial'_G(\chi_G({\rm SC}_{\Upsilon,\omega_\bullet}(A_{F/k}),h_{A,F})) = \chi_G(C_{\omega_\bullet}),\]
where $\partial'_{G}$ denotes the canonical connecting homomorphism $K_0(\ZZ[G],\RR[G]) \to K_0(\ZZ[G])$ of relative $K$-theory.

Then, given the vanishing of $\mu_{\Upsilon}(A_{F/k})$, the equality in ${\rm BSD}(A,F/k)$(iv) implies that
\[ \chi_G(C_{\omega_\bullet}) = \partial'_G\bigl(\partial_G(L_\Upsilon^*(A_{F/k},1)/\Omega_{\omega_\bullet}(A_{F/k}))).\]
%

However, the exactness of the lower row of diagram (\ref{E:kcomm}), with $\mathfrak{A} = \ZZ[G]$ and $A_E = \RR[G]$, implies that the composite homomorphism $\partial'_G\circ \partial_G$ is zero and so it follows that the Euler characteristic $\chi_G(C_{\omega_\bullet})$ must vanish.

Now, by a standard resolution argument, we may fix a bounded complex of finitely generated $G$-modules $C^\bullet$ that is isomorphic in $D(\ZZ[G])$ to $C_{\omega_\bullet}$ and is such that, for some integer $a$, all of the following properties are satisfied: $C^i = 0$ for all $i < a$; $C^a$ is projective of rank (over $\ZZ[G]$) at least two; $C^{i}$ is free for all $i \not= a$.

From the vanishing of $\chi_G(C_{\omega_\bullet}) = \chi_G(C^\bullet)$ it then follows that the class of $C^a$ in $K_0(\ZZ[G])$ coincides with that of a free $G$-module.

Thus, since the rank over $\ZZ[G]$ of $C^a$ is at least two, we may conclude from the Bass Cancellation Theorem (cf. \cite[(41.20)]{curtisr}) that $C^a$ is a free $G$-module, as required to prove claim (i).

Turning to claim (ii), we note that if $\ZZ'\otimes_\ZZ A(F)$ and $\ZZ'\otimes_\ZZ A^t(F)$ are torsion-free, then Proposition \ref{prop:perfect2} implies that the complex $C'_{\omega_\bullet} := \ZZ'\otimes_\ZZ C_{\omega_\bullet}$ is acyclic outside degrees one and two and that $H^1(C'_{\omega_\bullet})$ is torsion-free.

This implies that $C'_{\omega_\bullet}$ is isomorphic in $D^{\rm perf}(\ZZ'[G])$ to a complex of finitely generated $\ZZ'[G]$-modules of the form $(C')^1 \xrightarrow{d'} (C')^2$ where $(C')^1$ is projective and $(C')^2$ is free.

The vanishing of the Euler characteristic of $C'_{\omega_\bullet}$ then implies, by the same argument as in claim (i), that the module $(C')^1$ is free.

In addition, the fact that the $\RR[G]$-modules generated by $H^1(C'_{\omega_\bullet})$ and $H^2(C'_{\omega_\bullet})$ are isomorphic implies that the free modules $(C')^1$ and $(C')^2$ must have the same rank.

Given this, claim (ii) follows from the tautological exact sequence
\[ (C')^1 \xrightarrow{d'} (C')^2 \to \ZZ'\otimes_\ZZ H^2({\rm SC}_{\Upsilon}(A_{F/k};\mathcal{X}_{\Upsilon}(\omega_\bullet))) \to 0.\]
\end{proof}

\begin{remark}{\em An explicit description of the module $\ZZ'\otimes_\ZZ H^2({\rm SC}_{\Upsilon}(A_{F/k};\mathcal{X}_{\Upsilon}(\omega_\bullet)))$ that occurs in Proposition \ref{gec}(ii) can be found in Remark \ref{can structure groups}.}\end{remark}

\subsection{Refined Deligne-Gross-type conjectures}

In this section we address a problem raised by Dokchitser, Evans and Wiersema in \cite[Rem. 14]{vdrehw} by explaining how ${\rm BSD}(A_{F/k})$ leads to an explicit formula for the fractional ideal that is generated by the product of the leading coefficients of Hasse-Weil-Artin $L$-series
 by a suitable combination of `isotypic' periods and regulators. (See, in particular, Remark \ref{evans} below.)

\subsubsection{}We fix a character $\psi$ in $\widehat {G}$. We also then fix a subfield $E$ of $\bc$ that is both Galois and of finite degree over
$\QQ$ and also large enough to ensure that, with
$\mathcal{O}$ denoting the ring of algebraic integers of $E$, there exists a finitely generated ${\mathcal
O}[G]$-lattice $T_\psi$ that is free over $\mathcal{O}$ and such that the $\bc[G]$-module
$V_\psi:= \bc \otimes_{\mathcal{O}}T_\psi$ has character $\psi$.

We then obtain a left, respectively right,
exact functor from the category of $G$-modules to the category of $\mathcal{O}$-modules by setting
\[ X^\psi := \Hom_{{\mathcal O}}(T_\psi,{\mathcal O}
\otimes_{\ZZ} X)^G\,\,\,\text{ and }\,\,\, X_\psi := \Hom_{{\mathcal O}}(T_\psi,{\mathcal O}\otimes_{\ZZ} X)_G,
\]
where the $\Hom$-sets are endowed with the natural diagonal
$G$-action.

It is easily seen that for any $G$-module $X$ there is a natural isomorphism of $\mathcal{O}$-modules
\begin{equation}\label{func iso} \Hom_\ZZ(X,\ZZ)_\psi \cong \Hom_\mathcal{O}(X^{\check{\psi}},\mathcal{O}).\end{equation}

For a given odd prime number $p$, each maximal ideal $\mathfrak p$ of $\mathcal{O}$ that divides $p$ and each $\mathcal{O}$-module $X$ we set $X_\mathfrak{p} := \mathcal{O}_\mathfrak{p}\otimes_{\mathcal{O}}X$.

We also write $I(\mathcal{O}_\mathfrak{p})$ for the multiplicative group of invertible $\mathcal{O}_\mathfrak{p}$-submodules of $\CC_p$ and we use the composite homomorphism of abelian groups
\[ \rho_{\mathfrak{p}}^{\psi}: K_0(\ZZ_p [G],\CC_p[G]) \to K_0(\mathcal{O}_\mathfrak{p} ,\CC_p) \xrightarrow{\iota_\mathfrak{p}} I(\mathcal{O}_\mathfrak{p}).\]
Here the first map is induced by the composite functor $X \mapsto X^\psi\to (X^\psi)_\mathfrak{p}$ and $\iota_\mathfrak{p}$ is the canonical isomorphism induced by
 the upper row of (\ref{E:kcomm}) with $\A = {\mathcal O}_\mathfrak{p}$ and $E' = \bc_p$ and the canonical
isomorphisms $K_1(\bc_p) \xrightarrow{\sim} \bc_p^\times$ and
$K_1({\mathcal O}_\mathfrak{p}) \xrightarrow{\sim} {\mathcal O}_\mathfrak{p}^\times$.

For any finite ${\mathcal O}$-module $X$ we also set
\[ {\rm char}_{\mathfrak{p}}(X) := \mathfrak{p}^{{\rm length}_{{\mathcal O}_{{\mathfrak
p}}}(X_{\mathfrak p})}.\]

\subsubsection{}\label{explicit ec section} Using the isomorphism (\ref{func iso}), we define $R^\psi_A$ to be the determinant, with respect to a choice of $\mathcal{O}$-bases of $A^t(F)^{\psi}$ and  $A(F)^{\check\psi}$ of the isomorphism of $\CC$-spaces
\[ h^\psi_{A,F}: \CC\cdot A^t(F)^{\psi} \cong \CC\cdot \Hom_\ZZ(A(F),\ZZ)^\psi \cong \CC\cdot \Hom_\mathcal{O}(A(F)^{\check\psi},\mathcal{O})\]
that is induced by the N\'eron-Tate height pairing of $A$ relative to $F$.

Motivated by \cite[Def. 12]{vdrehw}, we then define a non-zero complex number by setting
\[ \mathcal{L}^\ast(A,\psi) := \frac{L^\ast(A,\check\psi,1)\cdot \tau^\ast(\QQ,\psi)^d}{\Omega_A^\psi\cdot w^d_\psi\cdot R_A^\psi}.\]

Finally, after recalling the integral Selmer group $X_\ZZ(A_F)$ of $A$ over $F$ that is defined by Mazur and Tate \cite{mt} (and discussed in \S\ref{perfect selmer integral}), we note that if $\sha(A_F)$ is finite then the kernel $\sha_\psi(A_F)$ of the natural surjective homomorphism of $\mathcal{O}$-modules
\[ X_\ZZ(A_F)_\psi \to \Hom_\ZZ(A(F),\ZZ)_\psi \cong \Hom_\mathcal{O}(A(F)^{\check\psi},\mathcal{O})\]
is finite.

We can now state the main result of this section.

\begin{proposition}\label{ref deligne-gross} Assume ${\rm BSD}(A_{F/k})$. Then, for each character $\psi$ in $\widehat{G}$, the following claims are valid. 
\begin{itemize}
\item[(i)] For every $\omega$ in $G_\QQ$ one has $\mathcal{L}^\ast(A,\omega\circ \psi) = \omega(\mathcal{L}^\ast(A,\psi))$. In particular, the complex number $\mathcal{L}^\ast(A,\psi)$ belongs to $E$.

\item[(ii)] Assume that no place of $k$ at which $A$ has bad reduction is ramified in $F$. Then for every odd prime $p$ that satisfies the conditions (H$_1$)-(H$_4$) listed in \S\ref{tmc} and for which neither $A(F)$ nor $A^t(F)$ has a point of order $p$, and every maximal ideal $\mathfrak{p}$ of $\mathcal{O}$ that divides $p$, there is an equality of fractional $\mathcal{O}_\mathfrak{p}$-ideals

\[ \mathcal{L}^\ast(A,\psi)\cdot \mathcal{O}_{\mathfrak{p}} = \frac{{\rm char}_\mathfrak{p}(\sha_\psi(A_F))\cdot \prod_{v\in S^*_{p,{\rm u}}}\rho_\mathfrak{p}^\psi(\mu_v(A_{F/k}))}{|G|^{r_{A,\psi}}\cdot \prod_{v \in S_k^p\cap S_k^F}\varrho_{v,\psi}^d\cdot\prod_{v\in S_k^F\cap S_k^f}P_v(A,\check\psi,1)}.\]
Here $$r_{A,\psi}:={\rm dim}_\CC(\CC\cdot A^t(F)^\psi)$$
while $S^*_{p,{\rm u}}$ is the set of $p$-adic places of $k$ that are unramified in $F$ but divide the different of $k/\QQ$ and, for every $v\in S_k^F\cap S_k^f$, $P_v(A,\check\psi,1)$ is the value at $z=1$ of the Euler factor at $v$ of the $\check\psi$-twist of $A$, while each term $\mu_v(A_{F/k})$ is defined in (\ref{localFM}) and each term $\varrho_{v,\psi}$ in (\ref{revisionVARRHO}).
\end{itemize}
\end{proposition}

\begin{proof} The first assertion of claim (i) is equivalent to asserting that the element
\[ \mathcal{L}^\ast := \sum_{\psi \in \widehat{G}}\mathcal{L}^\ast(A,\psi)\cdot e_\psi\]
belongs to the subgroup $\zeta(\QQ[G])^\times$ of $\zeta(\RR[G])^\times$.

Recalling that $\zeta(\QQ[G])^\times$ is the full pre-image under $\delta_G$ of the subgroup $K_0(\ZZ[G],\QQ[G])$ of $K_0(\ZZ[G],\RR[G])$, it is therefore enough to prove that $\delta_G(\mathcal{L}^\ast)$ belongs to $K_0(\ZZ[G],\QQ[G])$.

To do this we fix any basis of differentials $\omega_\bullet$ as in the statement of ${\rm BSD}(A_{F/k})$ and write $\mathcal{L}^\ast$ as a product $(\mathcal{L}^\ast_1)^{-1}\cdot \mathcal{L}^\ast_2\cdot (\mathcal{L}^\ast_3)^{-1}$ with
\[ \mathcal{L}^\ast_1 := {\rm Nrd}_{\RR[G]}(\Omega_{\omega_\bullet}(A_{F/k}))^{-1} \cdot \sum_{\psi\in \widehat{G}} \Omega_A^\psi\cdot w^d_\psi\cdot\tau^\ast(\QQ,\psi)^{-d}\cdot e_\psi,\]
\[ \mathcal{L}^\ast_2 := {\rm Nrd}_{\RR[G]}(\Omega_{\omega_\bullet}(A_{F/k}))^{-1}\cdot \sum_{\psi \in \widehat{G}}L^\ast(A,\check\psi,1)\cdot e_\psi,\]
and
\[ \mathcal{L}^\ast_3 := \sum_{\psi\in \widehat{G}}R_A^\psi\cdot e_\psi.\]

Proposition \ref{lms}(i) implies $\mathcal{L}^\ast_1$ belongs to $\zeta(\QQ[G])^\times$. In addition, for any set of places $S$ as in the statement of ${\rm BSD}(A_{F/k})$, the element $\mu_S(A_{F/k})$ belongs to $K_0(\ZZ[G],\QQ[G])$. We next note that $\delta_G(\mathcal{L}^\ast_2)$ differs from the left hand side of the equality in ${\rm BSD}(A_{F/k})$(iv) by
$$\sum_{v\in S\cap S_k^f}\delta_G(L_v(A,F/k)),$$ where $L_v(A,F/k)\in\zeta(\QQ[G])^\times$ is the equivariant Euler factor of $(A,F/k)$ at $v$ (see Appendix \ref{consistency section} below).

This difference belongs to $K_0(\ZZ[G],\QQ[G])$ and therefore the validity of  ${\rm BSD}(A_{F/k})$ implies that $$\delta_G(\mathcal{L}^\ast_2)- \chi_G({\rm SC}_{S,\omega_\bullet}(A_{F/k}),h_{A,F})$$ also belongs to $K_0(\ZZ[G],\QQ[G])$.

To prove $\mathcal{L}^\ast \in \zeta(\QQ[G])^\times$ it is therefore enough to note that $\delta_G(\mathcal{L}^\ast_3)$ also differs from $\chi_G({\rm SC}_{S,\omega_\bullet}(A_{F/k}),h_{A,F})$ by an element of $K_0(\ZZ[G],\QQ[G])$ (as one verifies by straightforward computation).

The second assertion of claim (i) is true since if $\omega$ is any element of $G_{\QQ^c/E}$, then $\omega\circ \psi = \psi$ and so $\omega(\mathcal{L}^\ast(A,\psi)) = \mathcal{L}^\ast(A,\omega\circ\psi) = \mathcal{L}^\ast(A,\psi)$.

Turning to claim (ii) we note that the given hypotheses imply that the data $A, F/k$ and $p$ satisfy the conditions of Theorem \ref{bk explicit}. To prove claim (ii) it is therefore enough to show that, if the difference between the left and right hand sides of the equality in Theorem \ref{bk explicit} belongs to $K_0(\ZZ_p[G],\QQ_p[G])$, then its image under $\rho_{\mathfrak{p}}^{\psi}$ is the equality in claim (ii). Since this image is independent of the choice of isomorphism $j:\CC\cong\CC_p$ we will omit it from all notations.

The group $I(\mathcal{O}_\mathfrak{p})$ is torsion-free and so Proposition \ref{basic props}(ii) implies that each term $R_{F/k}(\tilde A^t_v)$ that occurs in Theorem \ref{bk explicit} belongs to $\ker(\rho_{\mathfrak{p}}^{\psi})$.

Thus, since $\mathcal{L}^\ast(A,\psi)$ differs from the element $\mathcal{L}^\ast_{A,\psi}$ defined in (\ref{bkcharelement}) by the equality

\[ \mathcal{L}^\ast(A,\psi) = \mathcal{L}^\ast_{A,\psi}\cdot (R_A^\psi)^{-1}\prod_{v \in S_k^p\cap S_k^F}\varrho_{v,\psi}^{-d}\cdot\prod_{v\in S_k^F\cap S_k^f}P_v(A,\check\psi,1)^{-1}\]
the claimed result will follow if we can show that $\rho_{\mathfrak{p}}^{\psi}$ sends the element
\[ \chi_{G,p}( {\rm SC}_p(A_{F/k}),h_{A,F})-\delta_{G,p}(\mathcal{L}^\ast_3)\]
of $K_0(\ZZ_p[G],\QQ_p[G])$ to the ideal $|G|^{-r_{A,\psi}}\cdot {\rm char}_\mathfrak{p}(\sha_\psi(A_F))$.

Now, under the given hypotheses, Proposition \ref{explicitbkprop} implies that the complex $C := {\rm SC}_p(A_{F/k})$ is acyclic outside degrees one and two and has cohomology $A^t(F)_p$ and $X_\ZZ(A_F)_p=\Sel_p(A_F)^\vee$ in these respective degrees.

In particular, since $A^t(F)_p$ is torsion-free we can fix a representative of $C$ of the form $C^1 \xrightarrow{d} C^2$, where $C^1$ and $C^2$ are free $\ZZ_p[G]$-modules of the same rank.

Then the tautological exact sequence
\begin{equation} 0 \rightarrow H^1(C) \xrightarrow{\iota} C^1 \xrightarrow{d}
C^2 \xrightarrow{\pi} H^2(C) \rightarrow
0\label{tatseq}\end{equation}
induces a commutative diagram of ${\mathcal O}_p$-modules with exact rows
\[\begin{CD}
@. @. C^1_{\psi} @> d_\psi >> C^2_{\psi} @> \pi_\psi
>> H^2(C)_\psi @>  >> 0\\ @. @. @V {t^1_\psi} VV @V
{t^2_\psi} VV  \\ 0 @> >> H^{1}(C)^\psi @> \iota^\psi >>
 C^{1,\psi} @> d^\psi >> C^{2,\psi}.\end{CD}\]
Each vertical morphism $t^i_\psi$ here is induced by sending each $x$ in $\Hom_{{\mathcal O}_p}(T_{\psi,p}, {\mathcal O}_p\otimes_{\ZZ_p}C^i)$ to $\sum_{g \in G}g(x)$ and is bijective since the $\ZZ_p[G]$-module $C^i$ is free.

This diagram gives rise to an exact sequence of ${\mathcal O}_\mathfrak{p}$-modules
\begin{equation}\label{scal}
0\rightarrow
H^{1}(C)^\psi_\mathfrak{p}
\xrightarrow{\iota^\psi}C_\mathfrak{p}^{1,\psi}\xrightarrow{d^\psi}
C_\mathfrak{p}^{2,\psi} \xrightarrow{\pi_\psi\circ (t^2_\psi)^{-1}} H^2(C)_{\psi,\mathfrak{p}}
\rightarrow 0\end{equation}
which in turn implies that
\begin{equation}\label{firstform} \rho_{\mathfrak{p}}^{\psi}(\chi_{G,p}( {\rm SC}_p(A_{F/k}),h_{A,F})) = \iota_\mathfrak{p}\bigl(\chi_{\mathcal{O}_\mathfrak{p}}(C^{\bullet,\psi}_{\mathfrak{p}}, \tilde h^\psi)\bigr).\end{equation}
Here $C^{\bullet,\psi}_{\mathfrak{p}}$ denotes the complex $C^{1,\psi}_\mathfrak{p} \xrightarrow{d^\psi} C^{2,\psi}_\mathfrak{p}$ and $\tilde h^\psi$ the composite isomorphism of $\CC_p$-modules
\[ \CC_p\cdot H^1(C^{\bullet,\psi}_{\mathfrak{p}}) \cong \CC_p\cdot (A^t(F)^\psi)_\mathfrak{p} \xrightarrow{h^\psi} \Hom_{\CC_p}(\CC_p\cdot (A(F)^{\check\psi})_\mathfrak{p},\CC_p) \cong \CC_p\cdot H^2(C^{\bullet,\psi}_{\mathfrak{p}})\]
in which the first and third isomorphisms are induced by the maps in (\ref{scal}) and $h^\psi$ is induced by the isomorphism $h^\psi_{A,F}$.

Given the definition of each term $R_A^\psi$ it is, on the other hand, clear that
\begin{equation}\label{secondform} \rho^\psi_\mathfrak{p}(\delta_{G,p}(\mathcal{L}^\ast_3)) = \iota_\mathfrak{p}\bigl(\chi_{\mathcal{O}_\mathfrak{p}}(H^1(C^{\bullet,\psi}_{\mathfrak{p}})[-1]\oplus H^2(C^{\bullet,\psi}_{\mathfrak{p}})
_{\rm tf}[-2], h^\psi)\bigr).\end{equation}

Now, since $\mathcal{O}_\mathfrak{p}$ is a discrete valuation ring it is straightforward to construct an exact triangle in $D(\mathcal{O}_\mathfrak{p})$ of the form
\[ H^2(C^{\bullet,\psi}_{\mathfrak{p}})
_{\rm tor}[-2] \to C^{\bullet,\psi}_{\mathfrak{p}} \to H^1(C^{\bullet,\psi}_{\mathfrak{p}})[-1]\oplus H^2(C^{\bullet,\psi}_{\mathfrak{p}})
_{\rm tf}[-2] \to H^2(C^{\bullet,\psi}_{\mathfrak{p}})
_{\rm tor}[-1]\]
and Lemma \ref{fk lemma} applies to this triangle to imply that
\begin{align}\label{thirdform}
  &\chi_{\mathcal{O}_\mathfrak{p}}(C^{\bullet,\psi}_{\mathfrak{p}}, \tilde h^\psi) - \chi_{\mathcal{O}_\mathfrak{p}}(H^1(C^{\bullet,\psi}_{p})[-1]\oplus H^2(C^{\bullet,\psi}_{\mathfrak{p}})
_{\rm tf}[-2], h^\psi)\notag\\
=\,&\chi_{\mathcal{O}_\mathfrak{p}}(H^2(C^{\bullet,\psi}_{\mathfrak{p}})
_{\rm tf}[-1]\oplus H^2(C^{\bullet,\psi}_{\mathfrak{p}})
_{\rm tf}[-2], \tilde h^\psi\circ ( h^\psi)^{-1}) +\chi_{\mathcal{O}_\mathfrak{p}}(H^2(C^{\bullet,\psi}_{\mathfrak{p}})
_{\rm tor}[-2], 0).\end{align}

Next we note the definition of $\sha_\psi(A_F)$ ensures $\sha_\psi(A_F)_\mathfrak{p} = H^2(C^{\bullet,\psi}_{\mathfrak{p}})
_{\rm tor}$ and hence that
\begin{equation}\label{fourthform}\iota_\mathfrak{p}\bigl(\chi_{\mathcal{O}_\mathfrak{p}}(H^2(C^{\bullet,\psi}_{\mathfrak{p}})
_{\rm tor}[-2], 0)\bigr) =  {\rm char}_\mathfrak{p}(\sha_\psi(A_F)).\end{equation}

In addition, after identifying both $\CC_p\cdot H^2(C)_{\psi,\mathfrak{p}}$ and $\CC_p\cdot H^2(C)^\psi_{\mathfrak{p}}$ with $e_\psi(\CC_p\cdot H^2(C)_{\mathfrak{p}})$ in the natural way, the map $t^2_\psi$ that occurs in (\ref{scal}) induces upon the latter space the map given by multiplication by $|G|$.

To derive claim (ii) from the displayed formulas (\ref{firstform}), (\ref{secondform}), (\ref{thirdform}) and (\ref{fourthform}) it is thus enough to note that
\[ \dim_{\CC_p}(\CC_p\cdot H^2(C^{\bullet,\psi}_{\mathfrak{p}}))=\dim_{\CC_p}(\CC_p\cdot (A^t(F)^\psi)_\mathfrak{p})=r_{A,\psi},\]
and hence that
\begin{multline*} \chi_{\mathcal{O}_\mathfrak{p}}(H^2(C^{\bullet,\psi}_{\mathfrak{p}})
_{\rm tf}[-1]\oplus H^2(C^{\bullet,\psi}_{\mathfrak{p}})
_{\rm tf}[-2], \tilde h^\psi\circ ( h^\psi)^{-1})\\ = \chi_{\mathcal{O}_\mathfrak{p}}(H^2(C^{\bullet,\psi}_{\mathfrak{p}})
_{\rm tf}[-1]\oplus H^2(C^{\bullet,\psi}_{\mathfrak{p}})
_{\rm tf}[-2], \cdot |G|^{-1})\end{multline*}
is sent by $\iota_\mathfrak{p}$ to the ideal generated by $|G|^{-r_{A,\psi}}$.
\end{proof}

\begin{remark}\label{evans}{\em Fix a Galois extension $F$ of $k = \QQ$ and an elliptic curve $A$ whose conductor $N_A$ is prime to the discriminant $d_F$ of $F$ and is such that $A(F)$ is finite. Then for each $\psi$ in $\widehat{G}$ one has $r_{A,\psi}=0$, and hence $R_A^\psi = 1$, so that the complex number $\mathcal{L}^\ast(A,\psi)$ agrees up to a unit of $\mathcal{O}$ with the element $\mathcal{L}(A,\psi)$ that is defined in \cite[Def. 12]{vdrehw}. Now fix an odd prime $p$ that is prime to $d_F$, to $N_A$, to the order of $A(F)$, to the order of the group of points of the reduction of $A$ at any prime divisor of $d_F$ and to the Tamagawa number of $A_F$ at each place of bad reduction. Then the data $A, F/k$ and $p$ satisfy the hypotheses of Proposition \ref{ref deligne-gross}(ii) and the sets  $S^*_{p,{\rm u}}$ and $S_k^p\cap S_k^F$ are empty (the former since $k = \QQ$). The explicit formula in the latter result therefore simplifies to give
\[ \mathcal{L}(A,\psi)\cdot \mathcal{O}_{\mathfrak{p}} = {\rm char}_\mathfrak{p}(\sha_\psi(A_F)) \cdot\prod_{\ell\mid d_F}P_\ell(A,\check\psi,1)^{-1}\]
where in the product $\ell$ runs over all prime divisors of $d_F$. This formula shows that, in any such case, the fractional $\mathcal{O}$-ideal generated by $\mathcal{L}(A,\psi)$ should depend crucially on the structure of $\sha(A_F)$ as a $G$-module, as already suggested in this context by Dokchitser, Evans and Wiersema in \cite[Rem. 40]{vdrehw}. In particular, this observation is both consistent with, and helps to clarify, the result of [loc. cit., Th. 37(2)].

We also note that Wiersema and Wuthrich \cite{hwcw} have recently studied the integrality properties of $\mathcal{L}(A,\psi)$. }
\end{remark}

\section{Abelian congruence relations and module structures}\label{congruence sec} 
%

In both this and the next section we apply the general results of Sano, Tsoi and the first author in \cite{bst} to derive from the assumed validity of ${\rm BSD}_p(A_{F/k})$(iv) families of $p$-adic congruences that can be much more explicit than those discussed in Remark \ref{cons1}.

In particular, in this section we focus on predictions relating to the Galois structures of Selmer and Tate-Shafarevich groups.


In contrast to \S\ref{tmc}, it is no longer essential for us to assume any of the hypotheses (H$_1$)-(H$_5$),  although they do allow us to make some of the predictions more explicit. However, we will, unless explicitly stated otherwise, restrict to the case that $G$ is abelian and so will not distinguish between the leading coefficient element $L_S^*(A_{F/k},1)$ in $K_1(\RR[G])$ and its reduced norm $\sum_{\psi \in \widehat G}L_S^*(A,\check\chi,1)\cdot e_\psi$ in $\RR[G]^\times$. (An extension of the results of this section to the general (non-abelian) setting will be given in the upcoming article \cite{dmckwt}.)

Throughout this section we fix an odd prime $p$ and an isomorphism of fields $\CC\cong\CC_p$ (that we will usually not mention). We also fix a finite set $S$ of places of $k$ with
\[ S_k^\infty\cup S_k^p\cup  S_k^F \cup S_k^A\subseteq S.\]

 Writing $d$ for the dimension of $A$, we assume to be given an ordered $k$-basis $\{\omega'_j: j \in [d]\}$ of $H^0(A^t,\Omega^1_{A^t})$ and use this basis to define a classical period $\Omega_A^{F/k}$ in $\CC[G]^{\times}$ as in (\ref{period def}).

The main result of this section is stated as Theorem \ref{big conj} in \S\ref{8.1}. However, for the benefit of those  mainly interested in the sort of concrete predictions that this general result suggests, we first record a selection of predictions that follow upon appropriate specialization.

\subsection{Some explicit predictions}\label{newsection} The motivation for each of the predictions listed here is explained   in \S \ref{justifications}. 

For a non-negative integer $a$ we write $\widehat{G}_{A,(a)}$ for the subset of $\widehat{G}$ comprising characters $\psi$ for which the  $L$-series $L(A,\psi,z)$ vanishes at $z=1$ to order at least $a$. This definition ensures that the $\CC[G]$-valued function
\begin{equation}\label{revisionLa} L^{(a)}_{S}(A_{F/k},z) := \sum_{\psi \in \widehat{G}_{A,(a)}}z^{-a}L_S(A,\check\psi,z)\cdot e_\psi\end{equation}
%
%
%
is holomorphic at $z=1$.

For $a\geq 0$ and ordered $a$-tuples $P_\bullet = \{P_i: i \in [a]\}$ of $A^t(F)$ and $Q_\bullet = \{Q_j: j \in [a]\}$ of $A(F)$  we also define a matrix in ${\rm M}_a(\RR[G])$ by setting
\begin{equation}\label{regulatormatrix} h_{F/k}(P_\bullet, Q_\bullet) := (\sum_{g \in G}\langle g(P_i),Q_j\rangle_{A_F}\cdot g^{-1})_{1\le i,j\le a},\end{equation}
where $\langle -,-\rangle_{A_F}$ denotes the N\'eron-Tate height pairing for $A$ over $F$.

In the following predictions we shall consider the following list of assumptions on  $A$, $F/k$, $p$ and a given non-negative integer $a$:

\begin{itemize}
\item[(A$_1$)] all $p$-adic places of $k$ are at most tamely ramified in $F$;
\item[(A$_2$)] there exist $a$-tuples in $A^t(F)$ and $A(F)$ that are each linearly independent over $\QQ[G]$ (see Remark \ref{e=1 case} below for examples of such tuples), and neither the formal group $\widehat{A^t}(\wp_{F_p})$ in (\ref{formalgp}) below nor $A(F)$ has an element of order $p$;
\item[(A$_3$)] $A$, $F/k$ and $p$ satisfy the hypotheses (H$_1)$-(H$_5$) that are listed in \S\ref{tmc} and, in addition, $p$ is unramified in $k$ and $A(F)[p]=A^t(F)[p]=0$.
\end{itemize}

\smallskip

\begin{prediction}\label{new add}{\em 
Assume (A$_1)$ and (A$_2$). Fix $a\geq 0$ and ordered $a$-tuples $P_\bullet$ of $A^t(F)$ and $Q_\bullet$ of $A(F)$ that are each linearly independent over $\QQ[G]$. Then 
any element in the set
\begin{equation*}\label{explicit ann}
\frac{L^{(a)}_{S}(A_{F/k},1)\cdot \bigl(\tau^\ast(F/k)\cdot \prod_{v \in S_k^p}\varrho_v(F/k)\bigr)^d}{\Omega_A^{F/k}\cdot w_{F/k}^d\cdot {\rm det}(h_{F/k}(P_\bullet, Q_\bullet))} \cdot {\rm Fit}^0_{\ZZ[G]}((A^t(F)/\langle P_\bullet\rangle)^\vee_{\rm tor})\cdot{\rm Fit}^0_{\ZZ[G]}((A(F)/\langle Q_\bullet\rangle)^\vee_{\rm tor})\end{equation*}
belongs to ${\rm Fit}^a_{\ZZ_{p}[G]}({\rm Sel}_{p}(A_{F})^\vee)$ and annihilates $\sha(A^t_{F})[p^\infty]$. Here the elements $\varrho_v(F/k)$ are as defined in (\ref{varrho def}), $\langle P_\bullet\rangle$ and $\langle Q_\bullet\rangle$ denote the $G$-modules that are generated by the sets $P_\bullet$ and $Q_\bullet$ and ${\rm Fit}^a_{\ZZ_{p}[G]}({\rm Sel}_{p}(A_{F})^\vee)$ denotes the $a$-th Fitting ideal of the $\ZZ_p[G]$-module ${\rm Sel}_p(A_F)^\vee$.}\end{prediction}

\begin{example}\label{wuthrich example}{\em Christian Wuthrich kindly supplied us with the following concrete instances of Prediction \ref{new add}.  Set $k = \QQ$ and $K = \QQ(\sqrt{229})$ and write $F$ for the Galois closure of the field $L = \QQ(\alpha)$ with $\alpha^3-4\alpha+1 = 0$. Then $K \subset F$ and the group $G := G_{F/k}$ is dihedral of order six. Let $A$ denote either of the curves 3928b1 (with equation $y^2 = x^3-x^2 + x + 4$) or 5864a1 (with equation $y^2 = x^3-x^2 -24x + 28$). Then ${\rm rk}(A_\QQ)= 2$, ${\rm rk}(A_K) = {\rm rk}(A_L) = 3$ and ${\rm rk}(A_F) = 5$ and, since $\sha_3(A_K)$ vanishes (as can be shown via a computation with Heegner points on the quadratic twist of $A$), these facts combine with \cite[Cor. 2.10(i)]{bmw0} to imply the $\ZZ_{3}[G]$-module $A(F)_{3}$ is isomorphic to $\ZZ_{3}[G](1-\tau) \oplus \ZZ_{3}\oplus \ZZ_{3}$, with $\tau$ the unique non-trivial element in $G_{F/L}$. In particular, if we set $\Gamma := G_{F/K}$, then we can choose a point $P$ that generates over $\ZZ_{3}[\Gamma]$ a free direct summand of $A(F)_{3}$. In addition, ${\rm Fit}^1_{\ZZ_{3}[\Gamma]}({\rm Sel}_{3}(A_{F})^\vee)$ is contained in
\[ {\rm Fit}^1_{\ZZ_{3}[\Gamma]}(\ZZ_{3}[G](1-\tau) \oplus \ZZ_{3}\oplus \ZZ_{3}) = {\rm Fit}^0_{\ZZ_{3}[\Gamma]}(\ZZ_{3}\oplus \ZZ_{3}) = I_{3}(\Gamma)^2,\]
where $I_{3}(\Gamma)$ is the augmentation ideal of $\ZZ_{3}[\Gamma]$. Finally, we note that $3$ splits in $K$ and is unramified in $F$ so that for each $3$-adic place $v$ of $K$ the element $\varrho_v(F/K)$ is equal to $3$. After taking account of these facts, Prediction \ref{new add} (with $F/k$ taken to be $F/K$, $p$ to be $3$, $a$ to be $1$ and $P_1 = Q_1$ to be $P$) states that 
\[ 9\cdot \tau^\ast(F/K)\cdot\frac{L^{(1)}_{S}(A_{F/K},1)}{\Omega_A^{F/k}} = x\cdot \sum_{g \in G}\langle g(P),P\rangle_{A_F}\cdot g^{-1}\]
for an element $x$ of $I_{3}(\Gamma)^2$ that annihilates the $3$-primary component of $\sha(A_{F})$. Here we have also used the fact that $w_{F/K}=1$ because each of the real places $v$ of $K$ has trivial decomposition group in $\Gamma$ so $\psi^-_v(1)=1-1=0$ and thus $w_\psi=1$ for each $\psi\in\widehat\Gamma$.}
\end{example}

\smallskip

To formulate another prediction we define, for each $a\geq 0$, idempotents of $\QQ[G]$ by setting
\[ e_{(a)} = e_{F,(a)} := \sum_{\psi \in \widehat{G}_{A,(a)}}e_\psi \,\,\,\,\text{ and }\,\,\,\, e_{a} = e_{F,a}:= \sum_{\psi \in \widehat{G}_{A,(a)}\setminus \widehat{G}_{A,(a+1)}}e_\psi\]
(so that $e_{(a)} = \sum_{b \ge a}e_b$).

We also observe that, 
%
%
for each non-negative integer $a$, the N\'eron-Tate height (\ref{height triv}) for $A$ over $F$ combines with our fixed isomorphism of fields $\CC\cong \CC_p$ to induce an isomorphism of $\CC_p[G]$-modules
\begin{equation}\label{athpower} {\rm ht}^{a}_{A_{F/k}}: \CC_p\cdot {\bigwedge}^a_{\ZZ_p[G]}\Hom_{\ZZ_p[G]}(A(F)_p,\ZZ_p[G]) \cong \CC_p\cdot{\bigwedge}^a_{\ZZ_p[G]} A^t(F)_p.\end{equation}

In the sequel we shall also use (the scalar extension of) the canonical `evaluation' pairing $${\bigwedge}^a_{\ZZ_p[G]} A^t(F)_p\times{\bigwedge}^a_{\ZZ_p[G]} \Hom_{\ZZ_p[G]}(A^t(F)_p,\ZZ_p[G])\to \ZZ_p[G],$$
which, in the usual way, we write as $(x,y) \mapsto y(x)$. 
\smallskip

\begin{prediction}\label{new remark SC}{\em 
Assume (A$_1)$ and (A$_3$). Then 
for any given non-negative integer $a$, any (ordered) subsets $\{\theta_j: j \in [a]\}$ and $\{\phi_i: i\in [a]\}$ of $\Hom_{\ZZ_p[G]}(A^t(F)_p,\ZZ_p[G])$ and $\Hom_{\ZZ_p[G]}(A(F)_p,\ZZ_p[G])$ and any element $\alpha$ of $\ZZ_p[G]\cap \ZZ_p[G]e_{(a)}$, the product
\begin{equation*}\label{key product2} \alpha^{1+2a}\cdot(e_{F,a}\cdot\mathcal{L}_{A,F/k}^*)\cdot (\wedge_{j=1}^{j=a}\theta_j)({\rm ht}^{a}_{A_{F/k}}(\wedge_{i=1}^{i=a}\phi_i))\end{equation*}
belongs to ${\rm Fit}^a_{\ZZ_p[G]}(\Sel_p(A_F)^\vee)$ and annihilates $\sha(A_F^t)[p^\infty]$. Here $\mathcal{L}_{A,F/k}^*$ is the element defined in (\ref{bkcharelement}) and so is related to $L$-values
truncated only at the places in $S_{\rm r} :=S_k^f\cap S_k^F$ rather than at all places in $S$.
}\end{prediction}
\smallskip

\begin{prediction}\label{new add2}{\em Assume (A$_1)$, (A$_2$) and (A$_3$). 
Then 
Prediction \ref{new add} is true after replacing the terms $L^{(a)}_{S}(A_{F/k},1)$ and $S_k^p$ in (\ref{explicit ann}) by $L^{(a)}_{S_{\rm r}}(A_{F/k},1)$ and $S_{p,{\rm r}}:= S_k^p\cap S_k^F$ respectively. %
}\end{prediction}

%
%
\smallskip

\begin{prediction}\label{integrality rk}{\em Assume (A$_1)$ and (A$_2$), respectively (A$_1)$ and (A$_3$), and write $\mathcal{L}$ for the displayed expressions in Predictions \ref{explicit ann} and \ref{new remark SC} respectively. Then, for every $g$ in $G$, the elements $\mathcal{L}_\psi$ of $\CC_p$ that are defined via the equality $\mathcal{L} = \sum_{\psi\in \widehat{G}_{A,(a)}}\mathcal{L}_\psi\cdot e_\psi$ are related by the congruence 
\[ \sum_{\psi \in \widehat{G}_{A,(a)}}\psi(g)\mathcal{L}_\psi \equiv 0 \,\,\,({\rm mod}\,\, |G|\cdot \ZZ_p).\]
(Note that, in both cases, each term $\mathcal{L}_\psi$ can be explicitly related to the value at $z=1$ of the function $z^{-a}L(A,\check\psi,z)$).
}\end{prediction}
\smallskip

\subsection{Statement of the main result}\label{8.1}


%


\subsubsection{}

In order to state the main result of this section we must first extend the definition of logarithmic resolvents given in (\ref{log resol abelian}) to the setting of abelian varieties.

To do this we do not require $F/k$ to be abelian. 
For each index $j\in[d]$ we write ${\rm log}_{\omega'_j}$ for the formal logarithm of $A^t$ over $F_p$ that is defined with respect to $\omega_j'$.

We also fix an ordering of $\Sigma(k)$. We set $n := [k:\QQ]$, write $\CC_p[G]^{nd}$ for the direct sum of $nd$ copies of $\CC_p[G]$ and fix a bijection between the standard basis of this module and the lexicographically-ordered direct product $[d]\times \Sigma(k)$.

Then for any ordered subset
\begin{equation}\label{setx} x_\bullet:= \{x_{(i,\sigma)}: (i,\sigma) \in [d]\times\Sigma(k)\}\end{equation}
of $A^t(F_p)^\wedge_p$ we define a logarithmic resolvent element of $\zeta(\CC_p[G])$ by setting
\[ \mathcal{LR}^p_{A^t_{F/k}}(x_\bullet) := {\rm Nrd}_{\QQ^c_p[G]}\left(\bigl(\sum_{g \in G} \hat \sigma(g^{-1}({\rm log}_{\omega_j'}(x_{(j',\sigma')})))\cdot g \bigr)_{(j,\sigma),(j',\sigma')}\right) \]
where the indices $(j,\sigma)$ and $(j',\sigma')$ run over $[d]\times \Sigma(k)$ and ${\rm Nrd}_{\QQ^c_p[G]}(-)$ denotes the reduced norm of the given matrix in ${\rm M}_{dn}(\QQ_p^c[G])$.

It is clear that if $A$ is an elliptic curve (so $d=1$) and $F/k$ is abelian, then the `$\psi$-component' of this definition agrees with (\ref{log resol abelian}).

\subsubsection{}\label{statementstructure} We may now state the main result of this section. For each non-archimedean place $v$ of $k$ that does not ramify in $F/k$ and at which $A$ has good reduction we define an element of $\QQ [G]$ by setting
\[ P_v(A_{F/k},1) := 1-\Phi_v^{-1}\cdot a_v +  \Phi_v^{-2}\cdot {\rm N}v^{-2}.\]
Here $a_v$ is the integer $1 + {\rm N}v - | A(\kappa_v)|$.

This result will concern the value $L^{(a)}_{S}(A_{F/k},1)$ at $z=1$ of the function (\ref{revisionLa}), the fixed period $\Omega_A^{F/k}$, the element $w_{F/k}$ in (\ref{root number def}) and the $a$-th Fitting ideal of the Pontryagin dual of the $p$-primary Selmer module ${\rm Sel}_p(A_F)$.

The proof of this result will be given in \S\ref{proof of big conj}.

\begin{theorem}\label{big conj} Fix an ordered maximal subset $x_\bullet:= \{x_{(i,\sigma)}: (i,\sigma) \in [d]\times\Sigma(k)\}$ of $A^t(F_p)^\wedge_p$ that is linearly independent over $\ZZ_p[G]$ and a finite non-empty set $T$ of places of $k$ that is disjoint from $S$.

If ${\rm BSD}(A_{F/k})$ is valid, then for any non-negative integer $a$, any subsets $\{\theta_j: j \in [a]\}$ and $\{\phi_i: i\in [a]\}$ of $\Hom_{\ZZ_p[G]}(A^t(F)_p,\ZZ_p[G])$ and $\Hom_{\ZZ_p[G]}(A(F)_p,\ZZ_p[G])$ respectively, and any element $\alpha$ of $\ZZ_p[G]\cap \ZZ_p[G]e_{(a)}$ the product
\begin{equation}\label{key product} \alpha^{1+2a}\cdot (\prod_{v \in T}P_v(A_{F/k},1)) \cdot  \frac{L^{(a)}_{S}(A_{F/k},1)}{\Omega_A^{F/k}\cdot w_{F/k}^d}\cdot \mathcal{LR}^p_{A^t_{F/k}}(x_\bullet)\cdot (\wedge_{j=1}^{j=a}\theta_j)({\rm ht}^{a}_{A_{F/k}}(\wedge_{i=1}^{i=a}\phi_i))\end{equation}
belongs to ${\rm Fit}^a_{\ZZ_p[G]}(\Sel_p(A_F)^\vee)$  and annihilates $\sha(A^t_{F})[p^\infty]$.
\end{theorem}


We remark on several ways in which this result either simplifies or becomes more explicit. 

\begin{remarks}\label{more explicit rem}{\em \

\noindent{}(i) If $A(F)$ does not contain an element of order $p$, then our methods will show that the prediction in Theorem \ref{big conj} should remain true if the term $\prod_{v \in T}P_v(A_{F/k},1)$ is omitted. For more details see Remark \ref{omit T} below.

\noindent{}(ii) If one fixes a subset $\{\phi_i: i \in [a]\}$ of $\Hom_{\ZZ_p[G]}(A(F)_p,\ZZ_p[G])$ of cardinality $a$ that generates a free direct summand of rank $a$, then our approach combines with \cite[Th. 3.10(ii)]{bst} to suggest that, as the subset $\{\theta_j: j \in [a]\}$ varies, elements of the form (\ref{key product}) can be used to give an explicit description of ${\rm Fit}^a_{\ZZ_p[G]}(\Sel_p(A_F)^\vee)$.}\end{remarks}

\begin{remark}\label{e=1 case}{\em In special cases one can show that there exist $a$-tuples in $A^t(F)$ and $A(F)$ that are each linearly independent over $\QQ[G]$, as in assumption (A$_2$) in \S \ref{newsection}. In any such cases one can then either show, or is led to predict (for example via ${\rm BSD}(A_{F/k})$(ii)), that the idempotent $e_{(a)}$ is equal to $1$ and so the element $\alpha$ in (\ref{key product}) can be taken to be $1$. 

This is, for example, the case if $a = 0$, since each function $L(A, \psi,z)$ is holomorphic at $z=1$ and, in the setting of abelian extensions of $\QQ$, this case will be considered in detail in \S\ref{mod sect}. This situation can also arises naturally in cases with $a > 0$, such as the following.

\noindent{}(i) If $F$ is a ring class field of an imaginary quadratic field $k$ and suitable hypotheses are satisfied by an elliptic curve $A/\QQ$ and the extension $F/\QQ$, then the existence of a Heegner point in $A(F)$ with non-zero trace to $A(k)$ combines with the theorem of Gross and Zagier to imply that $e_{(1)} = 1$. This case will be considered in detail in \S\ref{HHP}.

\noindent{}(ii) As a generalization of (i), if $F$ is a generalized dihedral extension of a field $F'$, $k$ is the unique quadratic extension of $F'$ in $F$, all $p$-adic places split completely in $k/F'$ and the rank of $A(k)$ is odd, then the result of Mazur and Rubin in \cite[Th. B]{mr2} shows the existence of a point in $A(F)$ that generates a free $\QQ[G]$-module.

\noindent{}(iii) Let $F$ be a finite extension of $k$ inside a $\ZZ_p$-extension $k_\infty$, set $\Gamma:= G_{k_\infty/k}$ and write $r_\infty$ for the corank of ${\rm Sel}_p(A_{k_\infty})$ as a module over the Iwasawa algebra $\ZZ_p[[\Gamma]]$. Then, if $A$, $F/k$  and $p$ satisfy all of the hypotheses listed at the beginning of \S\ref{tmc}, one can show that the inverse limit $\varprojlim_{F'}A^t(F')_p$, where $F'$ runs over all finite extensions of $F$ in $k_\infty$, is a free $\ZZ_p[[\Gamma]]$-module of rank $r_\infty$ and this in turn implies the existence of suitable $r_\infty$-tuples of points. 
}
\end{remark}

%
%

\begin{remark}\label{integrality rk II}{\em To obtain concrete congruences from the result of Theorem \ref{big conj} one can proceed exactly as in Prediction \ref{integrality rk}. 
%
%
}\end{remark}

\subsection{Explicit regulator matrices}Motivated by Remark \ref{e=1 case}, and in order to justify Predictions \ref{new add} and \ref{new add2} (see \S \ref{justifications} below), we consider in more detail the case in which there exist $a$-tuples in $A^t(F)$ and $A(F)$ that are each linearly independent over $\QQ[G]$. In this case we may interpret the expressions $(\wedge_{j=1}^{j=a}\theta_j)({\rm ht}^a_{A_{F/k}}(\wedge_{i=1}^{i=a}\phi_i))$ in Theorem \ref{big conj} in terms of classical N\'eron-Tate heights, as follows.

\begin{lemma}\label{height pairing interp} Fix a natural number $a$ such that there exist ordered $a$-tuples $P_\bullet$ of $A^t(F)$ and $Q_\bullet$ of $A(F)$ that are each linearly independent over $\QQ[G]$. Then the matrix $e_a\cdot h_{F/k}(P_\bullet, Q_\bullet)$ belongs to ${\rm GL}_a(\RR[G]e_a)$ and one has
\begin{multline*}  e_a\cdot\left\{(\wedge_{j=1}^{j=a}\theta_j)({\rm ht}^a_{A_{F/k}}(\wedge_{i=1}^{i=a}\phi_i))\mid \theta_j\in \Hom_{\ZZ[G]}(A^t(F),\ZZ[G]),\phi_i\in \Hom_{\ZZ[G]}(A(F),\ZZ[G])\right\}\\
= {\rm det}(e_a\cdot h_{F/k}(P_\bullet, Q_\bullet))^{-1}\cdot{\rm Fit}^0_{\ZZ[G]}((A^t(F)_{\rm tf}/\langle P_\bullet\rangle)^\vee_{\rm tor})\cdot{\rm Fit}^0_{\ZZ[G]}((A(F)_{\rm tf}/\langle Q_\bullet\rangle)^\vee_{\rm tor}). \end{multline*}
%
\end{lemma}

\begin{proof} We write $N(P_\bullet)$ and $N(Q_\bullet)$ for the quotients of $A^t(F)_{\rm tf}$ and $A(F)_{\rm tf}$ by $\langle P_\bullet\rangle$ and $\langle Q_\bullet\rangle$. Then, by taking $\ZZ$-linear duals of the tautological short exact sequence
\[ 0 \to \langle P_\bullet\rangle \xrightarrow{\iota_{P_\bullet}} A^t(F)_{\rm tf} \to N(P_\bullet) \to 0\]
one obtains an exact sequence $$A^t(F)^\ast \xrightarrow{\iota_{P_\bullet}^\ast} \langle P_\bullet\rangle^\ast \to N(P_\bullet)_{\rm tor}^\vee \to 0$$ and hence an equality
\[ \im({\bigwedge}^a_{\ZZ[G]}\iota_{P_\bullet}^\ast) = {\rm Fit}^0_{\ZZ[G]}(N(P_\bullet)_{\rm tor}^\vee).\]

In the same way one derives an equality
\[ \im({\bigwedge}^a_{\ZZ[G]}\iota_{Q_\bullet}^\ast) = {\rm Fit}^0_{\ZZ[G]}(N(Q_\bullet)_{\rm tor}^\vee).\]

Since the maps $e_{a}(\QQ\otimes_\ZZ\iota_{P_\bullet}^\ast)$ and $e_{a}(\QQ\otimes_\ZZ\iota_{Q_\bullet}^\ast)$ are bijective, these equalities imply that the lattice
\[ e_a\cdot\left\{(\wedge_{j=1}^{j=a}\theta_j)({\rm ht}^a_{A_{F/k}}(\wedge_{i=1}^{i=a}\phi_i))\mid \theta_j\in A^t(F)^\ast,\phi_i\in A(F)^\ast\right\}\]
is equal to the product
\begin{multline*} \left\{e_a(\wedge_{j=1}^{j=a}\theta_j)({\rm ht}^a_{A_{F/k}}(e_a(\wedge_{i=1}^{i=a}\phi_i)))\mid \theta_j\in \langle P_\bullet\rangle^\ast,\phi_i\in \langle Q_\bullet\rangle^\ast\right\}\\
\times {\rm Fit}^0_{\ZZ[G]}(N(P_\bullet)_{\rm tor}^\vee)\cdot{\rm Fit}^0_{\ZZ[G]}(N(Q_\bullet)_{\rm tor}^\vee). \end{multline*}

This implies the claimed result since, writing $P_i^\ast$ and $Q_j^\ast$ for the elements of $\langle P_\bullet\rangle^\ast$ and $\langle Q_\bullet\rangle^\ast$ that are respectively dual to $P_i$ and $Q_j$, one has
\begin{multline*} \left\{e_a(\wedge_{j=1}^{j=a}\theta_j)({\rm ht}^a_{A_{F/k}}(e_a(\wedge_{i=1}^{i=a}\phi_i)))\mid \theta_j\in \langle P_\bullet\rangle^\ast,\phi_i\in \langle Q_\bullet\rangle^\ast\right\}\\
=  \ZZ[G]\cdot (e_a\wedge_{i=1}^{i=a}P_i^\ast)({\rm ht}^a_{A_{F/k}}(e_a\wedge_{i=1}^{i=a}Q^\ast_i))\end{multline*}
and
\begin{align*} e_a(\wedge_{i=1}^{i=a}P_i^\ast)({\rm ht}^a_{A_{F/k}}(e_a\wedge_{i=1}^{i=a}Q^\ast_i)) =\, &{\rm det}\bigl( (e_aP_i^\ast(h^{-1}_{A_{F/k}}(e_a Q^\ast_j)))_{1\le i, j\le a}\bigr)\\
 =\, &{\rm det}(e_a\cdot h_{F/k}(P_\bullet, Q_\bullet))^{-1}.\end{align*}
This last equality is true because the definition of ${\rm ht}^a_{A_{F/k}}$ implies that for every $j$ one has
\[ h^{-1}_{A_{F/k}}(e_a Q^\ast_j) = \sum_{b=1}^{b=a} ((e_a\cdot h_{F/k}(P_\bullet, Q_\bullet))^{-1})_{bj}P_b.\]
\end{proof}

\subsection{Bounds on logarithmic resolvents}\label{p-adic sec} 


In this section we discuss certain explicit interpretations of logarithmic resolvents.
For subsequent purposes (related to the upcoming article \cite{dmckwt}) we do not here require $G$ to be abelian.

\subsubsection{} We start by deriving an easy consequence of the arguments in Proposition \ref{lms} and Lemma \ref{ullom}. For each natural number $i$ we set $\wp_{F_p}^i:=\prod_{w'\in S_F^p}\wp_{F_{w'}}^i$, where $\wp_{F_{w'}}$ denotes the maximal ideal in the valuation ring of $F_{w'}$. We also set \begin{equation}\label{formalgp}\hat A^t(\wp_{F_p}^i):=\prod_{w'\in S_F^p}\hat A^t_{w'}(\wp_{F_{w'}}^i),\end{equation} where $\hat A^t_{w'}$ denotes the formal group of $A^t_{/F_{w'}}$.

\begin{proposition}\label{explicit log resolve} If all $p$-adic places are tamely ramified in $F/k$ and $\hat A^t(\wp_{F_p})$ is torsion-free, then there exists an ordered $\ZZ_p[G]$-basis $x_\bullet$ of $\hat A^t(\wp_{F_p})$ for which one has
\[ \mathcal{LR}^p_{A^t_{F/k}}(x_\bullet) = \bigl(\tau^\ast(F/k)\cdot \prod_{v \in S_k^p}\varrho_v(F/k)\bigr)^d,\]
where the elements $\varrho_v(F/k)$ of $\zeta(\QQ[G])^\times$ are as defined in (\ref{varrho def}).
\end{proposition}

\begin{proof} Since all $p$-adic places of $k$ are tamely ramified in $F$ Lemma \ref{ullom} implies that the $\ZZ_p[G]$-module $\hat A^t(\wp^{i}_{F_p})$ is cohomologically-trivial for all $i$. Hence, if $\hat A^t(\wp_{F_p})$ is torsion-free, then it is a projective $\ZZ_p[G]$-module (by \cite[Th. 8]{cf}) and therefore free of rank $nd$ (since $\QQ_p\otimes_{\ZZ_p}\hat A^t(\wp_{F_p})$ is isomorphic to $F^d_p$).

In this case we fix an ordered basis $x_\bullet$ of $\hat A^t(\wp_{F_p})$ and regard it as a $\QQ_p^c[G]$-basis of $\QQ_p^c\otimes_{\ZZ_p}\hat A^t(\wp_{F_p})$.

We also regard $\{(i,\hat\sigma): (i,\sigma) \in [d]\times\Sigma(k)\}$ as a $\QQ_p^c[G]$-basis of the direct sum $(\QQ_p^c\cdot Y_{F/k,p})^d$ of $d$ copies of $\QQ_p^c\cdot Y_{F/k,p}$.

Then, with respect to these bases, the matrix
\begin{equation}\label{log resolve matrix} \bigl(\sum_{g \in G} \hat \sigma(g^{-1}({\rm log}_{\omega_j'}(x_{(j',\sigma')})))\cdot g \bigr)_{(j,\sigma),(j',\sigma')}\end{equation}
represents the composite isomorphism of $\QQ^c_p[G]$-modules
\[ \mu': \QQ_p^c\otimes_{\ZZ_p}\hat A^t(\wp_{F_p}) \xrightarrow{({\rm log}_{\omega_j'})_{j\in [d]}}
(\QQ_p^c\otimes_\QQ F)^d \xrightarrow{(\mu)_{i \in [d]}} (\QQ_p^c\cdot Y_{F/k,p})^d,\]
where $\mu$ is the isomorphism $\QQ_p^c\otimes_\QQ F \cong \QQ_p^c\cdot Y_{F/k,p}$ that sends each element $\lambda\otimes f$ to $(\lambda\cdot \hat\sigma(f))_{\sigma\in \Sigma(k)}$. This fact implies that
\begin{equation}\label{eq 1}\delta_{G,p}(\mathcal{LR}^p_{A^t_{F/k}}(x_\bullet)) = [\hat A^t(\wp_{F_p}) , Y^d_{F/k,p}; \mu'] \end{equation}
in $K_0(\ZZ_p[G],\QQ^c_p[G])$.

In addition, since for any large enough integer $i$ the image of $\hat A^t(\wp^i_{F_p})$ under each map ${\rm log}_{\omega_j'}$ is equal to $\wp^i_{F,p}$, the `telescoping' argument of Lemma \ref{ullom} implies that
\begin{align}\label{eq 2} [\hat A^t(\wp_{F_p}) , Y^d_{F/k,p}; \mu'] = \, &d\cdot [\wp_{F_p} , Y_{F/k,p}; \mu]\\
                                                                 = \, &d\cdot [\mathcal{O}_{F,p}, Y_{F/k,p}; \mu] + d\cdot \chi_{G,p}
                                                                 \bigl((\mathcal{O}_{F,p}/\wp_{F_p})[0],0\bigr).\notag
                                                                 \end{align}
%

Next we note that if $f_v$ is the absolute residue degree of a $p$-adic place $v$, then the normal basis theorem for
$\mathcal{O}_{F_w}/\wp_{F_w}$ over the field with $p$ elements implies that there exists a short
exact sequence of $G_w/I_w$-modules

\[ 0 \to \ZZ [G_w/I_w] ^{f_{v}} \xrightarrow{\times p} \ZZ[G_w/I_w]^{f_{v}} \to \mathcal{O}_{F_w}/\wp_{F_w} \to 0.\]
By using these sequences (for each such $v$) one computes that

\begin{equation*}\label{eq 3} \chi_{G,p}\bigl((\mathcal{O}_{F,p}/\wp_{F_p})[0],0\bigr) = \delta_{G,p}\bigl(\prod_{v \in S_k^p}\varrho_v(F/k)\bigr).\end{equation*}

%
%

Upon combining this equality with (\ref{eq 1}) and (\ref{eq 2}) we deduce that
\[ \delta_{G,p}(\mathcal{LR}^p_{A^t_{F/k}}(x_\bullet)\cdot (\prod_{v \in S_k^p}\varrho_v(F/k))^{-d}) = d\cdot [\mathcal{O}_{F,p}, Y_{F/k,p}; \mu]\]
and from here one can deduce the claimed result via the argument of Proposition \ref{lms}.\end{proof}

\subsubsection{}It is possible to prove less precise versions of Proposition \ref{explicit log resolve} without making any hypotheses on ramification and to thereby obtain more explicit versions of the prediction that is made in Theorem \ref{big conj}.

For example, if $F_p^\times$ has no element of order $p$, $\mathfrak{M}$ is any choice of maximal $\ZZ_p$-order in $\QQ_p[G]$ that contains $\ZZ_p[G]$ and $x$ is any element of $\zeta(\ZZ_p[G])$ such that $x\cdot \mathfrak{M} \subseteq \ZZ_p[G]$, then one can deduce from the result of \cite[Cor. 7.8]{bleyburns} that for every element $y$ of the ideal
\[ \zeta(\ZZ_p[G])\cdot p^d\cdot x^{1+2d}((1-e_G)+ |G|\cdot e_G)\]
there exists an ordered $nd$-tuple $x(y)_\bullet$ of points in $A^t(F_p)$ for which  one has
\[ \mathcal{LR}^p_{A^t_{F/k}}(x(y)_\bullet) = y\cdot \bigl(\tau^\ast(F/k)\cdot \prod_{v \in S_k^p}\varrho_v(F/k)\bigr)^d\]
in $\zeta(\QQ^c[G])$.

However, we shall not prove this result here both because it is unlikely to be the best possible `bound' that one can give on logarithmic resolvents in terms of Galois-Gauss sums and also because, from the perspective of numerical investigations, the following much easier interpretation of logarithmic resolvents in terms of Galois resolvents is likely to be more helpful.

\begin{lemma} Let $i_0$ be the least integer with $i_0 > e_{F,p}/(p-1)$, where $e_{F,p}$ is the maximal absolute ramification degree of any $p$-adic place of $F$. Let $z$ be an element of $F$ that belongs to the $i_0$-th power of every prime ideal of $\mathcal{O}_F$ above $p$.

Then, for any integral basis $\{a_i:  i \in [n]\}$ of $\mathcal{O}_k$, there exists an ordered $nd$-tuple $x_\bullet$ of points in $A^t(F_p)$ for which  one has
\[ \mathcal{LR}^p_{A^t_{F/k}}(x_\bullet) = {\rm Nrd}_{\QQ_p^c[G]}\left( \bigl(\sum_{g \in G} \hat \sigma(g^{-1}(z\cdot a_i))\cdot g\bigr)_{\sigma\in \Sigma(k),i\in [n]}\right)^d. \]
\end{lemma}

\begin{proof} The definition of $i_0$ ensures that the formal group logarithm of $A^t$ over $F_p$ gives an isomorphism of $\hat A^t(\wp_{F_p}^{i_0})$ with a direct sum $(\wp_{F_p}^{i_0})^d$ of $d$ copies of $\wp_{F_p}^{i_0}$ (cf. \cite[Th. 6.4(b)]{silverman}). Here the individual copies of $\wp_{F_p}^{i_0}$ in the sum are parametrised by the differentials $\{\omega_j': j \in [n]\}$ that are used to define ${\rm log}_{A^t}$.

The choice of $z$ also implies that $Z := \{z\cdot a_i:  i \in [n]\}$ is a subset of $\wp_{F_p}^{i_0}$. We may therefore choose a pre-image $x_\bullet$ in $\hat A^t(\wp_{F_p}^{i_0})$ of the ordered $nd$-tuple in $(\wp_{F_p}^{i_0})^d$ that is obtained by placing a copy of $Z$ in each of the $d$ direct summands.

For these points $x_\bullet$ the interpretation of the matrix (\ref{log resolve matrix}) that is given in the proof of  Proposition \ref{explicit log resolve} shows that it is equal to a $d\times d$ diagonal block matrix with each diagonal block equal to the matrix
\[  \bigl(\sum_{g \in G} \hat \sigma(g^{-1}(z\cdot a_i))\cdot g\bigr)_{\sigma\in \Sigma(k),i\in [n]}.\]

Since $\mathcal{LR}^p_{A^t_{F/k}}(x_\bullet)$ is defined to be the reduced norm of the matrix (\ref{log resolve matrix}) the claimed equality is therefore clear.\end{proof}

\subsection{The proof of Theorem \ref{big conj}}\label{proof of big conj}

\subsubsection{}We start by proving a technical result that is both necessary for the proof of Theorem \ref{big conj} and will also be of further use in the sequel.

In this result we use the terminology of `characteristic elements' from \cite[Def. 3.1]{bst}.

\begin{lemma}\label{modifiedlemma} Fix an ordered subset $x_\bullet$ of $A^t(F_p)^\wedge_p$ as in Theorem \ref{big conj}. Write $X$ for the $\ZZ_p[G]$-module generated by $x_\bullet$ and $C_{S,X}$ for the Selmer complex ${\rm SC}_S(A_{F/k};X,H_\infty(A_{F/k})_p)$ from Definition \ref{selmerdefinition}. Then the following claims are valid.

\begin{itemize}
\item[(i)] The module $H^1(C_{S,X})$ is torsion-free.
\item[(ii)] For any finite non-empty set of places $T$ of $k$ that is disjoint from $S$, there exists an exact triangle in $D^{\rm perf}(\ZZ_p[G])$ of the form
\begin{equation}\label{modifiedtriangle}\bigoplus_{v\in T}R\Gamma(\kappa_v,T_{p,F}(A^t)(-1))[-2]\to C_{S,X}\stackrel{\theta}{\to} C_{S,X,T}\to\bigoplus_{v\in T}R\Gamma(\kappa_v,T_{p,F}(A^t)(-1))[-1]\end{equation}
in which $C_{S,X,T}$ is acyclic outside degrees one and two and there are canonical identifications of $H^1(C_{S,X,T})$ with $H^1(C_{S,X})$ and of  $\Sel_p(A_F)^\vee$ with a subquotient of $H^2(C_{S,X,T})$ in such a way that $\QQ_p\cdot\Sel_p(A_F)^\vee=\QQ_p\cdot H^2(C_{S,X,T})$.
 \item[(iii)] Following claim (ii) we write
\[ h^{T}_{A,F}: \CC_p\cdot H^1(C_{S,X,T}) \to \CC_p\cdot H^2(C_{S,X,T})\]
for the isomorphism $(\CC_p\otimes_{\ZZ_p}H^2(\theta))\circ(\CC_p\otimes_{\RR,j}h_{A,F})$ of $\CC_p[G]$-modules. Then, if ${\rm BSD}_p(A_{F/k})$(iv) is valid, there exists a characteristic element $\mathcal{L}$ for  $(C_{S,X},h_{A,F})$ in $\CC_p[G]^\times$ with the property that for any non-negative integer $a$ one has
\[ e_a\cdot\mathcal{L}^{-1}=\frac{L^{(a)}_{S}(A_{F/k},1)}{\Omega_A^{F/k}\cdot w_{F/k}^d}\cdot \mathcal{LR}^p_{A^t_{F/k}}(x_\bullet).\]

In addition, in this case the element
\[ \mathcal{L}_T := (\prod_{v \in T}P_v(A_{F/k},1))^{-1}\cdot\mathcal{L}\]
is a characteristic element for  $(C_{S,X,T},h^{T}_{A,F})$.
%
\end{itemize}
\end{lemma}
\begin{proof} 

Since $p$ is odd there exists a homomorphism of $\ZZ_p[G]$-modules $\phi$ from the module $H_\infty(A_{F/k})_p = \bigoplus_{v\in S_k^\infty}H^0(k_v,T_{p,F}(A^t))$ to $A^t(F_p)^\wedge_p$ that sends the ordered $\ZZ_p[G]$-basis of $H_\infty(A_{F/k})_p$ specified at the end of \S\ref{gamma section} to $x_\bullet$.
%
%

Then, comparing the triangle (\ref{can tri}) with the construction of \cite[Prop. 2.8 (ii)]{bst} immediately implies that $C_{S,X}$ is isomorphic in $D^{\rm perf}(\ZZ_p[G])$ to the complex that is denoted in loc. cit. by $C_\phi(T_{p,F}(A^t))$.

Given this, claim (i) follows directly from \cite[Prop. 2.8 (ii)]{bst} and claim (ii) from \cite[Prop. 2.8 (iii)]{bst} with $C_{S,X,T}:=C_{\phi,T}(T_{p,F}(A^t))$.

Turning to claim (iii) we note that each place $v$ in $T$ is not $p$-adic, does not ramify in $F/k$ and is of good reduction for $A$.

Each complex $R\Gamma(\kappa_v,T_{p,F}(A^t)(-1))$ is therefore well-defined. Since these complexes are acyclic outside degree one, where they have finite cohomology, we may therefore apply Lemma \ref{fk lemma} to the triangle (\ref{modifiedtriangle}) to deduce that
\begin{align*}\chi_{G,p}(C_{S,X},h_{A,F})-\chi_{G,p}(C_{S,X,T},h^{T}_{A,F})= \, &\sum_{v\in T}\chi_{G,p}(R\Gamma(\kappa_v,T_{p,F}(A^t)(-1))[-2],0)\\
=\, &\delta_{G,p}({\det}_{\QQ_p[G]}(1-\Phi_v|\QQ_p\cdot T_{p,F}(A^t)(-1)))\\
=\, &\delta_{G,p}(P_v(A_{F/k},1))
\end{align*}
in $K_0(\ZZ_p[G],\CC_p[G])$.

Thus if $\mathcal{L}$ is a characteristic element of $(C_{S,X},h_{A,F})$, then $(\prod_{v \in T}P_v(A_{F/k},1))^{-1}\cdot\mathcal{L}$ is a characteristic element for  $(C_{S,X,T},h^{T}_{A,F})$, as claimed in the final assertion of claim (iii).

It is thus enough to deduce from ${\rm BSD}_p(A_{F/k})$(iv) the existence of a characteristic element $\mathcal{L}$ of $(C_{S,X},\CC_p\otimes_{\RR}h_{A,F})$ with the required interpolation property.

Now, since $S$ contains all $p$-adic places of $k$, the module $\mathcal{Q}(\omega_\bullet)_{S,p}$ vanishes and the $p$-primary component of the term $\mu_{S}(A_{F/k})$ is also trivial.

In addition, as the validity of {\rm BSD}$_p(A_{F/k})$(iv) is independent of the choice of global periods and we can assume firstly that $\omega_\bullet$ is the set $\{ z_i\otimes \omega'_j: i \in [n], j \in [d]\}$ fixed in Lemma \ref{k-theory period} and secondly that the image of $\mathcal{F}(\omega_\bullet)_p$ under the formal group exponential ${\rm exp}_{A^t,F_p}$ (defined with respect to the differentials $\{\omega_j': j \in [d]\})$ is contained in $X$.

Then the assumed validity of the equality (\ref{displayed pj}) in this case combines with the equality in Lemma \ref{k-theory period} to imply that the element

\begin{equation}\label{char el 1} \frac{L_S^*(A_{F/k},1)}{{\rm Nrd}_{\RR[G]}(\Omega_{\omega_\bullet}(A_{F/k}))}= \frac{L_S^*(A_{F/k},1)}{ \Omega_A^{F/k}\cdot w_{F/k}^{d}}\cdot {\rm det}\left( \left( (\sum_{g \in G} \hat \sigma(g^{-1}(z_i)))\cdot g\right)_{\sigma\in \Sigma(k),i\in [n]}\right)^{d}\end{equation}
of $\zeta(\CC_p[G])^\times$ is the inverse of a characteristic element of $(C_{S,\omega_\bullet},h_{A,F})$.

Here we write $C_{S,\omega_\bullet}$ for the Selmer complex ${\rm SC}_S(A_{F/k};\mathcal{X}(p),H_\infty(A_{F/k})_p)$ where $\mathcal{X}$ is the perfect Selmer structure $\mathcal{X}_S(\{\mathcal{A}^t_v\}_v,\omega_\bullet,\gamma_\bullet)$ defined in \S\ref{perf sel sect}. (In addition, the fact that it is the inverse of a characteristic element results from a comparison of our chosen normalisation of non-abelian determinants compared with that of \cite[(10)]{bst}, as described in Remark \ref{comparingdets}).

In particular, since $\mathcal{X}(p)$ is by definition equal to ${\rm exp}_{A^t,F_p}(\mathcal{F}(\omega_\bullet)_p)$ and hence, by assumption, contained in $X$, a comparison of the definitions of $C_{S,\omega_\bullet}$ and $C_{S,X}$ shows that there is an exact triangle
\[  C_{S,\omega_\bullet} \to C_{S,X} \to \bigl(X/\mathcal{X}(p))[-1] \to C_{S,\omega_\bullet}[1]\]
in $D^{\rm perf}(\ZZ_p[G])$.

Since the product

\[ \xi := \mathcal{LR}^p_{A^t_{F/k}}(x_\bullet)\cdot {\rm det}\left( \left( (\sum_{g \in G} \hat \sigma(g^{-1}(z_i)))\cdot g\right)_{\sigma\in \Sigma(k),i\in [n]}\right)^{-d}\]
is equal to the determinant of a matrix that expresses a basis of the free $\ZZ_p[G]$-module $\mathcal{X}(p)$ in terms of the basis $x_\bullet$ of $X$, the above triangle implies that the product
\[ \frac{L_S^*(A_{F/k},1)}{ \Omega_A^{F/k}\cdot w_{F/k}^{d}}\cdot \mathcal{LR}^p_{A^t_{F/k}}(x_\bullet)\]
of $\xi$ and the element (\ref{char el 1}) is the inverse of a characteristic element for  $(C_{S,X},h_{A,F})$.

The claimed interpolation formula is thus a consequence of the fact that $L_S^{(a)}(A_{F/k},1)$ is equal to $e_a\cdot L_S^*(A_{F/k},1)$. \end{proof}

\subsubsection{}We are now ready to prove Theorem \ref{big conj}.

To do this we will apply the general result of \cite[Th. 3.10(i)]{bst} to the complex $C_{S,X,T}$, isomorphism $h^{T}_{A,F}$ and characteristic element $\mathcal{L}_T$ constructed in Lemma \ref{modifiedlemma}.

 In order to do so, we fix an ordered subset $\Phi:=\{\phi_i: i \in [a]\}$ of $\Hom_{\ZZ_p[G]}(A(F)_p,\ZZ_p[G])$ of cardinality $a$. We fix a pre-image $\phi_i'$ of each $\phi_i$ under the surjective composite homomorphism
\[ H^2(C_{S,X})\to\Sel_p(A_F)^\vee\to \Hom_{\ZZ_p[G]}(A(F)_p,\ZZ_p[G]),\]
where the first arrow is the canonical map from Proposition \ref{prop:perfect}(iii) and the second is induced by the canonical short exact sequence
\begin{equation}\label{sha-selmer}
 \xymatrix{0 \ar[r] & \sha(A_F)[p^\infty]^\vee \ar[r] & \Sel_p(A_F)^\vee \ar[r]&  \Hom_{\ZZ_p}(A(F)_p,\ZZ_p)\ar[r] & 0.}
\end{equation}

We set $\Phi':=\{\phi'_i:i \in [a]\}$ and consider the image $H^2(\theta)(\Phi')$ of $\Phi'$ in $H^2(C_{S,X,T})$, where $\theta$ is the morphism that occurs in the triangle (\ref{modifiedtriangle}) (so that $H^2(\theta)$ is injective).

We next write $\iota:H^1(C_{S,X,T})=H^1(C_{S,X})\to A^t(F)_p$ for the canonical homomorphism in Proposition \ref{prop:perfect}(iii).

Then, with $\mathcal{L}_T$ the element specified in Lemma \ref{modifiedlemma}(iii), a direct comparison of the definitions of $h_{A,F}^{T}$ and ${\rm ht}_{A_{F/k}}^{a}$ shows that the `higher special element' that is associated via \cite[Def. 3.3]{bst} to the data $(C_{S,X,T},h^{T}_{A,F},\mathcal{L}_T,H^2(\theta)(\Phi'))$ coincides with the pre-image under the bijective map $\CC_p\cdot\bigwedge_{\ZZ_p[G]}^a\iota$ of the element
\begin{equation}\label{hse interpret}(\prod_{v \in T}P_v(A_{F/k},1))\cdot \frac{L^{(a)}_{S}(A_{F/k},1)}{\Omega_A^{F/k}\cdot w_{F/k}^d}\cdot \mathcal{LR}^p_{A^t_{F/k}}(x_\bullet)\cdot {\rm ht}^{a}_{A_{F/k}}(\wedge_{i=1}^{i=a}\phi_i).\end{equation}
(Here we have also used the fact that, since ${\rm BSD}(A_{F/k})$(ii) is assumed to be valid, the idempotent $e_a$ defined here coincides with the idempotent denoted $e_a$ in \cite[\S3.1]{bst} for the complex $C_{S,X,T}$.)

To proceed we fix an ordered subset $\{\theta_j: j \in [a]\}$ of $\Hom_{\ZZ_p[G]}(A^t(F)_p,\ZZ_p[G])$ and identify it with its image under the injective map
\[ \Hom_{\ZZ_p[G]}(A^t(F)_p,\ZZ_p[G]) \to \Hom_{\ZZ_p[G]}(H^1(C_{S,X,T}),\ZZ_p[G])\]
induced by $\iota$. We also set $\mathfrak{A} := \ZZ_p[G]e_{(a)}$ and $M := \mathfrak{A}\otimes_{\ZZ_p[G]}H^2(C_{S,X,T})$.

Then the above interpretation of the higher special element in terms of the product (\ref{hse interpret}) combines with the general result of \cite[Th. 3.10(i)]{bst} to imply that for any element $\alpha$ of $\ZZ_p[G]\cap\mathfrak{A}$ and any element $y$ of $\ZZ_p[G]$ that annihilates ${\rm Ext}^2_{\mathfrak{A}}(M,\mathfrak{A})$
the product
\[\alpha \cdot y^a \cdot(\prod_{v \in T}P_v(A_{F/k},1))\cdot \frac{L^{(a)}_{S_F}(A_{F/k},1)}{\Omega_A^{F/k}\cdot w_{F/k}^d}\cdot \mathcal{LR}^p_{A^t_{F/k}}(x_\bullet)\cdot (\wedge_{j=1}^{j=a}\theta_j)({\rm ht}^{a}_{A_{F/k}}(\wedge_{i=1}^{i=a}\phi_i))\]
both belongs to ${\rm Fit}^a_{\ZZ_p[G]}(\Sel_p(A_F)^\vee)$ and annihilates $(\Sel_p(A_F)^\vee)_{\rm tor}$.

In addition, the exact sequence (\ref{sha-selmer}) identifies $(\Sel_p(A_F)^\vee)_{\rm tor}$ with $\sha(A_F)[p^\infty]^\vee$ and the Cassels-Tate pairing identifies $\sha(A_F)[p^\infty]^\vee$ with $\sha(A^t_F)[p^\infty]$.

To deduce the result of Theorem \ref{big conj} from here, it is therefore enough to show that $\alpha^2$ annihilates ${\rm Ext}^2_{\mathfrak{A}}(M,\mathfrak{A})$.

To do this we use the existence of a convergent first quadrant cohomological spectral sequence
\[ E_2^{pq} = {\rm Ext}_{\mathfrak{A}}^p(M,{\rm Ext}^q_{\ZZ_p[G]}(\mathfrak{A},\mathfrak{A})) \Rightarrow {\rm Ext}^{p+q}_{\ZZ_p[G]}(M,\mathfrak{A})\]
(cf. \cite[Exer. 5.6.3]{weibel}).

In particular, since the long exact sequence of low degree terms of this spectral sequence gives an exact sequence of $\ZZ_p[G]$-modules
\[ \Hom_{\ZZ_p[G]}(M,{\rm Ext}^1_{\ZZ_p[G]}(\mathfrak{A},\mathfrak{A})) \to {\rm Ext}_{\mathfrak{A}}^2(M,\mathfrak{A}) \to {\rm Ext}^{2}_{\ZZ_p[G]}(M,\mathfrak{A}),\]
we find that it is enough to show that the element $\alpha$ annihilates both ${\rm Ext}^1_{\ZZ_p[G]}(\mathfrak{A},\mathfrak{A})$ and ${\rm Ext}^{2}_{\ZZ_p[G]}(M,\mathfrak{A})$.

To verify this we write $\mathfrak{A}^\dagger$ for the ideal $\{x \in \ZZ_p[G]: x\cdot e_{(a)} = 0\}$ so that there is a natural short exact sequence of $\ZZ_p[G]$-modules $0 \to \mathfrak{A}^\dagger \to \ZZ_p[G] \to \mathfrak{A} \to 0$.

Then by applying the exact functor ${\rm Ext}^\bullet_{\ZZ_p[G]}(-,\mathfrak{A})$ to this sequence one obtains a surjective homomorphism
\[ \Hom_{\ZZ_p[G]}(\mathfrak{A}^\dagger,\mathfrak{A}) \twoheadrightarrow {\rm Ext}^1_{\ZZ_p[G]}(\mathfrak{A},\mathfrak{A}).\]

In addition, since $\ZZ_p[G]$ is Gorenstein, by applying the exact functor ${\rm Ext}^{\bullet}_{\ZZ_p[G]}(M,-)$ to the above sequence one finds that  there is a natural isomorphism
\[ {\rm Ext}^{3}_{\ZZ_p[G]}(M,\mathfrak{A}^\dagger) \cong {\rm Ext}^{2}_{\ZZ_p[G]}(M,\mathfrak{A}).\]

To complete the proof of Theorem \ref{big conj} it is thus enough to note that the left hand modules in both of the last two displays are annihilated by $\alpha$ since the definition of $\mathfrak{A}^\dagger$ implies immediately that $\alpha\cdot \mathfrak{A}^\dagger = 0$.

\begin{remark}\label{omit T}{\em If $A(F)$ does not contain an element of order $p$, then \cite[Th. 3.10(i)]{bst} can be directly applied to the complex $C_{S,X}$ rather than to the auxiliary complex $C_{S,X,T}$. This shows that, in any such case, the prediction in Theorem \ref{big conj} should remain true if the term $\prod_{v \in T}P_v(A_{F/k},1)$ is omitted.}
\end{remark}

\subsection{Justification of the explicit predictions}\label{justifications}

In this section we briefly explain why, if ${\rm BSD}(A_{F/k})$ is valid, then so are Predictions \ref{new add}, \ref{new remark SC} and \ref{new add2}.

We begin by observing that ${\rm BSD}(A_{F/k})$(ii) ensures that for a given $a\geq 0$, the idempotent $e_{(a)}$ is equal to 1 if and only if there exist ordered $a$-tuples $P_\bullet$ of $A^t(F)$ and $Q_\bullet$ of $A(F)$ that are each linearly independent over $\QQ[G]$.

Assume first that (A$_1)$ and (A$_2$) are satisfied. 
Fix $a\geq 0$ and ordered $a$-tuples $P_\bullet$ of $A^t(F)$ and $Q_\bullet$ of $A(F)$ that are each linearly independent over $\QQ[G]$. Then since $e_{(a)}=1$ one can take $\alpha=1$ in (\ref{key product}) and thus Theorem \ref{big conj} and Remark \ref{more explicit rem}(i) directly combine with Lemma \ref{height pairing interp} and with
Proposition \ref{explicit log resolve} to show that if ${\rm BSD}(A_{F/k})$ is valid, then so is Prediction \ref{new add}.

Assume now that (A$_1)$ and (A$_3$) are satisfied. Then, after taking account of the equality in Remark \ref{emptysets}, the argument that is used to prove Theorem \ref{big conj} can be directly applied to the Selmer complex ${\rm SC}_p(A_{F/k})$ rather than to the complex $C_{S,X}$ that occurs in \S\ref{proof of big conj}. Indeed, if ${\rm BSD}_p(A_{F/k})$(iv) is valid then $\mathcal{L}_{A,F/k}^*$ is the inverse of a characteristic element for the pair  $({\rm SC}_p(A_{F/k}),h_{A,F})$ and thus, in this way, one finds that if ${\rm BSD}(A_{F/k})$ is valid, then so is Prediction \ref{new remark SC}.

We next assume that (A$_1)$, (A$_2$) and (A$_3$) are all satisfied.
Then by using the claim of Prediction \ref{new remark SC} in place of that of Theorem \ref{big conj}, the same approach that was used to justify Prediction \ref{new add} shows that if ${\rm BSD}(A_{F/k})$ is valid, then each element of the product 
\begin{align*}&\,\,\frac{e_{F,a}\cdot\mathcal{L}_{A,F/k}^*}{{\rm det}(h_{F/k}(P_\bullet, Q_\bullet))} \cdot {\rm Fit}^0_{\ZZ[G]}((A^t(F)/\langle P_\bullet\rangle)^\vee_{\rm tor})\cdot{\rm Fit}^0_{\ZZ[G]}((A(F)/\langle Q_\bullet\rangle)^\vee_{\rm tor})\\
=\,\,&\frac{L^{(a)}_{S_{\rm r}}(A_{F/k},1)\cdot \bigl(\tau^\ast(F/k)\cdot \prod_{v \in S_{p,{\rm r}}}\varrho_v(F/k)\bigr)^d}{\Omega_A^{F/k}\cdot w_{F/k}^d\cdot {\rm det}(h_{F/k}(P_\bullet, Q_\bullet))} \\
& \hskip 1truein \times  {\rm Fit}^0_{\ZZ[G]}((A^t(F)/\langle P_\bullet\rangle)^\vee_{\rm tor})\cdot{\rm Fit}^0_{\ZZ[G]}((A(F)/\langle Q_\bullet\rangle)^\vee_{\rm tor})
\end{align*}
belongs to ${\rm Fit}^a_{\ZZ_p[G]}(\Sel_p(A_F)^\vee)$ and annihilates $\sha(A_F^t)[p^\infty]$, as stated in Prediction \ref{new add2}.

We finally justify Prediction \ref{integrality rk}. Under each of the stated sets of assumptions, we have predicted that the element $\calL$ belongs to $\ZZ_p[G]$. But then one also knows that the sum
$$\sum_{g\in G}\left(|G|^{-1}\sum_{\psi\in\widehat{G}_{A,(a)}}\psi(g)\calL_\psi\right)\cdot g^{-1}=\sum_{\psi\in\widehat{G}_{A,(a)}}\calL_\psi e_\psi=\calL$$ belongs to $\ZZ_p[G]$ and it follows directly that each term $\sum_{\psi\in\widehat{G}_{A,(a)}}\psi(g)\calL_\psi$ would belong to $|G|\cdot\ZZ_p$, as required.



%
%

\section{Abelian congruence relations and height pairings}\label{mrsconjecturesection}

In this section we continue to investigate the $p$-adic congruence relations between the leading coefficients of Hasse-Weil-Artin $L$-series that are encoded by ${\rm BSD}_p(A_{F/k})$(iv) in the case that $F/k$ is abelian.

More concretely in \S\ref{mtchd} we will show that, beyond the integrality properties that are discussed in Theorem \ref{big conj} and Remark \ref{integrality rk II}, elements of the form (\ref{key product}) can be expected to satisfy additional congruence relations in the integral augmentation filtration that  involve Mazur-Tate regulators.


Then in \S \ref{cycliccongssection} we specialise to the case of cyclic extensions in order to make these additional congruence relations fully explicit.

In particular, in this way we render the equality ${\rm BSD}_p(A_{F/k})$(iv) amenable to (numerical) verification even in cases in which it incorporates a thorough-going mixture of both archimedean phenomena and delicate $p$-adic congruences.

Finally, in \S\ref{dihedral} we explain how these results extend to certain families of non-abelian extensions.

Throughout this section, just as in \S\ref{congruence sec}, we give ourselves a fixed odd prime $p$ and isomorphism of fields $j:\CC\cong\CC_p$ (explicit mention of which we usually omit), a finite set $S$ of places of $k$ with
\[ S_k^\infty\cup S_k^p\cup  S_k^F \cup S_k^A\subseteq S\]
and a fixed ordered $k$-basis $\{\omega_j'\}_{j\in[d]}$ of $H^0(A^t,\Omega^1_{A^t})$ with associated classical period $\Omega_A^{F/k}$.

Except in \S\ref{dihedral} we shall always assume in this section that $F/k$ is abelian. In addition, we shall always assume that $p$ is chosen so that neither $A(F)$ nor $A^t(F)$ has a point of order $p$.


\subsection{A Mazur-Tate conjecture for higher derivatives}\label{mtchd}

In this section we formulate a Mazur-Tate type conjecture for higher derivatives of Hasse-Weil-Artin $L$-series. We then show that, under the hypotheses listed in \S \ref{tmc}, this conjecture would follow from the validity of ${\rm BSD}(A_{F/k})$.

\subsubsection{}

We first quickly recall the construction of canonical height pairings of Mazur and Tate \cite{mt0}.

To do this we fix a subgroup $J$ of $G$ and set $E := F^J$. 
We recall that the subgroups of `locally-normed elements' of $A(E)$ and $A^t(E)$ respectively are defined by setting \begin{equation}\label{localnorms}U_{F/E}:=\bigcap_v \bigl(A(E)\cap N_{F_w/E_v}(A(F_w))\bigr),\,\,\,\,U^t_{F/E}:=\bigcap_v \bigl(A^t(E)\cap N_{F_w/E_v}(A^t(F_w))\bigr).\end{equation}
Here each intersection runs over all (finite) primes $v$ of $E$ and $w$ is a fixed prime of $F$ above $v$. In addition, $N_{F_w/E_v}$ denotes the norm map of $F_w/E_v$ and each intersection of the form $A(E)\cap N_{F_w/E_v}(A(F_w))$, resp. $A^t(E)\cap N_{F_w/E_v}(A^t(F_w))$, takes place inside $A(E_v)$, resp. $A^t(E_v)$.

We recall from Lemma \ref{useful prel}(i) that each of the expressions displayed in (\ref{localnorms}) is in general a finite intersection of subgroups of $A(E)$, resp. of $A^t(E)$, and that the subgroups $U_{F/E}$ and $U^t_{F/E}$ have finite index in $A(E)$ and $A^t(E)$ respectively.

We note for later use that, whenever $A$, $F/k$ and $p$ satisfy the hypotheses (H$_1$)-(H$_5$) listed in \S \ref{tmc}, then Proposition \ref{explicitbkprop}(ii) (together with the duality of these hypotheses) implies that $$U_{F/E,p}:=\ZZ_p\otimes_{\ZZ}U_{F/E}\,\,\text{ and }\,\,U^t_{F/E,p}:=\ZZ_p\otimes_{\ZZ}U^t_{F/E}$$ are equal to $A(E)_p$ and to $A^t(E)_p$ respectively (for every given subgroup $J$ of $G$).

In general, Mazur and Tate \cite{mt0} construct, by using the theory of biextensions, a canonical height pairing \begin{equation}\label{tanpairing}\langle\,,\rangle^{\rm MT}_{F/E}:U^t_{F/E}\otimes_\ZZ U_{F/E}\to J.\end{equation}
This pairing will be a key ingredient of our conjectural congruence relations. To formulate our conjecture we must first describe how to make reasonable choices of points on which to evaluate the Mazur-Tate pairing.

\begin{definition}\label{separablepair}{\em Fix a subgroup $J$ of $G$ and set $E:=F^J$. We define a `$p$-separable choice of points of $A$ for $F/E$ of rank $(a,a')$' (with $a'\geq a\geq 0$) to be a pair $(\mathcal{Y},\mathcal{Y}')$ chosen as follows.

Let $\mathcal{Y}= \{y_i: i \in [a]\}$ be any ordered finite subset of $A(F)_p$ that generates a $\ZZ_p[G]$-direct-summand $Y$ of $A(F)_p$ that is free of rank equal to $|\mathcal{Y}|=a$. Then $\Tr_J(Y)=Y^J$ is a $\ZZ_p[G/J]$-direct-summand of $A(E)_p$.

We then let \begin{equation}\label{revisionWi}\mathcal{Y}'=\Tr_J(\mathcal{Y})\cup\{w_i:a<i\leq a'\}\end{equation} be any ordered finite subset of $U_{F/E,p}$, containing $\Tr_J(\mathcal{Y}):=\{\Tr_J(y_i):i \in [a]\}$, that generates a $\ZZ_p[G/J]$-direct-summand $Y'$ of $A(E)_p$ that is free of rank equal to $|\mathcal{Y}'|=a'\geq a$.}
\end{definition}




\begin{remarks}\label{p-separable}{\em 
To identify $p$-separable choices of points as above, one can proceed as follows. Write $G=P\times H$ where $P$ is the Sylow $p$-subgroup of $G$. Then the argument of \cite[Prop. 2.1]{burns nagoya} shows that a finite subset $\mathcal{Y}$ of $A(F)_p$ generates a free $\ZZ_p[G]$-direct-summand $Y$ of $A(F)_p$ of rank equal to $|\mathcal{Y}|=a$ if and only if the images in $A(F^P)_p/p\cdot A(F^P)_p$ of the elements $$h\cdot\sum_{g\in P}g(y),$$ for each $h\in H$ and each $y\in\mathcal{Y}$, are linearly independent over $\mathbb{F}_p$. An analogous observation applies to the choice of the set $\mathcal{Y}'$ in a $p$-separable choice of points $(\mathcal{Y},\mathcal{Y}')$ of $A$ for $F/E$ of rank $(a,a')$.}
\end{remarks}

Given data as in Definition \ref{separablepair}, and for each $a<i\leq a'$, we may also define composite homomorphisms of $\ZZ_p[G/J]$-modules
\begin{equation}\label{htan}{\rm ht}^{{\rm MT}}_{w_i}:U^t_{F/E,p}\to \ZZ_p\otimes_\ZZ J\cong I_p(J)/I_p(J)^2\to\mathcal{I}_p(J)/\mathcal{I}_p(J)^2\end{equation} by using our elements $w_i\in U_{F/E,p}$ fixed in (\ref{revisionWi}) to set
\begin{equation}\label{minussignMT}{\rm ht}^{{\rm MT}}_{w_i}(P):=-\langle P,w_i\rangle^{\rm MT}_{F/E}.\end{equation} Here $\mathcal{I}_p(J)$ denotes the ideal of $\ZZ_p[G]$ generated by the augmentation ideal $I_p(J)$ of $\ZZ_p[J]$, the isomorphism maps each $g$ in $J$ to the class of $g-1$ and the last arrow is induced by the inclusion $I_p(J)\subseteq\mathcal{I}_p(J)$. (We also note that the occurrence of the minus sign in (\ref{minussignMT}) will be clarified by the result of Theorem \ref{thecomptheorem} below).

\subsubsection{}

In the rest of \S \ref{mrsconjecturesection}, for any $\ZZ_p[G]$-module $M$ we write $M^*$ for the linear dual $\Hom_{\ZZ_p}(M,\ZZ_p)$, endowed with the natural contragredient action of $G$.

For given integers $a'\geq a\geq 0$ and a given $p$-separable choice of points $(\mathcal{Y},\mathcal{Y}')$ of $A$ for $F/E$ of rank $(a,a')$, we always use the following notation.
We fix a decomposition $A(F)_p=Y\oplus Z$. Since $A(F)_p^*=Y^*\oplus Z^*$ and
$Y^*$ is $\ZZ_p[G]$-free of rank $a$ we may and will denote by $y_i^*$ the element of $Y^*\subseteq A(F)_p^*$ which maps $y_i$ to 1 and $y_j$ to 0 for $j\neq i$.

In a similar way, to each element $y'\in\mathcal{Y}'$ we associate a canonical dual element $(y')^*$ of $A(E)_p^*$.

Then for any maximal subset $x_\bullet$ of $A^t(F_p)^\wedge_p$ that is linearly independent over $\ZZ_p[G]$ we may also set $x_\bullet^J:= \{\Tr_J(x):x\in x_\bullet\}\subseteq A^t(E_p)^\wedge_p$ and, by using the isomorphism ${\rm ht}_{A_{F/k}}^{a}$ considered in (\ref{athpower}) as well as the analogous isomorphism ${\rm ht}_{A_{E/k}}^{a'}$, one obtains elements
\[ \eta_{\mathcal{Y},x_\bullet} := \frac{L^{(a)}_{S}(A_{F/k},1)}{\Omega_A^{F/k}\cdot w_{F/k}^d}\cdot \mathcal{LR}^p_{A^t_{F/k}}(x_\bullet) \cdot {\rm ht}_{A_{F/k}}^{a}(\wedge_{y\in\mathcal{Y}}y^*)\in\CC_p\cdot\bigwedge_{\ZZ_p[G]}^aA^t(F)_p\]
and
\[ \eta_{\mathcal{Y}',x_\bullet^J} := \frac{L^{(a')}_{S}(A_{E/k},1)}{\Omega_A^{E/k}\cdot w_{E/k}^d}\cdot \mathcal{LR}^p_{A^t_{E/k}}(x_\bullet^J)\cdot {\rm ht}_{A_{E/k}}^{a'}(\wedge_{y'\in\mathcal{Y}'}(y')^*)\in\CC_p\cdot\bigwedge_{\ZZ_p[G/J]}^{a'}A^t(E)_p.\]
Here we have also used the value $L^{(a)}_{S}(A_{F/k},1)$ at $z=1$ of the function (\ref{revisionLa}), while the value $L^{(a')}_{S}(A_{E/k},1)$ is defined analogously.


\subsubsection{}

For any finitely generated $\ZZ_p[G]$-module $M$ and any non-negative integer $r$ we write $\bigcap_{\ZZ_p[G]}^rM$ for the $r$-th `Rubin lattice' of $M$ inside $\QQ_p\cdot\bigwedge_{\ZZ_p[G]}^rM$, as defined originally by Rubin \cite{rub} (see also \cite[\S 4.2]{bks}).

Then the homomorphisms ${\rm ht}^{{\rm MT}}_{w_i}$ in (\ref{htan}) can be combined as in \cite[\S3.4.1]{bst} to give a canonical homomorphism
\begin{equation}\label{canbecombined}
{\rm ht}_{\mathcal{Y}'}^{\rm MT}: {\bigcap}_{\ZZ_p[G/J]}^{a'} U^t_{F/E,p} \longrightarrow ({\bigcap}_{\ZZ_p[G/J]}^{a} U^t_{F/E,p})\otimes_{\ZZ_p} I_p(J)^{a'-a}/I_p(J)^{1+a'-a}.
\end{equation}

Following an original idea of Darmon (see \cite{darmon0}), we define the `norm operator'
$$\mathcal{N}_J : {\bigcap}_{\ZZ_p[G]}^a A^t(F)_p \longrightarrow ({\bigcap}_{\ZZ_p[G]}^a A^t(F)_p) \otimes_{\ZZ_p} \ZZ_p[J]/I_p(J)^{1+a'-a}$$
to be the homomorphism induced by sending each $x$ to $\sum_{\sigma \in J} \sigma (x) \otimes \sigma^{-1}$.

We also use \cite[Prop. 4.12]{bks} to obtain a canonical injective homomorphism of $\ZZ_p[G]$-modules
\begin{multline*} \nu_J:({\bigcap}_{\ZZ_p[G/J]}^{a} A^t(E)_p)\otimes_{\ZZ_p}I_p(J)^{a'-a}/I_p(J)^{1+a'-a}\\ \rightarrow ({\bigcap}_{\ZZ_p[G]}^{a} A^t(F)_p)\otimes_{\ZZ_p}I_p(J)^{a'-a}/I_p(J)^{1+a'-a}.\end{multline*}
%


We can now state the central conjecture of this section.

\begin{conjecture}\label{mrsconjecture} Assume that neither $A(F)$ nor $A^t(F)$ has a point of order $p$.
Fix any subgroup $J$ of $G$ (with $E := F^J$) and any $p$-separable choice of points $(\mathcal{Y},\mathcal{Y}')$ of $A$ for $F/E$ of rank $(a,a')$. Then one has
\[ \eta_{\mathcal{Y},x_\bullet}\in {\bigcap}_{\ZZ_p[G]}^{a}A^t(F)_p\]
and
\[ \eta_{\mathcal{Y}',x_\bullet^J}\in {\bigcap}_{\ZZ_p[G/J]}^{a'}U^t_{F/E,p}\]
and in $({\bigcap}_{\ZZ_p[G]}^{a} A^t(F)_p)\otimes_{\ZZ_p}I_p(J)^{a'-a}/I_p(J)^{1+a'-a}$ there is an equality
\[ \mathcal{N}_J(\eta_{\mathcal{Y},x_\bullet})=(-1)^{a(a'-a)}\cdot \nu_J({\rm ht}_{\mathcal{Y}'}^{\rm MT}(\eta_{\mathcal{Y}',x_\bullet^J})).\]
\end{conjecture}

\begin{remark}\label{explicit pairings}{\em
Let $k = \QQ$, $J = G$, $a=0$, $a' = {\rm rk}_\ZZ(A(k))$ and $\mathcal{Y}=\{y_i: i \in [a']\}$ be any $\ZZ$-basis of $\Hom_\ZZ(A(k),\ZZ)$.  
Then, after taking account of Prediction \ref{new add}, one finds that the congruence in Conjecture \ref{mrsconjecture} specialises in this case to refine and extend the conjectures formulated (for certain elliptic curves $A$ and abelian fields $F$) by Mazur and Tate in \cite{mt}. If instead $k$ is an imaginary quadratic field, then the congruence in Conjecture \ref{mrsconjecture} also incorporates the main conjecture of Darmon \cite{darmon0}. }\end{remark}

\begin{remark}\label{indeptpoints}
{\em Fix any set $x_\bullet$, subgroup $J$ of $G$ and integers $a'\geq a\geq 0$ as in Conjecture \ref{mrsconjecture}.
Then the general result of \cite[Prop. 4.17]{bks} can be used to show that the validity of Conjecture \ref{mrsconjecture} is independent of the $p$-separable choice of points $(\mathcal{Y},\mathcal{Y}')$ of $A$ for $F/E$ of rank $(a,a')$.}
\end{remark}

\subsubsection{}
The following result establishes conditions under which Conjecture \ref{mrsconjecture} is a consequence of ${\rm BSD}(A_{F/k})$.

The proof of this result will be given in \S \ref{mrsproof} below.

\begin{theorem}\label{rbsdimpliesmt} Fix an odd prime number $p$ and assume that $A$, $F/k$ and $p$ satisfy the hypotheses (H$_1$)-(H$_5$) listed in \S \ref{tmc}. Assume also that neither $A(F)$ nor $A^t(F)$ has a point of order $p$ and that Hypothesis \ref{hypbd} below is satisfied.

Then if ${\rm BSD}(A_{F/k})$ is valid, so is Conjecture \ref{mrsconjecture}.
\end{theorem}

\begin{remarks}\label{mrstheorem}{\em Each of the hypotheses (H$_1$)-(H$_5$) as well as Hypothesis \ref{hypbd} is satisfied widely  (see Remarks \ref{satisfying H} and \ref{rmkbd} respectively). However, without assuming any of these hypotheses it is also possible to deduce from ${\rm BSD}(A_{F/k})$ less explicit versions of the congruence relations that are predicted by Conjecture \ref{mrsconjecture}.

To be more precise we fix an odd prime number $p$ for which neither $A(F)$ nor $A^t(F)$ has a point of order $p$ and set $$C_{S,x_\bullet}:={\rm SC}_S(A_{F/k};\langle x_\bullet\rangle_{\ZZ_p[G]},H_\infty(A_{F/k})_p).$$ Then, without any additional hypotheses on $A$, $F/k$ or $p$, the same argument as used to prove Theorem \ref{rbsdimpliesmt} shows ${\rm BSD}(A_{F/k})$ implies that
\[ \eta_{\mathcal{Y},x_\bullet}\in {\bigcap}_{\ZZ_p[G]}^{a}H^1(C_{S,x_\bullet})\subseteq{\bigcap}_{\ZZ_p[G]}^{a}A^t(F)_p\]
and
\[ \eta_{\mathcal{Y}',x_\bullet^J}\in {\bigcap}_{\ZZ_p[G/J]}^{a'}H^1(C_{S,x_\bullet})^J\subseteq{\bigcap}_{\ZZ_p[G/J]}^{a'}A^t(E)_p\]
and that in $\left({\bigcap}_{\ZZ_p[G/J]}^{a} A^t(F)_p\right)\otimes_{\ZZ_p} I_p(J)^{a'-a}/I_p(J)^{1+a'-a}$ one has an equality
\[ \mathcal{N}_J(\eta_{\mathcal{Y},x_\bullet})=(-1)^{a(a'-a)}\cdot \nu_J({\rm ht}_{\mathcal{Y}'}^{C_{S,x_\bullet}}(\eta_{\mathcal{Y}',x_\bullet^J})).\]
Here ${\rm ht}_{\mathcal{Y}'}^{C_{S,x_\bullet}}$ is the `Bockstein homomorphism' (\ref{htcy}) associated to the Nekov\'a\v r-Selmer complex $C_{S,x_\bullet}$  in the course of the proof of Theorem \ref{rbsdimpliesmt} below. In this context we recall also from Remark \ref{mrselmer} that the cohomology group $H^1(C_{S,x_\bullet})$ may be interpreted as a Selmer group in the sense of Mazur and Rubin \cite{MRkoly}. }
\end{remarks}


\subsection{Cyclic congruence relations}\label{cycliccongssection}

In this section we make the congruences in Conjecture \ref{mrsconjecture} more explicit in the case of certain cyclic extensions.

We always assume the hypotheses (H$_1$)-(H$_6$) listed in \S\ref{tmc} and, motivated by
Remark \ref{emptysets} and Prediction \ref{new remark SC}, we further strengthen the hypothesis (H$_3$) by assuming that $p$ is unramified in $F$.

In this situation, the argument that is used in the proof of Theorem \ref{rbsdimpliesmt} below may be applied to the classical Selmer complex.

In this way, the equality of ${\rm BSD}_p(A_{F/k})$(iv) implies (via Theorem \ref{bk explicit}) congruence relations in the spirit of Conjecture \ref{mrsconjecture} that circumvent the choice of semi-local points and apply directly to the elements $\calL^*_{A,F/k}$ and $\calL^*_{A,E/k}$ as defined in (\ref{bkcharelement}).

We shall also now assume that the following additional conditions on $A$, $F/k$ and  $p$  are satisfied:

\begin{itemize}
\item[(H$_7$)] $F/k$ is cyclic of order $p^m$ (for some natural number $m$);
\item[(H$_8$)] neither $A(k)$ nor $A^t(k)$ has a point of order $p$;
\item[(H$_{9}$)] for every field $L$ with $k\subseteq L\subset F$ and $L \not=F$ the group $\sha(A_L)[p^\infty]$ vanishes;
\item[(H$_{10}$)] the group $H^1({\rm Gal}(k(A[p^{m}])/k),A[p^{m}])$ vanishes.
\end{itemize}

We impose the hypothesis (H$_{10}$) since it recovers Hypothesis \ref{hypbd} below in this specific setting and hence will allow us to apply the result of Theorem \ref{rbsdimpliesmt}. We also note that this hypothesis is satisfied widely (see Remark \ref{rmkbd}). 

%
%

 %

On the other hand, the hypotheses (H$_7$), (H$_8$) and (H$_9$) are introduced since they allow us to apply the main result of Yakovlev \cite{yakovlev} in order to deduce that $A(F)_p$ and $A^t(F)_p$ are (isomorphic) permutation $\ZZ_p[G]$-modules and hence that it is possible to fix an optimal collection of points as in Definition \ref{separablepair}.

Then, by combining such a choice of points with the methods of \cite{bleymc,dmc}, one can make fully explicit the congruence relations that are implied by ${\rm BSD}_p(A_{F/k})$(iv).

In order to stress the key ideas we shall focus on the simple cases $m=1$ and $m=2$. In particular, in the case $m=1$ we will show that the validity of  ${\rm BSD}_p(A_{F/k})$(iv) is under certain natural hypotheses equivalent to the validity of a single explicit congruence relation.

\subsubsection{}
We first assume that $m=2$. In this case we denote by $J$ the subgroup of $G$ of order $p$ and set $E := F^J$ and $\Gamma := G/J$. For each $i\in \{0,1,2\}$ we also write $F_i$ for the intermediate field of $F/k$ with $[F_i:k] = p^i$ and set $\Gamma_i:= G_{F_i/k}$ (so that $\Gamma_2 = G$ and $\Gamma_1 = \Gamma$).

Then, under the given hypotheses, Yakovlev's result implies (via the general approach of \cite[\S7.2.5]{ltav}) that for each $i \in \{0,1,2\}$ there exists a non-negative integers $m_i$ and subsets $\calP_{(i)}=\{P_{i,j}\mid j \in [m_i]\}$ of $A(F_i)_p$ and $\calP^t_{(i)}=\{P^t_{i,j}:j\in[m_i]\}$ of $A^t(F_i)_p$ such that the $\ZZ_p[\Gamma_i]$-modules generated by each of the points $P_{i,j}$ and $P^t_{i,j}$ are free of rank one and there are direct sum decompositions of $\ZZ_p[G]$-modules
%
\begin{equation}\label{global points}
 A(F)_p = \bigoplus_{i=0}^{i=2}\bigoplus_{j=1}^{j=m_i}\ZZ_p[\Gamma_i]\cdot P_{i,j}\end{equation}
 and

\begin{equation}\label{global points2} A^t(F)_p = \bigoplus_{i=0}^{i=2}\bigoplus_{j=1}^{j=m_i}\ZZ_p[\Gamma_i]\cdot P^t_{i,j},
\end{equation}
%
so that the integers $m_i$ are uniquely determined by the equalities
$${\rm rk}(A_k)=m_0+m_1+m_2,\,\,\,\,\,\,{\rm rk}(A_E)=m_0+p(m_1+m_2),\,\,\,\,\,\,{\rm rk}(A_F)=m_0+pm_1+p^2m_2.$$
For a full proof of these facts see \cite[Prop. 2.2]{bleymc}.

We then obtain `regulator matrices' \begin{align*}R_{F}(\calP_{\bullet}):=\, &h_{F/k}(\calP^t_{(2)};\,\calP_{(2)})\in {\rm M}_{m_2}(\RR[G]),\\\ R_{E}(\calP_{\bullet}):=\, &h_{E/k}(\calP^t_{(1)}\cup\calP_{(2)}^{t,J};\,\calP_{(1)}\cup \calP_{(2)}^J)\in {\rm M}_{m_1+m_2}(\RR[\Gamma])\\
R_{k}(\calP_{\bullet}) := \, &h_{k/k}(\calP^t_{(0)}\cup \calP^{t,\Gamma}_{(1)}\cup \calP_{(2)}^{t,G};\, \calP_{(0)}\cup \calP^\Gamma_{(1)}\cup \calP^G_{(2)}) \in {\rm M}_{{\rm rk}(A_k)}(\RR)\end{align*}
as in (\ref{regulatormatrix}), where we write $\calP_{(2)}^{t,J}$ for the set $\{\sum_{g \in J}g(P_{2,j}^t)\mid j \in [m_i]\}$ and define $\calP_{(2)}^J$, $\calP^{t,\Gamma}_{(1)}$, $\calP_{(2)}^{t,G}$, $\calP^\Gamma_{(1)}$ and $\calP^G_{(2)}$ similarly, and in each union we order the points  lexicographically.

We define an idempotent
\[
e_{\langle 2\rangle} := \begin{cases}1,& \text{ if } m_0=m_1=0,\\ 1-e_G,& \text{ if } m_0\neq 0,m_1=0,\\ 1-e_J,& \text{ if } m_1\neq 0, \end{cases}\]
in $\QQ[G]$ and
\[ e_{\langle 1\rangle} := \begin{cases}1,& \text{ if } m_0=0,\\ 1-e_\Gamma,& \text{ if } m_0\neq 0, \end{cases}
\]
in $\QQ[\Gamma]$.

Then the matrix $e_{\langle i\rangle}\cdot R_{F_i}(\calP_{\bullet})$ belongs to ${\rm GL}_{m_2}(\CC[G]e_{\langle 2\rangle})$ if $i=2$ and to  ${\rm GL}_{m_1+ m_2}(\CC[\Gamma]e_{\langle 1\rangle})$ if $i = 1$ and so we may set
\begin{align}\label{lmin}
\mathcal{L}^{(m_2)}_{F,\calP_{\bullet}} := \, &({\rm det}(e_{\langle 2\rangle}\cdot R_{F}(\calP_{\bullet})))^{-1}\cdot e_{\langle 2\rangle}\calL^*_{A,F/k}\in e_{\langle 2\rangle}\cdot\CC[G]^\times\\ \notag
\mathcal{L}^{(m_1+m_2)}_{E,\calP_{\bullet}} := \, &({\rm det}(e_{\langle 1\rangle}\cdot R_{E}(\calP_{\bullet})))^{-1}\cdot e_{\langle 1\rangle}\calL^*_{A,E/k}\in e_{\langle 1\rangle}\cdot\CC[\Gamma_1]^\times\\ \notag
\calL^\ast_{k,\calP_{\bullet}} :=\, &{\rm det}(R_k(\calP_{\bullet}))^{-1}\cdot\calL^*_{A,k}\in\RR^\times.\end{align}
%

\begin{remark} {\em This notation is motivated by fact that the decompositions in (\ref{global points}) and (\ref{global points2}) imply $e_{\langle 2\rangle}$ and $e_{\langle 1\rangle}$ are equal to the idempotents $e_{F,m_2}$ and $e_{E,m_1+m_2}$ that occur in \S \ref{statementstructure} with respect to $A$ and the extensions $F/k$ and $E/k$ respectively and so the terms $e_{\langle 2\rangle}\calL^*_{A,F/k}$ and $e_{\langle 1\rangle}\calL^*_{A,E/k}$ involve the values at $z=1$ of the $m_1$-th and $(m_1+m_2)$-th derivatives of Hasse-Weil-Artin $L$-series. We further recall from Remark \ref{bsdinvariants} that the term $\calL^*_{A,k}$ can be explicitly described in terms of invariants occurring in the classical Birch and Swinnerton-Dyer conjecture for $A$ over $k$.}\end{remark}


By using the approach of the second author in \cite{dmc}, one finds that Theorem \ref{rbsdimpliesmt} implies the explicit congruences in the following result. We do not give the proof of this result but rather refer the reader to loc. cit..

In this result we use the canonical identification
$$\mathcal{I}_p(J)^{m_1}/\mathcal{I}_p(J)^{m_1+1}\cong\ZZ_p[\Gamma]\otimes_{\ZZ_p}I_p(J)^{m_1}/I_p(J)^{m_1+1}.$$

\begin{theorem}\label{cyclicthm} 
Fix an odd prime number $p$ that is unramified in $F$ and such that the data $A, F/k$ and $p$ satisfy the hypotheses (H$_1$)-(H$_{10}$) with $m=2$. Then if ${\rm BSD}_p(A_{F/k})$(iv) is valid one has
\begin{equation*}\label{totallyexplicit}\mathcal{L}^{(m_2)}_{F,\calP_{\bullet}}\equiv  (-1)^{m_1}\cdot\bigl(\mathcal{L}^{(m_1+m_2)}_{E,\calP_{\bullet}}\otimes {\rm det}\bigl(\langle \calP_{(1)}^t,\calP_{(1)}\rangle^{\rm MT}_{F/E}\bigr)\bigr)\,\,\,\,\,({\rm mod}\,\,\mathcal{I}_p(J)^{m_1+1})\end{equation*}
and
\begin{equation*}\label{totallyexplicit}\mathcal{L}^{(m_2)}_{F,\calP_{\bullet}}\equiv (-1)^{m_0+m_1}\cdot\calL^\ast_{k,\calP_{\bullet}}\cdot {\rm det}\bigl(\langle\calP_{(0)}^t\cup \calP_{(1)}^{t,\Gamma},\calP_{(0)}\cup \calP^\Gamma_{(1)}\rangle^{\rm MT}_{F/k}\bigr)\,\,\,\,\, ({\rm mod}\,\,I_p(G)^{m_0+m_1+1}).\end{equation*}

\end{theorem}


\begin{remark}\label{changeofptsremark}{\em
The Wedderburn components of the elements $\mathcal{L}^{(m_2)}_{F,\calP_{\bullet}}$ and $\mathcal{L}^{(m_1+m_2)}_{E,\calP_{\bullet}}$  recover the products that are studied by Fearnley and Kisilevsky \cite{kisilevsky,kisilevsky2}. In particular, the numerical computations in loc. cit., as well as the numerical computations performed by Bley and the second author in \cite{bleymc}, can be interpreted as supporting evidence for ${\rm BSD}(A_{F/k})$.
}
\end{remark}

\subsubsection{}


We now assume that $m=1$. In this case Yakovlev's result implies, under the given hypotheses, that for each $i \in \{0,1\}$ there exists a non-negative $m_i$ and subsets $\calP_{(i)}$ of $A(F_i)_p$ and $\calP^t_{(i)}$ of $A^t(F_i)_p$ for which there are direct sum decompositions of $\ZZ_p[G]$-modules
\[ A(F)_p = \bigoplus_{j=1}^{j=m_0}\ZZ_p\cdot P_{0,j}\oplus\bigoplus_{j=1}^{j=m_1}\ZZ_p[G]\cdot P_{1,j}\]
and
\[ A^t(F)_p = \bigoplus_{j=1}^{j=m_0}\ZZ_p\cdot P^t_{0,j}\oplus\bigoplus_{j=1}^{j=m_1}\ZZ_p[G]\cdot P^t_{1,j}.\]
%
%

By using such a choice of points as well as the idempotent $$e_{F,m_1}:= \begin{cases}1,& \text{ if } m_0=0,\\ 1-e_G,& \text{ if } m_0\neq 0, \end{cases}$$ one obtains elements $\calL^{(m_1)}_{F,\calP_\bullet}$ of $e_{F,m_1}\cdot\CC[G]^\times$ and $\calL^*_{k,\calP_\bullet}$ of $\RR^\times$ exactly as in (\ref{lmin}).


In this special case we obtain the following refinement of Theorem \ref{cyclicthm}. This result will be proved in \S\ref{proofmequals1} below.

\begin{theorem}\label{mequals1} Fix an odd prime number $p$ that is unramified in $F$ and such that the data $A, F/k$ and $p$ satisfy the hypotheses (H$_1$)-(H$_{10}$) with $m=1$. Then the following claims are valid.

\begin{itemize}\item[(i)] If ${\rm BSD}_p(A_{F/k})$(iv) is valid one has
\begin{equation*}\calL^{(m_1)}_{F,\calP_\bullet}\equiv (-1)^{m_0}\cdot\calL^\ast_{k,\calP_{\bullet}}\cdot {\rm det}\bigl(\langle\calP_{(0)}^t,\calP_{(0)}\rangle^{\rm MT}_{F/k}\bigr)\,\,\,\,\, ({\rm mod}\,\,I_p(G)^{m_0+1}).\end{equation*}

\item[(ii)] Assume that the $p$-part of the Birch and Swinnerton-Dyer Conjecture is valid for $A$ over both $k$ and $F$ and that the group $\sha(A_F)[p^\infty]$ vanishes.

Then ${\rm BSD}_p(A_{F/k})$(iv) is valid if and only if the element $\calL^{(m_1)}_{F,\calP_\bullet}$ belongs to $I_p(G)^{m_0}$ and satisfies the displayed congruence in claim (i).\end{itemize}
\end{theorem}

\begin{remark}\label{equivalenceeasy}{\em The argument that we use to prove Theorem \ref{mequals1} can be extended, under suitable hypotheses, to the setting of cyclic extensions of arbitrary $p$-power degree. In this way one obtains an explicit family of congruence relations (of the same general form as in Theorem \ref{mequals1}(i)) whose validity is equivalent to that of ${\rm BSD}_p(A_{F/k})$(iv) whenever the $p$-part of the Birch and Swinnerton-Dyer Conjecture holds for $A$ over each intermediate field of $F/k$. This will be explained in the article \cite{bleymcIII} of Bley and the second author.
}\end{remark}

\begin{remark}\label{bleyexamples rem}{\em Theorems \ref{cyclicthm} and \ref{mequals1} both illustrate the fact that, unless the relevant Mordell-Weil groups are projective as Galois modules, the equality ${\rm BSD}_p(A_{F/k})$(iv) encodes a family of congruence relations between the values at $z=1$ of higher derivatives of Hasse-Weil-Artin $L$-series that also involve (non-trivial) Mazur-Tate regulators. As far as we are aware, no such congruence relations have previously been identified, let alone verified, in the literature.}\end{remark}

\begin{examples}\label{bleyexamples} {\em In view of the last remark, we are very grateful to Werner Bley for providing us with the following list of examples for which he has been able to numerically verify the conjectural equality ${\rm BSD}_p(A_{F/k})$(iv) by means of the congruence relation in Theorem \ref{mequals1}.  For more details of these computations, see \cite[\S 4.4]{bleymcIII} and Werner Bley's webpage. 

Throughout this list, $A$ is the elliptic curve with indicated Cremona reference, $k=\QQ$, $p=3$ and $F$ is the unique degree three Galois extension of $\QQ$ that is contained in $\QQ(\zeta_\ell)$ for any of the indicated primes $\ell$. 

We first provide some examples in which ${\rm rk}(A_F)={\rm rk}(A_\QQ)=1$, although this list is by no means exhaustive.

\begin{itemize}\item[-] $A$ is the curve $37a1$ (with Weierstrass equation $y^2+y=x^3-x$) and $\ell\in\{13,19\}.$
\item[-] $A$ is  $43a1$ ($y^2+y=x^3+x^2$) and $\ell\in\{7,13,37\}.$
\item[-] $A$ is $53a1$ ($y^2+xy+y=x^3-x^2$) and $\ell\in\{13,19,31,43\}.$
\item[-] $A$ is $58a1$ ($y^2+xy=x^3-x^2-x+1$) and $\ell\in\{7,13,19,31,43\}.$
\item[-] $A$ is $61a1$ ($y^2+xy=x^3-2x+1$) and $\ell\in\{7,13,43\}.$
\item[-] $A$ is $65a1$ ($y^2+xy=x^3-x$) and $\ell\in\{19,37,43\}.$
\item[-] $A$ is $65a2$ ($y^2+xy=x^3+4x+1$) and $\ell\in\{19,37,43\}.$
\item[-] $A$ is $77a1$ ($y^2+y=x^3+2x$) and $\ell\in\{19,37\}.$
\end{itemize}

We next provide some examples in which ${\rm rk}(A_F)={\rm rk}(A_\QQ)=2$, so the matrix that occurs in the right-hand side of the congruence in Theorem \ref{mequals1}(i) has dimension $2$.
\begin{itemize}
\item [-]  $A$ is $389a1$ ($y^2+y=x^3+x^2-2x$) and $\ell \in \{7,13,43 \}$.
\item [-]  $A$ is $433a1$ ($y^2+xy=x^3+1$) and $\ell \in \{7, 13, 31, 37 \}$.
\item [-]  $A$ is $446d1$ ($y^2+xy=x^3-x^2-4x+4$) and $\ell =19$.
\end{itemize}

}
\end{examples}

\subsubsection{}\label{proofmequals1}
In this section we prove Theorem \ref{mequals1}.

At the outset we set $C:={\rm SC}_p(A_{F/k})$ and note that, under the hypotheses of Theorem \ref{mequals1}, Theorem \ref{bk explicit} (and Remark \ref{emptysets}) implies that ${\rm BSD}_p(A_{F/k})$(iv) is valid if and only if the element
$$\xi:=\delta_{G,p}(\calL^*_{A,F/k})-\chi_{G,p}(C,h_{A,F})$$
of $K_0(\ZZ_p[G],\CC_p[G])$ vanishes.

It is a routine exercise to deduce claim (i) of Theorem \ref{mequals1} from this fact by combining Theorem \ref{thecomptheorem} below with the approach used by the second author to prove \cite[Cor. 3.5]{dmc} (and so we leave the details to the reader).

To prove claim (ii) we assume that the $p$-part of the Birch and Swinnerton-Dyer Conjecture holds for $A$ over both $k$ and $F$ and that $\sha(A_F)[p^\infty]$ vanishes. We also assume that $\calL^{(m_1)}_{F,\calP_\bullet}$ belongs to $I_p(G)^{m_0}$ and satisfies the congruence displayed in claim (i), and then proceed to prove that $\xi$ vanishes.

The first key step is to explicitly compute the second term that occurs in the definition of $\xi$ and to do this we  use the canonical identifications %
\[ r_1:A^t(F)_p\to H^1(C)\,\,\text{ and }\,\, r_2:H^2(C)\to \Sel_p(A_F)^\vee\stackrel{\sim}{\to}A(F)_p^*\]
that are described in Proposition \ref{explicitbkprop}(iv). Here the second arrow in the definition of $r_2$ is the canonical map and so is bijective since $\sha(A_F)[p^\infty]$ is assumed to vanish.

These identifications imply that the complex $C$ corresponds to a unique element $\varepsilon = \varepsilon_{C,r_1,r_2}$ of the Yoneda Ext-group ${\rm Ext}^2_{\ZZ_p[G]}(A(F)_p^*,A^t(F)_p)$.

For each pair $(i,j)$ with $i\in\{0,1\}$ and $j\in[m_i]$ we define a dual point $P_{i,j}^*$ of $A(F)_p^*$ in the natural way and obtain a direct sum decomposition of $\ZZ_p[G]$-modules
$$A(F)_p^*=\bigoplus_{j=1}^{j=m_0}\ZZ_p\cdot P_{0,j}^*\oplus\bigoplus_{j=1}^{j=m_1}\ZZ_p[G]\cdot P_{1,j}^*$$
as in \cite[Lem. 4.1]{bleymc}.

We also fix a generator $\sigma$ of $G$ and a free $\ZZ_p[G]$-module $P:=\bigoplus_{i,j}\ZZ_p[G]b_{i,j}$ of rank $m_0+m_1$ and consider the exact sequence
\begin{equation}\label{syzygy}0\to A^t(F)_p\stackrel{\iota}{\to}P\stackrel{\Theta}{\to}P\stackrel{\pi}{\to}A(F)_p^*\to 0,\end{equation} where we set $\iota(P^t_{0,j})=\Tr_G(b_{0,j})$, $\iota(P^t_{1,j})=b_{1,j}$, $\Theta(b_{i,j})=(\sigma^{p^i}-1)b_{i,j}$ and $\pi(b_{i,j})=P^*_{i,j}$.

This sequence defines a canonical isomorphism
\begin{equation}\label{computeext}{\rm Ext}^2_{\ZZ_p[G]}(A(F)_p^*,A^t(F)_p)\cong{\rm End}_{\ZZ_p[G]}(A^t(F)_p)/\iota_*\bigl(\Hom_{\ZZ_p[G]}(P,A^t(F)_p)\bigl),\end{equation}
where $\iota_*$ denotes composition with $\iota$.

Then, since Proposition \ref{explicitbkprop}(ii) implies $C$ belongs to $D^{\rm perf}(\ZZ_p[G])$ the results of \cite[Lem. 4.2 and Lem. 4.3]{bleymc} imply that there exists an automorphism $\phi$ of the $\ZZ_p[G]$-module $A^t(F)_p$ that represents the image of $\varepsilon$ under (\ref{computeext}) and also fixes the element $P^t_{1,j}$ for every $j$ in $[m_1]$.

It follows that the exact sequence
\begin{equation}\label{yonedarep}0\to A^t(F)_p\stackrel{\iota\circ\phi^{-1}}{\to}P\stackrel{\Theta}{\to}P\stackrel{\pi}{\to}A(F)_p^*\to 0\end{equation}
is a representative of the extension class $\varepsilon$.

In particular, for any choice of $\CC_p[G]$-equivariant splittings
$$s_1:\CC_p\cdot P\to \CC_p\cdot A^t(F)_p\oplus\CC_p\cdot\im(\Theta)$$
and
$$s_2:\CC_p\cdot P\to \CC_p\cdot A(F)^*_p\oplus\CC_p\cdot\im(\Theta)$$
of the scalar extensions of the canonical exact sequences
$$0\to A^t(F)_p\stackrel{\iota}{\to}P\stackrel{\Theta}{\to}\im(\Theta)\to 0$$
and
$$0\to\im(\Theta)\to P\stackrel{\pi}{\to}A(F)_p^*,$$
the definition of non-abelian determinants implies that
\begin{align}\label{separatingstuff}&-\chi_{G,p}(C,h_{A,F})\\
=&-\delta_{G,p}({\det}_{\CC_p[G]}(s_2^{-1}\circ(h_{A,F}\oplus{\rm id}_{\CC_p\cdot\im(\Theta)})\circ((\CC_p\cdot\phi)\oplus{\rm id}_{\CC_p\cdot\im(\Theta)})\circ s_1))\notag\\
=&\,\delta_{G,p}({\det}_{\CC_p[G]}(s_1^{-1}\circ(h^{-1}_{A,F}\oplus{\rm id}_{\CC_p\cdot\im(\Theta)})\circ s_2))\notag\\
&\hskip 1.5truein +\delta_{G,p}({\det}_{\QQ_p[G]}((\QQ_p\cdot\phi^{-1})\oplus{\rm id}_{\QQ_p\cdot\im(\Theta)}))\notag.\end{align}

In addition, by an explicit computation as in \cite[Prop. 4.4]{bleymc} one shows that
\begin{align}\label{theregulatorterm}&{\det}_{\CC_p[G]}(s_1^{-1}\circ(h^{-1}_{A,F}\oplus{\rm id}_{\CC_p\cdot\im(\Theta)})\circ s_2)\\
=\, &\det(e_{F,m_1}\cdot h_{F/k}(\calP^t_{(1)};\calP_{(1)}))^{-1}\cdot(\sigma-1)^{-m_0}\notag\\
&\hskip 1truein +(1-e_{F,m_1})\cdot\det(h_{k/k}(\calP^t_{(0)}\cup\calP^{t,G}_{(1)};\calP_{(0)}\cup\calP^{G}_{(1)}))^{-1}.\notag\end{align}

To compute the second term that occurs in the final equality of (\ref{separatingstuff}) we fix, as we may, any elements $\Psi_{(0,j),(0,k)}$ of $\ZZ_p[G]$ with the property that
\begin{equation}\label{matrixPsi}\phi^{-1}(P^t_{0,k})=\sum_{j\in[m_0]}\Psi_{(0,j),(0,k)}(P^t_{0,j}).\end{equation}
Then the chosen properties of the representative $\phi$ of $\varepsilon$ fixed above imply that
\begin{equation}\label{leftovertermMT}{\det}_{\QQ_p[G]}((\QQ_p\cdot\phi^{-1})\oplus{\rm id}_{\QQ_p\cdot\im(\Theta)})=
e_{F,m_1}+(1-e_{F,m_1})\cdot\det\bigl((\Psi_{(0,j),(0,k)})_{j,k\in[m_0]}\bigr).\end{equation}

The equalities (\ref{separatingstuff}), (\ref{theregulatorterm}) and (\ref{leftovertermMT}) now combine with the definition of $\xi$ to imply that
 $\xi=\delta_{G,p}(\mathcal{L})$ with
\[ \mathcal{L} := \calL_{F,\calP_\bullet}^{(m_1)}\cdot(\sigma-1)^{-m_0}+(1-e_{F,m_1})\cdot\calL^*_{k,\calP_\bullet}
\cdot\det\bigl((\Psi_{(0,j),(0,k)})_{j,k\in[m_0]}\bigr).\]

We shall use this description to show $\xi$ belongs to the finite group $K_0(\ZZ_p[G],\QQ_p[G])_{\rm tor}$, or equivalently that $\mathcal{L}$ belongs to the unit group of the maximal $\ZZ_p$-order $\mathcal{M}_p(G)$ in $\QQ_p[G]$.

To do this we note first that the assumed containment $\calL_{F,\calP_\bullet}^{(m_1)}\in I_{G,p}^{m_0}$ and validity of the $p$-part of the Birch and Swinnerton-Dyer Conjecture for $A$ over $k$ combine to imply $\mathcal{L}$ belongs to $\QQ_p[G]^\times$ and hence that $\xi$ belongs to
$K_0(\ZZ_p[G],\QQ_p[G])$.

Next we recall from the general result of \cite[Thm. 4.1]{ewt} that, since $G$ has order $p$, the group $K_0(\ZZ_p[G],\QQ_p[G])_{\rm tor}$ is the kernel of the homomorphism
\[ K_0(\ZZ_p[G],\QQ_p[G]) \xrightarrow{(\rho,q)}  K_0(\ZZ_p,\QQ_p)\oplus K_0(\ZZ_p,\QQ_p)\]
where $\rho$ is the canonical restriction of scalars map and $q$ is the homomorphism induced by mapping the class of a triple $(P,\phi,Q)$ to the class of $(P^G,\phi^G,Q^G)$.

Given this, the claimed containment $\xi \in K_0(\ZZ_p[G],\QQ_p[G])_{\rm tor}$ follows from the fact that Remark \ref{consistency remark}(ii) and (iii) combine to imply that the validities of the $p$-part of the Birch and Swinnerton-Dyer Conjecture for $A$ over $F$ and $k$ imply that  $\rho(\xi)$ and $q(\xi)$ both vanish.

At this stage we know that $\mathcal{L}$ belongs to $\mathcal{M}_p(G)^\times$ and hence that it belongs to $\ZZ_p[G]^\times$ (so that $\xi$ vanishes) if and only if it belongs to $\ZZ_p[G]$.

Now if $m_0=0$ then $\mathcal{L} =\calL_{F,\calP_\bullet}^{(m_1)}$ and $\ZZ_p[G]=I_p(G)^{m_0}$ and so our hypotheses imply directly that  $\mathcal{L} \in \ZZ_p[G]$. In the sequel we will therefore assume that $m_0\neq 0$, and hence that $e_{F,m_1}=1-e_G$.

In this case $\mathcal{L}$ belongs to $\ZZ_p[G]$ if and only if one has
\begin{equation*}\label{totallyexplicit1}\calL^{(m_1)}_{F,\calP_\bullet}\cdot(\sigma-1)^{-m_0}\equiv \calL^\ast_{k,\calP_{\bullet}}\cdot \det\bigl((\Psi_{(0,j),(0,k)})_{j,k\in[m_0]}\bigr)\,\,\,\,\, ({\rm mod}\,\,I_p(G))\end{equation*}
and so the assumed validity of the congruence in Theorem \ref{mequals1}(i) reduces us to verifying that, for any $j$ and $k$ in $[m_0]$, one has
\begin{equation}\label{PsicomputesMT}\Psi_{(0,j),(0,k)}\cdot(\sigma-1)+I_p(G)^2=-\langle P^t_{0,k},P_{0,j}\rangle^{\rm MT}_{F/k}\end{equation} in $I_p(G)/I_p(G)^2$.

To verify these equalities we use the fact, proved in Theorem \ref{thecomptheorem} below, that $\langle \,,\rangle^{\rm MT}_{F/k}$
coincides with the inverse of the pairing induced by
\begin{align*}\beta:A^t(k)_p\to I_p(G)/I_p(G)^2\otimes_{\ZZ_p}(A(F)_p^*)_{G} \to  I_p(G)/I_p(G)^2\otimes_{\ZZ_p}A(k)_p^*,\end{align*}
where the first arrow is the Bockstein homomorphism associated to the complex $C$ and the maps $r_1$ and $r_2$ (see \S\ref{bocksection} below), while the second arrow is induced by restriction to $A(k)_p$.

Now $\beta$ may be computed, through the representative (\ref{yonedarep}) of the extension class $\varepsilon$, as the connecting homomorphism that  arises when applying the snake lemma to the following commutative diagram (in which both rows and the third column are
exact and the first column is a complex)
\[\begin{CD}
@. @. @. A^t(k)_p\\ @. @. @. @VV (\iota\circ\phi^{-1})^{G} V  \\ 0 @>
>> I_p(G)\otimes_{\ZZ_p[G]}P @> \subseteq >> P @> \Tr_{G}  >>
 P^G  @> >> 0\\
@. @VV {\rm id}\otimes_{\ZZ_p[G]}\Theta V @VV \Theta V @VV\Theta^{G} V\\
0 @>
>> I_p(G)\otimes_{\ZZ_p[G]}P @> \subseteq >> P @> \Tr_{G} >>
 P^G  @> >> 0\\
@. @VV ({\rm id}\otimes_{\ZZ_p[G]}\pi)_{G}  V \\ @.
I_p(G)/I_p(G)^2\otimes_{\ZZ_p} (A(F)_p^*)_{G}.
\end{CD}\]

In addition, the equality (\ref{matrixPsi}) implies that
$$\iota(\phi^{-1}(P^t_{0,k}))=\iota(\sum_{l\in[m_0]} \Psi_{(0,l),(0,k)} P^t_{0,l})=\sum_{l\in[m_0]} \Psi_{(0,l),(0,k)}\cdot\Tr_G(b_{0,l})$$
and so, since $$P^*_{0,l}(P_{0,j})=\begin{cases}1,\,\,\,\,\,\,\,\,l=j,\\
0,\,\,\,\,\,\,\,\,l\neq j,\end{cases}$$
we can finally use the above diagram to compute that
\begin{align*}-\langle P^t_{0,k},P_{0,j}\rangle_{F/k}^{\rm MT}=&\left(({\rm id}\otimes_{\ZZ_p[G]}\pi)_{G}\left(\Theta(\Psi_{(0,j),(0,k)}\cdot b_{0,j})\right)\right)(P_{0,j})\\
=&\left(({\rm id}\otimes_{\ZZ_p[G]}\pi)_{G}\left((\sigma-1)\Psi_{(0,j),(0,k)}\cdot b_{0,j}\right)\right)(P_{0,j})\\
=&\left(\left((\sigma-1)+I_p(G)^2\right)\otimes\Psi_{(0,j),(0,k)}\cdot P^*_{0,j}\right)(P_{0,j})\\
=&\Psi_{(0,j),(0,k)}\cdot(\sigma-1)+I_p(G)^2.
\end{align*}
Here the first equality is given by Theorem \ref{thecomptheorem}. This verifies the equalities (\ref{PsicomputesMT}) and thus completes the proof of Theorem \ref{mequals1}.

\subsection{Dihedral congruence relations}\label{dihedral} With a view to extending the classes of extensions $F/k$ for which the equality of ${\rm BSD}_p(A_{F/k})$(iv) can be made fully explicit we consider the case that $F/k$ is generalized dihedral of order $2p^n$.

We recall that this means the Sylow $p$-subgroup $P$ of $G$ is abelian and of index two and that the conjugation action of any lift to $G$ of the generator of $G/P$ inverts elements of $P$. We write $K$ for the unique quadratic extension of $k$ in $F$.

In this setting we shall show that, in certain situations, the validity of ${\rm BSD}_p(A_{F/k})$(iv) can be checked by verifying congruences relative to the abelian extension $F/K$.

In order to state the precise result we fix a finite Galois extension $E$ of $\QQ$ in $\bc$ that is large enough to ensure that, with $\mathcal{O}$ denoting the ring of algebraic integers of $E$, there exists for each character $\psi$ of $\widehat{G}$ a finitely generated $\mathcal{O}[G]$-lattice that is free over $\mathcal{O}$ and spans a $\bc[G]$-module of character $\psi$.

For each $\psi$ in $\widehat{G}$ we recall the non-zero complex number $\mathcal{L}^\ast(A,\psi)$ defined in \S\ref{explicit ec section}.

\begin{proposition}\label{dihedral prop} Let $F/k$ be generalized dihedral of degree $2p^n$ as above. Assume that $\sha(A_F)$ is finite and that no place of $k$ at which $A$ has bad reduction is ramified in $F$. Assume also that $p$ satisfies the conditions (H$_1$)-(H$_4$) listed in \S\ref{tmc} and that neither $A(K)$ nor $A^t(K)$ has a point of order $p$.

Then the equality of ${\rm BSD}_p(A_{F/k})$(iv) is valid if the following three conditions are satisfied.

\begin{itemize}
\item[(i)] For every $\psi$ in $\widehat{G}$ and $\omega$ in $G_\QQ$, one has $\mathcal{L}^\ast(A,\omega\circ \psi) = \omega(\mathcal{L}^\ast(A,\psi))$.
\item[(ii)] For every $\psi$ in $\widehat{G}$ and every prime ideal $\mathfrak{p}$ of $\mathcal{O}$ that divides $p$, the explicit formula for $\mathcal{L}^\ast(A, \psi)\cdot \mathcal{O}_\mathfrak{p}$ that is given in Proposition \ref{ref deligne-gross}(ii) is valid.
\item[(iii)] The equality of ${\rm BSD}_p(A_{F/K})$(iv) is valid.
\end{itemize}
\end{proposition}

\begin{proof} Since $F/K$ is an extension of $p$-power degree, the assumption that neither $A(K)$ nor $A^t(K)$ has a point of order $p$ implies that neither $A(F)$ nor $A^t(F)$ has a point of order $p$.

Hence, in this case, the given assumptions imply that the data $A$, $F/k$ and $p$ satisfy all of the hypotheses of Proposition \ref{ref deligne-gross}.

In particular, if we write $\xi$ for the difference between the left and right hand sides of the equality in Theorem \ref{bk explicit}, then the argument of Proposition \ref{ref deligne-gross} shows that the assumed validity of the given conditions (i) and (ii) implies that $\xi$ belongs to
 $K_0(\ZZ_p[G],\QQ_p[G])$ and also to the kernel of the homomorphism $\rho_\mathfrak{p}^\psi$ for every $\psi$ in $\widehat{G}$ and every prime ideal $\mathfrak{p}$ of $\mathcal{O}$ that divides $p$.

These facts combine with the general result of \cite[Th. 4.1]{ewt} to imply that $\xi$ belongs to the finite group $K_0(\ZZ_p[G],\QQ_p[G])_{\rm tor}$.

We next recall from \cite[Lem. 5.12(ii)]{bmw} that, since $G$ is assumed to be dihedral, the natural restriction map ${\rm res}^G_P:K_0(\ZZ_p[G],\QQ_p[G])_{\rm tor} \to K_0(\ZZ_p[P],\QQ_p[P])$ is injective.

It follows that $\xi$ vanishes, and hence by Theorem \ref{bk explicit} that ${\rm BSD}_p(A_{F/k})$(iv) is valid, if the element ${\rm res}^G_P(\xi)$ vanishes.

To complete the proof, it is therefore enough to note that the functorial behaviour of the conjecture ${\rm BSD}(A_{F/k})$ under change of extension, as described in Remark \ref{consistency remark}(ii) (and justified via Remark \ref{consistency}(ii)), implies that ${\rm res}^G_P(\xi)$ vanishes if and only if the equality of ${\rm BSD}_p(A_{F/K})$(iv) is valid. \end{proof}

\begin{remark}{\em If $P$ is cyclic, then Proposition \ref{dihedral prop} shows that in certain situations the validity of ${\rm BSD}_p(A_{F/k})$(iv) for the non-abelian extension $F/k$ can be checked by verifying the relevant cases of the refined Deligne-Gross Conjecture formula in Proposition \ref{ref deligne-gross}(ii) together with explicit congruences for the cyclic extension $F/K$ of the form that are discussed in \S\ref{cycliccongssection}. In addition, if $P$ is cyclic, then the main result of Yakovlev in \cite{yakovlev2} can be used to show that if the groups $\sha(A_{F'})[p^\infty]$ vanish for all proper subfields of $F$ that contain $K$, then the $\ZZ_p[G]$-module $A(F)_p$ is a `trivial source module' and so has a very explicit structure. }\end{remark}

\begin{example}{\em For the examples described in Example \ref{wuthrich example} the field $F$ is a dihedral extension of $k = \QQ$ of degree $6$ and both of the given elliptic curves $A$ satisfy all of the hypotheses that are necessary to apply Proposition \ref{dihedral prop} (in the case $p=3$ and $n=1$). In this way one finds that the validity of ${\rm BSD}_3(A_{F/\QQ})$(iv) implies, and if $\sha(A_F)[3^\infty]$ vanishes is equivalent to, the validity of the relevant cases of the refined Deligne-Gross Conjecture together with the validity of an explicit congruence of the form described in Theorem \ref{mequals1} for the cyclic extension $F/K$ (and with $m_1= 1$ and $m_0 = 2$). Unfortunately, however, since $F/\QQ$ is of degree $6$ it seems that for the given curves $A$ the latter congruences are at present beyond the range of numerical investigation via the methods that have been used to verify the cases discussed in Example \ref{bleyexamples}. }\end{example}

\subsection{The proof of Theorem \ref{rbsdimpliesmt}}\label{mrsproof}

We fix a subgroup $J$ of $G$ and set $E := F^J$ and a maximal subset $x_\bullet$ of $A^t(F_p)^\wedge_p$ that is linearly independent over $\ZZ_p[G]$.

To study the relationship between ${\rm BSD}(A_{F/k})$ and Conjecture \ref{mrsconjecture} we set $X:=\langle x_\bullet\rangle_{\ZZ_p[G]}$ and consider both of the complexes $C_{S,X}:={\rm SC}_S(A_{F/k};X,H_\infty(A_{F/k})_p)$ and $C_{S,X,J}:={\rm SC}_S(A_{E/k};X^J,H_\infty(A_{E/k})_p)$.

Then the definition of $C_{S,X}$ as the mapping fibre of the morphism (\ref{fibre morphism}), the well-known properties of \'etale cohomology under Galois descent and the fact that $X$ is a free $\ZZ_p[G]$-module (see also (\ref{global descent}), (\ref{local descent}) and the argument that precedes them) imply that the object $$(C_{S,X})_J:=\ZZ_p[G/J]\otimes_{\ZZ_p[G]}^{\mathbb{L}}C_{S,X}$$ of $D^{\rm perf}(\ZZ_p[G/J])$ is isomorphic to $C_{S,X,J}$. This fact combines with \cite[Lem. 3.33]{bst} to give canonical identifications
$$H^1(C_{S,X})^J=H^1((C_{S.X})_J)=H^1(C_{S,X,J})$$ and
$$H^2(C_{S,X})_J=H^2((C_{S,X})_J)=H^2(C_{S,X,J}).$$
%

In addition, our assumption that neither $A(F)$ nor $A^t(F)$ has a point of order $p$ combines with Proposition \ref{prop:perfect} to imply that,
in the terminology of \cite{bst}, the complex $C_{S,X}$ is a `strictly admissible' complex of $\ZZ_p[G]$-modules. The approach of \S 3.4.1 in loc. cit. therefore gives a `Bockstein homomorphism' of $\ZZ_p[G/J]$-modules \begin{equation}\label{htc}{\rm ht}^{C_{S,X}}:H^1(C_{S,X})^J\to \mathcal{I}_p(J)/\mathcal{I}_p(J)^2\otimes_{\ZZ_p[G/J]}H^2(C_{S,X,J}).\end{equation}

For any $p$-separable choice of points $(\mathcal{Y},\mathcal{Y}')$ of $A$ for $F/E$ of rank $(a,a')$
and each index $i$ with $a< i \le a'$ we then use the dual point $w_i^*$ to construct a composite homomorphism of $\ZZ_p[G/J]$-modules \begin{equation}\label{projection}H^2(C_{S,X,J})\to\Sel_p(A_E)^\vee\to A(E)_p^*\to (Y')^*\to\ZZ_p[G/J].\end{equation} Here the first arrow is the canonical homomorphism of Proposition \ref{prop:perfect}(iii), the second arrow is the canonical homomorpism occurring in (\ref{sha-selmer}), the third arrow is the natural restriction maps and the fourth arrow maps an element of $(Y')^*$ to its coefficient at the basis element $w_i^*$.

Upon composing ${\rm ht}^{C_{S,X}}$ with each map (\ref{projection}) we thereby obtain, for each index $i$ with $a<i\leq a'$, a homomorphism of $\ZZ_p[G/J]$-modules
\begin{equation}\label{htcwi} {\rm ht}^{C_{S,X}}_{w_i}: H^1(C_{S,X})^J \to \mathcal{I}_p(J)/\mathcal{I}_p(J)^2 \end{equation}
that corresponds to the map denoted by ${\rm Boc}^C_x$ in \cite[\S 3.4.1]{bst}.
Just as in (\ref{canbecombined}) these homomorphisms can then be combined to give a canonical homomorphism \begin{equation}\label{htcy}{\rm ht}_{\mathcal{Y}'}^{C_{S,X}}: {\bigcap}_{\ZZ_p[G/J]}^{a'} H^1(C_{S,X})^J \longrightarrow ({\bigcap}_{\ZZ_p[G/J]}^{a} H^1(C_{S,X})^J)\otimes_{\ZZ_p} I_p(J)^{a'-a}/I_p(J)^{1+a'-a}.\end{equation}

%
%
%


We will apply the general result of \cite[Th. 3.34(ii)]{bst} to the complex $C_{S,X}$, the isomorphism $h_{A,F}$ and the characteristic element $\mathcal{L}_{F/k}$ for the pair $(C_{S,X},h_{A,F})$ constructed in Lemma \ref{modifiedlemma}(iii) under the assumption that ${\rm BSD}(A_{F/k})$ is valid. We recall that \begin{equation}\label{charelementFk}e_{F,a}\cdot\mathcal{L}_{F/k}^{-1}=\frac{L^{(a)}_{S}(A_{F/k},1)}{\Omega_A^{F/k}\cdot w_{F/k}^d}\cdot \mathcal{LR}^p_{A^t_{F/k}}(x_\bullet).\end{equation}

Similarly, if ${\rm BSD}(A_{F/k})$ is valid then the observations made in Remark \ref{consistency remark}(ii) (see also Remark \ref{consistency} below for more details) imply that ${\rm BSD}(A_{E/k})$ is also valid. Applying Lemma \ref{modifiedlemma}(iii) to the complex $C_{S,X,J}$ and the isomorphim $h_{A,E}$ therefore implies the existence of a characteristic element $\mathcal{L}_{E/k}$ for this pair with the property that \begin{equation}\label{charelementEk}e_{E,a'}\cdot\mathcal{L}_{E/k}^{-1}=\frac{L^{(a')}_{S}(A_{E/k},1)}{\Omega_A^{E/k}\cdot w_{E/k}^d}\cdot \mathcal{LR}^p_{A^t_{E/k}}(x_\bullet^J).\end{equation}

We next fix any (injective) section $s$ to the canonical surjective composite homomorphism $$H^2(C_{S,X})\to\Sel_p(A_F)^\vee\to A(F)_p^*\to Y^*,$$ where the last arrow is the natural restriction map. The set $s(\mathcal{Y}^*):=\{s(y^*):y\in\mathcal{Y}\}$ is then (in the sense of \cite[Def. 3.6]{bst}) a `separable subset' of $H^2(C_{S,X})$ (of cardinality $a$).
The `higher special element' associated via \cite[Def. 3.3]{bst} to the data $$(C_{S,X},h_{A,F},\mathcal{L}_{F/k},s(\mathcal{Y}^*))$$ is then, by (\ref{charelementFk}) and a direct comparison of the definitions of $h_{A,F}$ and ${\rm ht}_{A_{F/k}}^{a}$, equal to $\eta_{\mathcal{Y},x_\bullet}$.

We finally fix any (injective) section $s'$ to the canonical surjective composite homomorphism $$H^2(C_{S,X,J})\to\Sel_p(A_E)^\vee\to A(E)_p^*\to (Y')^*,$$ where the last arrow is the natural restriction map. The set $s'((\mathcal{Y}')^*):=\{s'(y^*):y\in\mathcal{Y}'\}$ is then a separable subset of $H^2(C_{S,X,J})$ (of cardinality $a'$). The higher special element associated via \cite[Def. 3.3]{bst} to the data $$(C_{S,X,J},h_{A,E},\mathcal{L}_{E/k},s'((\mathcal{Y}')^*))$$ is then, by (\ref{charelementEk}) and a direct comparison of the definitions of $h_{A,E}$ and ${\rm ht}_{A_{E/k}}^{a'}$, equal to $\eta_{\mathcal{Y}',x_\bullet^J}$.

The general result \cite[Th. 3.34(ii)]{bst} therefore implies that
\[ \eta_{\mathcal{Y},x_\bullet} \in {\bigcap}_{\ZZ_p[G]}^{a} H^1(C_{S,X})\subseteq{\bigcap}_{\ZZ_p[G]}^{a}A^t(F)_p,\]
and
\[ \eta_{\mathcal{Y}',x_\bullet^J}\in {\bigcap}_{\ZZ_p[G/J]}^{a'} H^1(C_{S,X})^J\subseteq{\bigcap}_{\ZZ_p[G/J]}^{a'}A^t(E)_p,\]
and that 
%
\begin{equation}\label{forreferenceinremark} \mathcal{N}_J(\eta_{\mathcal{Y},x_\bullet})=(-1)^{a(a'-a)}\cdot \nu_J({\rm ht}_{\mathcal{Y}'}^{C_{S,X}}(\eta_{\mathcal{Y}',x_\bullet^J}))\end{equation}
in $\left({\bigcap}_{\ZZ_p[G/J]}^{a} A^t(F)_p\right)\otimes_{\ZZ_p} I_p(J)^{a'-a}/I_p(J)^{1+a'-a}$. 

To complete the proof of Theorem \ref{rbsdimpliesmt} we will combine the equality (\ref{forreferenceinremark}) with the comparison of height pairings proved in Theorem \ref{thecomptheorem} below.


We first note that $A$, $F/E$ and $p$ satisfy the hypotheses (H$_1$)-(H$_5$) listed in \S \ref{tmc} (see \cite[Lem. 3.4]{bmw0}). We also note that the validity of these hypotheses (and their duality) combines with Proposition \ref{explicitbkprop}(ii) to imply that $U_{F/E,p}=A(E)_p$ and $U^t_{F/E,p}=A^t(E)_p$.

It will thus be enough to prove that, for any index $a<i\leq a'$, the restriction to $H^1(C_{S,X})^J$ of the homomorphism
$${\rm ht}_{w_i}^{{\rm MT}}:A^t(E)_p\to \mathcal{I}_p(J)/\mathcal{I}_p(J)^2$$ is equal to the homomorphism ${\rm ht}_{w_i}^{C_{S,X}}$ of (\ref{htcwi}).

By Theorem \ref{thecomptheorem} below applied to $F/E$, for any $P$ in $A(E)_p$ one has $${\rm ht}_{w_i}^{{\rm MT}}(P):=-\langle P,w_i\rangle^{\rm MT}_{F/E}=\langle P,w_i\rangle_{p}^{\rm N}$$ in $\ZZ_p\otimes_\ZZ J$, where $\langle \,,\rangle_{p}^{\rm N}$ is the Bockstein pairing associated to the classical Selmer complex ${\rm SC}_p(A_{F/E})$ as in (\ref{thepairing}) below.

It is therefore be enough to prove that, for any index $a<i\leq a'$ and for any $P\in H^1(C_{S,X})^J$ one has \begin{equation*}{\rm ht}_{w_i}^{C_{S,X}}(P)=\langle P,w_i\rangle_{p}^{\rm N}.\end{equation*} But this equality follows easily from the compatibility of the maps $$H^1(C_{S,X})^J=H^1(C_{S,X,J})\to H^1({\rm SC}_p(A_{F/E}))$$ and $$H^2(C_{S,X,J})\to H^2({\rm SC}_p(A_{F/E}))$$ occurring in the exact sequence (\ref{useful1}) with all the other maps involved in the construction of both Bockstein homomorphisms (and the naturality of such homomorphisms). This completes the proof of Theorem \ref{rbsdimpliesmt}.


\section{Height pairing comparisons}\label{comparison section}

In this section we continue to assume that $F/k$ is abelian.

We fix an odd prime number $p$ and assume that $A$, $F/k$ and $p$ satisfy the hypotheses (H$_1$)-(H$_6$) listed in \S \ref{tmc}. We also assume that neither $A(F)$ not $A^t(F)$ has a point of order $p$.

In this situation, Proposition \ref{explicitbkprop} implies that the classical Selmer complex ${\rm SC}_p(A_{F/k})$ belongs to $D^{\rm perf}(\ZZ_p[G])$ and is acyclic outside degrees one and two, where its cohomology modules canonically identify with $A^t(F)_p$ and $\Sel_p(A_F)^\vee$ respectively.

Via the general formalism outlined in Appendix \ref{bocksection} one therefore obtains an associated canonical algebraic height pairing
$$\langle\,,\rangle^{\rm N}_{p}:A^t(k)_p\otimes_{\ZZ_p} A(k)_p\to \ZZ_p\otimes_{\ZZ}G$$
that we refer to as the `Nekov\'a\v r height pairing'.

Proposition \ref{explicitbkprop}(ii) also implies that the $p$-completion of the height pairing (\ref{tanpairing}) (for $J=G$) that is defined using the theory of bi-extensions by Mazur and Tate in \cite{mt0} gives a well-defined pairing
$$\langle\,,\rangle_p^{\rm MT}:A^t(k)_p\otimes_{\ZZ_p} A(k)_p\to \ZZ_p\otimes_{\ZZ}G.$$

In this section we shall prove that the Nekov\'a\v r height pairing coincides with the inverse of the Mazur-Tate height pairing. This comparison result was already used in the proof of Theorem \ref{rbsdimpliesmt} and is also, we believe, of some independent interest.

\subsection{Statement of the result}

To make the construction of the Nekov\'a\v r height pairing more precise we fix a finite subset $\Sigma$ of $S_k^f$ that contains $S_k^p$, $S_k^F\cap S_k^f$ and $S_k^A$.

Proposition \ref{explicitbkprop} implies that the Selmer complex $C_{\Sigma}={\rm SC}_{\Sigma,p}(A_{F/k})$ belongs to $D^{\rm perf}(\ZZ_p[G])$ and is acyclic outside degrees one and two.

In addition,
we can use Proposition \ref{explicitbkprop}(iv) to interpret the global Kummer map as an isomorphism \begin{equation}\label{r1}r_1:A^t(F)_p\to H^1(C_{\Sigma}),\end{equation} and also to define a canonical composite homomorphism \begin{equation}\label{r2}r_2:H^2(C_{\Sigma})\stackrel{\sim}{\to}\Sel_p(A_F)^\vee{\to} A(F)_p^*\to A(k)_p^*.\end{equation} Here the first arrow is the canonical isomorphism occurring in the diagram (\ref{Selmerdiagram}) and the last arrow is simply the restriction to $A(k)_p$ of any element of the linear dual $A(F)_p^*$.

By using the general approach (and notation) of \S \ref{bocksection}, we can therefore define the Nekov\'a\v r height pairing by setting
\begin{equation}\label{thepairing}\langle\,,\rangle^{\rm N}_{p}:=\langle\,,\rangle_{C_\Sigma,r_1,r_2}:A^t(k)_p\otimes_{\ZZ_p} A(k)_p\to I_{p}(G)/I_{p}(G)^2\cong \ZZ_p\otimes_{\ZZ}G,\end{equation}
where $I_p(G)$ denotes the augmentation ideal in $\ZZ_p[G]$.

Then Lemma \ref{independenceofsigma} and Remark \ref{indeptremark} combine with the naturality of Bockstein pairings to imply that this pairing is  independent of the choice of set $\Sigma$.

We write $m_p(G)$ for the non-negative integer for which $p^{m_p(G)}$ is the exponent of the group $\ZZ_p\otimes_{\ZZ}G$.

Then the proof of our comparison of height pairings is greatly simplified by imposing the following additional technical hypothesis:

\begin{hyp}\label{hypbd}
The group $H^1({\rm Gal}(F(A[p^{m_p(G)}])/F),A[p^{m_p(G)}])$ vanishes.
\end{hyp}

\begin{remarks}\label{rmkbd}{\em
Hypothesis \ref{hypbd} is motivated by the result \cite[Lem. 2.13]{bert} of Bertolini and Darmon and  
 is satisfied whenever multiplication-by-`$-1$' belongs to the image of the canonical Galois representation $G_F\to\Aut_{\mathbb{F}_p}(A[p])$. In particular, if $A$ is an elliptic curve, then the hypothesis only excludes finitely many primes $p$ by a famous result of Serre \cite{serre}. Moreover, if $A$ is an elliptic curve and $F$ does not contain any $p$-th roots of unity, Lawson and Wuthrich have recently shown that Hypothesis \ref{hypbd} is valid in all but certain exceptional cases that, in particular, all have $p\leq 11$ (see \cite[Thm. 2 and \S 6]{lw}).}
\end{remarks}



\begin{theorem}\label{thecomptheorem} Fix an odd prime number $p$ and assume that $A$, $F/k$ and $p$ satisfy the hypotheses (H$_1$)-(H$_6$) listed in \S \ref{tmc}. Assume also that neither $A(F)$ nor $A^t(F)$ has a point of order $p$ and that Hypothesis \ref{hypbd} is satisfied.

Then the Nekov\'a\v r pairing $\langle\,,\rangle^{\rm N}_{p}$ is equal to the inverse of the Mazur-Tate pairing $\langle\,,\rangle^{\rm MT}_p$.
\end{theorem}

The proof of Theorem \ref{thecomptheorem} will occupy the rest of this section. We thus assume throughout that all the hypotheses listed in Theorem \ref{thecomptheorem} are valid.

\subsection{An auxiliary height pairing}

In this section we use constructions due to Bertolini and Darmon \cite{bert,bert2} to
define an auxiliary $p$-adic height pairing $\langle\,,\rangle^{\rm BD}_p$. The explicit computation carried out in Lemma \ref{bdcomp} below will in fact easily show that this definition is just an alternate description of the pairing $\langle\,,\rangle_1$ from \cite[\S 3.4.1]{bert} or \cite[\S 2.2]{bert2}.


In \S\ref{compproof} below we will then show that $\langle\,,\rangle_p^{\rm BD}$ coincides with the inverse of $\langle\,,\rangle_p^{\rm N}$. By then appealing to comparison results due to Bertolini and Darmon and to Tan, we will finally obtain the desired comparison of pairings.

We will now follow closely some of the constructions from \cite{bert,bert2} and for this reason we will be rather brief in some of our arguments.

\subsubsection{} We again write $m:=m_p(G)$ for the non-negative integer such that $p^m$ is the exponent of the finite abelian $p$-group $\ZZ_p\otimes_{\ZZ}G$.

We set $Z:=\ZZ/p^m\ZZ$ and $R:=Z[G]$ and write $I_{R}$ for the augmentation ideal in $R$. We will also require all of the following notation.

For any finite set $S$ of non-archimedean places of $k$, we define $\Sel^S_{p^m}(A^t/F)$ to be the kernel of the localisation map $$H^1\bigl(F,A^t[p^m]\bigr)\to\bigoplus_{w'\notin S(F)}H^1(F_{w'},A^t)$$ (with the sum running over all non-archimedean places of $F$ that do not belong to $S(F)$). When $S$ is taken to be empty one recovers the usual Selmer group associated to multiplication by $p^m$, which we denote by $\Sel^{(p^m)}(A^t_F)$ instead.

For such a set $S$ and $B\in\{A,A^t\}$ we also define $R$-modules
$$B_S(\mathbb{A}_F)/p^m:=\prod_{w'\in S(F)}B(F_{w'})/p^mB(F_{w'}),$$ $$H^1_{S}\bigl(\mathbb{A}_F,B[p^m]\bigr):=\prod_{w'\in S(F)}H^1(F_{w'},B[p^m])$$ and $$H^1_{S}\bigl(\mathbb{A}_F,B\bigr)[p^m]:=\prod_{w'\in S(F)}H^1(F_{w'},B)[p^m].$$

By abuse of notation, we denote by $\lambda_{S}$ both the localisation map $$\Sel^S_{p^m}(A^t/F){\to}H^1_{S}\bigl(\mathbb{A}_F,A^t[p^m]\bigr)$$ and the induced map $$\Sel^S_{p^m}(A^t/F){\to}H^1_{S}\bigl(\mathbb{A}_F,A^t\bigr)[p^m].$$

We also write $$\kappa:A^t_S(\mathbb{A}_F)/p^m\to H^1_{S}\bigl(\mathbb{A}_F,A^t[p^m]\bigr)$$ and $$\kappa':H^1_{S}\bigl(\mathbb{A}_F,A\bigr)[p^m]^\vee\to H^1_{S}\bigl(\mathbb{A}_F,A[p^m]\bigr)^\vee$$ for the canonical maps induced by the local Kummer sequences.

We recall (see Corollary \ref{Tatepoitouexplicit} below) that the local Tate duality isomorphism $$u_S:H^1_{S}\bigl(\mathbb{A}_F,A^t[p^m]\bigr)\to H^1_{S}\bigl(\mathbb{A}_F,A[p^m]\bigr)^\vee$$ induces an isomorphism
\begin{equation}\label{inducedbyum}H^1_{S}\bigl(\mathbb{A}_F,A^t\bigr)[p^m]\cong\cok(\kappa)\stackrel{\sim}{\to}\cok(\kappa')\cong\bigl(A_S(\mathbb{A}_F)/p^m\bigr)^\vee.\end{equation}
By composing this isomorphism with the dual of the restriction map $$\Sel^{(p^m)}(A_F)\to A_S(\mathbb{A}_F)/p^m$$ one finally obtains a map
$$v_S:H^1_{S}\bigl(\mathbb{A}_F,A^t\bigr)[p^m]{\to}
{\rm Sel}^{(p^m)}(A_F)^\vee.$$

\subsubsection{}
In the sequel we denote by $P$ the Sylow $p$-subgroup of $G$ and set $K:=F^P$.

\begin{lemma}\label{214} The restriction map $H^1(K,A[p^m])\to H^1(F(A[p^m]),A[p^m])$ is injective.\end{lemma}
\begin{proof} As in \cite[Cor. 2.14]{bert}, the result follows immediately upon combining the validity of Hypothesis \ref{hypbd} with the fact that $A(F)[p^m]$ vanishes (because $A(F)$ does not have a point of order $p$).
\end{proof}

Lemma \ref{214} now implies the existence of `admissible' auxiliary sets of primes, as in \cite[Def. 1.5]{bert2} or \cite[Def. 2.20, Def. 2.22]{bert}:

\begin{lemma}\label{tm} There is a finite set $T$ of non-archimedean places of $K$ such that:
\begin{itemize}
\item[(i)] Every place in $T$ splits completely in $F/K$.
\item[(ii)] If $v\in T$ then $A(K_v)/p^mA(K_v)$ is a free $Z$-module of rank $2\dim(A)$.
\item[(iii)] The restriction map \begin{equation}\label{Kres}\Sel^{(p^m)}(A_K)\to\bigoplus_{v\in T}A(K_v)/p^mA(K_v)\end{equation} is injective.
\end{itemize}
\end{lemma}

The existence of such a set $T$ (of infinitely many such sets $T$, in fact) is proved in \cite[Lem. 2.21]{bert} using a Tcheboratev density argument (and Lemma \ref{214}).

In the sequel, we fix such an admissible set $T$ and, by abuse of notation, also denote by $T$ the set of places of $k$ below places in $T$.

The existence of the admissible set $T$ now leads us to consider the canonical complex of $R$-modules $$C_{\rm BD}:=\bigl[\Sel^T_{p^m}(A^t/F)\stackrel{\lambda_T}{\to}H^1_{T}\bigl(\mathbb{A}_F,A^t\bigr)[p^m]\bigr],$$ with first term placed in degree 1. The desired auxiliary pairing $\langle\,,\rangle_p^{\rm BD}$ will then be obtained as a Bockstein pairing associated to $C_{\rm BD}$.

\begin{lemma}\label{bdlemma} The following claims are valid.
\begin{itemize}\item[(i)] Let $S$ be any finite set of non-archimedean places of $k$ which contains $T$. Then the sequence
\begin{equation}\label{notexact}0\to{\rm Sel}^{(p^m)}(A^t_F)\to \Sel^S_{p^m}(A^t/F)\stackrel{\lambda_S}{\to}H^1_{S}\bigl(\mathbb{A}_F,A^t\bigr)[p^m]\stackrel{v_S}{\to}
{\rm Sel}^{(p^m)}(A_F)^\vee\to 0\end{equation} is exact. 
\item[(ii)] The $R$-modules $\Sel^T_{p^m}(A^t/F)$, $H^1_{T}\bigl(\mathbb{A}_F,A^t\bigr)[p^m]$ and $A_T(\mathbb{A}_F)/p^m$ are finitely generated and $G$-cohomologically-trivial.
\item[(iii)] Let $S\supseteq T$ be any set as in claim (i). Then there is a canonical quasi-isomorphism from the complex $C_{\rm BD}$ to the complex $$C_{{\rm BD},S}:=\bigl[\Sel^S_{p^m}(A^t/F)\stackrel{\lambda_S}{\to}H^1_{S}\bigl(\mathbb{A}_F,A^t\bigr)[p^m]\bigr],$$ with first term placed in degree 1, which induces the identity map both on ${\rm Sel}^{(p^m)}(A^t_F)$ and on ${\rm Sel}^{(p^m)}(A_F)^\vee$ (viewed as the cohomology modules of $C_{\rm BD}$ and of $C_{{\rm BD},S}$ via (\ref{notexact})).
\end{itemize}
\end{lemma}
\begin{proof} The exactness of (\ref{notexact}) at the first three modules is proved in \cite[Lem. 6.15]{milne}, and is independent of the fact that the set $S$ contains $T$. See also \cite[Prop. 1.2]{bert2} and the discussion that follows it.

In order to prove claim (i), it is hence enough to show that $v_{T}$ is surjective (since then so must be $v_S$ for any set $S$ containing $T$), or equivalently that the restriction map \begin{equation}\label{tmrestriction}\Sel^{(p^m)}(A_F)\to A_{T}(\mathbb{A}_F)/p^m\end{equation} is injective. This may be achieved by following the approach of \cite[Lem. 2.25, Lem. 2.26, Lem 3.1]{bert} or \cite[Prop. 1.6]{bert2}. We limit ourselves to giving a sketch of this argument.

Proposition \ref{explicitbkprop}(ii) implies that, for any non-archimedean place $v$ of $k$, the module $A(F_v)$ is $P$-cohomologically-trivial.





This fact then combines with our assumption that $A(F)$ has no point of order $p$ to imply the restriction map $\Sel^{(p^m)}(A_K)\to\Sel^{(p^m)}(A_F)^P$ is bijective, by the same argument used to prove \cite[Lem 2.19]{bert} or \cite[Lem. 1.4]{bert2}.

Since the restriction map (\ref{Kres}) is injective, one finally deduces that the restriction map (\ref{tmrestriction}) is injective, just as in \cite[Lem 2.25]{bert} or \cite[Prop. 1.6]{bert2}.

To next consider claim (ii), it is enough to show that each of the modules $\Sel^T_{p^m}(A^t/F)$, $H^1_{T}\bigl(\mathbb{A}_F,A^t\bigr)[p^m]$ and $A_T(\mathbb{A}_F)/p^m$ is finitely generated and free over $Z[P]$. These facts in turn follow from the arguments of \cite[\S 3.1]{bert}.

Indeed, one first uses properties (i) and (ii) of the set $T$, as given in Lemma \ref{tm}, to see that $A_T(\mathbb{A}_F)/p^m$ is (finitely generated and) $Z[P]$-free.
The duality (\ref{inducedbyum}) between $A_{T}(\mathbb{A}_F)/p^m$ and $H^1_{T}\bigl(\mathbb{A}_F,A^t\bigr)[p^m]$ then implies that the same property is true for the latter module.

Since neither $A(F)$ nor $A^t(F)$ has a point of order $p$, the canonical exact sequences
$$0\to A(F)/p^mA(F)\to\Sel^{(p^m)}(A_F)\to\sha(A_F)[p^m]\to 0$$ and
$$0\to A^t(F)/p^mA^t(F)\to\Sel^{(p^m)}(A^t_F)\to\sha(A^t_F)[p^m]\to 0$$ combine with the Cassels-Tate pairing to imply that the finite groups ${\rm Sel}^{(p^m)}(A^t_F)$ and ${\rm Sel}^{(p^m)}(A_F)^\vee$ have the same order. One may therefore combine the fact that $H^1_{T}\bigl(\mathbb{A}_F,A^t\bigr)[p^m]$ is $Z[P]$-free with the exactness of (\ref{notexact}) (for $S=T$) to deduce, as in \cite[Thm. 3.2]{bert}, that $\Sel^{T}_{p^m}(A^t/F)$ is also $Z[P]$-free.

We finally consider claim (iii). The canonical inclusions give vertical arrows in the commutative diagram
\begin{equation*}\xymatrix{
\Sel^{T}_{p^m}(A^t/F) \ar[d] \ar[r]^(0.45){\lambda_T} & H^1_{T}\bigl(\mathbb{A}_F,A^t\bigr)[p^m] \ar[d] \\
\Sel^{S}_{p^m}(A^t/F)  \ar[r]^(0.45){\lambda_S} & H^1_{S}\bigl(\mathbb{A}_F,A^t\bigr)[p^m]
}\end{equation*}
and hence define a morphism of complexes $C_{\rm BD}\to C_{{\rm BD},S}$.

It is clear that this morphism  induces the identity map both on ${\rm Sel}^{(p^m)}(A^t_F)$ and on ${\rm Sel}^{(p^m)}(A_F)^\vee$, so by the exactness of the respective sequences of the form (\ref{notexact}) it is in particular a quasi-isomorphism, as required.
%
\end{proof}

\subsubsection{}
To define the useful auxiliary pairing we let $S$ be any finite set of non-archimedean places of $k$ containing $T$.

We may then combine Lemma \ref{bdlemma} with the general formalism outlined in \S \ref{bocksection} below to define a pairing
$\langle\,,\rangle_p^{\rm BD}$ as follows. 

We let $r_{1,m}$ denote the canonical composite map
$$A^t(k)_p/p^m\to A^t(F)_p/p^m\to \Sel^{(p^m)}(A^t_F)\to H^1(C_{{\rm BD},S}),$$  and $r_{2,m}$ denote the canonical composite map $$H^2(C_{{\rm BD},S})\stackrel{v_{S}}{\to}{\rm Sel}^{(p^m)}(A_F)^\vee\to(A(F)_p/p^m)^\vee\to(A(k)_p/p^m)^\vee.$$

As in \S \ref{bocksection}, these canonical choices induce a canonical Bockstein homomorphism \begin{equation*}\label{thebdbock}\beta_{C_{{\rm BD},S}}:A^t(k)_p/p^m{\to}I_{R}/I_{R}^2\otimes_{Z} (A(k)_p/p^m)^\vee\end{equation*} and hence a canonical pairing
$$\langle\,,\rangle_{C_{{\rm BD},S}}:A^t(k)_p/p^m\otimes_{Z} A(k)_p/p^m\to I_{R}/I_{R}^2
\cong Z\otimes_{\ZZ}G= \ZZ_p\otimes_{\ZZ}G,$$
or equivalently a canonical pairing
$$\langle\,,\rangle_p^{{\rm BD}}:A^t(k)_p\otimes_{\ZZ_p} A(k)_p\to  \ZZ_p\otimes_{\ZZ}G.$$

In particular, Lemma \ref{bdlemma}(iii) combines with the naturality of Bockstein homomorphisms to ensure that this pairing is independent of the choice of set $S$ containing $T$. 

\subsubsection{}
 In order to compare $\langle\,,\rangle_p^{\rm BD}$ to the Nekov\'a\v r height pairing we shall need to clarify the relationship between complexes of the form $C_{{\rm BD},S}$ and classical Selmer complexes.

To do this we fix a finite subset $\Sigma$ of $S_k^f$ with
$$S_k^p\cup (S_k^F \cap S_k^f)\cup S_k^A\cup T\subseteq\Sigma.$$

We then use the classical Selmer complex $C_\Sigma$ to define a complex of $R$-modules $$C_{\Sigma}/p^m:=Z\otimes_{\ZZ_p}^{\mathbb{L}}C_{\Sigma}.$$ We refer the reader to Lemma \ref{reducedbock} in Appendix \ref{bocksection} below for the definition of the maps $q_{1}'$ and $q_{2}'$ that are associated, by using the fixed homomorphisms $r_1$ and $r_2$ displayed in (\ref{r1}) and (\ref{r2}) respectively, to the construction of this complex.

\begin{lemma}\label{morphismphi} There exists a canonical morphism of complexes
\[ \phi:C_{{\rm BD},\Sigma}\to C_{\Sigma}/p^m\]
with the following properties:
\begin{itemize}\item[(i)] The restriction of $q'_{1}$ to $A^t(k)_p/p^m$ is equal to the composition $$A^t(k)_p/p^m\stackrel{r_{1,m}}{\to}H^1(C_{{\rm BD},\Sigma})\stackrel{H^1(\phi)}{\to}H^1(C_{\Sigma}/p^m).$$
\item[(ii)] $r_{2,m}$ is equal to $-1$ times the composition $$H^2(C_{{\rm BD},\Sigma})\stackrel{H^2(\phi)}{\to}H^2(C_{\Sigma}/p^m)\stackrel{q'_{2}}{\to}A(k)_p^*/p^m\to(A(k)_p/p^m)^\vee,$$ where the last arrow is the canonical map.
\end{itemize}
\end{lemma}
\begin{proof} In order to explicitly define the morphism $\phi$ we first describe an explicit representative of the classical Selmer complex $C_\Sigma$. The claimed result is clearly independent of such a choice.

To do so we use the notation given in \S \ref{expcohcomp} below. In particular, for any $w'\in\Sigma(F)$ and any $i\geq 0$ we let $R\Gamma^i(F_{w'},T_p(A^t))$ denote the group of inhomogeneous $i$-cochains of $G_{w'}$ with coefficients in $T_p(A^t)$, and also write $d_{w'}^i$ for the usual differential.

For any $v\in\Sigma$ we then set
$$C_v^2:=\frac{\prod_{w'\in S_F^v}\bigl(R\Gamma^{1}(F_{w'},T_p(A^t))/\im(d^0_{w'})\bigr)}{\kappa\bigl(A^t(F_v)_p^\wedge\bigr)}$$
and also, for any $i\geq 3$,
$$C_v^i:=\prod_{w'\in S_F^v}R\Gamma^{i-1}(F_{w'},T_p(A^t)).$$ We note that it is straightforward, by using the relevant exact sequences of the form (\ref{kummerseq}), to verify that the module $C_v^2$ is $\ZZ_p$-torsion-free.

If, for each $v\in\Sigma$, we set $D_v^i=C_v^{i-1}$, we let $\partial_v^{i}$ be induced by $\prod_{w'\in S_F^v}d_{w'}^{i-1}$ and consider the complex $D^\bullet_v=(D_v^{i},\partial_v^{i})$, then there is a canonical exact triangle
$$A^t(F_v)_p^\wedge[-1]\stackrel{\kappa}{\to}R\Gamma(k_v,T_{p,F}(A^t))\to D^\bullet_v\to A^t(F_v)_p^\wedge[0]$$ in $D(\ZZ_p[G])$.

Similarly, we write $G_{F,\Sigma}$ for the Galois group of the maximal Galois extension of $F$ wich is unramified outside $\Sigma(F)\cup S_F^\infty$, set $U_{F,\Sigma}:=\mathcal{O}_{F,\Sigma\cup S_F^\infty}$, write $R\Gamma^i(U_{F,\Sigma},T_p(A^t))$ for the group of inhomogeneous $i$-cochains of $G_{F,\Sigma}$ with coefficients in $T_p(A^t)$, and also write $d^i$ for the usual differential.


Then, since $R\Gamma^i(U_{F,\Sigma},T_p(A^t))$ is acyclic in degrees less that one, we may obtain an explicit representative of $C_\Sigma$ as the $-1$-shift of the explicit mapping cone of the diagonal localisation morphism
$$\tau_{\geq 1}\bigl(R\Gamma^i(U_{F,\Sigma},T_p(A^t))\bigr)\stackrel{\lambda}{\to}\bigoplus_{v\in\Sigma}D_v^\bullet.$$
(Here $\tau_{\geq 1}$ denotes the truncation in degree greater than or equal to 1 preserving cohomology).

To be explicit about this construction of a representative of $C_\Sigma$, we set
\[
C_{\Sigma}^i :=
\begin{cases}
  0,    &\text{ if $i\leq 0$}\\
  R\Gamma^1\bigl(U_{F,\Sigma},T_{p}(A^t)\bigr)/{\rm im}(d^0),    &\text{ if $i=1$}\\
\bigl(\bigoplus_{v\in \Sigma}C^i_v\bigr)\oplus R\Gamma^i\bigl(U_{F,\Sigma},T_{p}(A^t)\bigr),    &\text{ if $i\geq 2$.}
\end{cases}
\]
We let the differential $d_{\Sigma}^1$ in degree $1$ be given by
\[d_{\Sigma}^1(\overline{c}):=\bigl(-(\overline{c_{w'}})_{w'\in \Sigma(F)},d^1(c)\bigr)\]
and the differential $d_{\Sigma}^i$ in degree $i\geq 2$ be given by
\[d_{\Sigma}^i(\bigl((a_{w'})_{w'\in \Sigma(F)},b\bigr)):=\bigl(-(d^{i-1}_{w'}(a_{w'})+b_{w'})_{w'\in \Sigma(F)},d^i(b)\bigr).\]
Here, in any degree $i$ and for any place $w'\in \Sigma(F)$, we have written $x_{w'}$ for the localisation at $w'$
of an inhomogeneous $i$-cochain $x$ in $R\Gamma^i\bigl(U_{F,\Sigma},T_{p}(A^t)\bigr)$.

The exactness of the sequence (\ref{longexact}) combines with our assumption that $A^t(F)$ contains no point of order $p$ to imply that the module $H^1\bigl(U_{F,\Sigma},T_{p}(A^t)\bigr)$ is $\ZZ_p$-torsion-free. One may deduce from this fact that the module $C_\Sigma^1$ is also $\ZZ_p$-torsion-free.

Now, since each module $C_{\Sigma}^i$ is $\ZZ_p$-torsion-free, the complex $C_{\Sigma}/p^m$ may be represented by the sequence of modules $Z\otimes_{\ZZ_p}C_{\Sigma}^i$ together with the differentials $Z\otimes_{\ZZ_p}d_{\Sigma}^i$.

In order to define the required morphism $\phi$, we first define $$\phi^1:\Sel^{\Sigma}_{p^m}(A^t/F)\to Z\otimes_{\ZZ_p}C_{\Sigma}^1$$ as follows.

We first note that, since the set $\Sigma$ is assumed to contain $S_k^p\cup S_k^A$, the argument of \cite[Prop. 6.5]{milne} proves that the canonical inflation map defines an isomorphism $$\iota:H^1(U_{F,\Sigma},A^t[p^m])\stackrel{\sim}{\to}\Sel^{\Sigma}_{p^m}(A^t/F).$$

For any $x\in\Sel^{\Sigma}_{p^m}(A^t/F)$ we may hence fix a 1-cocycle $y\in R\Gamma^1(U_{F,\Sigma},A^t[p^m])$ representing $\iota^{-1}(x)$ and use the map $l_\Sigma$ defined in Lemma \ref{lifts}(i) below to define $\phi^1(x)$ as the class in $Z\otimes_{\ZZ_p}C_{\Sigma}^1$ of the lift $$l_\Sigma(y)\in Z\otimes_{\ZZ_p}R\Gamma^1(U_{F,\Sigma},T_p(A^t)).$$

Before proceeding we note that Lemma \ref{lifts}(ii) ensures both that $\phi^1$ is well-defined and also that
\begin{equation}\label{itisacocycle}(Z\otimes_{\ZZ_p}d^1)(l_\Sigma(y))=0.
\end{equation}

We next define $$\phi^2:H^1_{\Sigma}\bigl(\mathbb{A}_F,A^t\bigr)[p^m]\to Z\otimes_{\ZZ_p}C_{\Sigma}^2$$ as follows.

For any $(x_{w'})_{w'\in\Sigma(F)}$ in
\[ H^1_{\Sigma}\bigl(\mathbb{A}_F,A^t\bigr)[p^m]=\frac{H^1_{\Sigma}\bigl(\mathbb{A}_F,A^t[p^m]\bigr)}
{\kappa\bigl(A^t_\Sigma(\mathbb{A}_F)/p^m\bigr)}\]
we fix a family of 1-cocycles
$$(y_{w'})_{w'\in\Sigma(F)}\in\prod_{w'\in\Sigma(F)}R\Gamma^1(F_{w'},A^t[p^m])$$ representing $(x_{w'})_{w'\in\Sigma(F)}$.
We then use the maps $l_{w'}$, as $w'$ ranges through $\Sigma(F)$, defined in Lemma \ref{lifts}(i) below to obtain a lift
$$l_{w'}(y_{w'})\in Z\otimes_{\ZZ_p}R\Gamma^1(F_{w'},T_p(A^t)).$$

We finally define $\phi^2((x_{w'})_{w'\in\Sigma(F)})$ as the class in $Z\otimes_{\ZZ_p}C_{\Sigma}^2$ of $(-(l_{w'}(y_{w'}))_{w'\in\Sigma(F)},0)$. Again one can use Lemma \ref{lifts}(ii) to verify that $\phi^2$ is well-defined.
%
%

Our explicit description of the differential $d_{\Sigma}^1$ then implies that, for any $x\in\Sel^{\Sigma}_{p^m}(A^t/F)$, one has
\begin{align*} \bigl((Z\otimes_{\ZZ_p}d_{\Sigma}^1)\circ\phi^1\bigr)(x)=&\bigl(-(\overline{l_\Sigma(y)_{w'}})_{w'},(Z\otimes_{\ZZ_p}d^1)(l_\Sigma(y))\bigr)\\
=&\bigl(-(\overline{l_\Sigma(y)_{w'}})_{w'},0\bigr)\\
=&\bigl(-(\overline{l_{w'}(y_{w'})})_{w'},0\bigr)\\
=&\phi^2((\overline{y_{w'}})_{w'})\\
=&\phi^2(\lambda_\Sigma(\iota(\overline{y})))\\
=&(\phi^2\circ\lambda_\Sigma)(x).\end{align*}
Here again $y$ is a 1-cocycle representing $\iota^{-1}(x)$, the second equality follows from (\ref{itisacocycle}) and the third equality holds since, by Lemma \ref{lifts}(ii), the maps $l_\Sigma$ and $l_{w'}$ commute with localisation at each $w'$.

Putting then $\phi^i=0$ for $i\neq 1,2$ therefore defines a morphism of complexes of $R$-modules $$\phi:C_{{\rm BD},\Sigma}\to C_{\Sigma}/p^m.$$ Indeed, using Lemma \ref{lifts}(ii) one verifies that
\[ \bigl((Z\otimes_{\ZZ_p}d_{\Sigma}^2)\circ\phi^2\bigr)((x_{w'})_{w'})=0\] for any $(x_{w'})_{w'}$ in $H^1_{\Sigma}\bigl(\mathbb{A}_F,A^t\bigr)[p^m]$.

We next prove that $\phi$ satisfies the stated properties. Now, $q'_{1}$ is by definition equal to the composition
$$A^t(F)_p/p^m\stackrel{r_1/p^m}{\to}H^1(C_\Sigma)/p^m\stackrel{q_1}{\to}H^1(C_\Sigma/p^m)$$
with $q_{1}$ induced by the canonical projection $$R\Gamma^1\bigl(U_{F,\Sigma},T_p(A^t)\bigr)\to Z\otimes_{\ZZ_p}R\Gamma^1\bigl(U_{F,\Sigma},T_p(A^t)\bigr).$$ The defining property of the lifting map $l_\Sigma$ given in Lemma \ref{lifts}, combined with our definition of $\phi^1$, therefore implies that $q'_1$ coincides with the canonical composition
$$A^t(F)_p/p^m\to\Sel^{(p^m)}(A^t_F)\to H^1(C_{{\rm BD},\Sigma})\stackrel{H^1(\phi)}{\to}H^1(C_\Sigma/p^m).$$
This fact directly gives property (i) of the morphism $\phi$.

Turning now to property (ii), we write $q_{2}$ for the canonical isomorphism $$H^2(C_\Sigma)/p^m\to H^2(C_\Sigma/p^m)$$ induced by the canonical projection morphism $C_\Sigma\to C_\Sigma/p^m$ as in Lemma \ref{reducedbock} below.

Then an explicit computation using our definition of $\phi^2$ shows that the composition
\begin{multline*}H^1_\Sigma(\mathbb{A}_F,A^t[p^m])\to H^1_\Sigma(\mathbb{A}_F,A^t)[p^m]\to H^2(C_{{\rm BD},\Sigma}) \stackrel{H^2(\phi)}{\to}H^2(C_\Sigma/p^m)\stackrel{q_{2}^{-1}}{\to}H^2(C_{\Sigma})/p^m\\ \stackrel{\sim}{\to}\cok(\Delta)/p^m\to\frac{H^2_c\bigl(U_{F,\Sigma},A^t[p^m]\bigr)}{\im\bigl(\kappa(A^t_\Sigma(\mathbb{A}_F)/p^m)\bigr)}\end{multline*}
coincides with $-1$ times the canonical map. In this display, the fifth arrow is the canonical isomorphism occurring in the diagram (\ref{Selmerdiagram}) and the sixth arrow is the canonical projection.

The commutativity of the diagram (\ref{Selmerdiagram}), as well as of the diagram in Corollary \ref{Tatepoitouexplicit} below, can now be combined with the above given fact to show that the composition
\begin{equation*} H^2(C_{{\rm BD},\Sigma}) \stackrel{H^2(\phi)}{\to}H^2(C_{\Sigma}/p^m)\stackrel{q_{2}^{-1}}{\to}H^2(C_{\Sigma})/p^m\stackrel{\sim}{\to}\Sel_p(A_F)^\vee/p^m\to\Sel^{(p^m)}(A_F)^\vee\end{equation*} is equal to $-v_{\Sigma}$.

Given this fact, and since $q'_{2}$ is by definition equal to $r_2/p^m\circ q_{2}^{-1}$, the required equality of composite maps is now clear from their definitions.
\end{proof}

\subsection{The proof of Theorem \ref{thecomptheorem}}\label{compproof}

As a first step we use Lemma \ref{morphismphi} to show that the Nekov\'a\v r height pairing $\langle\,,\rangle^{\rm N}_{p}$ is equal to the inverse of the pairing $\langle\,,\rangle_p^{\rm BD}$.

To do this we recall that we have set $m:=m_p(G)$ and also fixed a finite set $T$ of places of $k$ as in Lemma \ref{tm}.

Then for any finite subset $\Sigma$ of $S_k^f$ that contains $$S_k^p\cup (S_k^F\cap S_k^f)\cup S_k^A\cup T,$$ any $x\in A^t(k)_p$ and $y\in A(k)_p$ one computes that
\begin{align}\label{explicit comp} \langle x,y\rangle^{\rm N}_{p}=\, &\langle x,y\rangle_{C_{\Sigma},r_1,r_2}\\ =\, &\langle x+p^m,y+p^m\rangle_{C_{\Sigma}/p^m,q'_{1},q'_{2}}\notag\\
= \, &\langle x+p^m,y+p^m\rangle_{C_{{\rm BD},\Sigma},r_{1,m},-r_{2,m}}\notag\\
=  &-\langle x+p^m,y+p^m\rangle_{C_{{\rm BD},\Sigma},r_{1,m},r_{2,m}}\notag\\
=  &-\langle x,y\rangle_p^{\rm BD}.\notag\end{align}

Here the second equality is given by Lemma \ref{reducedbock} below and the third equality, after fixing a morphism $\phi:C_{{\rm BD},\Sigma}\to C_{\Sigma}/p^m$ as in Lemma \ref{morphismphi}, is immediate from the naturality of Bockstein homomorphisms and the description of $H^1(\phi)$ and $H^2(\phi)$ given in that result.

Bertolini and Darmon have defined in \cite[\S 3.4.1]{bert} and \cite[\S 2.2]{bert2}
a pairing \[\langle\,,\rangle_1:A^t(k)\times A(k)\to  G.\] (Although the definition of this pairing is only given in the case that $A$ is an elliptic curve and $F/k$ is an abelian $p$-extension, it extends naturally to this more general setting).

In addition, the results of Bertolini and Darmon in \cite[Thm. 2.8, Rem. 2.10]{bert2} and of Tan in \cite[Prop. 3.1]{kst} combine to imply directly that the pairing obtained by tensoring $\langle\,,\rangle_1$ with $\ZZ_p$ coincides with the Mazur-Tate pairing $\langle\,,\rangle_p^{\rm MT}$.

Given the latter fact, and the equality of (\ref{explicit comp}), the proof of Theorem \ref{thecomptheorem} is therefore completed by means of the following result. 

\begin{lemma}\label{bdcomp}  The pairing obtained by tensoring $\langle\,,\rangle_{1}$ with $\ZZ_p$ coincides with $\langle\,,\rangle^{\rm BD}_p$.
\end{lemma}
\begin{proof}
As necessary preparation for the proof of this result, we fix a minimal set of generators $\gamma_1,\ldots,\gamma_r,\gamma_{r+1},\ldots,\gamma_s$ of $G$, ordered so that for $1\leq i\leq r$ the order $o(\gamma_i)$ of $\gamma_i$ is a power of $p$ and for $r<i\leq s$ the order $o(\gamma_i)$ of $\gamma_i$ is prime to $p$.

We also denote by $J$ the cartesian product of the sets $\{0,1,\ldots,o(\gamma_i)-1\}$ as $i$ ranges through all indices between 1 and $s$.

For each index $1\leq i\leq s$ we write $F_i$ for the fixed field in $F$ of the subgroup of $G$ generated by $\{\gamma_j:j\neq i\}$.

For each index $1\leq i\leq r$ we let $$D_i^{(1)}:=-\sum_{j=0}^{o(\gamma_i)-1}j\gamma_i^j\in R$$ be the `derivative operator' of \cite[Lem. 2.1]{bert2}.

To proceed with the proof, it is enough to verify that, for any points $P\in A^t(k)$ and $Q\in A(k)$, the image of $\langle P,Q\rangle_1$ in $\ZZ_p\otimes_{\ZZ}G\cong I_R/I_R^2$ is equal to $\langle P+p^m,Q+p^m\rangle_{C_{{\rm BD},T}}$, since the latter element is by definition equal to $\langle P,Q\rangle_p^{{\rm BD}}$.

To compute the term $\langle P,Q\rangle_1$, we use the fact that the $G$-modules $\Sel^T_{p^m}(A^t/F)$ and $A_T(\mathbb{A}_F)/p^m$ are cohomologically-trivial established in Lemma \ref{bdlemma}(ii).

Let $a\in \Sel^{T}_{p^m}(A^t/F)$ be chosen so that $\Tr_G(a)$ is equal to the image of $P+p^m$ under the composite map
$$A^t(k)_p/p^m\to\bigl(A^t(F)_p/p^m\bigr)^G\to\Sel^{T}_{p^m}(A^t/F)^G.$$

Let $b\in A_{T}(\mathbb{A}_F)/p^m$ be chosen so that $\Tr_G(b)$ is equal to the image $Q_{\rm loc}$ of $Q+p^m$ under the composite map
$$A(k)_p/p^m\to\bigl(A(F)_p/p^m\bigr)^G\to\bigl(A_{T}(\mathbb{A}_F)/p^m\bigr)^G.$$

After interpreting the local duality pairing $\langle\,,\rangle_{T}$ of (\ref{localpairing}) below as taking values in $Z$ in the canonical way as is done throughout \cite{bert2}, the image of $\langle P,Q\rangle_1$ in $I_{R}/I_{R}^2$ is by definition equal to
\begin{align}\label{bdcomput}\notag &\bigl(\sum_{g\in G}\langle g^{-1}(\lambda_{T}(a)),b\rangle_{T}\cdot g\bigr)+I_{R}^2\\
\notag =&\bigl(\sum_{j_\bullet\in J}\langle(\gamma_1^{-j_1}\cdot\ldots\cdot\gamma_s^{-j_s})(\lambda_T(a)),b\rangle_T\cdot(\gamma_1^{j_1}\cdot\ldots\cdot\gamma_s^{j_s})\bigr)+I_{R}^2\\
\notag =&\bigl(\sum_{i=1}^{i=r}\sum_{j_\bullet\in J}\langle(\gamma_1^{-j_1}\cdot\ldots\cdot\gamma_s^{-j_s})(\lambda_T(a)),b\rangle_T\cdot j_i(\gamma_i-1)\bigr)+I_{R}^2\\
=&\bigl(\sum_{i=1}^{i=r}\langle(D_i^{(1)}\cdot\Tr_{G_{F/F_i}})(\lambda_T(a)),b\rangle_T\cdot(\gamma_i-1)\bigr)+I_{R}^2.
\end{align}

We now let, as we may, $a'_1,\ldots,a'_s\in H^1_{T}\bigl(\mathbb{A}_F,A^t\bigr)[p^m]$ be chosen so that
\begin{equation}\label{augmentationdivision}\sum_{l=1}^{l=s}(\gamma_l-1)\cdot a_l'=\lambda_{T}(a).\end{equation}

Then (\ref{bdcomput}) implies that the image of $\langle P,Q\rangle_1$ in $I_{R}/I_{R}^2$ is equal to

\begin{align*}&\bigl(\sum_{i=1}^{i=r}\langle(D_i^{(1)}\cdot\Tr_{G_{F/F_i}})\bigl(\sum_{l=1}^{l=s}(\gamma_l-1)\cdot a_l'\bigr),b\rangle_T\cdot(\gamma_i-1)\bigr)+I_{R}^2\\
=&\bigl(\sum_{i=1}^{i=r}\langle(D_i^{(1)}\cdot\Tr_{G_{F/F_i}}\cdot(\gamma_i-1)) (a_i'),b\rangle_T\cdot(\gamma_i-1)\bigr)+I_{R}^2\\
=&\bigl(\sum_{i=1}^{i=r}\langle \Tr_G(a_i'),b\rangle_T\cdot(\gamma_i-1)\bigr)+I_{R}^2\\
=&\bigl(\sum_{i=1}^{i=r}\langle a_i',\Tr_G(b)\rangle_T\cdot(\gamma_i-1)\bigr)+I_{R}^2\\
=&\bigl(\sum_{i=1}^{i=r}\langle a_i',Q_{\rm loc}\rangle_T\cdot(\gamma_i-1)\bigr)+I_{R}^2\\
=&\bigl(\sum_{i=1}^{i=r}Q_{\rm loc}(u_T(a_i'))\cdot(\gamma_i-1)\bigr)+I_{R}^2.
\end{align*}
In this last expression we have identified $Q_{\rm loc}$ with its double-dual.

The claimed result is thus equivalent to the equality $$\langle P+p^m,Q+p^m\rangle_{C_{{\rm BD},T}}=\bigl(\sum_{i=1}^{i=r}Q_{\rm loc}(u_T(a_i'))\cdot(\gamma_i-1)\bigr)+I_{R}^2$$ in $I_{R}/I_{R}^2$. But, taking into account the equality (\ref{augmentationdivision}), and given that the $R$-modules $\Sel^T_{p^m}(A^t/F)$ and $H^1_{T}\bigl(\mathbb{A}_F,A^t\bigr)[p^m]$ are finitely generated and $G$-cohomologically-trivial by Lemma \ref{bdlemma}(ii),
this equality follows directly from a diagram chase to compute the Bockstein homomorphism $\beta_{C_{{\rm BD},T}}$.
\end{proof}

\section{Modular symbols}\label{mod sect} In this section we shall use the theory of modular symbols to study ${\rm BSD}(A_{F/k})$ in the case $A$ is an elliptic curve and $F$ is a totally real tamely ramified abelian extension of $k = \QQ$. The results in this section will, in particular, extend the main results of Bley in \cite{Bley3}.

For each natural number $d$ we set $\zeta_d := {\rm exp}(2\pi i/d)$. We write $\QQ(d)$ for $\QQ(\zeta_d)$ and $\QQ(d)^+$ for the
maximal real subfield of $\QQ(d)$ and set $G_d := G_{\QQ(d)/\QQ}$ and $G^+_d := G_{\QQ(d)^+/\QQ}$.

For any element $a$ of $(\ZZ/d\ZZ)^\times$ we write $\sigma_a$ for the element of $G_d$ that sends $\zeta_d$ to $\zeta_d^a$.

\subsection{Regularized Mazur-Tate elements} We first recall the definition and relevant properties of modular symbols as introduced by Mazur and Tate in \cite{mt} and normalized by Darmon in \cite{darmon} and explain their connection to the leading term element that occurs in Theorem \ref{bk explicit}.

\subsubsection{}\label{ell cirve section}We fix an elliptic curve $A$ that is defined over $\QQ$ and has conductor $N$. We also fix a modular parametrisation $\varphi_A: X_0(N) \to A$ and write $c(\varphi_A)$ for the associated Manin constant.

We write $\Lambda$ for a N\'eron lattice for $A$ in $\CC$ and then identify $A(\CC)$ with $\CC/\Lambda$ and the N\'eron differential $\omega_0$ of $A$ with ${\rm d}z$.

We define strictly positive real numbers $\Omega^+$ and $\Omega^-$ by the condition that $\Lambda$ is a sublattice of $\ZZ\cdot\Omega^+ + \ZZ\cdot i\Omega^-$ with minimal index.

Then the pullback under $\varphi_A$ of $\omega_0$ defines a cusp form of weight $2$ on $X_0(N)$ so
\[ \varphi_A^*(\omega_0) = c(\varphi_A)f(q){\rm d}q/q = c(\varphi_A)2\pi if (z){\rm d}z\]
where $f(q) = \sum_{n \ge 1}a_nq^n$, with $q = {\rm exp}(2\pi iz)$, is normalized such that $a_1 = 1$.

We now fix a natural number $c$ that is square-free and prime to $N$. For each integer $a$ we follow Mazur and Tate \cite{mt} in defining the modular symbols $[ \frac{a}{c}]^+$ and $[ \frac{a}{c}]^{-}$ via the equality
\[ 2\pi i\int_{a/c}^{a/c + i\infty}f(z){\rm d}z = \left[ \frac{a}{c}\right]^+\Omega^+ + i\left[ \frac{a}{c}\right]^{-}\Omega^-.\]
%
%
Recalling that $[\frac{a}{c}]^+$ depends only on the residue class of $a$ modulo $t$, we then follow Darmon \cite{darmon} in defining the regularized modular
symbols

\[ \left[ \frac{a}{c}\right]^* := \sum_{t\mid c}\mu(c/t)\left[\frac{a(c/t)^{-1}}{t}\right]^+.\]
Here $t$ runs over all natural numbers that divide $c$, $\mu(-)$ is the M\"obius function and $(c/t)^{-1}$ denotes the inverse of $c/t$ modulo $t$ (which exists since $c$ is assumed to be squarefree).

We finally define the `regularized Mazur-Tate element of $A$ at level $c$' to be the element
\[ \Theta^{\rm MT}_{A,c} := \frac{1}{2}\sum_{a\in (\ZZ/c\ZZ)^\times}\left[\frac{a}{c}\right]^*\sigma_a\]
of $\QQ[G_c]$. In the sequel we identify $\Theta^{\rm MT}_{A,c}$ with its image under the canonical projection $\QQ[G_c]\to\QQ[G_c^+]$.

\subsubsection{} In the next result we describe the precise link between the Mazur-Tate element $\Theta^{\rm MT}_{A,c}$ and the leading term element $\calL^*_{A,\QQ(c)^+/\QQ}$ that occurs in Theorem \ref{bk explicit} (as defined in (\ref{bkcharelement}) with respect to a fixed odd prime $p$).

To state the result we write $\Upsilon_{c,0}$ for the subset of $\widehat{G_c^+}$ comprising all characters $\psi$
for which $L(A,\check\psi,1)$ does not vanish and define an idempotent of $\CC[G_c^+]$ by setting
\[ e_{c,0} := \sum_{\psi\in \Upsilon_{c,0}} e_\psi.\]
(It is then in fact straightforward to check that $e_{c,0}$ belongs to $\QQ[G_c^+]$.)



In the sequel, we denote by $c_\psi$ the conductor of a Dirichlet character $\psi$.

\begin{lemma}\label{gammatheta} For any odd prime number $p$, in $\QQ[G_c^+]$ one has
\[ \Theta^{\rm MT}_{A,c}=  (c(\varphi_A)/c_\infty)\cdot \bigl(\sum_{\psi\in\widehat{G_c^+}}(c/c_\psi)m_p^\psi \cdot e_\psi\big)\cdot e_{c,0}\cdot\calL^*_{A,\QQ(c)^+/\QQ}\]
where $c_\infty$ is the number of connected components of $A(\RR)$ and we set
\[ m_p^\psi := \begin{cases} p^{-1}, &\text{ if $p$ divides $c$ but not $c_\psi$},\\
                             1, &\text{otherwise.}\end{cases}\]
\end{lemma}

\begin{proof} The key point is that, by Darmon \cite[Prop. 2.3]{darmon}, for every $\psi\in\widehat{G_c^+}$ one has \begin{equation}\label{chiMT}(\Theta^{\rm MT}_{A,c})_\psi=c(\varphi_A)\cdot\frac{c}{c_\psi}\cdot\frac{L_c(A,\check{\psi},1)\tau_c(\psi)}{2\Omega^+}.\end{equation}
Here the $L$-function $L_c(A,\check{\psi},z)$ is truncated at all rational primes dividing $c$ and we use the (in general, imprimitive) Gauss sum
\[ \tau_c(\psi):= \sum_{a\in (\ZZ/c\ZZ)^\times}\psi(a)\cdot \zeta^a_c.\]

Since the definition of $e_{c,0}$ implies that $L_c(A,\check{\psi},1)\not= 0$ if and only if $e_\psi e_{c,0}\not= 0$, the equalities (\ref{chiMT}) combine to imply that $\Theta^{\rm MT}_{A,c} = e_{c,0}\cdot\Theta^{\rm MT}_{A,c}$.

In addition, since each character $\psi$ in $\widehat{G_c^+}$ is even one has both $w_{\infty,\psi}=1$ and
\[ 2\Omega^+ = c_\infty\cdot\Omega_{A,\infty}^+ = c_\infty\cdot \Omega_{A,\infty}^\psi\]
whilst, if $p$ divides $c$, an explicit computation of each factor $\rho_{p,\psi}$ occurring in (\ref{bkcharelement}) shows that it is equal to $(m_p^\psi)^{-1}$.

The claimed equality now follows directly upon substituting these facts and the result of Lemma \ref{imprimitive GGS} below into the equality of (\ref{chiMT}), recalling the explicit definition (\ref{bkcharelement}) of the element $\mathcal{L}^*_{A,\QQ(c)^+/\QQ}$ and noting that, in terms of the notation of \S\ref{somr tmc sec}, one has $L_c(A,\check{\psi},1) = L_{S_{\rm r}}(A,\check{\psi},1)$.
%
%
\end{proof}

\begin{lemma}\label{imprimitive GGS} For each $\psi$ in $\widehat{G_c^+}$ one has $\tau_c(\psi) = \tau^\ast(\QQ,\psi)$.\end{lemma}

\begin{proof} Since $\psi$ is linear, the definition of the Galois-Gauss sum $\tau(\QQ,\psi)$ implies that it is equal to the primitive Gauss sum $\tau_{c_\psi}(\psi)$.

Upon recalling the explicit definition of $\tau^\ast(\QQ,\psi)$ from \S\ref{mod GGS section}, it is thus enough to show that
\[ \tau_c(\psi) = \tau_{c_\psi}(\psi)\cdot \prod_{\ell \mid (c/c_\psi)}(-\psi(\sigma_\ell)^{-1})\]
where $\ell$ runs over all prime divisors of $c/c_\psi$.

To prove this it is in turn enough to show that for each prime divisor $\ell$ of $c/c_\psi$ one has an equality $\tau_c(\psi) = -\psi(\sigma_\ell)^{-1}\cdot \tau_{c'}(\psi)$ with $c' := c/\ell$.

To show this we  fix integers $s$ and $t$ such that $s\ell + tc' = 1$. Then
$\zeta_c = \zeta_\ell^t\cdot \zeta_{c'}^s$ and hence

\begin{align*} \tau_c(\psi) = \, &\sum_{a\in (\ZZ/c\ZZ)^\times}\psi(a)\cdot \zeta_c^a\\
= \, &\sum_{a_1\in (\ZZ/\ell\ZZ)^\times}\sum_{a_2\in (\ZZ/c'\ZZ)^\times}\psi(a_1a_2)\cdot \zeta_\ell^{ta_1}\zeta_{c'}^{sa_2}\\
= \, &\sum_{a_2\in (\ZZ/c'\ZZ)^\times}\left(\sum_{a_1\in (\ZZ/\ell\ZZ)^\times}\zeta_\ell^{ta_1}\right)\psi(a_2)\cdot\zeta_{c'}^{sa_2}\\
= \, &-\sum_{a_2\in (\ZZ/c'\ZZ)^\times}\psi(a_2)\cdot\zeta_{c'}^{sa_2}\\
= \, &-\psi(s)^{-1}\sum_{a_2\in (\ZZ/c'\ZZ)^\times}\psi(a_2)\cdot\zeta_{c'}^{a_2}\\
= \, &-\psi(\sigma_\ell)^{-1}\cdot\tau_{c'}(\psi).
\end{align*}
The third equality here follows from the fact that $\psi(a_1) = 1$ for all $a_1$ in $(\ZZ/\ell\ZZ)^\times$, the fourth from the equality $\sum_{a_1\in (\ZZ/\ell\ZZ)^\times}\zeta_\ell^{ta_1} = -1$ and the last from the fact that $\sigma_\ell$ acts as multiplication by $s$ on $(\ZZ/c'\ZZ)^\times$.
\end{proof}

\begin{remark}\label{euler factor rem}{\em For any abelian extension $F$ of $\QQ$ and any subfield $F'$ of $F$ we write $\pi_{F/F'}$ for the natural projection map $\QQ[G_{F/\QQ}] \to \QQ[G_{F'/\QQ}]$.

Then, if $p$ is any prime that does not divide $c$, the interpolation formula (\ref{chiMT}) implies that
\[\pi_{\QQ(pc)/\QQ(c)}(\Theta^{\rm MT}_{A,pc}) =- \sigma_p(p -\sigma_p^{-1}a_p + \sigma_p^{-2})\cdot \Theta_{A,c}^{\rm MT}\]
with $a_p=1+p-|A(\mathbb{F}_p)|$.
}\end{remark}

\subsection{Mazur-Tate elements and refined BSD}

\subsubsection{}We now fix a tamely ramified real abelian extension $F$ of $\QQ$ and set $G := G_{F/\QQ}$.

We also fix an odd prime $p$, write $F'$ for the maximal subfield of $F$ whose degree over $\QQ$ is prime to $p$ and set $G' := G_{F'/\QQ}$.

We fix an elliptic curve $A$ that is defined over $\QQ$ and whose conductor $N$ is prime to the conductor $c$ of $F$. We always assume that the Tate-Shafarevich group $\sha(A_{F'})$ of $A$ over $F'$ is finite.

Then, since the conductor $c$ of $F$ is square-free, we can define an element of $\QQ[G]$ by setting
\[ \Theta^{\rm MT}_{A,F} := \pi_{\QQ(c)/F}(\Theta^{\rm MT}_{A,c}).\]
We also write $e_{F,0}$ for the idempotent $\pi_{\QQ(c)^+/F}(e_{c,0})$ of $\QQ[G]$.

%
%

We recall that we have written ${\rm Sel}_p(A_F)$ for the $p$-primary Selmer group of $A$ over $F$ and that ${\rm Fit}^0_{\ZZ_p[G]}(M)$ is the (initial) Fitting ideal of a finitely generated $\ZZ_p[G]$-module $M$.

In this section we shall prove the following results.

\begin{theorem}\label{MT result} We assume that $A$, $F$ and $p$ are such that $(c,N)=1$ and also satisfy all of the following conditions.

\begin{itemize}
\item[(a)]   $p$ is prime to $|A(F')_{\rm tor}|$;
\item[(b)]  $p$ is prime to $N$ and to the Tamagawa number of $A$ at any prime divisor of $N$;
\item[(c)] If $p$ divides $c$, then $p$ is prime to $|A(\kappa_v)|$ for any $p$-adic place $v$ of $F'$ and, in addition, $A$ has ordinary reduction at $p$;
\item[(d)] $p$ is prime to $|A(\kappa_v)|$ for any place $v$ of $F'$ of residue characteristic dividing $c$.
 \item[(e)] $p$ is prime to the Manin constant $c(\varphi_A)$.
\end{itemize}

Then the following assertions are valid.
\begin{itemize}
\item[(i)] The projection to $K_0(\ZZ_p[G]e_{F,0},\CC_p[G]e_{F,0})$ of the equality in ${\rm BSD}_p(A_{F/\QQ})$(iv) is valid if and only if one has
\begin{equation*}\label{refMT}\ZZ_p[G]\cdot\Theta^{\rm MT}_{A,F} = e_{F,0}\cdot{\rm Fit}^0_{\ZZ_p[G]}({\rm Sel}_p(A_{F})^\vee).\end{equation*}


\item[(ii)] If $L(A_{F},1)$ is non-zero, then ${\rm BSD}_p(A_{F/\QQ})$(iv) is valid if and only if $\Theta^{\rm MT}_{A,F}$ belongs to ${\rm Fit}^0_{\ZZ_p[G]}(\sha(A_{F})[p^\infty])$ and, in addition, the $p$-part of the Birch and Swinnerton-Dyer conjecture is valid for $A$ over $F'$.
\end{itemize}
\end{theorem}

The following explicit corollary of Theorem \ref{MT result} extends the main results of Bley in \cite{Bley3} (for more details see Remark \ref{mod bley rem} below).

\begin{corollary}\label{bleycor} Fix an elliptic curve $A$ over $\QQ$ with $L(A_{\QQ},1)\not=0$. Then for each natural number $n$ there are infinitely many primes $p$, and for each such prime $p$ infinitely many abelian extensions $F/\QQ$ that satisfy all of the following properties.

\begin{itemize}
\item[$\bullet$] ${\rm BSD}_p(A_{F/\QQ})$(iv) is valid.
\item[$\bullet$] The exponent of $G_{F/\QQ}$ is divisible by $p^n$.
\item[$\bullet$] The $p$-rank of $G_{F/\QQ}$ is at least $n$.
\item[$\bullet$] The order of $G_{F/\QQ}$ is divisible by at least $n$ distinct primes.
\item[$\bullet$] The discriminant of $F$ is divisible by at least $n$ distinct primes.
\end{itemize}
\end{corollary}

\subsubsection{} We shall now prove Theorem \ref{MT result}.

At the outset we note that the given assumptions (a)-(e) imply that the hypotheses (H$_1$)-(H$_5$) listed in \S\ref{tmc} are satisfied and so Proposition \ref{explicitbkprop} implies that the complex %
\[ C_{F,0}:=\ZZ_p[G]e_{F,0}\otimes_{\ZZ_p[G]}^{\mathbb{L}}{\rm SC}_p(A_{F/\QQ})\]
belongs to $D^{\rm perf}(\ZZ_p[G]e_{F,0})$.

In addition, by the main result of Kato in \cite{kato}, one knows that $e_\psi(\CC_p\otimes_{\ZZ_p} A(F)_p)$ vanishes for every character $\psi$
 in $\Upsilon_{c,0}$ and hence that the complex $\QQ_p\otimes_{\ZZ_p}C_{F,0}$ is acyclic.

It follows that for every field isomorphism $j:\CC\cong\CC_p$, one has
\begin{equation*}\label{zeroEuler}\pi_{0,*}(\chi_{G,p}({\rm SC}_p(A_{F/\QQ}),h^j_{A,F}))=\chi_{\ZZ_p[G]e_{F,0}}(C_{F,0},0)\end{equation*}
in $K_0(\ZZ_p[G]e_{F,0},\CC_p[G]e_{F,0})$, where $$\pi_{0,\ast}:K_0(\ZZ_p[G],\CC_p[G])\to K_0(\ZZ_p[G]e_{F,0},\CC_p[G]e_{F,0})$$ denotes the canonical projection homomorphism.
%

Now, in Theorem \ref{bk explicit}, the sets $S_{p,{\rm w}}$ and $S_{p,{\rm u}}^\ast$ are empty (since $F$ is a tamely ramified extension of $k = \QQ$) and so the above equality implies that the image under $\pi_{0,\ast}$ of the equality in ${\rm BSD}_p(A_{F/\QQ})$(iv) is valid if and only if one has
\begin{equation}\label{first approx} \pi_{0,\ast}(\delta_{G,p}(\calL^*_{A,F/\QQ})) = \chi_{\ZZ_p[G]e_{F,0}}(C_{F,0},0).\end{equation}

To interpret this equality we identify $K_0(\ZZ_p[G]e_{F,0},\CC_p[G]e_{F,0})$ with the multiplicative group of invertible $\ZZ_p[G]e_{F,0}$-lattices in $\CC_p[G]e_{F,0}$ (as in Remark \ref{comparingdets}).

In particular, since the description of the cohomology of ${\rm SC}_p(A_{F/\QQ})$ given in (\ref{bksc cohom}) combines with hypothesis (a) and the acyclicity of $\QQ_p\otimes_{\ZZ_p}C_{F,0}$ to imply
\[ H^i(C_{F,0}) = \begin{cases} \ZZ_p[G]e_{F,0}\otimes_{\ZZ_p[G]}\Sel_p(A_{F})^\vee, &\text{if $i=2$,}\\
0, &\text{otherwise,}\end{cases}\]
it follows that the right hand side of (\ref{first approx}) identifies with the lattice
\[{\rm Fit}^0_{\ZZ_p[G]e_{F,0}}(\ZZ_p[G]e_{F,0}\otimes_{\ZZ_p[G]}\Sel_p(A_{F})^\vee) = e_{F,0}\cdot {\rm Fit}^0_{\ZZ_p[G]}(\Sel_p(A_{F})^\vee).\]

In addition, the left hand side of (\ref{first approx}) identifies with the $\ZZ_p[G]$-module generated by $e_{F,0}\calL^*_{A,F/\QQ}$.
 Thus, since both of the terms $c_\infty$ and $c(\varphi_A)$ belong to $\ZZ_p^\times$ (the former because $p$ is odd and the latter by hypothesis (e)), to derive claim (i) from the formula of Lemma \ref{gammatheta} it is enough to show that $$\pi_{\QQ(c)^+/F}\bigl(\sum_{\psi\in\widehat{G_c^+}}(c/c_\psi)m_p^\psi \cdot e_\psi\bigr)$$ belongs to $\ZZ_p[G]^\times$.

To verify this assertion it is enough to show that $\sum_{\psi\in\widehat{G_c^+}}(c/c_\psi)m_p^\psi \cdot e_\psi$ belongs to $\ZZ_p[G_c^+]^\times$ and this fact follows directly from Lemma \ref{bley lemma} below with $A=G_c^+$, $i=1$ and, for each positive divisor $d$ of $c$, the subgroup $H_d$ of $G_c^+$ taken to be $G_{\QQ(c)^+/\QQ(d)^+}$.

In order to prove claim (ii), we assume that $L(A_{F},1)\not= 0$. This implies, by Kato \cite{kato}, that the groups $\sha(A_F)$ and $A(F)$ are both finite and also that the idempotent $e_{F,0}$ is $1$. In particular, 
the module $\Sel_p(A_F)^\vee$ is canonically isomorphic to $\sha(A_F)[p^\infty]^\vee$ (see the exact sequence (\ref{sha-selmer})). The module $\Sel_p(A_F)^\vee$  is therefore also canonically isomorphic to $\sha(A_F)[p^\infty]$ via the Cassels-Tate pairing.

Claim (i) therefore implies that ${\rm BSD}_p(A_{F/\QQ})$(iv) is valid if and only if ${\rm Fit}^0_{\ZZ_p[G]}({\rm Sel}_p(A_{F})^\vee)$ $ = {\rm Fit}^0_{\ZZ_p[G]}(\sha(A_{F})[p^\infty])$ is generated over $\ZZ_p[G]$ by $\Theta^{\rm MT}_{A,F}$.

Now, by Nakayama's Lemma, the latter assertion is valid if and only if ${\rm Fit}^0_{\ZZ_p[G]}({\rm Sel}_p(A_{F})^\vee)$ contains $\Theta^{\rm MT}_{A,F}$
 and, in addition, the module of coinvariants
 \[ \ZZ_p[G']\otimes_{\ZZ_p[G]}\bigl({\rm Fit}^0_{\ZZ_p[G]}({\rm Sel}_p(A_{F})^\vee)/(\ZZ_p[G]\cdot \Theta^{\rm MT}_{A,F})\bigr)\]
vanishes, or equivalently
\begin{equation}\label{conv reinter} \pi_{F/F'}\bigl({\rm Fit}^0_{\ZZ_p[G]}({\rm Sel}_p(A_{F})^\vee)\bigr) = \ZZ_p[G']\cdot \pi_{F/F'}(\Theta^{\rm MT}_{A,F}).\end{equation}

We next note that, since ${\rm SC}_p(A_{F/\QQ})$ is acyclic outside degree two, the isomorphism
\[ \ZZ_p[G']\otimes^{\mathbb{L}}_{\ZZ_p[G]}{\rm SC}_p(A_{F/\QQ})\cong {\rm SC}_p(A_{F'/\QQ})\]
in $D(\ZZ_p[G'])$ coming from Proposition \ref{explicitbkprop}(iii) induces an isomorphism of $\ZZ_p[G']$-modules
\begin{multline}\label{comp iso} \ZZ_p[G']\otimes_{\ZZ_p[G]}{\rm Sel}_p(A_{F})^\vee \cong H^2(\ZZ_p[G']\otimes^{\mathbb{L}}_{\ZZ_p[G]}{\rm SC}_p(A_{F/\QQ})) \\ \cong H^2({\rm SC}_p(A_{F'/\QQ})) \cong {\rm Sel}_p(A_{F'})^\vee , \end{multline}
and hence an equality
\begin{equation}\label{proj1} \pi_{F/F'}\bigl({\rm Fit}^0_{\ZZ_p[G]}({\rm Sel}_p(A_{F})^\vee)\bigr) = {\rm Fit}^0_{\ZZ_p[G']}({\rm Sel}_p(A_{F'})^\vee).\end{equation}

In addition, if we write $c'$ for the conductor of $F'$, then the distribution relation in Remark \ref{euler factor rem} implies that
\begin{equation}\label{dist rel 2} \pi_{F/F'}(\Theta^{\rm MT}_{A,F}) = \Theta^{\rm MT}_{A,F'}\cdot \prod_{\ell \mid (c/c')}-\sigma_\ell(\ell -\sigma_\ell^{-1}a_\ell + \sigma_\ell^{-2}).\end{equation}

We claim next that for each prime factor $\ell$ of $c/c'$ the element $\ell -\sigma_\ell^{-1}a_\ell + \sigma_\ell^{-2}$ is a unit in $\ZZ_p[G']$. This is because any such $\ell$ cannot be equal to $p$ (since $F/F'$ is a $p$-extension and $F/\QQ$ is tamely ramified) so that $\ell -\sigma_\ell^{-1}a_\ell + \sigma_\ell^{-2}$ is a generator of the Fitting ideal of the $\ZZ_p[G']$-module $\bigoplus_{v\in S_{F'}^\ell} A(\kappa_v)$, whilst the given condition (d) implies that the latter direct sum contains no element of order $p$.

Given this, the equality (\ref{dist rel 2}) implies that $\pi_{F/F'}(\Theta^{\rm MT}_{A,F})$ differs from $\Theta^{\rm MT}_{A,F'}$ by a unit in $\ZZ_p[G']$ and thus combines with (\ref{proj1}) to imply that (\ref{conv reinter}) is valid if and only if the invertible $\ZZ_p[G']$-sublattice
\begin{equation}\label{last ideal} {\rm Fit}^0_{\ZZ_p[G']}({\rm Sel}_p(A_{F'})^\vee)\cdot (\ZZ_p[G']\cdot (\Theta^{\rm MT}_{A,F'}))^{-1}\end{equation}
of $\QQ_p[G']$ is trivial.

In addition, if ${\rm Fit}^0_{\ZZ_p[G]}(\sha(A_F)[p^\infty])={\rm Fit}^0_{\ZZ_p[G]}({\rm Sel}_p(A_{F})^\vee)$ contains $\Theta^{\rm MT}_{A,F}$ then the above arguments would also imply that ${\rm Fit}^0_{\ZZ_p[G']}({\rm Sel}_p(A_{F'})^\vee)$ contains $\Theta^{\rm MT}_{A,F'}$. Since $p\nmid |G|'$, the lattice (\ref{last ideal}) would then be trivial if and only if its image under the natural restriction map $\varrho_{G'}$ to the group of invertible $\ZZ_p$-sublattices of $\QQ_p$ is trivial.

In particular, since the argument of claim (i) shows that the triviality of (\ref{last ideal}) is equivalent to the validity of ${\rm BSD}_p(A_{F'/\QQ})$(iv), the observations in Remark \ref{consistency remark}(ii) and (iii) imply that the triviality of its image under $\varrho_{G'}$ is equivalent to the validity of the $p$-part of the Birch and Swinnerton-Dyer Conjecture for $A$ over $F'$, as required to complete the proof of Theorem \ref{MT result}.

The following algebraic result was used above and will also be useful in \S\ref{HHP}.

\begin{lemma}\label{bley lemma} 
Let $A$ be a finite abelian group, $c$ a square-free positive integer and $(c_\psi)_{\psi\in\widehat{A}}$ a family of positive divisors of $c$. We fix an odd prime number $p$, we let $n_\psi$ be equal to $p$ if $p$ divides $c$ but not $c_{\psi}$, and we let $n_{\psi}$ be equal to 1 otherwise.

Fix $i\in\{1,2\}$ and assume that, for each positive divisor $d$ of $c$, there exists a subgroup $H_d$ which satisfies both of the following properties:
\begin{itemize}\item[(i)] $d$ is divisible by an integer $c_{\psi}$ if and only if $H_d\subseteq\ker(\psi)$;
\item[(ii)] the order of $H_d$ divides $\prod_{\ell\mid\frac{c}{d}}(l+(-1)^i)$ (where $\ell$ runs over primes dividing $\frac{c}{d}$).
\end{itemize}

Then the element $\sum_{\psi \in \widehat{A}}((c_{\psi}/c) n_{\psi})^i\cdot e_{\psi}$ belongs to $\ZZ_p[A]^\times$.
\end{lemma}

\begin{proof} Since $((c_{\psi}/c)n_{\psi})^i$ is a $p$-adic unit for all $\psi$ it suffices to prove that the inverse $x_i := \sum_{\psi \in \widehat{A}}(c_{\psi}/c)^{-i} n^{-i}_{\psi}\cdot e_{\psi}$ of the given sum belongs to $\ZZ_p[A]$.

For each divisor $\delta$ of $c$ we set
\[ f_{p,i}(\delta) := \begin{cases} \bigl(c/\delta\bigr)^ip^{-i}, &\text{ if $p$ divides $c/\delta$,}\\
                          \bigl(c/\delta\bigr)^i, &\text{ otherwise.}\end{cases}\]

Then $x_i$ is equal to $\sum_{\psi \in \widehat{A}}f_{p,i}(c_{\psi})\cdot e_{\psi}$ and so, when combined with assumption (i), the approach of Bley in \cite[Prop. 3.1(a)]{Bley3} implies the required containment is true if for each positive divisor $d$ of $c$, the sum
\begin{equation}\label{tricky sum} {\sum}_{t \mid \frac{c}{d}} f_{p,i}\left(dt\right)\mu(t)=\sum_{\{t:p\nmid\frac{c/d}{t}\}}\left(\frac{c/d}{t}\right)^i\mu(t)+\sum_{\{t:p\mid\frac{c/d}{t}\}}\left(\frac{c/d}{t}\right)^ip^{-i}\mu(t)\end{equation}
is divisible by $|H_d|$, where $\mu(-)$ denotes the M\"obius function on square-free natural numbers.

If $d$ is such that $p$ does not divide $c/d$, then the sum (\ref{tricky sum}) is equal to
\[\sum_{\{t:p\nmid\frac{c/d}{t}\}}\left(\frac{c/d}{t}\right)^i\mu(t)=(\frac{c}{d})^i\sum_{t\mid\frac{c}{d}}t^{-i}\mu(t)=(\frac{c}{d})^i\prod_{\ell\mid\frac{c}{d}}(1-\ell^{-i})=\prod_{\ell\mid\frac{c}{d}}(\ell^{i}-1). \]

The claimed result is thus true in this case since, using assumption (ii), $|H_d|$ divides (\ref{tricky sum}) both if $i=1$ and also if $i=2$ because $\ell+1$ divides $\ell^2-1$ for each prime divisor $\ell$ of $c/d$.

We now assume that $p$ divides $c/d$. In this case $p$ divides $c$ but not $d$ so the sum (\ref{tricky sum}) is equal to
\begin{align*} &{\sum}_{s \mid \frac{p^{-1}c}{d}} f_p(ds)\mu(s)
+ {\sum}_{s \mid \frac{p^{-1}c}{d}} f_p(dsp)\mu(sp)\\
= \, &{\sum}_{s \mid \frac{p^{-1}c}{d}} f_p(ds)\mu(s)
- {\sum}_{s \mid \frac{p^{-1}c}{d}} f_p(dsp)\mu(s)\\
=\, &{\sum}_{s \mid \frac{p^{-1}c}{d}} \left(\frac{c}{dsp}\right)^i\mu(s)
- {\sum}_{s \mid \frac{p^{-1}c}{d}} \left(\frac{c}{dsp}\right)^i\mu(s)\\
= \, &0.
\end{align*}
and the claimed result is clear.
\end{proof}

\subsubsection{} In this section we prove Corollary \ref{bleycor}.


If $L(A_\QQ,1)$ is non-zero, then Kolyvagin has proved that $\sha(A_\QQ)$ and $A(\QQ)$ are both finite and one knows that the $p$-part of the Birch and Swinnerton-Dyer Conjecture for $A_\QQ$ is valid for all but finitely many primes $p$.


To prove Corollary \ref{bleycor} it is therefore enough to fix a natural number $n$, assume to be given an odd prime $p'$ and a tamely ramified real abelian extension $F'/\QQ$ that has discriminant $d_{F'}$ prime to $N$ and is such that $L(A_{F'},1)$ is non-zero and all of the stated conditions are satisfied after replacing $p, F$ and $n$ by $p', F'$ and $n-1$, and then deduce from this assumption the assertion of Corollary \ref{bleycor} for the fixed $n$. Indeed, the validity of Corollary \ref{bleycor} would then follow from the previous paragraph by induction since ${\rm BSD}_p(A_{\QQ/\QQ})$(iv) is equivalent to $p$-part of the Birch and Swinnerton-Dyer Conjecture for $A_\QQ$ (see Remark \ref{consistency remark}(iii)).

We hence fix a natural number $n$ as well as $p'$ and $F'$ with the assumed given properties.
Then the results of Kato \cite{kato} imply that $\sha(A_{F'})$ and $A(F')$ are both finite and so we may fix a prime $p$ that is large enough to ensure that all of the following conditions are satisfied: the $p$-torsion subgroups of $A(F')$ and ${\rm Sel}(A_{F'})^\vee$ are trivial, $p$ is prime to the degree of $F'/\QQ$, to $d_{F'}$, to $N$, to $c(\varphi_A)$ and to the Tamagawa numbers of $A$ at any place of $F'$ that divides $N$ and the reduction of $A$ at any place $v$ of $F'$ that divides $d_{F'}$ has no element of order $p$.

In addition, if we set $m := |A(\QQ)_{\rm tor}|$ and $G':=G_{F'/\QQ}$, then $\Theta^{\rm MT}_{A,F'}$ belongs to $\ZZ[1/2m][G']$ and so the interpolation property (\ref{chiMT}) combines with the non-vanishing of $L(A_{F'},1)$ to imply that $\Theta^{\rm MT}_{A,F'}$ generates over $\ZZ[1/2m][G']$ a finite index submodule of
 $\ZZ[1/2m][G']$.

This means that we can also assume the prime chosen above to be such that both $\Theta^{\rm MT}_{A,F'}$ belongs to $\ZZ_p[G']^\times$ and the $p$-part of the Birch and Swinnerton-Dyer Conjecture for $A$ over $F'$ is valid.


Next we note that if $A$ does not have complex multiplication, then a result of Serre \cite{serre} implies that for all sufficiently large primes $q$ the group
$G_{\QQ(A[q])/\QQ}$ is isomorphic to ${\rm GL}_2(\mathbb{F}_q)$ and so has order $(q^2-1)(q^2-q)/2$, whilst if $A$ has complex multiplication by an order in an imaginary quadratic field $K$, then the theory of complex multiplication implies that for all sufficiently large $q$ the degree of $\QQ(A[q])/\QQ$ is either $(q-1)^2/w_K$ (if $q$ splits in $K$) or $(q^2-1)/w_K$ (if $q$ is inert in $K$), where $w_K$ is the number of roots of unity in $K$.

It follows that for all sufficiently large primes $q$ the degree of $\QQ(A[q])/\QQ(\mu_q)$ is either $(q^2-1)q/2$, $(q-1)/w_K$ or $(q+1)/w_K$.
 In particular, for sufficiently large $q$ the degree of $F'(\mu_{q^n})/\QQ(\mu_q)$ is not divisible by the degree of $\QQ(A[q])/\QQ(\mu_q)$ and so $\QQ(A[q])$ is not contained in $F'(\mu_{q^n})$.

In particular, we may assume that the prime $p$ fixed earlier has this property and then, by the Tchebotarev density theorem, it follows that there is an infinite set $\mathfrak{L}$ of primes that each split completely in $F'(\mu_{p^n})$ but not in $\QQ(A[p])$.

Then for any such prime $\ell\in\mathfrak{L}$ one has $\ell \equiv 1$ modulo $p^n$ and, by the argument used prove \cite[Th. 1.3]{Bley3}, also $a_\ell \not\equiv 2$ modulo $p$ so that the order of $A(\kappa_v) = A(\mathbb{F}_\ell)$ is prime to $p$ for any $\ell$-adic place $v$ of $F'$.

For any given ordered subset $\{\ell_i: i \in [n]\}$ of $\mathfrak{L}$ of cardinality $n$ and for each index $i$ we write $F_i$ for the unique subextension of $\QQ(\ell_i)$ of degree $p^n$.

We claim that the compositum $F$ of $F'$ with each of the fields $F_i$ has all of the  properties stated in Corollary \ref{bleycor} with respect to the fixed $n$ and $p$.

It is certainly clear that this field satisfies all of the claimed properties except possibly for the assertion that ${\rm BSD}_p(A_{F/\QQ})$(iv) is valid.

In addition, our choice of $p$ implies the hypotheses of Theorem \ref{MT result} are satisfied by the data $A, F/\QQ$ and $p$.

In particular, since the degree of $F/F'$ is a power of $p$, and our choice of $p$ implies the $p$-part of the Birch and Swinnerton-Dyer Conjecture is valid for $A$ over $F'$, the criterion of Theorem \ref{MT result}(ii) reduces us to showing that $\Theta^{\rm MT}_{A,F}$ belongs to ${\rm Fit}^0_{\ZZ_p[G]}(\sha(A_{F})[p^\infty])$ and $L(A_{F},1)$ is non-zero.

Now the isomorphism (\ref{comp iso}) is valid in this case and combines with our choice of $p$ to imply that $\sha(A_{F})[p^\infty]$ is trivial and hence that ${\rm Fit}^0_{\ZZ_p[G]}(\sha(A_{F})[p^\infty])$ is equal to $\ZZ_p[G]$ and so contains $\Theta^{\rm MT}_{A,F}$. In addition, the interpolation formula (\ref{chiMT}) implies that $L(A_{F},1)$ is non-zero if and only if $\Theta^{\rm MT}_{A,F}$ is an invertible element of $\QQ[G]$. To complete the proof of Corollary \ref{bleycor} it is therefore enough to show that $\Theta^{\rm MT}_{A,F}$ belongs to $\ZZ_p[G]^\times$.

But, since $\Theta^{\rm MT}_{A,F}$ belongs to $\ZZ_p[G]$ and the order of $G_{F/F'}$ is a power of $p$, it belongs to $\ZZ_p[G]^\times$ if and only if its image under $\pi_{F/F'}$ belongs to $\ZZ_p[G']^\times$. In addition, since each prime $\ell_i$ splits completely in $F'$, the distribution relation (\ref{dist rel 2}) implies that
\[ \pi_{F/F'}(\Theta^{\rm MT}_{A,F}) = \Theta^{\rm MT}_{A,F'}\cdot \prod_{i=1}^{i=n}-(\ell_i- a_{\ell_i} + 1).\]
In particular, since our choice of $p$ guarantees that both $\Theta^{\rm MT}_{A,F'}$ belongs to $\ZZ_p[G']^\times$ and for each $i$ one has
\[ \ell_i - a_{\ell_i} + 1 \equiv 2-a_{\ell_i} \not\equiv 0 \,\, \text{modulo} \,\, p,\]
it follows that $\Theta^{\rm MT}_{A,F}$ belongs to $\ZZ_p[G]^\times$, as required.

\begin{remark}\label{mod bley rem}{\em The results of Bley in \cite[Cor. 1.4 and Rem. 4.1]{Bley3} show that ${\rm BSD}_p(A_{F/\QQ})$(iv) is valid for an infinite family of abelian extensions $F/\QQ$ of exponent $p$ and rely on the explicit computations of Fearnley, Kisilevsky and Kuwata in \cite{fkk} concerning the twists of modular elements by Dirichlet characters of prime order. The above argument avoids these computations by directly using the theorem of Kato. In addition, whilst the argument above only constructs fields $F$ in which $p$ is unramified, it is possible to modify the approach to construct families of abelian fields $F$ in which $p$ can in principle be tamely ramified and ${\rm BSD}_p(A_{F/\QQ})$(iv) is valid if and only if $L(A_F,1)$ does not vanish. }\end{remark}

\section{Heegner points}\label{HHP}

In this section we again fix an elliptic curve $A$ that is defined over $\QQ$ and has conductor $N$ and also assume that the modular parametrisation $\varphi_A$ that was fixed in \S\ref{mod sect} is of minimal degree.

We also fix a quadratic imaginary field $K$ in which all prime divisors of $N$ split and a square-free product $c$ of primes that are both inert in $K$ and coprime to $N$. We write $K_c$ for the ring class field of $K$ of conductor $c$: this is an abelian extension of $K$ that is ramified exactly at the primes dividing $c$; moreover $K_c$ contains the Hilbert class field $K_1$ of $K$ and class field theory gives a natural isomorphism
\begin{equation}\label{explicit iso} G_{K_c/K_1} \cong (\mathcal{O}_K/c\mathcal{O}_K)^\times/(\ZZ/c\ZZ)^\times. \end{equation}
(cf. \cite[\S3]{gross_koly}).

We set $G_c := G_{K_c/K}$ and $h_c := |G_c|$.

\subsection{Zhang's Theorem and a conjecture of Bradshaw and Stein}\label{thm gzz} The above hypothesis on $K$ implies that the ideal $N\mathcal{O}_K$ factorises as $\mathcal{N}\overline{\mathcal{N}}$ where the natural map $\ZZ/N\ZZ\to \mathcal{O}_K/\mathcal{N}$ is bijective. We write $\mathcal{O}_c$ for the order $\ZZ + c\mathcal{O}_K$ in $K$ and set ${\mathcal{N}}_c := \mathcal{N}\cap \mathcal{O}_c$.

Fixing an embedding of $K$ into $\CC$ the pair $(\CC/\mathcal{O}_c,{\mathcal{N}}_c^{-1}/\mathcal{O}_c)$ defines a CM elliptic curve equipped with a cyclic subgroup of order $N$ and the isomorphism class of this pair defines a point $x_c$ on $X_0(N)(K_c)$.

We define the `higher Heegner point' $y_c$ to be the element $\varphi_A(x_c)$ of $A(K_c)$.

\subsubsection{}
Assuming $L(A_K,z)$ vanishes to order one at $z=1$, Zhang's generalisation \cite{zhang01, zhang} of the seminal results of Gross and Zagier in \cite{GZ} implies that for every $\psi$ in $\widehat{G_c}$ the function $L(A_K,\psi,z)$ vanishes to order one at $z=1$, that the element $e_\psi (y_c)$ of $\CC\cdot A(K_c)$ is non-zero and that the complex number $\epsilon_{A,c,\psi}$ defined by the equality
\begin{equation}\label{e-def} L'(A_{K},\check{\psi},1) = \epsilon_{A,c,\psi}\cdot
\frac{\Omega^+\Omega^-C}{c_\psi\sqrt{|d_K|}}\cdot h_c\cdot\langle e_{\psi}(y_c),e_{\check\psi}(y_c)\rangle _{K_c}\end{equation}
satisfies both
\begin{equation}\label{stark ec} \epsilon_{A,c,\psi} \in \QQ(\psi)^\times \,\,\,\,\text{ and }\,\,\,\, (\epsilon_{A,c,\psi})^\omega = \epsilon_{A,c,\psi^\omega} \,\,\,\,\text{ for all $\omega\in G_{\QQ(\psi)/\QQ}$.}\end{equation}
Here we have set
\[ C := 4 \cdot c_\infty \cdot c(\varphi_A)^{-2}\cdot \bigl\vert\mathcal{O}_{K}^\times\bigr\vert^{-2},\]
and written $c_\psi$ for the least divisor of $c$ such that $\psi$ factors through the restriction map $G_c\to G_{c_\psi}$ and $\langle -,-\rangle_{K_c}$ for the $\CC$-linear extension of the N\'eron-Tate height pairing for $A$ relative to the field $K_c$.


To investigate the conjecture ${\rm BSD}(A_{K_c/K})$ it is convenient to consider a version of the equality (\ref{e-def}) that is adapted to truncated Hasse-Weil-Artin $L$-series.

To do this we follow Darmon \cite{darmon} in defining the `regularised higher Heegner points of $A$ at level $c$' by setting

\[ z_c := \sum_{m \mid c}\mu(m)y_m, \,\,\text{ and }\,\,\,\,z'_c := \sum_{m \mid c}y_m.\]
For each $\psi$ in $\widehat{G_c}$ we then define a non-zero complex number $\epsilon_{A,\psi}$ by means of the equality
\begin{equation}\label{u-def} L'_{c}(A_{K},\check{\psi},1) = \epsilon_{A,\psi}\cdot
\frac{\Omega^+\Omega^-C}{c_\psi\sqrt{|d_K|}}\cdot \frac{h_c}{(c/c_\psi)^2}\cdot\langle e_{\psi}(z_c),e_{\check\psi}(z'_c)\rangle _{K_c}
\end{equation}
where the $L$-series $L_{c}(A_{K},\check{\psi},z)$ is truncated at all places of $K$ corresponding to rational prime divisors of $c$.

The next result describes the precise connection between the algebraic numbers $\epsilon_{A,c,\psi}$ and $\epsilon_{A,\psi}$ that are defined via the equalities (\ref{e-def}) and (\ref{u-def}). This result shows, in particular, that $\epsilon_{A,\psi}$ depends only on the character $\psi$.

In the sequel, for each prime divisor $\ell$ of $c$ we write $\kappa_{(\ell)}$ for the residue field of the unique prime of $K$ above $\ell$.

\begin{lemma}\label{independence} Fix $\psi$ in $\widehat{G_c}$. Set $c':= c_\psi$ and write $\phi$ for $\psi$ regarded as an element of $\widehat{G_{c'}}$. Set $a_{c,c'} := \prod_{\ell}a_\ell$, where $\ell$ runs over all prime divisors of $c/c'$ and the terms $a_\ell$ are as defined at the beginning of \S\ref{mod sect}.

Then the following claims are valid.
\begin{itemize}
\item[(i)] $h_c\cdot\langle e_{\psi}(y_c),e_{\check\psi}(y_c)\rangle _{K_c} = a_{c,c'}^2\cdot h_{c'}\cdot\langle e_{\phi}(y_{c'}),e_{\check\phi}(y_{c'})\rangle _{K_{c'}}$.
\item[(ii)] $\epsilon_{A,c,\psi} = a_{c,c'}^{-2}\cdot\epsilon_{A,c',\phi}.$
\item[(iii)] $\epsilon_{A,\psi} = \epsilon_{A,c',\phi}$.
\end{itemize}
\end{lemma}

\begin{proof} We can clearly assume that $c \not= c'$ and hence that $c > 1$.

In this case, if we set ${\rm T}_{c,c'} := \sum_{g \in G_{K_c/K_{c'}}}g$, then the norm-compatibility of Heegner points implies that
\begin{equation}\label{nc heegner} {\rm T}_{c,c'}(y_c) = a_{c,c'}\cdot y_{c'}.\end{equation}

This implies $e_{\psi}(y_c) = (h_{c'}/h_{c})a_{c,c'}\cdot e_\phi(y_{c'})$ and $e_{\check\psi}(y_c) = (h_{c'}/h_{c})a_{c,c'}\cdot e_{\check\phi}(y_{c'})$ %
and hence 
\begin{align*} h_c\langle e_{\psi}(y_c),e_{\check\psi}(y_c)\rangle _{K_c} =\, &h_c(h_{c'}/h_{c})^2(a_{c,c'})^2\cdot  \langle e_{\phi}(y_{c'}),e_{\check\phi}(y_{c'})\rangle _{K_c}\\
=\, &h_c(h_{c'}/h_{c})(a_{c,c'})^2\cdot  \langle e_{\phi}(y_{c'}),e_{\check\phi}(y_{c'})\rangle _{K_{c'}}\\
=\, &(a_{c,c'})^2\cdot h_{c'}\langle e_{\phi}(y_{c'}),e_{\check\phi}(y_{c'})\rangle _{K_{c'}}.\end{align*}
where the second equality follows from the general result of \cite[Chap. VIII, Lem. 5.10]{silverman}. This proves claim (i).

Claim (ii) follows directly from claim (i) and the fact that the terms $L'(A_{K},\check{\psi},1)$ and $(\Omega^+\Omega^-C)/(c_\psi\sqrt{|d_K|})$ that occur in (\ref{e-def}) do not change if one replaces $\psi$ by $\phi$.

To prove claim (iii) we set $\epsilon_{A,\phi}:= \epsilon_{A,\psi}$. Then, since both $e_\phi(z_{c'}) = e_\phi(y_{c'})$ and $e_{\check\phi}(z'_{c'}) = e_{\check\phi}(y_{c'})$, an explicit comparison of the equalities (\ref{e-def}) and (\ref{u-def}) shows that it suffices to show (\ref{u-def}) remains valid if one replaces $c$ by $c'$ and $\psi$ by $\phi$.

We now set $d:=c/c'$. By a routine computation using (\ref{nc heegner}) one then finds that
\begin{equation}\label{trace heegner} {\rm T}_{c,c'}(z_c) = \prod_{\ell \mid d} (\ell + 1 - a_\ell)\cdot z_{c'} \,\,\,\text{ and }\,\,\,
{\rm T}_{c,c'}(z'_c) = \prod_{\ell \mid d} (\ell + 1 + a_\ell)\cdot z'_{c'}.\end{equation}

In addition, for each prime divisor $\ell$ of $c$ one has
\begin{equation}\label{trace heegner2} |A(\kappa_{(\ell)})| = (\ell + 1 - a_\ell)(\ell + 1 + a_\ell)\end{equation}
and so

\begin{align*} \frac{h_c}{c^2}\langle e_{\psi}(z_c),e_{\check\psi}(z'_c)\rangle _{K_c} &=  \frac{h_c}{c^2}(\prod_{\ell \mid d}(\ell + 1 - a_\ell)(\ell + 1 + a_\ell))\langle e_{\phi}(z_{c'}),e_{\check\phi}(z'_{c'})\rangle _{K_c}\\
&= \bigl(\prod_{\ell \mid d}\frac{|A(\kappa_{(\ell)})|}{\ell^2}\bigr)\frac{h_{c'}}{(c')^2}\langle e_{\phi}(z_{c'}),e_{\check\phi}(z'_{c'})\rangle _{K_{c'}},\end{align*}
where the second equality uses \cite[Chap. VIII, Lem. 5.10]{silverman} and the fact $c = c'\cdot\prod_{\ell\mid d}\ell$.

The right hand side of (\ref{u-def}) therefore changes by a factor of $(\prod_{\ell \mid d}|A(\kappa_{(\ell)})|/\ell^{2})^{-1}$ if one replaces $c$ by $c'$ and $\psi$ by $\phi$.

To show that the left hand side of (\ref{u-def}) changes by the same factor we note that
\begin{align*} L'_{c'}(A_{K},\check{\phi},1)L'_{c}(A_{K},\check{\psi},1)^{-1} =\, &L'_{c'}(A_{K},\check{\psi},1)L'_{c}(A_{K},\check{\psi},1)^{-1}\\ =\, &\prod_{\ell\mid d}P_\ell(A_{K},\check{\psi},1)^{-1}\end{align*}
where $P_\ell(A_{K},\check{\psi},t)$ denotes the Euler factor at (the unique prime of $K$ above) $\ell$ of the $\psi$-twist of $A$.

Now $K_c$ is a dihedral extension of $\QQ$ and so any prime $\ell$ that is inert in $K$ must split completely in the maximal subextension of $K_c$ in which it is unramified. In particular, for each prime divisor $\ell$ of $d$ this implies that $P_\ell(A_{K},\check{\psi},t)$ coincides with the Euler factor $P_\ell(A_{K},t)$ at $\ell$ of $A_{K}$ and hence that
\[ P_\ell(A_{K},\check{\psi},1) = P_\ell(A_{K},1) = \frac{|A(\kappa_{(\ell)})|}{{\rm N}_{K/\QQ}(\ell)} = \frac{|A(\kappa_{(\ell)})|}{\ell^2},\]
as required.
\end{proof}

\subsubsection{}\label{bs conj section}

If $c=1$, then for each character $\psi$ in $\widehat{G_c}$ one has $c_\psi = 1$ and the results of Gross and Zagier in \cite[see, in particular, \S I, (6.5) and the discussion on p. 310]{GZ} imply directly that $\epsilon_{A,c,\psi}=1$.

In addition, for $c > 1$ the work of Zhang in \cite{zhang01, zhang} implies for each $\psi$ in $\widehat{G_c}$ a formula for the algebraic number $\epsilon_{A,c,\psi}$.

However, as observed by Bradshaw and Stein in \cite[\S2]{BS}, this formula is difficult to make explicit and is discussed in the literature in several mutually inconsistent ways.

In particular, it is explained in loc. cit. that the earlier articles of Hayashi \cite{hayashi} and Jetchev, Lauter and Stein \cite{JLS} together contain three distinct formulas for the elements $\epsilon_{A,c,\psi}$ that are mutually inconsistent and all apparently incorrect.

In an attempt to clarify this issue, in \cite[Conj. 6]{BS} Bradshaw and Stein conjecture that for every non-trivial character $\psi$ in $\widehat{G_c^+}$ one should have
\begin{equation}\label{bs conj} \epsilon_{A,c,\psi} = 1,\end{equation}
and Zhang has asserted that the validity of this conjecture can indeed be deduced from his results in \cite{zhang} (see, in particular, \cite[Rem. 7]{BS}).

However, if $c > 1$, then Lemma \ref{independence}(ii) implies $\epsilon_{A,c,\psi}$ is not always equal to $\epsilon_{A,c_\psi,\psi}$ and hence that the conjectural equalities (\ref{bs conj}) are in general mutually compatible only if one restricts to characters $\psi$ with $c_\psi = c$.

For further comments in this regard see Remark \ref{bs conj rem} below.

\subsection{Heegner points and refined BSD} In this section we interpret the complex numbers $\epsilon_{A,\psi}$ defined above in terms of our refined Birch and Swinnerton-Dyer Conjecture.

%
%


We define an element of $\CC[G_c]$ by setting
\[ \epsilon_{A,c} := \sum_{\psi\in \widehat{G_c}}\epsilon_{A,\psi}\cdot e_\psi.\]
Lemma \ref{independence}(iii) combines with the properties (\ref{stark ec}) to imply $\epsilon_{A,c}$ belongs to $\QQ[G_c]$.

We also define an element of $\QQ[G]^\times$ by setting
\[ u_{K,c}:= (-1)^{n(c)}\sum_{\psi\in \widehat{G_c}}(-1)^{n(c_\psi)}\cdot e_\psi,\]
where $n(d)$ denotes the number of rational prime divisors of a natural number $d$.

We recall that ${\rm Sel}_p(A_F)$ denotes the $p$-primary Selmer group of $A$ over $F$ and  ${\rm Fit}^a_{\ZZ_p[G]}(M)$ the $a$-th Fitting ideal of a finitely generated $\ZZ_p[G]$-module $M$. 

\begin{theorem}\label{h-rbsd} Let $F$ be an abelian extension of $K$ of conductor $c$ and set $G := G_{F/K}$. Fix an odd prime $p$ and assume that all of the following conditions are satisfied:
\begin{itemize}
\item[$\bullet$] the data $A$, $F/K$ and $p$ satisfy the hypotheses (H$_1$)-(H$_6$) listed in \S\ref{tmc}.
\item[$\bullet$] $A(F)$ has no point of order $p$.
\item[$\bullet$] The trace to $K$ of $y_1$ is non-zero.
\item[$\bullet$] $p$ is unramified in $K$.
\end{itemize}
%
Set $z_{F} := {\rm Tr}_{K_c/F}(z_c)$ and $z'_{F} := {\rm Tr}_{K_c/F}(z'_c)$. Then the following claims are valid.

\begin{itemize}
\item[(i)] 
If ${\rm BSD}_p(A_{F/K})$(iv) is valid then every element of
\[ {\rm Fit}^0_{\ZZ_p[G]}\left( \bigl(A(F)_p/\langle z_F\rangle\bigr)^\vee\right)\cdot  {\rm Fit}^0_{\ZZ_p[G]}\left(\bigl( A(F)_p/\langle z'_F\rangle\bigr)^\vee\right)\cdot C\cdot u_{K,c}\cdot\epsilon_{A,c} \]
belongs to ${\rm Fit}^1_{\ZZ_p[G]}({\rm Sel}_p(A_{F})^\vee)$ and annihilates $\sha(A_{F})[p^\infty]$.

\item[(ii)] Assume that $F/K$ is of $p$-power degree and that $p$ does not divide the trace to $K$ of $y_1$. Then ${\rm BSD}_p(A_{F/K})$(iv) is valid if and only if one has $${\rm Fit}^0_{\ZZ_p[G]}(\sha(A_{F})[p^\infty]) = \ZZ_p[G]\cdot C\cdot u_{K,c}\cdot \epsilon_{A,c}.$$
\end{itemize}
\end{theorem}

%
%
%
\begin{proof} Since the extension $F/K$ is tamely ramified we shall derive claim (i) as a consequence of the fact, justified in \S \ref{justifications}, that if ${\rm BSD}_p(A_{F/K})$(iv) is valid then so is Prediction \ref{new add2}.



We first note that the assumed non-vanishing of the trace to $K$ of $y_1=z_1$ combines with the trace compatibilities in (\ref{trace heegner}) to imply that the elements $e_\psi(z_c)$ and $e_\psi(z'_c)$ are non-zero for all $\psi$ in $\widehat{G}$.

Taken in conjunction with (\ref{u-def}), this fact implies directly that each function $L_{S_{\rm r}}(A,\check{\psi},z)$ vanishes to order one at $z=1$, where as in Prediction \ref{new add2} we have set $S_{\rm r}=S_k^F\cap S_k^F$. It also combines with the main result of Bertolini and Darmon in \cite{BD} to imply that the $\ZZ_p[G]$-modules generated by $z_c$ and $z_c'$ each have finite index in $A(K_c)$. This implies, in particular, that the idempotent $e_{(1)}$ is equal to $1$.


Since every prime divisor of $c$ is inert in $K$ and then splits completely in the maximal subextension of $K_c$ in which it is unramified, the conductor of each character $\psi$ in $\widehat{G_c}$ is a divisor $c_\psi$ of $c$ and the unramified characteristic $u_\psi$ defined in \S\ref{mod GGS section} is equal to $(-1)^{n(c) + n(c_\psi)}$.

By using \cite[(21)]{bmw} one can then compute that for every $\psi$ in $\widehat{G_c}$ one has
\begin{align}\label{gauss sums_eq}
\tau^\ast \bigl(\QQ,\psi\bigr)\cdot w_\psi^{-1}=\, &u_\psi\cdot \tau \bigl(\QQ,\psi\bigr)\cdot w_\psi^{-1}
\\ =\, &u_\psi\sqrt{\vert d_{K}\vert} \sqrt{{\rm N}c_\psi}\notag\\
 =\, &(-1)^{n(c) + n(c_\psi)}c_\psi\sqrt{\vert d_{K}\vert}.\notag\end{align}

In addition, for each $\psi$ in $\widehat{G}$ one has $\Omega_A^\psi= \Omega^+\Omega^-$ and
\begin{align*} e_\psi\cdot h_{F/K}(z_F,z_F') =\, &\sum_{g \in G}\langle g(z_F),z_F'\rangle_{F}\cdot \psi(g)^{-1}e_\psi\\
=\, & |G|\langle e_\psi(z_F),z_F'\rangle_{F}\cdot e_\psi\\
=\, & |G|\langle e_\psi(z_F),e_{\check\psi}(z_F')\rangle_{F} \cdot e_\psi\\
=\, & |G|(|G|/h_c)\langle e_\psi(z_F),e_{\check\psi}(z_F')\rangle_{K_c} \cdot e_\psi\\
=\, & h_c\cdot\langle e_\psi(z_c),e_{\check\psi}(z_c')\rangle_{K_c} \cdot e_\psi,\end{align*}
where in the last equality $\psi$ and $\check\psi$ are regarded as characters of $G_c$.

Setting $S_{p,{\rm r}}=S_{\rm r}\cap S_k^p$ one may also explicitly compute, for $\psi\in\widehat{G}$, that $m_\psi:=\prod_{v\in S_{p,{\rm r}}}\varrho_{v,\psi}$ is equal to $p^2$ if $p$ divides $c$ but not $c_\psi$ and is equal to $1$ otherwise. We use this explicit description to extend the definition of $m_\psi$ to all characters $\psi\in\widehat{G_c}$.

The above facts combine with (\ref{u-def}) to imply that for any $\psi\in\widehat{G}$ one has

\begin{align}\label{explicit lt}&\left(\frac{L^{(1)}_{S_{\rm r}}(A_{F/K},1)\cdot\tau^*(F/K)\cdot\prod_{v\in S_{p,{\rm r}}}\varrho_{v}(F/k)}{\Omega_A^{F/K}\cdot w_{F/k}\cdot h_{F/K}(z_F,z_F')}\right)_\psi\\= \, &\frac{L'_{c}(A_K,\check{\psi},1)\cdot\tau^*(\QQ,\psi)\cdot m_\psi}{\Omega_A^\psi\cdot w_\psi\cdot h_c\cdot\langle e_\psi(z_c),e_{\check\psi}(z_c')\rangle_{K_c}}\notag\\
 = \, & \frac{L'_{c}(A_{K},\check{\psi},1)(-1)^{n(c)+ n(c_\psi)}c_\psi\sqrt{|d_K|}}{\Omega^+\Omega^-\cdot h_c\cdot\langle e_\psi(z_c),e_{\check\psi}(z_c')\rangle_{K_c}} \cdot m_\psi\notag\\
 = \, & (-1)^{n(c)+ n(c_\psi)}\epsilon_{A,\psi}\cdot C \cdot (c_\psi/c)^2m_\psi \notag\\
 = \, & (-1)^{n(c)}(-1)^{n(c_\psi)}\epsilon_{A,\psi}\cdot C \cdot (c_\psi/c)^2 m_\psi. \notag
  \end{align}
%

To deduce claim (i) from the validity of Prediction \ref{new add2} it is thus sufficient to show that the sum

\begin{equation}\label{unit sum} \sum_{\psi\in\widehat{G_c}}(c_\psi/c)^2m_\psi\cdot e_\psi\end{equation}
belongs to $\ZZ_p[G_c]^\times$. This fact follows from the result of Lemma \ref{bley lemma} with $A=G_c$ and $i=2$ (so that $n_\psi^i=m_\psi$) and, for each positive divisor $d$ of $c$, with the subgroup $H_d$ of $G_c$ specified to be $G_{K_c/K_d}$. Indeed, this choice of subgroups satisfies the assumption (ii) of Lemma \ref{bley lemma} because (\ref{explicit iso}) implies that $|G_{K_c/K_d}|$ is equal to $\prod_{\ell\mid (c/d)}(\ell+1)$.


To prove claim (ii) it suffices to show that, under the given hypotheses, the equality in Theorem \ref{bk explicit} is valid if and only if one has ${\rm Fit}^0_{\ZZ_p[G]}(\sha(A_{F})[p^\infty]) = \ZZ_p[G]\cdot C\cdot u_{K,c}\cdot\epsilon_{A,c}.$

Now, in Theorem \ref{bk explicit}, the set $S_{p,{\rm w}}$ is empty since $F$ is a tamely ramified extension of $k =K$ and the set $S_{p,{\rm u}}^\ast$ is empty since we are assuming that $p$ is unramified in $K$.

We next note that $z_1 = z'_1 = y_1$ and hence that (\ref{trace heegner}) implies
\[ {\rm Tr}_{F/K}(z_F) = {\rm Tr}_{K_c/K}(z_c) = \mu_c\cdot{\rm Tr}_{K_1/K}(z_1) = \mu_c\cdot {\rm Tr}_{K_1/K}(y_1)\]
with $\mu_c = \prod_{\ell \mid c} (\ell + 1 - a_\ell)$ and, similarly, that
\[ {\rm Tr}_{F/K}(z'_F) = \mu_c'\cdot {\rm Tr}_{K_1/K}(y_1)\]
with $\mu'_c = \prod_{\ell \mid c} (\ell + 1 +a_\ell)$.

In particular, since (\ref{trace heegner2}) implies that $\mu_c\cdot \mu'_c = \prod_{\ell\mid c}|A(\kappa_{(\ell)})|$, our assumption that the hypotheses (H$_3$) and (H$_4$) hold for $F/K$ means that $\mu_c\cdot \mu'_c$ is not divisible by $p$. This fact in turn combines with our assumption that ${\rm Tr}_{K_1/K}(y_1)$ is not divisible by $p$ in $A(K)$ to imply that neither ${\rm Tr}_{F/K}(z_F)$ nor ${\rm Tr}_{F/K}(z'_F)$ is divisible by $p$ in $A(K)_p$.

 Since our hypotheses imply that $A(K)_p = A(F)_p^{G}$ is a free $\ZZ_p$-module of rank one, it follows that ${\rm Tr}_{F/K}(z_F)$ and ${\rm Tr}_{F/K}(z'_F)$ are both $\ZZ_p$-generators of $A(F)_p^G$, and hence, by Nakayama's Lemma, that $A(F)_p$ is itself a free rank one $\ZZ_p[G]$-module that is generated by both $z_F$ and $z'_F$.

Taken in conjunction with the explicit descriptions of cohomology given in (\ref{bksc cohom}) (that are valid under the present hypotheses), these facts imply that the Euler characteristic that occurs in Theorem \ref{bk explicit} can be computed as
\[ \chi_{G,p}({\rm SC}_p(A_{F/k}),h_{A,F}) = {\rm Fit}^0_{\ZZ_p[G]}(\sha(A_F)[p^\infty])\cdot h_{F/k}(z_F,z_F'),\]
where we have identified $K_0(\ZZ_p[G],\CC_p[G])$ with the multiplicative group of invertible $\ZZ_p[G]$-lattices in $\CC_p[G]$ (as in Remark \ref{comparingdets}).

When combined with the equality (\ref{explicit lt}) these facts imply that the product $$\mathcal{L}^*_{A,F/K}\cdot h_{F/k}(z_F,z_F')^{-1},$$ where $\mathcal{L}^*_{A,F/K}$ is the leading term element defined in (\ref{bkcharelement}), is equal to the projection to $\ZZ_p[G]$ of $C\cdot u_{K,c}\cdot\epsilon_{A,c}$ multiplied by the sum (\ref{unit sum}).

Claim (ii) is therefore a consequence of Theorem \ref{bk explicit} and the fact that, as already observed above, the sum (\ref{unit sum}) belongs to $\ZZ_p[G_c]^\times$.
\end{proof}

\begin{remark}\label{bs conj rem}{\em If $p$ is prime to all factors in $C$, then the hypotheses of Theorem \ref{h-rbsd}(ii) combine with an argument of Kolyvagin to imply $\sha(A/K)[p^\infty]$ vanishes (cf. \cite[Prop. 2.1]{gross_koly}). This fact combines with the projectivity of $A(F)^\ast$ to imply $\sha(A/F)[p^\infty]$ vanishes and hence, via Theorem \ref{h-rbsd}(ii), that ${\rm BSD}_p(A_{F/K})$(iv) is valid if and only if the product $u_{K,c}\cdot\epsilon_{A,c}$ projects to a unit of $\ZZ_p[G]$. In the case that $F/K$ is unramified, this observation was used by Wuthrich and the present authors to prove the main result of \cite{bmw}. In the general case, it is consistent with an affirmative answer to the question of whether for every $\psi$ in $\widehat{G}$ one should always have
\[ \epsilon_{A,\psi} = (-1)^{n(c_\psi)}?\]
We observe that such an equality would, if valid, constitute a functorially well-behaved, and consistent, version of the conjecture of Bradshaw and Stein that was discussed in \S\ref{bs conj section}.}\end{remark}


\appendix

\section{Refined BSD and equivariant Tamagawa numbers}\label{consistency section}

\subsection{Statement of the main result and consequences} We shall assume throughout this section that $\sha(A_F)$ is finite.

In this case the approach of \cite[\S3.4]{bufl99} gives a well-defined element $R\Omega(h^1(A_{F})(1),\ZZ[G])$ of $K_0(\Z[G],\RR[G])$ and we shall now compute this element in terms of elements that occur in ${\rm BSD}(A_{F/k})$.

To state the result we fix a prime $p$ and then define, for each non-archimedean place $v$ of $k$, an element of $\zeta(\QQ_p[G])^\times$ by setting
\[ L_v(A,F/k) := \begin{cases} \Nrd_{\QQ_p[G]}\bigl(1- \Phi_v^{-1}\bigm\vert V_{p,F}(A^t)^{I_{w}}\bigr), &\text{if $v\nmid p$}\\
\Nrd_{\QQ_p[G]}\bigl(1- \varphi_v\bigm\vert D_{{\rm cr},v}(V_{p,F}(A^t))\bigr), &\text{if $v\mid p$,}\end{cases}\]
where $\varphi_v$ is the crystalline Frobenius at $v$.

We also fix an isomorphism of fields $j:\CC\cong \CC_p$ and use the induced homomorphism of abelian groups $j_*: K_0(\ZZ[G],\RR[G]) \to K_0(\ZZ_p[G],\CC_p[G])$.

In the main result of this section we compute the element $R\Omega(h^1(A_{F})(1),\ZZ[G])$ in terms of the complex ${\rm SC}_{S,\omega_\bullet}(A_{F/k})$ in (\ref{revisionSC}), the morphism of non-abelian determinants $h^{{\rm det}}_{A,F}$ that is induced by the N\'eron-Tate height (\ref{height triv}), the period $\Omega_{\omega_\bullet}(A_{F/k})$ in \S \ref{revisionO} and the element $\mu_{S}(A_{F/k})$ in (\ref{revisionMU}).

The proof of this result will be given in \S \ref{A2} below.

\begin{proposition}\label{etnc} Fix an ordered $\QQ[G]$-basis $\omega_\bullet$ of the space $H^0(A^t_F,\Omega^d_{A^t_F})$ and a finite set $S$ of places of $k$ that contains $S_k^\infty, S_k^F$ and $S_k^A$.

Then in $K_0(\ZZ_p[G],\CC_p[G])$ one has
\begin{multline*}\label{enoughromega} \chi_{G,p}({\rm SC}_{S,\omega_\bullet}(A_{F/k})_p,h^j_{A,F})+\mu_{S}(A_{F/k})_p\\
 = -j_*\bigl(R\Omega(h^1(A_{F})(1),\ZZ[G])\bigr)-
 \partial_{G,p}(j_*(\Omega_{\omega_\bullet}(A_{F/k})))+ \sum_{v\in S\setminus S_k^\infty}\delta_{G,p}(L_v(A,F/k)),
\end{multline*}
where ${\rm SC}_{S,\omega_\bullet}(A_{F/k})_p$ denotes $\ZZ_p\otimes_\ZZ {\rm SC}_{S,\omega_\bullet}(A_{F/k})$ and $h^{j}_{A,F}$ denotes $\CC_p\otimes_{\RR,j}h^{{\rm det}}_{A,F}$.
\end{proposition}

This result has the following useful consequence.

\begin{corollary}\label{rbsd=etnc} The weaker version of ${\rm BSD}(A_{F/k})$ that is discussed in Remark \ref{weaker BSD} is equivalent to the equivariant Tamagawa number conjecture for the pair $(h^1(A_F)(1),\ZZ[G])$.\end{corollary}

\begin{proof} Claim (i) of ${\rm BSD}(A_{F/k})$ simply asserts that $\sha(A_F)$ is finite, and we are assuming this to be true.

It is also well-known that claim (ii) of ${\rm BSD}(A_{F/k})$ coincides with ~\cite[Conj. 4(i),(ii)]{bufl99} for the pair $(h^1(A_F)(1),\ZZ[G])$.

It is thus enough to relate the equality predicted in Remark \ref{weaker BSD} to that predicted by \cite[Conj. 4(iv)]{bufl99}.

To do this we write $L^\ast(_{\QQ[G]}h^1(A_{F})(1),0)$ for the leading coefficient  at $s=0$ of the $\zeta\bigl(\CC[G]\bigr)$-valued $L$-function $L\bigl(_{\QQ[G]}h^1(A_{F})(1),s\bigr)$ defined in
\cite[\S4.1]{bufl99}.

Then the definition of the element $\calL_S^*(A_{F/k},1)$ of $\zeta(\RR[G])^\times$ given in Remark \ref{weaker BSD} directly implies that
\[ \calL_S^*(A_{F/k},1) = L^\ast(_{\QQ[G]}h^1(A_{F})(1),0)\cdot \prod_{v\in S\setminus S_k^\infty}L_v(A,F/k).\]

By applying the extended boundary homomorphism $\delta_G: \zeta(\RR[G])^\times \to K_0(\ZZ[G],\RR[G])$ we can therefore deduce that

\begin{align*} &\delta_G(L^\ast(_{\QQ[G]}h^1(A_{F})(1),0))\\
=\, &\delta_G(\calL_S^*(A_{F/k},1)) - \sum_{v\in S\setminus S_k^\infty}\delta_{G}(L_v(A,F/k))\\
=\, &\bigl(\delta_G(\calL_S^*(A_{F/k},1)) - \partial_G(\Omega_{\omega_\bullet}(A_{F/k}))\bigr) + \partial_G(\Omega_{\omega_\bullet}(A_{F/k}))- \sum_{v\in S\setminus S_k^\infty}\delta_{G}(L_v(A,F/k)).
\end{align*}

This implies that the equality of ~\cite[Conj. 4(iv)]{bufl99} for the pair $\bigl(h^1(A_{F})(1), \ZZ[G]\bigr)$ asserts that in $K_0(\ZZ[G],\RR[G])$ one has
\begin{align*} &-R\Omega(h^1(A_{F})(1), \ZZ[G])\\
 = \, &\delta_G(L^\ast(_{\QQ[G]}h^1(A_{F})(1),0)) \\
 =\, &\bigl(\delta_G(\calL_S^*(A_{F/k},1)) - \partial_G(\Omega_{\omega_\bullet}(A_{F/k}))\bigr) + \partial_G(\Omega_{\omega_\bullet}(A_{F/k}))- \sum_{v\in S\setminus S_k^\infty}\delta_{G}(L_v(A,F/k)).\end{align*}

Thus, since Remark \ref{weaker BSD} predicts the equality
\[ \delta_G(\calL_S^*(A_{F/k},1)) - \partial_G(\Omega_{\omega_\bullet}(A_{F/k})) = \chi_G({\rm SC}_{S,\omega_\bullet}(A_{F/k}),h_{A,F}) + \mu_{S}(A_{F/k}),\]
the claimed result will be true if we can show that Proposition \ref{etnc} implies that in $K_0(\ZZ[G],\RR[G])$ one has
\begin{multline*}
\chi_G({\rm SC}_{S,\omega_\bullet}(A_{F/k}),h_{A,F}) + \mu_{S}(A_{F/k}) = \\
 -R\Omega(h^1(A_{F})(1), \ZZ[G]) - \partial_G(\Omega_{\omega_\bullet}(A_{F/k})) + \sum_{v\in S\setminus S_k^\infty}\delta_{G}(L_v(A,F/k)).\end{multline*}

By the argument of Lemma \ref{pro-p lemma}, it is thus enough to show that for every prime $p$ and every isomorphism of fields $j:\CC\cong \CC_p$ the image under $j_*$ of the latter equality is implied by the displayed formula in Proposition \ref{etnc}.

This follows easily from the fact that the construction of $\delta_G$ ensures the commutativity of the diagram

\begin{equation*}\label{key commute}\begin{CD} \zeta(\RR[G])^\times @> \delta_{G} >> K_0(\bz
[G],\RR [G])\\
@V VV @VVj_* V\\
\zeta(\CC_p[G])^\times @> \delta_{G,p}>>
K_0(\bz_p [G],\CC_p [G])\end{CD}\end{equation*}
in which the left hand vertical map is induced by the ring inclusion $\RR[G] \to \CC_p[G]$ that sends each element $\sum_{g \in G}x_gg$ to $\sum_{g \in G}j(x_g)g$. 
%
\end{proof}

\begin{remark}\label{consistency}{\em The consistency results for ${\rm BSD}(A_{F/k})$ that are stated in Remark \ref{consistency remark} follow directly upon combining (the argument of) Corollary \ref{rbsd=etnc} with known results concerning equivariant Tamagawa numbers.

\begin{itemize}
\item[(i)] It is clear that $L^\ast(_{\QQ[G]}h^1(A_{F})(1),0)$ is independent of the choices of sets $\omega_\bullet$ and $S$ and that $R\Omega(h^1(A_{F})(1), \ZZ[G])$ is independent of $\omega_\bullet$. Remark \ref{consistency remark}(i) thus follows from the fact that $R\Omega(h^1(A_{F})(1), \ZZ[G])$ is independent of the choice of $S$, as is shown in \cite[\S3.4, Lem. 5]{bufl99}.
\item[(ii)] Remark \ref{consistency remark}(ii) is a consequence of the functorial properties of $L^\ast(_{\QQ[G]}h^1(A_{F})(1),0)$ and $R\Omega(h^1(A_{F})(1), \ZZ[G])$ that are established in
 \cite[\S3.5, Th. 3.1, \S4.4, Th. 4.1 and \S4.5, Prop. 4.1]{bufl99}.
\item[(iii)] Remark \ref{consistency remark}(iii) follows from the main result of Kings in \cite{kings}.
\end{itemize}
}
\end{remark}

\subsection{The proof of Proposition \ref{etnc}}\label{A2}

\subsubsection{}
To compute $j_*\bigl(R\Omega\bigl(h^1(A_{F})(1),\ZZ[G])\bigr)$ we first recall relevant facts concerning the formalism of virtual objects introduced by Deligne in~\cite{delignedet} (for more details see~\cite[\S2]{bufl99}).

With $R$ denoting either $\ZZ_p[G]$ or $\CC_p[G]$ we write $V(R)$ for the category of virtual objects over $R$ and $[C]_R$ for the object of $V(R)$ associated to each $C$ in $D^{\rm perf}(R)$.

We recall that $V(R)$ is a Picard category with $\pi_1(V(R))$ naturally isomorphic to $K_1(R)$ (see~\cite[(2)]{bufl99}) and we write
\[ (X,Y) \mapsto X\cdot Y\]
for its product and $\eins_R$ for the unit object $[0]_R$.

Writing $\mathcal{P}_0$ for the
Picard category with unique object $\eins_{\mathcal{P}_0}$
and $\Aut_{\mathcal{P}_0}(\eins_{\mathcal{P}_0}) = 0$ we use the isomorphism of abelian groups
\begin{equation}\label{virtual iso}
 \pi_0\Bigl(V\bigl(\ZZ_p[G]\bigr)\times_{V(\CC_p[G])}\mathcal{P}_0\Bigr)\cong K_0\bigl(\ZZ_p[G],\CC_p[G]\bigr)
\end{equation}
that is described in~\cite[Prop. 2.5]{bufl99}. In particular, via this isomorphism, each pair comprising an object $X$ of $V\bigl(\ZZ_p[G]\bigr)$ and a morphism $\iota: \CC_p[G]\otimes_{\ZZ_p[G]}X \to \eins_{\CC_p[G]}$ gives rise to a canonical element $[X,\iota]$ of $K_0\bigl(\ZZ_p[G],\CC_p[G]\bigr)$.

Now if $C$ belongs to $D^{\rm perf}\bigl(\ZZ_p[G]\bigr)$ and $\QQ_p\otimes_{\ZZ_p}C$ is acyclic outside degrees $a$ and $a+1$ (for any integer $a$), then any isomorphism of $\CC_p[G]$-modules
\[ \lambda: H^a\bigl(\CC_p\otimes_{\ZZ_p}C\bigr) \cong H^{a+1}\bigl(\CC_p\otimes_{\ZZ_p}C\bigr)\]
induces a canonical morphism in $V\bigl(\CC_p[G]\bigr)$ of the form
\[ \lambda_{\rm Triv}: [\CC_p\otimes_{\ZZ_p}C]_{\CC_p[G]} \to \eins_{\CC_p[G]}.\]
The associated element $\bigl[[C]_{\ZZ_p[G]},\lambda_{{\rm Triv}}\bigr]$ of  $K_0\bigl(\ZZ_p[G],\CC_p[G]\bigr)$ coincides with the Euler characteristic $\chi_{G,p}\bigl(C,\lambda^{(-1)^a}\bigr)$ defined in~\cite[Def. 5.5]{additivity}.

In particular, from Proposition 5.6(3) in loc. cit. it follows that
\[ \bigl[[C[-1]]_{\ZZ_p[G]},\lambda_{{\rm Triv}}\bigr] = -\bigl[[C]_{\ZZ_p[G]},\lambda_{{\rm Triv}}\bigr],\]
whilst the results of Theorem 6.2 and Lemma 6.3 in loc. cit. combine to imply that if the module $H^a(C)$ is torsion-free, then one has
\begin{multline}\label{normalisation}
  \bigl[[C]_{\ZZ_p[G]},\lambda_{{\rm Triv}}\bigr]\\ = (-1)^a\chi_{G,p}(C,\lambda) + \delta_{G,p}\biggl(\prod_{i \equiv 1,2 \bmod{4} }\Nrd_{\QQ_p[G]} \bigl(-{\id}\bigm\vert \QQ_p\otimes_{\ZZ_p}H^i(C)\bigr)^{(-1)^i}\biggr),\end{multline}
where $\chi_{G,p}(-,-)$ is the non-abelian determinant discussed in \S\ref{nad sec}.

\subsubsection{} We now prove Proposition \ref{etnc}. In fact, for simplicity of exposition, we shall assume that both the set of places $S$ of $k$ that occurs in the statement (Conjecture \ref{conj:ebsd}) of ${\rm BSD}(A_{F/k})$(iv), and the associated set of places $\Upsilon$ that is defined in \S\ref{Galoiscomplexes}, are closed under the action of $G_\QQ$. (The general case is proved by a natural extension of the argument presented here but involves rather heavier notation.)

In the sequel we abbreviate the perfect Selmer structure $\mathcal{X}_S(\{\mathcal{A}_v^t\}_v,\omega_\bullet,\gamma_\bullet)$ to $\mathcal{X}$.

We recall that $j_*\bigl(R\Omega\bigl(h^1(A_{F})(1),\ZZ[G])\bigr)$ is by definition equal to the image under (\ref{virtual iso}) of the pair given by $[R\Gamma_c(A_{F/k})]_{\ZZ_p[G]}\in V(\ZZ_p[G])$ and by a specific canonical trivialisation of this object. Here the compactly supported cohomology complex $$R\Gamma_c(A_{F/k}):=R\Gamma_c(\mathcal{O}_{k,S\cup S_k^p},T_{p,F}(A^t))$$ lies in a canonical exact triangle
\begin{multline}\label{adaptedtri}\mathcal{X}(p)[-2]\oplus \mathcal{X}(\infty)_p[-1]\to R\Gamma_c(A_{F/k})\to {\rm SC}_S(A_{F/k},\mathcal{X}(p),\mathcal{X}(\infty)_p)\\
\to \mathcal{X}(p)[-1]\oplus \mathcal{X}(\infty)_p[0]\end{multline}
in $D^{\rm perf}(\ZZ_p[G])$ exactly as in (\ref{can tri}).

The key strategy in our computation of $j_*\bigl(R\Omega\bigl(h^1(A_{F})(1),\ZZ[G])\bigr)$ is then to compare the scalar extension of this triangle to the exact triangle that is induced by the central column of the diagram in \cite[(26)]{bufl99} (for $V_p=V_{p,F}(A^t)$) in order to compute the relevant canonical trivialisation of $[R\Gamma_c(A_{F/k})]_{\ZZ_p[G]}$ in terms of the invariants occurring in the claimed formula of Proposition \ref{etnc}.

%
%

%
%

In order to do so we set $${\rm SC}_S(A_{F/k},\mathcal{X})_p:={\rm SC}_S(A_{F/k},\mathcal{X}(p),\mathcal{X}(\infty)_p)$$ and also
$$C^{\rm exp}_{A,F}:=\mathcal{X}(p)[-1]\oplus \mathcal{X}(\infty)_p[0]$$
and then define an isomorphism of $\CC_p[G]$-modules $h_{A,F}^{{\rm exp},j}$ as the following composition:

\begin{align*}\CC_p\cdot H^0(C^{\rm exp}_{A,F})=\CC_p\cdot\mathcal{X}(\infty)_p&=\CC_p\cdot H_\infty(A_{F/k})_p\\ 
&\stackrel{\mu_1}{\cong}\CC_p\cdot\Hom_F(H^0(A_F^t,\Omega^1_{A_F^t}),F)\\
&\stackrel{\mu_2}{\cong}\CC_p\cdot A^t(F_p)^\wedge_p=\CC_p\cdot\mathcal{X}(p)=\CC_p\cdot H^1(C^{\rm exp}_{A,F}).
\end{align*}
Here $\mu_1$ is induced (via $j$) by the period isomorphism and $\mu_2$ is the scalar extension of ${\rm exp}_{A^t,F_p}:=\prod_{v\in S_k^p}{\rm exp}_{A^t,F_v}$. 
%
%

We next compare the scalar extension $\QQ_p\cdot C^{\rm exp}_{A,F}$ of the complex $C^{\rm exp}_{A,F}$ to the term
$$\bigoplus_{v\in S\cup S_k^p}R\Gamma_f(k_v,V_{p,F}(A^t))$$
that occurs in the central column of the diagram \cite[(26)]{bufl99}.

For each $v\in S_k^\infty$ the complex $R\Gamma_f(k_v,V_{p,F}(A^t))$ is by definition equal to $R\Gamma(k_v,V_{p,F}(A^t))$ and thus the canonical comparison isomorphisms (\ref{cancompisom}) induce canonical morphisms
$$\QQ_p\cdot \mathcal{X}(\infty)_p[0]\to R\Gamma_f(k_v,V_{p,F}(A^t)).$$

For each $v\in S\setminus(S_k^\infty\cup S_k^p)$ the canonical quasi-isomorphism \cite[(19)]{bufl99} implies that the complex $R\Gamma_f(k_v,V_{p,F}(A^t))$ is acyclic.

Finally, for each $v\in S_k^p$, it follows from \cite[Ex. 3.11]{bk} that the canonical Kummer maps induce canonical morphisms
$$\QQ_p\cdot \mathcal{X}(p)[-1]\to R\Gamma_f(k_v,V_{p,F}(A^t)).$$

By combining these facts we obtain a canonical morphism of exact triangles
\begin{equation} \label{commtriangles}
\begin{CD} \QQ_p\cdot C^{\rm exp}_{A,F}[-1] @> >> \QQ_p\!\cdot\! R\Gamma_c(A_{F/k}) @>  >> \QQ_p\!\cdot\! {\rm SC}_S(A_{F/k},\mathcal{X})_p \\
@V VV @\vert @V VV  \\
\bigoplus_{v\in S\cup S_k^p}R\Gamma_f(k_v,V_{p,F}(A^t))[-1] @> >> \QQ_p\!\cdot\! R\Gamma_c(A_{F/k}) @>   >> R\Gamma_f(k,V_{p,F}(A^t))
\end{CD}
\end{equation}
in $D^{\rm perf}(\QQ_p[G])$. Here the top row is the scalar extension of the exact triangle (\ref{adaptedtri}), the bottom row is the central column of the diagram \cite[(26)]{bufl99} (in particular the complex $R\Gamma_f(k,V_{p,F}(A^t))$ may be defined by the exactness of this triangle), and the right-most vertical arrow is defined by the commutativity of the diagram.

To compute $j_*\bigl(R\Omega\bigl(h^1(A_{F})(1),\ZZ[G])\bigr)$ we consider the following diagram in $V\bigl(\CC_p[G]\bigr)$
\begin{equation}\label{comm diag} \xymatrix@C+1ex@R+1ex{
\bigl[\CC_p\cdot R\Gamma_c(A_{F/k})\bigr] \ar[d]^{\Delta} \ar[r]^{\alpha} & \eins_{\CC_p[G]}\ar[r]^{\alpha'}\ar[d]^{\iota} & \eins_{\CC_p[G]}\ar[d]^{\iota}\\
    \bigl[\CC_p\cdot C^{\rm exp}_{A,F}[-1]\bigr]\cdot \bigl[\CC_p\cdot {\rm SC}_S(A_{F/k},\mathcal{X})_p\bigr]
\ar[r]^(.68){\beta}
& \eins_{\CC_p[G]}\cdot\eins_{\CC_p[G]} \ar[r]^{\gamma}
& \eins_{\CC_p[G]}\cdot\eins_{\CC_p[G]} .}
\end{equation}
In this diagram we abbreviate $[-]_{\CC_p[G]}$ to $[-]$ and use the following morphisms: $\Delta$ is induced by the scalar extension of the exact triangle (\ref{adaptedtri}); $\beta$ is the morphism $\bigl[(h_{A,F}^{{\rm exp},j})_{\rm Triv}\bigr]\cdot \bigl[(h^{j}_{A,F})_{\rm Triv}\bigr]$; $\gamma$ is the morphism $\bigl(\prod_{v\in (S\cup S_k^p)\setminus S_k^\infty}L_v(A,F/k)\bigr)\cdot{\id}$;  $\iota$ denotes the canonical identification and, finally, the morphism $\alpha$ and $\alpha'$ are defined by the condition that the two squares commute.


We claim that this definition of $\alpha$ and $\alpha'$ implies that
\begin{equation}\label{comp eq}\alpha'\circ \alpha = (\CC_p\otimes_{\RR,j}\vartheta_\infty)\circ(\CC_p\otimes_{\QQ_p}\vartheta_p)^{-1}\end{equation}
where the morphisms $\vartheta_\infty$ and $\vartheta_p$ are as constructed in~\cite[\S3.4]{bufl99}.

Indeed, this fact is a consequence of the commutativity of the diagram (\ref{commtriangles}) and of a direct comparison of the explicit definitions of the morphisms $h_{A,F}^{{\rm exp},j}$ and $\vartheta_p$ and of $h^{j}_{A,F}$ and $\vartheta_\infty$.
After this it only remains to note the following fact.

For each place $v \in S\setminus (S_k^\infty\cup S_k^p)$, respectively $v \in S_k^p$, we write $V_{p,v}$ and $\phi_v$ for $V_{p,F}(A^t)^{I_{w}}$ and $\Phi_v^{-1}$, respectively $D_{{\rm cr},v}\bigl(V_{p,F}(A^t)\bigr)$ and $\varphi_v$, and then $V_{p,v}^\bullet$ for the complex $$V_{p,v}\xrightarrow{1-\phi_v}V_{p,v},$$ with the first term placed in degree zero.

Then our definition of $h_{A,F}^{{\rm exp},j}$ implicitly uses the morphism $[V^\bullet_{p,v}]_{\QQ_p[G]} \to \eins_{\QQ_p[G]}$ induced by the acyclicity of $V^\bullet_{p,v}$ whereas the definition of $\vartheta_p$ uses (via~\cite[(19) and (22)]{bufl99}) the morphism $[V^\bullet_{p,v}]_{\QQ_p[G]} \to \eins_{\QQ_p[G]}$ induced by the identity map on $V_{p,v}$;
the occurrence of the morphism $\alpha'$ in the equality~\eqref{comp eq} is thus accounted for by applying the remark made immediately after~\cite[(24)]{bufl99} to each of the complexes $V^\bullet_{p,v}$ and noting that $\Nrd_{\QQ_p[G]}(1-\phi_v\mid V_{p,v}) = L_v(A,F/k)$.

Now, taking into account the equality~\eqref{comp eq}, the element $j_*\bigl(R\Omega\bigl(h^1(A_{F})(1),\ZZ[G])\bigr)$ is defined in~\cite{bufl99} to be equal to
\[
 \Bigl[\bigl[R\Gamma_c(A_{F/k})\bigr]_{\ZZ_p[G]},\, (\CC_p\otimes_{\RR,j}\vartheta_\infty)\circ(\CC_p\otimes_{\QQ_p}\vartheta_p)^{-1}\Bigr] = \Bigl[\bigl[R\Gamma_c(A_{F/k})\bigr]_{\ZZ_p[G]},\,\alpha'\circ \alpha\Bigr].
\]
The product structure of $\pi_0\bigl(V(\ZZ_p[G])\times_{V(\CC_p[G])}\mathcal{P}_0\bigr)$ then combines with the commutativity of~\eqref{comm diag} to imply that it is also equal to

\begin{align*}
 &\hskip 0.2truein{\Bigl[ \bigl[R\Gamma_c(A_{F/k})\bigr]_{\ZZ_p[G]}, \,\alpha\Bigr]}  + \bigl[\eins_{\ZZ_p[G]},\,\alpha'\bigr]\\
& = \Bigl[ \bigl[C^{\rm exp}_{A,F}[-1]\bigr]_{\ZZ_p[G]},\,\bigl(h_{A,F}^{{\rm exp},j}\bigr)_{\rm Triv}\Bigr]
    + \Bigl[ \bigl[{\rm SC}_S(A_{F/k},\mathcal{X})_p\bigr]_{\ZZ_p[G]},\,\bigl(h^{j}_{A,F}\bigr)_{\rm Triv}\Bigr]
    + \bigl[\eins_{\ZZ_p[G]},\,\alpha'\bigr]\\
& = -\Bigl[\bigl[C^{\rm exp}_{A,F}\bigr]_{\ZZ_p[G]},\,\bigl(h_{A,F}^{{\rm exp},j}\bigr)_{\rm Triv}\Bigr]
   + \Bigl[\bigl[{\rm SC}_S(A_{F/k},\mathcal{X})_p\bigr]_{\ZZ_p[G]},\, \bigl(h^{j}_{A,F}\bigr)_{\rm Triv}\Bigr] \\
& \phantom{= - aa}   + \Bigl[\eins_{\ZZ_p[G]},\,\prod_{v\in (S\cup S_k^p)\setminus S_k^\infty}L_v(A,F/k)\Bigr].
\end{align*}

By using the general result of (\ref{normalisation}) one finds that there are equalities
\[ \Bigl[\bigl[C^{\rm exp}_{A,F}\bigr]_{\ZZ_p[G]},\,\bigl(h_{A,F}^{{\rm exp},j}\bigr)_{\rm Triv}\Bigr] = \chi_{G,p}(C^{\rm exp}_{A,F},h_{A,F}^{{\rm exp},j})\]
and
\[ \Bigl[\bigl[{\rm SC}_S(A_{F/k},\mathcal{X})_p\bigr]_{\ZZ_p[G]},\,\bigl(h^{j}_{A,F}\bigr)_{\rm Triv}\Bigr] =  -\chi_{G,p}({\rm SC}_S(A_{F/k},\mathcal{X})_p,h^j_{A,F}).\]
(The first of these is valid because $\delta_{G,p}\bigr(\Nrd_{\QQ_p[G]}\bigl(-{\id}\bigm\vert \QQ_p\otimes_{\ZZ_p}H^1(C^{\rm exp}_{A,F})\bigr)\bigr)$ vanishes since $\QQ_p\otimes_{\ZZ_p}H^1\bigl(C^{\rm exp}_{A,F}\bigr)\cong\QQ_p\cdot A^t(F_p)^\wedge_p$ is a free $\QQ_p[G]$-module and the second because one has $\prod_{i=1}^{i=2}\Nrd_{\QQ_p[G]}\bigl(-{\id}\bigm\vert \QQ_p\otimes_{\ZZ_p}H^i({\rm SC}_S(A_{F/k},\mathcal{X})_p)\bigr)^{(-1)^i} = 1$.)

In addition, the explicit description of the isomorphism (\ref{virtual iso}) implies that
\begin{align*} \Bigl[\eins_{\ZZ_p[G]},\,\prod_{v\in (S\cup S_k^p)\setminus S_k^\infty}L_v(A,F/k)\Bigr] =& \delta_{G,p}\Bigl(\prod_{v\in (S\cup S_k^p)\setminus S_k^\infty}L_v(A,F/k)\Bigr)\\ =& \sum_{v\in (S\cup S_k^p)\setminus S_k^\infty}\delta_{G,p}(L_v(A,F/k)).\end{align*}

Putting together all of the above computations, we find that

\begin{align*} &j_*\bigl(R\Omega\bigl(h^1(A_{F})(1),\ZZ[G])\bigr)\\ =\, &\sum_{v\in (S\cup S_k^p)\setminus S_k^\infty}\delta_{G,p}(L_v(A,F/k))-\chi_{G,p}(C^{\rm exp}_{A,F},h_{A,F}^{{\rm exp},j})-\chi_{G,p}({\rm SC}_S(A_{F/k},\mathcal{X})_p,h^j_{A,F})\\
=\, &\sum_{v\in (S\cup S_k^p)\setminus S_k^\infty}\delta_{G,p}(L_v(A,F/k))-[H_\infty(\gamma_\bullet)_p,\mathcal{F}(\omega_\bullet)_p;\mu_1]-\chi_{G,p}(\mathcal{F}(\omega_\bullet)_p[0]\oplus\mathcal{X}(p)[-1],\mu_2)\\ & \hskip 1truein -\chi_{G,p}({\rm SC}_S(A_{F/k},\mathcal{X})_p,h^j_{A,F})\\
=\, &\sum_{v\in (S\cup S_k^p)\setminus S_k^\infty}\delta_{G,p}(L_v(A,F/k))-\partial_{G,p}(j_*(\Omega_{\omega_\bullet}(A_{F/k})))
-\chi_{G,p}(\mathcal{F}(\omega_\bullet)_p[0]\oplus\mathcal{X}(p)[-1],\mu_2)\\ &\hskip 1truein -\chi_{G,p}({\rm SC}_S(A_{F/k},\mathcal{X})_p,h^j_{A,F}).
\end{align*}
%
%
%
Since Proposition \ref{prop:perfect2}(i) implies the existence of a canonical isomorphism
$${\rm SC}_{S,\omega_\bullet}(A_{F/k})_p\cong{\rm SC}_S(A_{F/k},\mathcal{X})_p\oplus\mathcal{Q}(\omega_\bullet)_{S,p}[0]$$ in $D^{\rm perf}(\ZZ_p[G])$, the proof of Proposition \ref{etnc} is therefore completed by the following result.

\begin{proposition} One has
\begin{multline}\label{enoughmu2}\chi_{G,p}(\mathcal{F}(\omega_\bullet)_p[0]\oplus\mathcal{X}(p)[-1],\mu_2)\\ =\sum_{v\in S_k^p\setminus S}\delta_{G,p}(L_v(A,F/k))
+\mu_{S}(A_{F/k})_p+\chi_{G,p}(\mathcal{Q}(\omega_\bullet)_{S,p}[0],0).\end{multline}
\end{proposition}

\begin{proof} Our simplifying assumption on the sets $S$ and $\Upsilon$ allow us to deal separately with the cases $p\in S$, $p\in \Upsilon\setminus  S$ and $p\notin \Upsilon$.

If, firstly, $p\in S$, then the terms $\sum_{v\in S_k^p\setminus S}\delta_{G,p}(L_v(A,F/k))$ and $\mu_{S}(A_{F/k})_p$ and module  $\mathcal{Q}(\omega_\bullet)_{S,p}$ all vanish, whilst the $\ZZ_p[G]$-module
\[ \mathcal{X}(p)={\rm exp}_{A^t,F_p}(\mathcal{F}(\omega_\bullet)_p)=\mu_2(\mathcal{F}(\omega_\bullet)_p)\]
is free.

The required equality is therefore true since, in $K_0(\ZZ_p[G],\QQ_p[G])$, one has
\[\chi_{G,p}(\mathcal{F}(\omega_\bullet)_p[0]\oplus\mathcal{X}(p)[-1],\mu_2)= [\mathcal{F}(\omega_\bullet)_p,\mu_2(\mathcal{F}(\omega_\bullet)_p);\mu_2] = 0.\]
%

We assume now that $p$ does not belong to $S$ and set $\mathcal{D}_F(\mathcal{A}^t)_p:=\bigoplus_{v\in S_k^p}\mathcal{D}_F(\mathcal{A}^t_v)$. In this case $\mathcal{X}(p)$ is equal to $A^t(F_p)^\wedge_p$, one has
\[ \sum_{v\in S_k^p\setminus S}\delta_{G,p}(L_v(A,F/k))
+\mu_{S}(A_{F/k})_p = \sum_{v \in S_k^p}\chi_{G,p}\bigl(\kappa_{F_v}^d[0]\oplus\tilde A^t_v(\kappa_{F_v})_p[-1],0\bigr),\]
the module $\mathcal{Q}(\omega_\bullet)_{S,p} = \mathcal{D}_F(\mathcal{A}^t)_p/\mathcal{F}(\omega_\bullet)_{p}$ vanishes if $p \notin \Upsilon$ and the $\ZZ_p[G]$-module $\mathcal{D}_F(\mathcal{A}^t)_p$ is free.


These facts combine to imply that, irrespective of whether $p$ belongs to $\Upsilon$ or not, the stated equality is valid if one has
\begin{equation}\label{almost last}\chi_{G,p}(\mathcal{D}_F(\mathcal{A}^t)_p[0]\oplus A^t(F_p)^\wedge_p[-1],\mu_2)
= \chi_{G,p}\bigl(\kappa_{F_p}^d[0]\oplus\tilde A^t_p(\kappa_{F_p})_p[-1],0\bigr).\end{equation}

To verify this we write $\hat{A^t}$ for the formal group of $A^t$ over $F_p$ (as in \S\ref{p-adic sec}). We also set $$Y_p := \bigoplus_{v \in S_k^p}\Hom_{\mathcal{O}_{k_v}}(H^0(\mathcal{A}_v^t,\Omega^1_{\mathcal{A}_v^t}), \mathcal{O}_{k_v})$$ so that $\mathcal{D}_F(\mathcal{A}^t)_p = \mathcal{O}_{F,p}\otimes_{\mathcal{O}_{k,p}}Y_p$.

Then, since $p$ is unramified in $F/k$,  all modules in the short exact sequences
\[ 0 \to \hat{A^t}(\wp_{F_p}) \to A^t(F_p)^\wedge_p \to \tilde A^t_p(\kappa_{F_p})_p \to 0\]
and
\[ 0 \to \wp_{F_p}\otimes_{\mathcal{O}_{k,p}}Y_p \to \mathcal{D}_F(\mathcal{A}^t)_p \to \kappa^d_{F_p} \to 0\]
are cohomologically-trivial and so the left hand side of (\ref{almost last}) is equal to
\[ \chi_{G,p}((\wp_{F_p}\otimes_{\mathcal{O}_{k,p}}Y_p)[0]\oplus \hat{A^t}(\wp_{F_p})[-1],\mu_2) + \chi_{G,p}\bigl(\kappa_{F_p}^d[0]\oplus\tilde A_p^t(\kappa_{F_p})_p[-1],0\bigr)\]

It is therefore enough to show that the first term in this sum vanishes and, since $p$ is unramified in $F/k$, this is proved by the argument of Proposition \ref{explicit log resolve}. \end{proof}



\section{Poitou-Tate duality}\label{ptduality}

In this section we review the constructions and results from Poitou-Tate duality that were used in both \S \ref{tmc} and \S \ref{comparison section}.

We fix an odd prime $p$ and a finite Galois extension $F/k$ of number fields with Galois group $G$. We also fix a finite set $\Sigma$ of non-archimedean places of $k$ containing all $p$-adic places and all that ramify in $F$.

By abuse of notation, we shall sometimes simply denote by $\Sigma$ the set $\Sigma(F)$ of places of $F$ lying above those in $\Sigma$. If $L$ denotes either $k$ or $F$, we 
write $G_{L,\Sigma}$ for the Galois group of the maximal Galois extension $L_\Sigma$ of $L$ which is unramified outside $\Sigma\cup S_L^\infty$.
For convenience, we abbreviate the ring $\mathcal{O}_{L,\Sigma\cup S_L^\infty}$ to $U_L$.

We also write $\mathcal{O}^\times_{\Sigma}$ for the group of units in the ring of integers of $F_\Sigma$, and $I_\Sigma$ for the direct limit of the $(\Sigma\cup S_F^\infty)$-id\`ele groups of all finite extensions of $F$ inside $F_\Sigma$. One then has a canonical short exact sequence of $G_{F,\Sigma}$-modules \begin{equation}\label{cft}0\to \mathcal{O}^\times_{\Sigma}\to I_{\Sigma}\to C_{\Sigma}\to 0
\end{equation} as in \cite[Ch. VIII, \S 3]{nsw}. (In accordance with the conventions of loc. cit., we only consider $I_{\Sigma}$ and $C_{\Sigma}$ as discrete $G_{F,\Sigma}$-modules).

Whenever we consider a fixed abelian variety defined over $k$ and the action of $G_{k,\Sigma}$ on a group of points, we do so under the additional assumption that $\Sigma$ contains all places of bad reduction.

\subsection{Explicit complexes}\label{expcohcomp}


In this section we recall an explicit description of both local and global \'etale cohomology complexes which will be useful in the Poitou-Tate duality constructions of the next section.

Again we write $Y_{F/k}$ for the module $\prod\ZZ$, with the product running over all $k$-embeddings $F\to k^c$, endowed with its natural action of $G\times G_k$.

Given a non-archimedean place $v$ of $k$ and a topological $\ZZ_p$-module $M$ on which $G_{v}$ acts continuously, we set $$M_F:=Y_{F/k,p}\otimes_{\ZZ_p}M,$$ where $G$ acts naturally on the first factor and $G_v$ acts diagonally.

The local \'etale cohomology complex $R\Gamma(k_v,M_F)$ defines an object of $D(\ZZ_p[G])$ which may be represented as an explicit complex of $\ZZ_p[G]$-modules as follows: for $i\geq 0$ and any place $w'$ in $S_F^v$, we let $R\Gamma^{i}(F_{w'},M)$ denote the group of inhomogeneous $i$-cochains of $G_{w'}=G_{F_{w'}^c/F_{w'}}$ with coefficients in $M$. 
The usual differentials $$d^i_{w'}:R\Gamma^{i}(F_{w'},M)\to R\Gamma^{i+1}(F_{w'},M)$$ between the respective groups of inhomogeneous cochains then induce differentials $d^{i}_v:=\prod_{w'\in S_F^v}d^{i}_{w'}$ in the explicit complex of $\ZZ_p[G]$-modules
\begin{equation}\label{localinhom}\ldots\to \prod_{w'\in S_F^v}R\Gamma^i(F_{w'},M)\stackrel{d^i_v}{\to}\prod_{w'\in S_F^v}R\Gamma^{i+1}(F_{w'},M)\to\ldots,\end{equation}
which represents $R\Gamma(k_v,M_F)$.

If instead we consider the analogous situation for an archimedean place $v$ of $k$, the object $R\Gamma(k_v,M_F)$ of $D(\ZZ_p[G])$ may simply be represented by \begin{equation}\label{archimedean}H^0(k_v,M_F)[0]\end{equation} since $p$ is odd, so in the sequel we identify these complexes.

We now let $M$ be any topological $\ZZ_p$-module on which $G_{k,\Sigma}$ acts continuously. We again set $M_F:=Y_{F/k,p}\otimes_{\ZZ_p}M,$ where $G$ acts naturally on the first factor and $G_{k,\Sigma}$ acts diagonally.

The global \'etale cohomology complex $R\Gamma(U_k,M_F)$ defines an object of $D(\ZZ_p[G])$ which, given our restrictions on the choice of set $\Sigma$, may be represented as an explicit complex of $\ZZ_p[G]$-modules by simply placing the group $R\Gamma^{i}(U_F,M)$ of inhomogeneous $i$-cochains of $G_{F,\Sigma}$ with coefficients in $M$ in degree $i$ (for $i\geq 0$), and letting $d^{i}$ denote the usual differentials.

The following explicit definition is consistent with our use of compactly supported \'etale cohomology throughout the rest of the article.

\begin{definition}\label{compactsupport}{\em Let $M$ be a topological $\ZZ_p$-module on which $G_{k,\Sigma}$ acts continuously.

We define a complex of $\ZZ_p[G]$-modules
$R\Gamma_{\Sigma}(\mathbb{A}_F,M)$ as the direct sum of the complexes (\ref{localinhom}) over all places $v$ in $\Sigma$ and the complexes (\ref{archimedean}) over all places $v$ in $S_k^\infty$.

We then also define a complex of $\ZZ_p[G]$-modules $R\Gamma_c(U_F,M)$ as the $-1$-shift of the explicit mapping cone of the natural
localisation map from the complex $R\Gamma(U_k,M_F)$
to the complex
$R\Gamma_{\Sigma}(\mathbb{A}_F,M)$.
We denote by $d^i_c$ the differential of $R\Gamma_c(U_F,M)$ in degree $i$.

Namely, the complex $R\Gamma_c(U_F,M)$ is defined by placing the module $$R\Gamma_c^0(U_F,M)=R\Gamma^0(U_F,M)=M$$ in degree zero, the module
$$R\Gamma_c^1(U_F,M):=\bigoplus_{v\in S_k^\infty}H^0(k_v,M_F)\oplus\bigl(\bigoplus_{w'\in\Sigma(F)}M\bigr)\oplus R\Gamma^1(U_F,M)$$ in degree 1 and the module
$$R\Gamma_c^i(U_F,M):=\bigl(\bigoplus_{w'\in\Sigma(F)}R\Gamma^{i-1}(F_{w'},M)\bigr)\oplus R\Gamma^i(U_F,M)$$ in degree $i\geq 2$.

The differential in degree $i\geq 1$ is given by
$$d_c^{i}=\left(\begin{array}{cc}
-\bigoplus_{v\in \Sigma}d^{i-1}_v & -(\lambda^{i}_{w'})_{w'\in \Sigma(F)}
\\
0 & d^{i}
\end{array}\right)$$ with each $\lambda_{w'}^{i}$ denoting the corresponding localisation map at $w'$ in degree $i$.
}
\end{definition}

The following compatibility property of inhomogeneous cochains with respect to projections is used in the main body of the article.

\begin{lemma}\label{lifts} Let $A$ be an abelian variety defined over $k$ and let $m$ be a natural number. Then the following claims are valid.
\begin{itemize}
\item[(i)] For any $i\geq 0$ and any non-archimedean place $w'$ of $F$ we use the following natural projection maps
\begin{align*}
 \mu_{w'} : &R\Gamma^i(F_{w'},T_p(A))\to R\Gamma^i(F_{w'},A[p^m]),\\
 \nu_{w'} : &R\Gamma^i(F_{w'},T_p(A))\to \ZZ/p^m\otimes_{\ZZ_p}R\Gamma^i(F_{w'},T_p(A)),\\
\mu_\Sigma : &R\Gamma^i(U_F,T_p(A))\to R\Gamma^i(U_F,A[p^m]),\\
\nu_\Sigma : &R\Gamma^i(U_F,T_p(A))\to \ZZ/p^m\otimes_{\ZZ_p}R\Gamma^i(U_F,T_p(A)).\end{align*}

Then there exist homomorphisms $$l_{w'}:R\Gamma^i(F_{w'},A[p^m])\to \ZZ/p^m\otimes_{\ZZ_p}R\Gamma^i(F_{w'},T_p(A))$$ and $$l_\Sigma:R\Gamma^i(U_F,A[p^m])\to \ZZ/p^m\otimes_{\ZZ_p}R\Gamma^i(U_F,T_p(A))$$ for which $l_{w'}\circ\mu_{w'}=\nu_{w'}$ and $l_\Sigma\circ\mu_\Sigma=\nu_\Sigma$.

\item[(ii)] If $w'$ belongs to $\Sigma(F)$ then the maps $l_\Sigma$ and $l_{w'}$ commute with localisation at $w'$. In addition, they restrict to induce homomorphisms
\begin{align*}
l_{w'} : &\ker(d^1_{w'})\to \ker(\ZZ/p^m\otimes_{\ZZ_p}d^1_{w'}),\\
l_{w'} : &\im(d^0_{w'})\to \im(\ZZ/p^m\otimes_{\ZZ_p}d^0_{w'}),\\
l_\Sigma : &\ker(d^1)\to \ker(\ZZ/p^m\otimes_{\ZZ_p}d^1),\\
l_\Sigma : &\im(d^0)\to \im(\ZZ/p^m\otimes_{\ZZ_p}d^0).\end{align*}
\end{itemize}
\end{lemma}
\begin{proof} To prove claim (i) it is enough to show that given an element $\rho$ of $R\Gamma^i(F_{w'},A[p^m])$, respectively of $R\Gamma^i(U_F,A[p^m])$, the element $1\otimes\rho'$ of $\ZZ/p^m\otimes_{\ZZ_p}R\Gamma^i(F_{w'},T_p(A))$, respectively of $\ZZ/p^m\otimes_{\ZZ_p}R\Gamma^i(U_F,T_p(A))$, is independent of the choice of continuous lift $\rho'$ of $\rho$ through the canonical projection $$\pi:T_p(A)\to A[p^m]$$ (such lifts exist since $A[p^m]$ is a discrete space). We let $\Gamma$ denote either $G_{w'}^i$ or $G_{F,\Sigma}^i$, so that in either case $\rho$ is a continuous function $\Gamma\to A[p^m]$, let $\rho'$ and $\rho''$ denote continuous lifts of $\rho$ through $\pi$ and proceed to prove that $1\otimes\rho'=1\otimes\rho''$.

Since $\pi\circ\rho'=\rho=\pi\circ\rho''$ we have $\pi\circ(\rho'-\rho'')=0$. It follows that $$(\rho'-\rho'')(\gamma)\in\ker(\pi)= p^m\cdot T_p(A)$$ for all $\gamma\in\Gamma$ and, since $T_p(A)$ is $\ZZ_p$-free, there exists a unique $t_\gamma\in T_p(A)$ for which $$(\rho'-\rho'')(\gamma)= p^mt_\gamma.$$ We then obtain a continuous function $t:\Gamma\to T_p(A)$ by setting $t(\gamma):=t_\gamma$. Then \begin{equation}\label{previousparagraph}1\otimes(\rho'-\rho'')=1\otimes p^mt=p^m\otimes t=0\otimes t,\end{equation} so $1\otimes\rho'=1\otimes\rho''$ as required.

To consider claim (ii) we first note that it is clear from their definition that $l_\Sigma$ and $l_{w'}$ commute with localisation at $w'$.

To continue the proof of claim (ii) we consider the global case and proceed to prove that $l_\Sigma$ restricts to induce the claimed homomorphisms.

We first let $c\in R\Gamma^1(U_F,A[p^m])$ be a 1-cocycle and we fix a lift $\alpha$ of $c$ through $\pi$. Using the explicit formula for the differentials $d^1$ one easily finds that $d^1(\pi\circ\alpha)=\pi\circ d^1(\alpha)$ and hence also that
$$\pi\circ d^1(\alpha)=d^1(\pi\circ\alpha)=d^1(c)=0=\pi\circ 0.$$
The argument from (\ref{previousparagraph}) therefore implies that $1\otimes d^1(\alpha)=1\otimes 0=0$ and thus we finally find that
$$(\ZZ/p^m\otimes d^1)(l_\Sigma(c))=(\ZZ/p^m\otimes d^1)(1\otimes\alpha)=1\otimes d^1(\alpha)=0.$$
The map $l_\Sigma$ therefore restricts to induce a homomorphism
$$\ker(d^1)\to \ker(\ZZ/p^m\otimes_{\ZZ_p}d^1),$$ as required.

We next let $z$ be an element of $R\Gamma^0(U_F,A[p^m])=A[p^m]$ and we fix $\beta\in R\Gamma^0(U_F,T_p(A))=T_p(A)$ with $\pi(\beta)=z$. Using the explicit formula for the differentials $d^0$ one easily finds that $\pi\circ d^0(\beta)=d^0(\pi(\beta))=d^0(z)$ and therefore $$l_\Sigma(d^0(z))=1\otimes d^0(\beta)=(\ZZ/p^m\otimes d^0)(1\otimes\beta).$$
The map $l_\Sigma$ therefore restricts to induce a homomorphism
$$\im(d^0)\to \im(\ZZ/p^m\otimes_{\ZZ_p}d^0),$$ as required.

We finally note that the precise same arguments hold in the local case after replacing $d^i$ and $l_\Sigma$ by $d^i_{w'}$ and $l_{w'}$ respectively. This completes the proof of claim (ii).
\end{proof}




\subsection{Compatibility of pairings}\label{b2}

We now review the Poitou-Tate duality constructions that will be relevant to our comparison of height pairings given in \S \ref{comparison section}.

In this section we let $M$ be a finite $\ZZ_p$-module on which $G_{k,\Sigma}$ acts continuously. In our applications in the main body of the article, we will always have $M=A[p^m]$ for an abelian variety $A$ defined over $k$ and some natural number $m$.

Since  $\mathcal{O}^\times_{\Sigma}$ is $p$-divisible by \cite[(8.3.4)]{nsw}, the exact sequence (\ref{cft}) induces a canonical short exact sequence of $G_{F,\Sigma}$-modules
\begin{equation}\label{pdiv}0\to M'\to I_{\Sigma}(M){\to} C_{\Sigma}(M)\to 0
\end{equation} with $M':={\rm Hom}_\ZZ(M,\mathcal{O}^\times_{\Sigma})$, $I_{\Sigma}(M):={\rm Hom}_\ZZ(M,I_{\Sigma})$ and $C_{\Sigma}(M):={\rm Hom}_\ZZ(M,C_{\Sigma})$.

The following result can be obtained by specialising \cite[Ch. II, Prop. 4.13]{milne}, but we choose to make the construction explicit instead.

\begin{proposition}\label{shap}
There is a canonical isomorphism $$s_M:H^1(U_F,C_{\Sigma}(M))\to H^2_c(U_F,M')$$ which makes the following diagram commute:
\begin{equation*}\label{shapdiagram}\xymatrix@C=1em{
H^1(U_F,M')\ar@{=}[d] \ar[r] & H^1(U_F,I_{\Sigma}(M))\ar[d]_-{\wr}^-{{\rm sh}_M^1} \ar[r] &  H^1(U_F,C_{\Sigma}(M))\ar[d]_-{\wr}^-{s_M} \ar[r]^-{\delta_M} & H^2(U_F,M')\ar[d]_-{\wr}^-{-1} \ar[r] & H^2(U_F,I_{\Sigma}(M))\ar[d]_-{\wr}^-{-{\rm sh}_M^2}\\
H^1(U_F,M') \ar[r] & H^1_{\Sigma}(\mathbb{A}_F,M') \ar[r] &  H^2_c(U_F,M') \ar[r] & H^2(U_F,M') \ar[r] & H^2_{\Sigma}(\mathbb{A}_F,M').
}\end{equation*}
Here the top row is the long exact cohomology sequence sequence associated to (\ref{pdiv}), the bottom row is the long exact cohomology sequence of the canonical exact triangle
$$R\Gamma_c(U_F,M')\to R\Gamma(U_F,M')\to R\Gamma_\Sigma(\mathbb{A}_F,M')\to R\Gamma_c(U_F,M')[1]$$
in $D(\ZZ_p[G])$ (determined by Definition \ref{compactsupport}), and ${\rm sh}_M^1$ and ${\rm sh}_M^2$ are the (bijective) Shapiro maps (see (8.5.4), (8.5.5) and the following Remarks in \cite{nsw}).
\end{proposition}
\begin{proof} The commutativity of the leftmost and rightmost squares of the diagram is an immediate consequence of the definition of the Shapiro maps, and the bijectivity of the map $s_M$ will follow from the Five lemma once we define it in such a way as to make the remaining squares commute.

In order to define the map $s_M$, let $\theta:G_{F,\Sigma}\to C_{\Sigma}(M)$ be an inhomogeneus 1-cocycle. Since $C_{\Sigma}$ is a discrete space, we may and will fix a continuous lift $\theta':G_{F,\Sigma}\to I_{\Sigma}(M)$ of $\theta$ through the canonical projection $$\pi:I_\Sigma(M)\to C_\Sigma(M)$$ occurring in the exact sequence (\ref{pdiv}).

Then $\pi\circ d^1(\theta')=d^1(\pi\circ\theta')=d^1(\theta)=0$, so $$d^1(\theta'):G_{F,\Sigma}^2\to I_{\Sigma}(M)$$ takes its values in $M'$ and in fact defines a $2$-cocycle $d^1(\theta')'$ in $R\Gamma^2(U_F,M')$ which, by definition of the connecting homomorphism $\delta_M$, satisfies \begin{equation}\label{rightrect}\delta_M(\overline{\theta})=\overline{d^1(\theta')'}\end{equation}
in $H^2(U_F,M')$.

Furthermore, given a place $w'\in\Sigma(F)$, the localisation $d^1(\theta')'_{w'}:=\lambda_{w'}(d^1(\theta')')$ of $d^1(\theta')'$ in $R\Gamma^2(F_{w'},M')$ may be computed as \begin{equation}\label{2coc}d^1(\theta')'_{w'}=d^1_{w'}((\theta')_{w'}'),\end{equation}
where the continuous function $(\theta')_{w'}'$ is defined by the following composition: \begin{equation}\label{compo}(\theta')_{w'}':G_{w'}\to G_{F,\Sigma}\stackrel{\theta'}{\to}I_{\Sigma}(M)\to\Hom_\ZZ(M,(F_{w'}^c)^\times)=\Hom_\ZZ(M,\mu_{p^m})=M'.\end{equation} Here the third arrow is induced by the projection map $I_{\Sigma}\to (F_{w'}^c)^\times$, $p^m$ is the exponent of $M$, $\mu_{p^m}$ denotes the group of $p^m$-th roots of unity in $F_\Sigma$ and the two identifications are induced by the canonical map $\mathcal{O}_{\Sigma}^\times\to  (F_{w'}^c)^\times$ together with the assumption that $\Sigma$ contains all $p$-adic places of $k$.

The equality (\ref{2coc}) therefore combines with the fact that $d^1(\theta')'$ is a 2-cocycle in $R\Gamma^2(U_F,M')$ to imply that the element \begin{equation}\label{2cocycle}(((\theta')_{w'}')_{w'\in\Sigma(F)},-d^1(\theta')')\end{equation} of $R\Gamma_c^2(U_F,M')$ is in fact a 2-cocycle. Its class in $H_c^2(U_F,M')$ is easily seen to be independent of the choice of lift $\theta'$ of $\theta$ so, in order to define the map $s_M$ by the equality $$s_M(\overline{\theta}):=\overline{(((\theta')_{w'}')_{w'\in\Sigma(F)},-d^1(\theta')')},$$ we simply need to verify that if $\theta$ is a 1-coboundary then (\ref{2cocycle}) is a 2-coboundary. Once we do this, the commutativity of the diagram will be clear: indeed, (\ref{rightrect}) will imply that the third square commutes, while the second square will commute because, in $H^1_\Sigma(\mathbb{A}_F,M')$, one has \begin{equation}\label{computeshap}{\rm sh}_M^1(\overline{\rho})=(\overline{\rho_{w'}'})_{w'\in\Sigma(F)}\end{equation} for any 1-cocycle $\rho$ in $R\Gamma^1(U_F,I_{\Sigma}(M))$ (with each $\rho_{w'}'$ obtained from $\rho$ through the composition (\ref{compo})).

So we finally assume that $\theta$ is a 1-coboundary and proceed to verify that the element (\ref{2cocycle}) of $R\Gamma_c^2(U_F,M')$ is a 2-coboundary. In this case one has $\theta(g)=(g-1)a$ for some $a\in C_{\Sigma}(M)$ and for every $g\in G_{F,\Sigma}$. We fix any pre-image $a'\in I_{\Sigma}(M)$ of $a$ through $\pi$ and define a lift $\theta'$ of $\theta$ through $\pi$ by setting $\theta'(g)=(g-1)a'$ for every $g\in G_{F,\Sigma}$. Then $d^1(\theta')'=d^1(\theta')=0$ so it is enough to prove that each $(\theta')_{w'}'$ is a 1-coboundary in $R\Gamma^1(F_{w'},M')$. But, by (\ref{computeshap}), one has $$(\overline{(\theta')_{w'}'})_{w'\in\Sigma(F)}={\rm sh}_M^1(\overline{\theta'})={\rm sh}_M^1(0)=0$$
in $H^1_\Sigma(\mathbb{A}_F,M')$, as required. 


\end{proof}

In the next result, we let $$u_M:H^1_{\Sigma}(\mathbb{A}_F,M'){\to} H^1_{\Sigma}(\mathbb{A}_F,M)^\vee$$ denote the canonical isomorphism induced by the pairing
\begin{equation}\label{localpairing}\langle\,,\rangle_{\Sigma}:H^1_{\Sigma}(\mathbb{A}_F,M')\times H^1_{\Sigma}(\mathbb{A}_F,M)\to\QQ_p/\ZZ_p\end{equation} defined as the sum over all places $w'\in\Sigma(F)$ of the perfect, Galois-equivariant, local Tate duality pairings $$H^1(F_{w'},M')\times H^1(F_{w'},M)\to\QQ_p/\ZZ_p.$$

We also have a canonical isomorphism
$$w_M:H^1(U_F,C_{\Sigma}(M)){\to}H^1(U_F,M)^\vee$$
induced by a perfect global duality pairing (see \cite[Thm. (10.11.8)]{nsw}).

Poitou-Tate duality now leads to the following result for the map $s_M$ constructed in Proposition \ref{shap}.

\begin{corollary}\label{Mfinite} The canonical isomorphism $s_M$ of Proposition \ref{shap} makes the following diagram (with bijective vertical arrows) commute:

\begin{equation*}\label{theMfdiagram}\xymatrix{
H^1_{\Sigma}\bigl(\mathbb{A}_F,M'\bigr) \ar[dd]^-{u_M} \ar[r] &
H^2_c\bigl(U_F,M'\bigr) 
\\
 &
H^1\bigl(U_F,C_{\Sigma}(M)\bigr) \ar[d]^-{w_M} \ar[u]_-{s_M} 
\\
H^1_{\Sigma}\bigl(\mathbb{A}_F,M\bigr)^\vee \ar[r] &
H^1\bigl(U_F,M\bigr)^\vee. 
}\end{equation*}

Here the bottom horizontal arrow is the dual of the localisation map induced on degree 1 cohomology.
\end{corollary}
\begin{proof}
The result is immediate upon combining Proposition \ref{shap} with the commutative diagram in \cite[Lem. (8.6.6)]{nsw}.
\end{proof}

\subsection{Abelian varieties}
In this section we specialise the general results of the previous section to the arithmetic setting of interest in this article: namely, for a given abelian variety $A$ defined over $k$ and any natural number $m$, the finite module $A[p^m]$ will play the role of $M$. To do so we assume that our fixed set $\Sigma$ contains all places at which $A$ has bad reduction.

In this case, for any natural number $m$, the dual module $A[p^m]'$ canonically identifies with $A^t[p^m]$.
The Mazur-Tate height pairing was in \S \ref{comparison section}, via the work of Bertolini and Darmon and of Tan, reinterpreted in terms of the local duality maps 
\begin{equation}\label{um}u_m:=u_{A[p^m]}:H^1_{\Sigma}(\mathbb{A}_F,A^t[p^m]){\to} H^1_{\Sigma}(\mathbb{A}_F,A[p^m])^\vee.\end{equation}
We also set $s_m:=s_{A[p^m]}$ for the canonical isomorphism defined in Proposition \ref{shap} and $w_m:=w_{A[p^m]}$ for the global duality map corresponding to our specific choice of module $M=A[p^m]$.

Furthermore, for every $w'\in\Sigma(F)$, the images of
$$A^t(F_{w'})/p^mA^t(F_{w'})$$ and of $$A(F_{w'})/p^mA(F_{w'})$$ in $H^1_{\Sigma}(\mathbb{A}_F,A^t[p^m])$ and in $H^1_{\Sigma}(\mathbb{A}_F,A[p^m])$ respectively, under the respective local Kummer maps, are the exact annihilators of each other under the corresponding local duality pairing.

Setting \[A^t_\Sigma\bigl(\mathbb{A}_F\bigr)/p^m:=\prod_{w'\in \Sigma(F)}A^t(F_{w'})/p^mA^t(F_{w'})\] and
$$H^1_{\Sigma}(\mathbb{A}_F,A):=\prod_{w'\in \Sigma(F)}H^1(F_{w'},A)$$ and canonically identifying $H^1_{\Sigma}(\mathbb{A}_F,A)[p^m]$ with the cokernel of the product of the local Kummer maps,
the map $u_m$ therefore restricts to an isomorphism $$t_m:A^t_\Sigma\bigl(\mathbb{A}_F\bigr)/p^m\stackrel{\sim}{\to}H^1_{\Sigma}(\mathbb{A}_F,A)[p^m]^\vee.$$

We finally let $$\kappa_m:A^t_\Sigma\bigl(\mathbb{A}_F\bigr)/p^m\to H^1_{\Sigma}\bigl(\mathbb{A}_F,A^t[p^m]\bigr)$$ and $$\kappa_m':H^1_{\Sigma}\bigl(\mathbb{A}_F,A\bigr)[p^m]^\vee\to H^1_{\Sigma}\bigl(\mathbb{A}_F,A[p^m]\bigr)^\vee$$ denote the maps induced by the respective local Kummer sequences.

Corollary \ref{Mfinite} then specialises to give the following useful result:

\begin{corollary}\label{Tatepoitouexplicit}
The following diagram (with bijective vertical maps) commutes.
\begin{equation*}\label{thepmdiagram}\xymatrix@C=1em{
A^t_\Sigma\bigl(\mathbb{A}_F\bigr)/p^m \ar[dd]^-{t_m} \ar[r]^-{\kappa_m} &
H^1_{\Sigma}\bigl(\mathbb{A}_F,A^t[p^m]\bigr) \ar[dd]^-{u_m} \ar[r] &
H^2_c\bigl(U_F,A^t[p^m]\bigr) 
\\
 &
 &
H^1\bigl(U_F,C_{\Sigma}(A[p^m])\bigr) \ar[d]^-{w_m} \ar[u]_-{s_m} 
\\
H^1_{\Sigma}(\mathbb{A}_F,A)[p^m]^\vee \ar[r]^-{\kappa'_m} &
H^1_{\Sigma}\bigl(\mathbb{A}_F,A[p^m]\bigr)^\vee \ar[r] &
H^1\bigl(U_F,A[p^m]\bigr)^\vee. 
}\end{equation*}
Here the unlabeled bottom horizontal arrow is the dual of the localisation map induced on degree 1 cohomology.
\end{corollary}

We end this section by noting that, as $m$ ranges over the natural numbers, the third column of the diagram in Corollary \ref{Tatepoitouexplicit}  defines morphisms of inverse systems, with the obvious transition maps: $$H^2_c\bigl(U_F,A^t[p^{m+1}]\bigr)\to H^2_c\bigl(U_F,A^t[p^m]\bigr)$$ induced by multiplication by $p$ on $A^t[p^{m+1}]$, $$H^1\bigl(U_F,C_{\Sigma}(A[p^{m+1}])\bigr)\to H^1\bigl(U_F,C_{\Sigma}(A[p^m])\bigr) \text{ and }H^1\bigl(U_F,A[p^{m+1}]\bigr)^\vee\to H^1\bigl(U_F,A[p^m]\bigr)^\vee$$ induced by the inclusion $A[p^m]\subseteq A[p^{m+1}]$. The compatibility of the maps $w_m$ and $s_m$ with these transition maps may be verified directly from their definitions.

We therefore obtain isomorphisms
\begin{multline}\label{themapw}w:=\varprojlim_m w_m:H^1\bigl(U_F,C_{\Sigma}(A[p^\infty])\bigr)=\varprojlim_m H^1\bigl(U_F,C_{\Sigma}(A[p^m])\bigr)\\ \to\varprojlim_m H^1\bigl(U_F,A[p^m]\bigr)^\vee =H^1\bigl(U_F,A[p^\infty]\bigr)^\vee\end{multline} and
\begin{multline}\label{themaps}s:=\varprojlim_m s_m:H^1\bigl(U_F,C_{\Sigma}(A[p^\infty])\bigr)=\varprojlim_m H^1\bigl(U_F,C_{\Sigma}(A[p^m])\bigr)\\ \to\varprojlim_m H^2_c\bigl(U_F,A^t[p^m]\bigr) =H^2_c\bigl(U_F,T_p(A^t)\bigr).\end{multline} Here we have used the fact that the canonical map $H^2_c\bigl(U_F,T_p(A^t)\bigr)\to\varprojlim_m H^2_c\bigl(U_F,A^t[p^m]\bigr)$ is bijective, as an easy consequence of the Five Lemma, in order to make the final identification.

\section{Bockstein homomorphisms}\label{bocksection}

In this section we review the explicit construction of the algebraic height pairing that was used to define the Nekov\'a\v r height pairing in \S \ref{comparison section}. This construction recovers that of \cite[\S 3.4.1]{bst} and originates with
the general formalism of algebraic height pairings introduced by Nekov\'a\v r \cite[\S11]{nek}.

We fix an odd prime $p$ and allow $Z$ to denote either the ring $\ZZ_p$, or the ring $\ZZ/p^m\ZZ$ for some fixed natural number $m$. We also fix a finite group $G$, set $R:=Z[G]$ and denote by $I_R$ the augmentation ideal in $R$.

We denote by $m_p(G)$ the logarithm in base $p$ of the exponent of the finite abelian $p$-group $\ZZ_p\otimes_{\ZZ}G^{\rm ab}$. Given an $R$-module $M$, we write $M^*$ for the linear dual $\Hom_{Z}(M,Z)$, endowed with the natural contragredient action of $G$.

In the sequel we shall say that an $R$-module is `perfect' if it is both finitely generated and $G$-cohomologically-trivial.

Let $C$ be a bounded complex of perfect $R$-modules that is acyclic outside degrees 1 and 2.
We set $C_G:=Z\otimes_{R}^{\mathbb{L}}C$.
Then the action of $\Tr_G:=\sum_{g\in G}g\in R$
induces a canonical isomorphism $T_C:H^1(C_G)\cong H^1(C)^G$.

There is also a canonical `Bockstein homomorphism' of $R$-modules
\begin{equation*}\label{bock}\beta_C:H^1(C)^G\stackrel{T_C^{-1}}{\to}H^1(C_G)\stackrel{\delta}{\to} H^2(I_R\otimes_{R}^{\mathbb{L}}C)\cong I_R\otimes_{R}H^2(C)\to I_R/I_R^{2}\otimes_Z H^2(C)_G.\end{equation*} Here $\delta$ is the connecting homomorphism arising from the canonical exact triangle
$$I_R\otimes_{R}^{\mathbb{L}}C\to C\to C_G\to (I_R\otimes_{R}^{\mathbb{L}}C)[1]$$ in $D(R)$, the unlabeled isomorphism is the canonical one (induced by the fact that $C$ is acyclic in degrees greater than 2), and the last arrow is induced by passing to $G$-coinvariants.

Let now $M$ and $N$ be $R$-modules and $r_1:M\to H^1(C)$ and $r_2:H^2(C)\to N$ be $R$-homomorphisms. Then $\beta_C$ induces a composite homomorphism of $R$-modules
\begin{equation*}\label{modifiedbock}\beta_{C,r_1,r_2}:M^G\stackrel{r_1^G}{\to}H^1(C)^G\stackrel{\beta_C}{\to}I_R/I_R^2\otimes_Z H^2(C)_G\stackrel{{\rm id}\otimes_Z r_{2,G}}{\to}I_R/I_R^2\otimes_Z N_G\end{equation*} and hence also a `Bockstein pairing' \begin{equation*}\label{bockpairing}\langle\,,\rangle_{C,r_1,r_2}:M^G\otimes_Z (N_G)^*\to I_R/I_R^2\cong Z\otimes_{\ZZ}G^{\rm ab},\end{equation*} where the isomorphism is induced by mapping the class of each element $g-1$ with $g$ in $G$ to the image of $g$ in $G^{\rm ab}$.

The construction of Bockstein pairings is natural in the following sense: if now $C$ is a complex of $R$-modules, acyclic outside degrees 1 and 2, that is isomorphic in $D(R)$ to a bounded complex of perfect $R$-modules, then such an isomorphism associates a well-defined Bockstein homomorphism, and thus also Bockstein pairing, to $C$.





The following result establishes the compatibility of our construction with respect to a natural change of ring.
\smallskip

\begin{lemma}\label{reducedbock} Let $C\in D^{\rm perf}(\ZZ_p[G])$ be acyclic outside degrees 1 and 2, and have the property that $H^1(C)$ is $\ZZ_p$-free.

We let $m$ denote any integer greater than or equal to $m_p(G)$ and define a complex of $(\ZZ/p^m)[G]$-modules by setting $C/p^m:=(\ZZ/p^m)\otimes_{\ZZ_p}^{\mathbb{L}}C.$ We also write $q:C\to C/p^m$ for the canonical projection morphism. Then the following claims hold.

\begin{itemize}\item[(i)] $C/p^m$ is isomorphic in $D((\ZZ/p^m)[G])$ to a bounded complex of perfect $(\ZZ/p^m)[G]$-modules and is acyclic outside degrees 1 and 2.
The map $H^1(q)$ induces a map $q_{1}$ fitting in a canonical short exact sequence of $(\ZZ/p^m)[G]$-modules $$0\to H^1(C)/p^m\stackrel{q_{1}}{\to} H^1(C/p^m)\to H^2(C)[p^m]\to 0,$$ while $H^2(q)$ induces a canonical isomorphism $\,q_{2}:H^2(C)/p^m\stackrel{\sim}{\to}H^2(C/p^m).$

\item[(ii)] For any $\ZZ_p[G]$-homomorphisms $r_1:M\to H^1(C)$ and $r_2:H^2(C)\to N$ and any elements $x\in M^G$ and $y\in (N_G)^*$ one has an equality $$\langle x,y\rangle_{C,r_1,r_2}=\langle x+p^m,y/p^m\rangle_{C/p^m,q'_{1},q'_{2}}$$ in $\ZZ_p\otimes_{\ZZ}G^{\rm ab}=\ZZ/p^m\otimes_{\ZZ}G^{\rm ab}$. In this equality, $x+p^m$ denotes the class of $x$ in $(M/p^m)^G$, $y/p^m$ denotes the element of $((N/p^m)_G)^*$ induced by $y$ after identifying $(N/p^m)_G$ with $(N_G)/p^m$, and $q'_{1}$ and $q'_{2}$ are the composite maps $$q'_{1}:M/p^m\stackrel{r_1/p^m}{\to}H^1(C)/p^m\stackrel{q_{1}}{\to} H^1(C/p^m)$$ and $$q'_{2}:H^2(C/p^m)\stackrel{q_{2}^{-1}}{\to}H^2(C)/p^m\stackrel{r_2/p^m}{\to} N/p^m.$$
\end{itemize}
\end{lemma}
\begin{proof}
The canonical exact triangle $C\stackrel{p^m}{\to}C\stackrel{q}{\to} C/p^m\to C[1]$ in $D^{\rm perf}(\ZZ_p[G])$ has long exact cohomology sequence
\begin{multline*}0\to H^0(C/p^m)\to H^1(C)\stackrel{p^m}{\to}H^1(C)\stackrel{H^1(q)}{\to} H^1(C/p^m) \\
\to H^2(C)\stackrel{p^m}{\to}H^2(C)\stackrel{H^2(q)}{\to} H^2(C/p^m)\to 0.\end{multline*}

Since $H^1(C)$ is assumed to be $\ZZ_p$-free, it is clear that $H^0(C/p^m)$ vanishes (so that $C/p^m$ is indeed acyclic outside degrees 1 and 2) and also that $H^1(q)$ and $H^2(q)$ induce maps with the properties described in claim (i).

Let us next note that, by a standard resolution argument, 
$C$ is isomorphic in $D(\ZZ_p[G])$ to a complex of the form $F^1\stackrel{\phi}{\to} F^2$, where $F^1$ and $F^2$ are finitely generated, free $\ZZ_p[G]$-modules and the first term is placed in degree one.

The complex $C/p^m$ is therefore isomorphic in $D((\ZZ/p^m)[G])$ to the complex of $(\ZZ/p^m)[G]$-modules that is equal to $F^1/p^m\stackrel{\phi/p^m}{\to}F^2/p^m,$ where the first term is placed in degree one. This completes the proof of claim (i).

Furthermore, claim (ii) now follows from a straightforward computation of the Bockstein homomorphisms associated to both $C$ and $C/p^m$ with respect to the given representatives of the complexes. 
\end{proof}


\end{document}